\documentclass[letterpaper,11pt,oneside,reqno]{amsart}

\usepackage[pdftex,backref=page,colorlinks=true,linkcolor=blue,citecolor=red]{hyperref}
\usepackage[alphabetic,nobysame]{amsrefs}

\usepackage{amsmath,amssymb,amsthm,amsfonts,mathtools}
\usepackage{graphicx,color}
\usepackage{upgreek}
\usepackage[mathscr]{euscript}
\usepackage{dutchcal}

\usepackage[all]{xy}
\xyoption{matrix}
\xyoption{arrow}
\usepackage{young}
\usepackage{hhline}
\usepackage{multirow}
\usepackage{shuffle}

\allowdisplaybreaks
\numberwithin{equation}{section}

\usepackage{tikz}
\usetikzlibrary{shapes,arrows,positioning,decorations.markings}

\usepackage{array}
\usepackage{adjustbox}
\usepackage{cleveref}
\usepackage{enumerate}

\usepackage[DIV=12]{typearea}
\usepackage{caption}
\usepackage[all]{xy}
\newcommand{\ssp}{\hspace{1pt}}
\renewcommand{\mid}{\,|\,}
\renewcommand{\Re}{\operatorname{Re}}

\newcommand{\bigzero}{\mbox{\normalfont\Large\bfseries 0}}
\newcommand{\rvline}{\hspace*{-\arraycolsep}\vline\hspace*{-\arraycolsep}}
\newcommand{\muI}{\mu_{\hspace{0.5pt}\mathrm{I}}}

\newenvironment{Blockmatrix}{
    \newcommand{\block}[6][draw]{
        \path[##1] (##3-1, -##2+1) rectangle (##5, -##4) node[pos=0.5] {$##6$};
    }
    \tikzpicture[
        x=4pt+width("C"), y=4pt+height("C"),
        baseline={([yshift=-1ex]current bounding box.center)}
    ]
}{
    \endtikzpicture
}

\newtheorem{proposition}{Proposition}[section]
\newtheorem{lemma}[proposition]{Lemma}
\newtheorem{corollary}[proposition]{Corollary}
\newtheorem{theorem}[proposition]{Theorem}
\newtheorem{conjecture}[proposition]{Conjecture}
\theoremstyle{definition}
\newtheorem{definition}[proposition]{Definition}
\newtheorem{remark}[proposition]{Remark}
\newtheorem{example}[proposition]{Example}
\newtheorem{problem}[proposition]{Problem}

\makeatletter
\renewcommand\part{%
  \clearpage
  \@startsection{part}{0}{\z@}%
  {3\linespacing\@plus\linespacing}%
  {2.5\linespacing}%
  {\centering\large\bfseries\scshape}}
\makeatother

\makeatletter
\def\tocsection#1#2#3{%
  \indentlabel{\@ifnotempty{#2}{\makebox[2em][l]{\ignorespaces#1 #2.\hfil}}}\ignorespaces#3}
\makeatother

\begin{document}
\title{Random Fibonacci Words via Clone Schur Functions}

\author{Leonid Petrov and Jeanne Scott}

\date{}

\begin{abstract}
	We investigate positivity and probabilistic properties
	arising from the Young--Fibo\-nacci lattice $\mathbb{YF}$,
	a 1-differential poset on words composed of 1's and 2's
	(Fibonacci words) and graded by the sum of the digits.
	Building on Okada's theory of clone Schur functions
	\cite{okada1994algebras}, we introduce clone coherent
	measures on $\mathbb{YF}$ which give rise to random
	Fibonacci words of increasing length.  Unlike coherent
	systems associated to classical Schur functions on the
	Young lattice of integer partitions, clone coherent
	measures are generally not extremal on~$\mathbb{YF}$.

	Our first main result is a complete characterization of
	Fibonacci positive specializations ---
	parameter sequences which yield positive clone Schur functions on $\mathbb{YF}$.

	Second, we establish a broad array of correspondences that
	connect Fibonacci positivity with: (i)~the total
	positivity of tridiagonal matrices; (ii)~Stieltjes moment
	sequences; (iii)~the combinatorics of set partitions; and
	(iv)~families of univariate orthogonal polynomials from the ($q$-)Askey scheme.
	We further link the moment sequences of broad classes of
	orthogonal polynomials to combinatorial structures on
	Fibonacci words, a connection that may be of independent
	interest.

	Our third family of results concerns
	the asymptotic behavior
	of random Fibonacci words derived from various Fibonacci
	positive specializations.
	We analyze several limiting regimes for specific examples,
	revealing stick-breaking-like processes (connected to GEM
	distributions), dependent stick-breaking processes of a
	new type,
	or limits supported on the discrete component of the Martin boundary of the Young--Fibonacci lattice.
	Our stick-breaking-like
	scaling limits significantly extend the
	result of Gnedin--Kerov \cite{gnedin2000plancherel} on
	asymptotics of the Plancherel measure
	on~$\mathbb{YF}$.

	We also establish Cauchy-like identities for clone Schur functions
	whose right-hand side is presented as a quadridiagonal determinant rather than a product,
	as in the case of classical Schur functions.
	We construct and analyze models of random permutations and involutions based on
	Fibonacci positive specializations along with a version of the
	Robinson--Schensted correspondence for~$\mathbb{YF}$.
\end{abstract}

\maketitle

\setcounter{tocdepth}{1}
\tableofcontents
\setcounter{tocdepth}{4}

\section{Introduction}

\subsection{Overview}

Branching graphs --- particularly the Young
lattice $\mathbb{Y}$ of integer partitions --- have long
held a central position at the crossroads of representation
theory, combinatorics, and probability. Indeed, the Young
lattice powers the representation theory of symmetric and
(by Schur-Weyl duality) general linear groups, giving rise to Schur
functions and driving profound connections to random matrix
theory and statistical mechanics.  In this landscape, the
\emph{Plancherel measure} on $\mathbb{Y}$ and its
generalization, the \emph{Schur measures}, have emerged as
foundational objects. These probability measures
(in the terminology of statistical mechanics,
\emph{ensembles of random partitions})
serve as a framework for
exploring phenomena such as limit shapes, random tilings,
and universal distributions governing eigenvalues in random
matrix ensembles.

\smallskip

Despite the prominence of the Young lattice,
kindred combinatorial structures remain relatively
underexplored from a probabilistic perspective.  One notable
example is the \textit{Young--Fibonacci lattice}
\(\mathbb{YF}\) \cite{Fomin1986},
\cite{stanley1988differential}, \cite{okada1994algebras},
\cite{KerovGoodman1997}.  Like the Young lattice,
\(\mathbb{YF}\) possesses a 1-differential poset structure defined, not on integer
partitions, but rather binary words
composed of the symbols \(1\) and \(2\)
(known as \emph{Fibonacci words}).
As a 1-differential poset, \(\mathbb{YF}\) carries
a \emph{Plancherel measure} on Fibonacci words.
The Young--Fibonacci
lattice also mirrors other structures found
in the classical Young lattice, such as versions of the
Robinson-Schensted (RS) correspondence,
multiparameter analogues of Schur functions, and a
representation-theoretic framework introduced by Okada
\cite{okada1994algebras}.  At the same time, \(\mathbb{YF}\)
exhibits novel combinatorial and probabilistic behaviors,
which are the main focus of the present work.

\smallskip

Our starting point is the theory of \textit{biserial clone Schur
functions} \cite{okada1994algebras}, a family of functions
$s_w(\vec{x}\mid \vec{y} \ssp)$
indexed by Fibonacci words \(w\in \mathbb{YF}\)
and involving two sequences of parameters \(\vec{x}= (x_1, x_2,
x_3, \ldots)\) and \(\vec{y}= (y_1, y_2, y_3, \ldots)\).
Clone Schur functions, first introduced by Okada, parallel the
classical Schur functions $s_\lambda(\vec z\ssp)$, $\lambda\in \mathbb{Y}$,
$\vec z=(z_1, z_2, z_3, \ldots)$,
in several
aspects.
Both Schur and
biserial clone Schur functions
branch according to a {\it Pieri rule}
and, more generally,
obey a {\it Littlewood-Richardson rule} which
reflects the
structure of covering
relations in
the $\Bbb{Y}$
and $\Bbb{YF}$ lattices.
As a consequence,
both Schur and biserial clone Schur functions give rise to \textit{harmonic} functions on $\mathbb{Y}$ and $\mathbb{YF}$, defined respectively by
\[
\lambda \mapsto
\frac{s_\lambda(\vec{z} \ssp)}{s_{\scriptscriptstyle \Box}^n(\vec{z} \ssp)},
\qquad
w \mapsto
\frac{s_w(\vec{x} \mid \vec{y} \ssp)}{x_1 \cdots x_n},
\]
where \(n\) is the \textit{rank} of \(\lambda \in \mathbb{Y}\) or \(w \in \mathbb{YF}\).
The Plancherel harmonic function (associated with the
Plancherel measure) for each lattice arises from a special
choice of parameters in Schur and clone Schur functions,
respectively.

\smallskip

Positive specializations of classical Schur
functions (sequences $\vec z$ for which
$s_\lambda(\vec z\ssp)$ is positive for all $\lambda\in \mathbb{Y}$)
are central in the study of the Young lattice.
They are related to total positivity
\cite{AESW51}, \cite{ASW52}, \cite{Edrei1952},
\cite{Edrei53}, characters of the infinite symmetric group
\cite{Thoma1964}, and asymptotic theory of characters of
symmetric groups of increasing order
\cite{VK81AsymptoticTheory}.

\smallskip

One important distinction with the $\mathbb{YF}$-lattice is
that the harmonic functions on $\mathbb{Y}$ associated with
positive specializations of classical Schur
functions are {\it extremal} (this property is also often called
\emph{ergodic}, or \emph{minimal}).
Extremality means that the functions $\lambda \mapsto
s_\lambda(\vec{z}\ssp ) / s_{\scriptscriptstyle
\Box}^n(\vec{z}\ssp)$ cannot be expressed as a nontrivial
convex combination of other nonnegative harmonic functions.
In contrast, this extremality property \emph{does not generally
hold} for positive harmonic functions arising from
specializations of biserial clone Schur
functions.
One of the initial motivations for our work was to
investigate the broad question of
how clone harmonic functions on \(\mathbb{YF}\)
decompose into extremal components. The classification of
extremal harmonic functions on \(\mathbb{YF}\) was
established in \cite{KerovGoodman1997}, and results
on the boundary of $\mathbb{YF}$ were strengthened
in the preprints \cite{BochkovEvtushevsky2020},
\cite{Evtushevsky2020PartII}.

\smallskip

Our first goal
is to investigate conditions
for which a specialization
$(\vec{x}, \vec{y} \ssp)$
is \textit{Fibonacci positive}
--- in the sense
that the biserial clone Schur functions
$s_w(\vec{x}\mid \vec{y} \ssp)$
are strictly postive for
all Fibonacci words $w \in \Bbb{YF}$.
Subsequently, we explore probabilistic and combinatorial properties
of Fibonacci positive specializations and related
ensembles of random Fibonacci words.
The present work extends and completes several well-studied classical topics
associated with the Young lattice $\mathbb{Y}$ and Schur functions,
adapting them to the Young--Fibonacci lattice $\mathbb{YF}$ and clone Schur functions.
Our four main contributions can be summarized as follows.

\smallskip\noindent
\textbf{1.~Characterization of Fibonacci positivity.}
		We establish a complete classification of
		specializations \((\vec{x}, \vec{y} \ssp)\)
		for which the clone Schur functions
		$s_w(\vec{x}\mid \vec{y} \ssp)$ are positive for all
		\(w\in \mathbb{YF}\).
		The concept of Fibonacci positivity strengthens the
		notion of total positivity of tridiagonal matrices whose
		subdiagonal consist entirely
					of 1's.
		We identify two
		classes of Fibonacci positive
		specializations, of \emph{divergent} and
		\emph{convergent} type.

\smallskip\noindent
\textbf{2.~Stieltjes moment sequences and orthogonal polynomials.}
		The connection to tridiagonal matrices
		places the Fibonacci positivity problem within the context of
		the Stieltjes moment
		problem and Jacobi continued fractions associated to totally
		positive tridiagonal matrices.
        We obtain two new combinatorial formulae for the
		moments of
		a Borel measure on $[0,\infty)$
        in terms of the recursion parameters
        of the corresponding family of orthogonal
        polynomials. The first
        is a general formula that involves non-crossing
        set partitions, while the second
        requires a Fibonacci positive specialization of divergent type and is expressed as a sum
        over all set partitions. The first formula also
        exhibits a novel {\it splitting} of an integer composition into a unique pair of Fibonacci words.

		For many examples of Fibonacci positive specializations,
		the Borel measure on $[0,\infty)$
		is (after performing a change of variables)
		the orthogonality measure for a family of orthogonal polynomials from the ($q$-)Askey scheme
		\cite{Koekoek1996}, including
		the Charlier, Type-I Al-Salam--Carlitz,
		Al-Salam--Chihara, and $q$-Charlier polynomials.
		In the latter two examples, the Fibonacci positivity enforces the
		atypical condition $q>1$ on the deformation parameter,
		in contrast to the usual
		assumption $|q|<1$ in the $q$-Askey scheme.

\smallskip\noindent
\textbf{3.~Asymptotics of random Fibonacci words.}
		We investigate the behavior of random Fibonacci words,
		(with respect to various Fibonacci positive
		specializations) in the limit as the word length
		grows. We find examples when growing random words $w\in \mathbb{YF}_n$, $n\to \infty$,
		exhibit one of the following three patterns:
		\begin{enumerate}[$\bullet$]
			\item $w=1^{r_1}21^{r_2}2\ldots $, where $r_i$
				scale proportionally to $n$;
			\item $w=2^{h_1}12^{h_2}1\ldots $, where $h_i$
				scale proportionally to $n$;
			\item $w=1^\infty v$,
				that is, the word has a single growing prefix of 1's,
				followed by a finite (random) Fibonacci word $v\in \mathbb{YF}$.
		\end{enumerate}
		In the first two cases, the joint scaling limit
		of either $(r_1,r_2,\ldots )$ or
		$(h_1,h_2,\ldots )$ displays a
		``stick-breaking''-type behavior.
		More precisely, we extend the result of Gnedin--Kerov
		\cite{gnedin2000plancherel} who showed that,
        in the case of the Plancherel measure on $\mathbb{YF}$,
		the sequence $(h_1,h_2,\ldots )$
		converges to the GEM distribution with parameter $\theta=\frac{1}{2}$.

        The Robinson-Schensted
        correspondence for the $\Bbb{YF}$-lattice
        introduced by \cite{nzeutchap2009young}
        (which we revisit in \Cref{sec:random_Fibonacci_words_to_random_permutations}),
        the approach using the Stieltjes moment sequences mentioned above, and
		the results of \cite{gnedin2000plancherel}
        connect four distinguished probability
        distributions:
        \begin{enumerate}[$\bullet$]
			\item the uniform measures on permutations of
				$\{1,2,\ldots,n\}$,
			\item the Plancherel measures on $\mathbb{YF}_n$ determined by the differential poset structure,
			\item the Poisson distribution on $\mathbb{Z}_{\ge0}$ with
            parameter $\rho=1$, and
			\item the stick-breaking GEM scheme powered by the Beta distributions.
		\end{enumerate}


		We extend and deform all these interrelated
		distributions. In particular, on the
		Beta distributions side, we find a new family of
		\emph{dependent} stick-breaking schemes.

		For random Fibonacci words almost surely behaving as $w=1^\infty v$ (the third pattern mentioned above),
		we determine the probability law
		of the random finite suffix $v \in \Bbb{YF}$ in the limit
		in terms of the parameters
		$(\vec{x}, \vec{y} \ssp)$
		of the Fibonacci positive specialization.
		This distribution on the $v$'s
		gives the desired decomposition
		of the clone harmonic function
		$w\mapsto s_w(\vec{x}\mid \vec{y} \ssp)/(x_1\cdots x_{|w|})$
		into the extremal components.

\smallskip\noindent
\textbf{4.~Clone Cauchy identities and random
		permutations.}
		We establish \emph{clone Cauchy identities}
		which are summation identities involving clone Schur functions,
		in parallel to the celebrated Cauchy identities for Schur functions.
		In the Young--Fibonacci setting, the right hand side of each
		Cauchy identity is expressed by a quadridiagonal determinant
		(and not a product, like for classical Schur functions).
		We employ clone Cauchy identities to
		study models of random permutations coming from
		random Fibonacci words and a Robinson--Schensted
		correspondence for the
		$\mathbb{YF}$-lattice introduced in
		\cite{nzeutchap2009young}. In particular, we compute the moment
		generating function for the number of two-cycles in a
		certain ensemble of random involutions,
		and explore its asymptotic behavior under a specific
		Fibonacci positive specialization.
		Other specializations may lead to interesting models of
		random permutations with pattern avoidance properties.

\medskip\noindent
In the remainder of the Introduction,
we formulate our main results in more detail.
Further discussion of possible extensions and open problems
is postponed to the last \Cref{sec:prospectives}.

\subsection{Clone Schur functions and Fibonacci positivity}

A Fibonacci word $w$ is a binary word composed of the symbols 1 and 2.
Its weight $|w|$ is the sum of the symbols. For example,
$|12112|=7$.
By $\mathbb{YF}_n$ we denote the set of Fibonacci words of weight $n$.
The lattice structure on $\mathbb{YF}$ is defined through
branching (covering) relations $w \nearrow w'$ on
pairs of Fibonacci words, where $|w'| = |w| + 1$. This
relation is recursively defined to hold if and only if
either $w' = 1w$, or $w' = 2v$ with $v \nearrow w$. The base
case is given by $\varnothing \nearrow 1$.
Let $\dim(w)$ denote the number of
saturated chains
$$\varnothing=w_0\nearrow w_1\nearrow \cdots \nearrow w_n=w$$
in the Young--Fibonacci lattice starting at
$\varnothing$ and ending at $w$.
See \Cref{fig:YF_lattice} for an illustration of the
$\mathbb{YF}$ up to level $n=5$.
A function $\varphi$ on $\mathbb{YF}$
is called {\it harmonic} if
$$\varphi(w)=\sum\limits_{w'\colon w \nearrow w'} \varphi(w')$$
for all $w\in \mathbb{YF}$.

Let $\vec x=(x_1, x_2, x_3, \ldots)$ and $\vec y=(y_1, y_2, y_3, \ldots)$
be two sequences of parameters.
Define the $\ell\times \ell$ tridiagonal determinants
by
\begin{equation*}
	\begin{split}
		A_\ell (\vec{x} \mid \vec{y}\ssp) &\coloneqq
		\det
		\underbrace{\begin{pmatrix}
		x_1 & y_1 & 0 & 0 & \cdots\\
		1 & x_2 & y_2 & 0 & \\
		0 & 1 & x_3 & y_3 & \\
		\vdots & & & & \ddots
		\end{pmatrix}}_{\text{$\ell \times \ell $ tridiagonal matrix}} ,
		\\[10pt]
		B_{\ell -1} (\vec{x} +r \mid \vec{y} +r ) &\coloneqq
		\det
		\underbrace{\begin{pmatrix}
		y_{r+1} & x_{r+1}y_{r+2} & 0 & 0 & \cdots\\
		1 & x_{r+3} & y_{r+3} & 0 &\\
		0 & 1 & x_{r+4} & y_{r+4} &\\
		\vdots & & & & \ddots
		\end{pmatrix}}_{\text{$\ell \times \ell$ tridiagonal matrix}} ,
	\end{split}
\end{equation*}
where nonzero elements in all rows in $A_{\ell}$ and all rows in
$B_{\ell-1}$ except for the first one
follow the pattern $(1,x_j,y_j)$.
Here and throughout the paper, $\vec x+r$ and $\vec y+r$ denote the
sequences with indices shifted by
$r \in \Bbb{Z}_{\geq 0}$.

The clone Schur function $s_w(\vec{x}\mid \vec{y} \ssp)$ is
defined by the following recurrence:
\begin{equation*}
		s_w (\vec{x} \mid \vec{y}\ssp)
		\coloneqq
		\begin{cases}
			A_k(\vec{x} \mid \vec{y}\ssp), & \textnormal{if $w=1^{k}$ for some $k\ge 0$},\\
			B_k(\vec{x} +r \mid \vec{y} +r )\cdot s_u (\vec{x} \mid \vec{y}\ssp)
			, & \textnormal{if $w=1^k2u$ for some $k\ge 0$ and $|u|=r$}.
		\end{cases}
\end{equation*}
The function
\begin{equation*}
	\varphi_{\vec x,\vec y} \ssp (w)\coloneqq
	\frac{
	s_w(\vec x\mid \vec y\ssp)}{x_1 x_2 \cdots x_{|w|} }
\end{equation*}
is harmonic on $\mathbb{YF}$.
It is normalized so that $\varphi_{\vec x,\vec y} \ssp (\varnothing)=1$.

Our first main result is a complete characterization of
the \emph{Fibonacci positive}
sequences $(\vec{x}, \vec{y} \ssp)$ for which the clone Schur functions
$s_w(\vec{x}\mid \vec{y} \ssp)$ are positive for all
\(w\in \mathbb{YF}\):
\begin{theorem}[{\Cref{thm:Fibonacci_positivity}}]
	\label{thm:positivity_introduction}
	All Fibonacci positive sequences
	$(\vec x,\vec{y}\ssp)$
	have the form
	\begin{equation*}
		x_k=c_k\ssp (1+t_{k-1}),\qquad y_k=c_k\ssp c_{k+1}\ssp t_k, \qquad k\ge1,
	\end{equation*}
	where $\vec c$ is an
	arbitrary positive sequence,
	and $\vec t=(t_1,t_2,\ldots )$ (with $t_0=0$, for convenience) is
	a positive real sequence of one of the two types:
	\begin{enumerate}[$\bullet$]
		\item (divergent type)
			The infinite series
			\begin{equation}
				\label{eq:divergent_type_A1}
				1+t_1+t_1t_2+t_1t_2t_3+\ldots
			\end{equation}
			diverges, and
			$t_{m+1}\ge 1+t_m$ for all $m\ge 1$;
		\item (convergent type)
			The series \eqref{eq:divergent_type_A1} converges, and
			\begin{equation*}
				 1+t_{m+3}+t_{m+3}t_{m+4}+
				t_{m+3}t_{m+4}t_{m+5}+\ldots
				\ge
				\frac{t_{m+1}}{t_{m+2}(1+t_m-t_{m+1})}
				,
				\qquad
				\textnormal{for all }
				m\ge0.
			\end{equation*}
	\end{enumerate}
	The sequences $\vec c$ and $\vec t$ are determined by $(\vec x,\vec{y}\ssp)$ uniquely.
\end{theorem}

The distinguished Plancherel harmonic function
\begin{equation*}
	\varphi_{_\mathrm{PL}}(w)= \frac{\dim (w)}{n!},\qquad  w\in \mathbb{YF}_n,
\end{equation*}
is obtained from clone Schur functions by setting
$x_k=y_k=k$, $k\ge 1$.
Throughout the paper we
are primarily concerned with
two deformations of
the Plancherel specialization, of convergent and divergent type, respectively:
the \emph{shifted
Plancherel} specialization
$x_k=y_k=k+\sigma-1$, $\sigma\in[1,\infty)$,
and the \emph{Charlier} specialization
$x_k=k+\rho-1$, $y_k=k\rho$, $\rho\in(0,1]$.
In \Cref{sec:examples_of_Fibonacci_positivity}, we
describe other examples of Fibonacci positive specializations, both of
divergent and convergent type.

\subsection{Stieltjes moment sequences and orthogonal polynomials}

As a corollary of
the Fibonacci positivity of a specialization
$(\vec{x}, \vec{y} \ssp)$, we see that the
infinite tridiagonal matrix
with the diagonals
$(1,1,\ldots)$, $(x_1,x_2,\ldots)$, and $(y_1,y_2,\ldots)$
is totally positive, that is, all its minors
which are not identically zero are positive.
It is known from
\cite{flajolet1980combinatorial},
\cite{viennot1983theorie},
\cite{sokal2020euler},
\cite{petreolle2023lattice}
that totally positive tridiagonal matrices
correspond to \emph{Stieltjes moment sequences}
$a_n=\int t^n \upnu (dt)$, $n\ge 0$,
where $\upnu$ is a Borel measure on $[0,\infty)$.
Moreover, the monic polynomials $P_n(t)$, $n\ge0$,
orthogonal with respect to $\upnu$
can be determined directly in terms of the parameters $(\vec{x}, \vec{y} \ssp)$:
\begin{equation*}
	P_{n+1}(t) = (t - x_{n+1})P_n(t) - y_n P_{n-1}(t),
	\quad n \geq 1,
	\qquad
	P_0(t) = 1, \quad P_1(t) = t - x_1.
\end{equation*}
We refer to \Cref{sec:Stieltjes_moment_sequences}
for a detailed discussion of the connection between
total positivity of tridiagonal matrices and Stieltjes moment sequences.

According to the ($q$-)Askey nomenclature
\cite{Koekoek1996},
in \Cref{sec:Fibonacci_and_Stieltjes_examples}
we find several Fibonacci positive specializations
whose orthogonal polynomials
are (up to a change of variables and parameters):
\begin{enumerate}[$\bullet$]
	\item Charlier polynomials;
	\item Type-I Al-Salam--Carlitz polynomials;
	\item Al-Salam--Chihara polynomials;
	\item $q$-Charlier polynomials.
\end{enumerate}
In these cases, we also explicitly determine the
orthogonality measures $\upnu$. For example,
in the Charlier case, the orthogonality measure
is simply the Poisson distribution with the parameter~$\rho$.
For the
Al-Salam--Chihara and $q$-Charlier polynomials,
the Fibonacci positivity enforces the
atypical condition $q>1$, in contrast to the usual
assumption $|q|<1$ in the $q$-Askey scheme.

The shifted Plancherel specialization
$x_k=y_k=k+\sigma-1$
which we consider in \Cref{sec:shifted_Charlier_Stieltjes}
corresponds to the
so-called associated Charlier polynomials
\cite{ismail1988linear},
\cite{ahbli2023new}. The orthogonality measure
$\upnu$ in this case is not explicit, but we find its
moment generating function (\Cref{prop:shifted_Charlier_moments}).
Notably, all orthogonality measures $\upnu$ related to known orthogonal polynomials
turn out to be discrete. However, not all orthogonal polynomial families
with discrete measures correspond to Fibonacci positive specializations
(for example, Meixner polynomials do not).
An open question is whether there exists a family of orthogonal polynomials with a continuous measure that is Fibonacci positive.

We also find combinatorial interpretations of the Stieltjes
moments $a_n$ associated with Fibonacci positive
specializations. There is a rich literature addressing the
combinatorics of moments coming from orthogonal polynomials,
notably \cite{wachs1991pq}, \cite{Zeng1993},
\cite{anshelevich2005linearization},
\cite{KimStantonZeng2006}, \cite{kasraoui2006distribution},
\cite{josuat2011crossings}, \cite{CorteelKimStanton2016}.
Most of the results in Sections
\ref{sec:Stieltjes_moments_general},
\ref{sec:Fibonacci_positive_moments}, and
\ref{sec:shifted_Charlier_Stieltjes} follow from techniques
developed in these references. In the case of the Charlier
specialization, $a_n$ is the Bell polynomial in the
parameter $\rho$ (often called the Touchard polynomial)
\begin{equation*}
	a_n=B_n(\rho) \coloneqq \sum\nolimits_{\pi \in \Pi(n)} \rho^{\# \mathrm{blocks}(\pi)},
\end{equation*}
where the sum is over all set partitions $\pi$ of $\{1, \dots, n\}$. This example is representative of the general behavior of Stieltjes moments arising from Fibonacci positive specializations of divergent type. We prove the following:
\begin{proposition}[\Cref{prop:fp-divergent-moment-formula}]
\label{prop:fp-divergent-moment-formula-introduction}
Let $(\vec{x}, \vec{y} \ssp )$ be a Fibonacci positive specialization of divergent type
\begin{equation}
		x_k=c_k\ssp (k +\epsilon_1 + \cdots + \epsilon_{k-1}),\qquad y_k=c_k\ssp c_{k+1}\ssp (k + \epsilon_1 +
        \cdots + \epsilon_k), \qquad k\ge1,
\end{equation}
where $\vec{c}$ and $\vec{\epsilon}$ are sequences of positive real numbers uniquely determined by \eqref{eq:Fib_via_epsilons} and \Cref{corollary:coefficientwise-total-positivity}. Then the associated $n$-th Stieltjes moment is given by
\begin{equation*}
a_n = \sum_{\pi \ssp \in \ssp
\Pi(n)} \,
\prod_{k \ssp \geq \ssp 1} \,
c_k^{\ssp \ell_{k-1}(\pi)} \ssp (1 + \epsilon_k)^{\ssp g_k(\pi)}.
\end{equation*}
\end{proposition}

As in the Charlier case, the sum is taken over all set partitions $\pi \in \Pi(n)$. The statistics $\ell_k(\pi)$ and $g_k(\pi)$ are defined in \Cref{def:g-ell-statistics} and are standard in the analysis of set partitions. Although the shifted-Charlier specialization is Fibonacci positive of convergent type (and not divergent type), its Stieltjes moments nevertheless obey a similar expansion formula involving all set partitions; see \Cref{prop:shifted_Charlier_moments_combinatorial_interpretation}.

\Cref{prop:tp-moment-formula} presents an expansion formula for the moments $a_n$ of a general Borel measure on $[0, \infty)$ involving {\it non-crossing} set partitions; this result does not require Fibonacci positivity as an assumption. \Cref{prop:tp-moment-formula-restatement} is a compressed version of \Cref{prop:tp-moment-formula} where $a_n$ is expressed instead as a sum of monomials indexed by integer compositions of $n$. This result exhibits an intriguing {\it splitting} which can be applied to an integer composition to obtain a unique pair of Fibonacci words; see \Cref{def:splitting+composition}.

While we have a complete description of Fibonacci positive specializations $(\vec{x}, \vec{y})$ and an understanding of their corresponding moment sequences $a_n$ (at least in the divergent case), characterizing their associated Borel measures $\upnu$ within the space of all Borel measures on $[0,\infty)$ remains an open problem.

\subsection{Asymptotics of random Fibonacci words}

We investigate asymptotic behavior of growing random
Fibonacci words
distributed according to clone coherent
probability measures $M_n$ on~$\mathbb{YF}_n$:
\begin{equation*}
  M_n(w) \coloneqq \dim(w)\,\varphi_{\vec x,\vec y}(w)
  = \dim(w)\,\frac{s_w(\vec x\mid \vec y\ssp)}
  {x_1 \ssp x_2\,\cdots\,x_n},
  \qquad w\in \mathbb{YF}_n.
\end{equation*}
The measures $M_n$ are called \emph{coherent} since they are
compatible for varying $n$; see \eqref{eq:coherence_property_standard_down_transitions}.

In
\Cref{sec:asymptotics_Charlier_specialization,sec:asymptotics_shifted_Plancherel_specialization},
we prove two limit theorems concerning the asymptotic
behavior of random Fibonacci words under the Charlier and the shifted Plancherel
specializations.
For the Charlier specialization
$x_k = k + \rho - 1$, $y_k = k\rho$  (which is of divergent type),
we decompose the random word
as $w=1^{r_1}21^{r_2}2\ldots $.

\begin{theorem}[{\Cref{thm:defomed_Plancherel_scaling}}]
	\label{thm:Charlier_scaling_introduction}
	Fix \(\rho \in (0,1)\). Let \(w\in\mathbb{YF}_n\) be distributed according to the
	Charlier clone coherent measure \(M_n\). Then for each fixed \(k\ge1\), the joint
distribution of runs \((r_1(w),\ldots,r_k(w))\) converges to
\[
  \frac{r_j(w)}{n - \sum_{i=1}^{j-1}r_i(w)}
  \xrightarrow[n\to\infty]{d}\eta_{\rho;j},
  \qquad j=1,\dots,k,
\]
where \(\eta_{\rho;1}, \eta_{\rho;2}, \dots\) are i.i.d.\ copies
of a random variable with the distribution
\begin{equation*}
			\rho \ssp\delta_0(\alpha) + (1 - \rho)\ssp \rho (1 - \alpha)^{\rho - 1} \ssp d\alpha,\qquad
			\alpha\in[0,1].
\end{equation*}
This distribution is a convex combination of the point mass at 0 and the Beta
random variable \(\mathrm{beta}(1, \rho)\), with weights
\(\rho\) and \(1 - \rho\).
\end{theorem}

Equivalently, we have \(\{r_j/n\}_{j\ge1}\to X_j\),
where $X_1=U_1$ and
$X_n=(1-U_1)\cdots(1-U_{n-1})U_n$ for $n\ge2$,
where $U_j=\eta_{\rho;j}$ are i.i.d.
The representation of the
vector $(X_1,X_2,\ldots )$ through the variables $U_j$
is called a \emph{stick-breaking} process.

Note that
if $U_j$ have the distribution \(\mathrm{beta}(1, \theta)\),
then the distribution of the vector $(X_1,X_2,\ldots )$
is called the Griffiths--Engen--McCloskey distribution $\mathrm{GEM}(\theta)$.
We refer to
\cite[Chapter~41]{Johnson1997}
for further discussion and applications of GEM distributions,
in particular, in population genetics.

We see that the runs of $1$'s under the Charlier specialization
scale to the $\mathrm{GEM}(\rho)$ vector with additional zero entries
inserted independently with density $1-\rho$.

\medskip

For the shifted Plancherel specialization
$x_k=y_k=k+\sigma-1$ (which is of convergent type), we decompose the random word
as $w=2^{h_1}12^{h_2}1\ldots $.
Denote $\tilde h_j=2h_j+1$.

\begin{theorem}[\Cref{thm:shifted_Plancherel_scaling}]
	\label{thm:shifted_Plancherel_scaling_introduction}
	Fix \(\sigma\ge1\). Under
	the shifted Plancherel clone coherent measure \(M_n\), we
	have for the joint distribution \((\tilde h_1(w),\ldots,\tilde h_k(w))\) for each fixed \(k\ge1\):
	\begin{equation*}
		\frac{\tilde h_j(w)}{n-\sum_{i=1}^{j-1}\tilde h_i(w)}\xrightarrow[n\to\infty]{d}\xi_{\sigma;j},\qquad j=1,\ldots,k.
	\end{equation*}
	The joint distribution of \((\xi_{\sigma;1},\xi_{\sigma;2},\ldots)\)
	can be described as follows.
	Toss a sequence of independent coins with probabilities of
	success $1,\sigma^{-1},\sigma^{-2},\ldots $.
	Let $N$ be the (random) number of successes until the first
	failure.
	Then, sample $N$ independent
	$\mathrm{beta}(1,\sigma/2)$ random variables.
	Set $\xi_{\sigma;k}$,
	$k=1,\ldots,N $, to be these random variables,
	while $\xi_{\sigma;k}=0$ for $k>N$.
\end{theorem}
When $\sigma>1$, the random variables $\xi_{\sigma;k}$ are not
independent, but $\xi_{\sigma;1},\ldots,\xi_{\sigma;n} $
are conditionally independent given $N=n$.
Almost surely, the sequence $\xi_{\sigma;1},\xi_{\sigma;2},\ldots$
contains only finitely many nonzero terms.

At \(\sigma=1\) (Plancherel measure), we have $N=\infty$ almost surely, so the random variables
$\xi_{1;k}$ are i.i.d.\ $\mathrm{beta}(1,\sigma/2)$.
Thus,
we recover the convergence to GEM\((1/2)\)
obtained in \cite{gnedin2000plancherel}.

\Cref{thm:Charlier_scaling_introduction,thm:shifted_Plancherel_scaling_introduction}
follow from product-like formulas for the joint
distributions of $r_j(w)$ and $h_j(w)$, respectively.
The product-like formulas are valid for arbitrary Fibonacci positive specializations,
but they greatly simplify in the Charlier and shifted Plancherel cases.

\medskip

Consider now generic specializations of convergent type
with an additional property that the infinite product
$\prod_{i=1}^\infty (1+t_i)$ converges to a finite non-zero value.

\begin{theorem}[\Cref{prop:measure_on_type_I_words,prop:full_type_I_support}]
	For a sequence
	$\vec t$ of convergent type
    subject to the assumption above,
    the behavior of a random word
	$w\in \mathbb{YF}_n$, sampled with respect to the corresponding clone coherent measure, in the limit as $n\to\infty$ is
	as follows:
	\begin{enumerate}[$\bullet$]
		\item either
			$w \to 1^\infty$,
		\item
			or
			$w \to 1^\infty 2v$,
			where $v \in\mathbb{YF}$ is a finite random Fibonacci word.
	\end{enumerate}
	That is, the growing word $w$ stabilizes to a random element of the set
	\begin{equation*}
		1^\infty\mathbb{YF}\coloneqq
		\{1^\infty\}\cup \{1^\infty 2u\colon u\in \mathbb{YF}\}.
	\end{equation*}
	Moreover, the
	distribution $\muI$ of the limiting random word $w\in 1^\infty\mathbb{YF}$
	is given by
	\begin{equation*}
		\muI\left( 1^\infty \right)= \prod_{i=1}^\infty(1+t_i)^{-1},
		\qquad
		\muI\left( 1^\infty2u \right)=
		\Bigg(
			\, \prod_{i \ssp = \ssp 1}^{|u|-1}(1+t_i)
		\Bigg)
		\big(|u|+1 \big)\ssp
		M_{|u|}(u) \ssp
		\frac{B_{\infty}(|u|)}{\prod_{i=1}^{\infty}(1+t_i)},
	\end{equation*}
	where $u\in \mathbb{YF}$ is arbitrary, and
	$B_\infty(m)$, $m\ge0$, is an infinite series defined below
	in \eqref{eq:A_infty_B_infty_series_definitions}.
	Moreover, $\muI$ is a probability measure on $1^\infty\mathbb{YF}$.
\end{theorem}

In \Cref{cor:typeI-expansion-convergent},
we obtain the following decomposition of the clone harmonic function
$\varphi_{\vec x, \vec y}$ for specializations of convergent type
satisfying $\prod_{i=1}^\infty (1+t_i)<\infty$:
\begin{equation*}
\varphi_{\vec x, \vec y}
=
\muI(1^\infty) \ssp \Phi_{1^\infty}
\, + \ssp
\sum_{u\in \mathbb{YF}}\muI(1^\infty 2u)\ssp \Phi_{1^\infty 2 u},\quad
\Phi_{1^\infty2u}(w) \coloneqq
\begin{cases}
	\displaystyle \frac{\dim(w, 1^k 2u)}{\dim(2u)}, & \text{if $w \unlhd 1^k u$, $k \geq 0$}, \\[5pt]
	0, & \text{otherwise}.
\end{cases}
\end{equation*}
Here,
$\unlhd$ denotes the partial
order on $\mathbb{YF}$ (induced
from the branching relation).
The functions
$\Phi_{1^\infty}$ and
$\Phi_{1^\infty 2u}$
for $v \in \Bbb{YF}$
are called \emph{Type-I harmonic functions}, and they are extremal.

\subsection{Clone Cauchy identities and random permutations}

In \Cref{sec:clone-cauchy-identities},
we
derive \emph{clone Cauchy identities}
generalizing the classical Cauchy-type summation formulas
for the usual Schur functions. Two identities are presented
in \Cref{prop:first-clone-cauchy,prop:second-clone-cauchy},
with the second being
\begin{equation}
	\label{eq:quadridiagonal_determinant_intro}
	\sum_{|w| = n}
	s_w (\vec{p} \mid \vec{q}\ssp)
	\ssp s_w (\vec{x} \mid \vec{y}\ssp)
	=
	\det \underbrace{\begin{pmatrix}
	\mathrm{A}_1 & \mathrm{B}_1 & \mathrm{C}_1 & 0 & \cdots \\
	1 & \mathrm{A}_2 & \mathrm{B}_2 & \mathrm{C}_2 & \\
	0 & 1 & \mathrm{A}_3 & \mathrm{B}_3 & \\
	0 & 0 & 1 & \mathrm{A}_4 & \\
	\vdots & & & & \ddots
	\end{pmatrix}}_{n \times n \ \mathrm{quadridiagonal \, matrix}},
\end{equation}
where
$\mathrm{A}_k = p_k x_k$, \,
$\mathrm{B}_k = q_k(x_k x_{k+1} - y_k) + y_k(p_k p_{k+1} - q_k)$, \,
$\mathrm{C}_k = p_k x_k q_{k+1} y_{k+1}$.

The identity in \eqref{eq:quadridiagonal_determinant_intro}
can be used to define clone analogues of \emph{Schur
measures},
extending the framework from harmonic functions on $\mathbb{YF}$.
Indeed, when
one of the specializations
in \eqref{eq:quadridiagonal_determinant_intro}
is Plancherel, \(p_k = q_k = k\),
identity \eqref{eq:quadridiagonal_determinant_intro}
reduces to the normalizing identity for the clone harmonic function
$\varphi_{\vec x, \vec y}$.
For the Young lattice, Schur
measures were
introduced in \cite{okounkov2001infinite} and
generalized to Schur processes (measures on sequences of partitions)
in
\cite{okounkov2003correlation}.
They found extensive applications in random matrices,
interacting particle systems,
random discrete structures like tilings,
geometry, and
other areas
\cite{okounkov2006gromov},
\cite{okounkov2006quantum},
\cite{BorFerr2008DF},
\cite{BorodinGorinSPB12}, \cite{BorodinPetrov2013Lect},
\cite{corwin2014brownian}.
We leave clone analogues of Schur measures and processes for future work.

\medskip

In \Cref{sec:random_Fibonacci_words_to_random_permutations},
we introduce ensembles of random permutations and involutions
by utilizing the
Young--Fibonacci RS
correspondence \cite{nzeutchap2009young} and positive
harmonic functions on $\mathbb{YF}$. In full generality, the
distribution of a permutation or involution depends, respectively, on a
triplet $(\uppi, \varphi, \psi)$ or a couple $(\uppi,
\varphi)$ of harmonic functions.
We do not treat the general case in the present work,
but
focus on the clone harmonic / Plancherel
random involutions, that is, corresponding to setting
$\uppi =
\varphi_{\vec x, \vec y}$
and
$\varphi =
\varphi_{_\mathrm{PL}}$, where
$(\vec x, \vec y)$ is a Fibonacci positive specialization.
Using clone Cauchy
identities,
in \Cref{sec:observables_from_Cauchy_identities}
we find the moment generating function for the
number of two-cycles in a random involution
$\boldsymbol\upsigma \in \mathfrak{S}_n$
(\Cref{proposition:expectation-cycles-involution}):
\begin{equation*}
		\operatorname{\mathbb{E}}\bigl[ \tau^{\# \ssp
		\mathrm{two \text{-} cycles}(\boldsymbol\upsigma) } \bigr]
		=
		(x_1 \cdots x_n)^{-1} \ssp
		\det \underbrace{\begin{pmatrix}
		x_1
		&(1 - \tau )  y_1
		& - \tau x_1 y_2
		&0
		&\cdots \\
		1
		&x_2
		&(1  - 2 \tau)y_2
		&- 2 \tau x_2 y_3
		& \\
		0
		&1
		&x_3
		&(1  - 3 \tau )y_3
		& \\
		0
		&0
		&1
		&x_4
		& \\
		\vdots & & & &\ddots
	\end{pmatrix}}_{n \times n \ \mathrm{quadridiagonal \, matrix}} ,
\end{equation*}
where $\tau$ is an auxiliary parameter.

When $(\vec{x}, \vec{y})$ is the shifted Plancherel specialization ($x_k = y_k = k + \sigma - 1$, $\sigma \in [1,\infty)$), the Young--Fibonacci shape
$w\in \mathbb{YF}_n$
of a random involution
$\boldsymbol\upsigma \in \mathfrak{S}_n$
under the RS correspondence
has the same distribution as a random Fibonacci word
considered in \Cref{thm:shifted_Plancherel_scaling_introduction} above.
In this way, we can compare the asymptotic behavior of the total
number of $2$'s in a random Fibonacci word (which is the same
as the number of two-cycles), and the
scaling limit of initial long sequences of $2$'s from
\Cref{thm:shifted_Plancherel_scaling_introduction}.
We establish
a law of large numbers
(\Cref{proposition:limit-two-cycles-shifted-plancherel})
for the total number of $2$'s:
\begin{equation}
	\label{eq:limit_two_cycles_shifted_plancherel_intro}
	\lim_{n\to\infty}
	\frac{\# \mathrm{two \text{-} cycles}(\boldsymbol\upsigma)}{n}
	=
	\frac{1}{\sigma+1}.
\end{equation}
For $\sigma > 1$, this value exceeds the expectation of
the sum of the
scaled quantities $h_j$ in
\Cref{thm:shifted_Plancherel_scaling_introduction}.
This
discrepancy
reveals
that additional digits of $2$ remain hidden
in the growing random Fibonacci word
after long sequences of $1$'s.
This behavior is
unaccounted for
in the scaling limit of
\Cref{thm:shifted_Plancherel_scaling_introduction} but
contributes to the law of large numbers
\eqref{eq:limit_two_cycles_shifted_plancherel_intro}.

\subsection*{Outline of the paper}

The paper is organized into three parts.

In \Cref{part:1} (\Cref{sec:YF_lattice_clone,sec:clone-cauchy-identities,sec:Fibonacci_positivity,sec:Fibonacci_positive_properties_part1,sec:examples_of_Fibonacci_positivity}),
we introduce the Young--Fibonacci lattice, the biserial clone Schur functions, and the associated clone coherent measures.
We prove the clone Cauchy summation identities, and give a complete characterization of Fibonacci positive specializations of both divergent and convergent type.
Representative examples include the shifted Plancherel and Charlier cases that will serve as running examples in later sections.

In \Cref{part:2} (\Cref{sec:Stieltjes_moments_general,sec:Stieltjes_moment_sequences,sec:Fibonacci_positive_moments,sec:Fibonacci_and_Stieltjes_examples,sec:shifted_Charlier_Stieltjes}),
we develop the link between Fibonacci positivity, total positivity of tridiagonal matrices, and Stieltjes moment sequences.
We obtain general combinatorial formulas for moments, identify families of orthogonal polynomials from the ($q$-)Askey scheme
(Charlier, Type-I Al-Salam--Carlitz, Al-Salam--Chihara, $q$-Charlier), and highlight the atypical $q>1$ regime enforced by Fibonacci positivity.
For these and other specializations, we determine the corresponding discrete Borel orthogonality measures and give explicit combinatorial interpretations of the moments.

In \Cref{part:3} (\Cref{sec:initial_part_of_random_Fibonacci_word,sec:asymptotics_Charlier_specialization,sec:asymptotics_shifted_Plancherel_specialization,sec:general_convergent_specializations,sec:random_Fibonacci_words_to_random_permutations,sec:observables_from_Cauchy_identities}),
we study the asymptotic behavior of random Fibonacci words under various Fibonacci positive specializations.
Our results include stick-breaking-type limit laws extending the GEM($\tfrac12$) scaling limit of~\cite{gnedin2000plancherel},
new dependent stick-breaking schemes, and the description of extremal components in terms of Type-I harmonic functions in the Martin boundary.
We further employ clone Cauchy identities to define and analyze models of random permutations and involutions, computing exact generating functions for cycle statistics and establishing laws of large numbers.

The paper concludes in \Cref{sec:prospectives} with a discussion of possible extensions and open problems.

\subsection*{Acknowledgements}

There a number of people we would like to thank for
discussions and input during the course of our research.

Part of the inspiration for this project stems from Philippe
Biane's lectures on asymptotic representation theory given
at the 2017 IHP trimester program {\it Combinatorics and
interactions}.  Much in this paper is an attempt to emulate
the approach to the Young lattice conveyed in these
lectures.  We thank him for his advice and hospitality while
visiting Laboratoire d'Informatique Gaspard-Monge
(Marne-la-Vall\'{e}e) in the Fall of 2021 and later in the
Winter of 2023, where preliminary results of our work were
discussed.  We also thank and acknowledge Florent Hivert at
Laboratoire Interdisciplinaire des Sciences du Num\'{e}rique
(Gif-sur-Yvette), who is working on an allied project which
focuses on the algebro-combinatorial side of the
$\Bbb{YF}$-lattice.

Institut Henri Poincar\'{e} (Paris) and Institut de Physique
Th\'{e}orique (Saclay) have been important locations where
progress has been made on this project.  JS is especially
grateful to Fran\c{c}ois David for his hospitality during
numerous visits to IPhT since 2017; in particular for the
opportunity to report on aspects of this project at the IPhT
Math-Physics seminar in 2023.  Many thanks to Ariane
Carrance and Olya Mandelshtam for helpful discussions at
IPhT and IHP in 2023.

Part of our research was made during the Spring 2024 program
``Geometry, Statistical Mechanics, and Integrability'' at
the Institute for Pure and Applied Mathematics (Los
Angeles), which is supported by the NSF grant DMS-1925919.
We were fortunate to have a number of fruitful discussions
with fellow participants at IPAM, including Philippe Di
Francesco, Rinat Kedem, Soichi Okada, Sri Tata, and Harriet
Walsh.

We would like to thank Natasha Blitvic, Mourad Ismail, and
Dennis Stanton, for their advice and input on moment
sequences and the combinatorics of set partitions, as well
as Michael Somos and Qiaochu Yuan for their insights offered
on the Math StackExchange.

JS would also like to thank LP for his hospitality while
visiting the University of Virginia (Charlottesville) in the
Fall of 2023.

%

LP was partially supported by the NSF grant DMS-2153869 and
by the Simons Collaboration Grant for Mathematicians 709055.

\part{Young--Fibonacci Lattice and Fibonacci Positivity}
\label{part:1}

In this part we recall, and when necessary introduce, the
principal objects of our study: the Young-Fibonacci lattice
$\mathbb{YF}$, the clone Schur functions, and the clone
coherent measures on Fibonacci words. We then characterize
those specializations of the clone Schur functions that give
rise to positive harmonic functions on $\mathbb{YF}$.
Finally, we present several examples of Fibonacci positive
specializations, including the ones whose asymptotics we
analyze in \Cref{part:3}.

\section{Young--Fibonacci Lattice and Clone Schur Functions}
\label{sec:YF_lattice_clone}

In this preliminary section we review the Young--Fibonacci lattice $\mathbb{YF}$ (also referred to as the Young--Fibonacci branching graph)~\cite{Fomin1986}, \cite{stanley1988differential}, \cite{KerovGoodman1997} and the clone Schur functions introduced in~\cite{okada1994algebras}. The biserial clone Schur functions are harmonic on $\mathbb{YF}$, and we use them to define coherent probability measures on Fibonacci words.

\subsection{Young--Fibonacci lattice and harmonic functions}
\label{sub:YF_lattice}

A \emph{Fibonacci word}
$w=w_1\ldots w_\ell$
is any binary word with letters $w_j\in \left\{ 1,2 \right\}$.
The integer $|w|\coloneqq w_1+\ldots+w_{\ell} =n$ is called the \emph{weight} of the word $w$.
The total number of Fibonacci words of weight $n$ is equal to the $n$-th Fibonacci number,\footnote{With the convention that $F_0=F_1=1$.}
hence the name.
Denote the set of all Fibonacci words of weight $n$ by $\mathbb{YF}_{n}$, where $n\ge0$.

\begin{definition}
	\label{def:YF_lattice}
	The \emph{Young--Fibonacci lattice} $\mathbb{YF}$ is the union of all sets $\mathbb{YF}_{n}$, $n\ge0$.
	In this lattice,
	$w\in \mathbb{YF}_n$ is connected to $w'\in \mathbb{YF}_{n+1}$
	if and only if $w'$ can be obtained from $w$ by one of the following three operations:
	\begin{enumerate}[{\bf{}F1.\/}]
		\item $w'=1w$.
		\item $w'=2^{k+1}v$ if $w=2^k1v$ for some $k\ge0$ and an arbitrary Fibonacci word $v$.
		\item $w'=2^\ell 1 2^{k-\ell}v$ if $w=2^kv$ for some $k\ge1$ and an arbitrary Fibonacci word $v$.
			While \textbf{F1} and \textbf{F2} each generate at most one edge,
			this
			rule generates $k$ edges indexed by $\ell=1,\ldots,k$.
	\end{enumerate}
	We denote this relation by $w\nearrow w'$ (equivalently, $w'\searrow w$).
	An example of the Young--Fibonacci lattice up to level $n=5$ is given in Figure~\ref{fig:YF_lattice}.
\end{definition}

\begin{figure}[htpb]
	\centering
	\begin{equation*}
		\xymatrixrowsep{0.3in}
		\xymatrixcolsep{0.15in}
		\xymatrix{ \varnothing \ar@{-}[d] & & & & & & &  \\
		1 \ar@{-}[d] \ar@{-}[drrrrrrr] & & & & & & & \\
		11  \ar@{-}[d] \ar@{-}[drrrr] & & & & & & & 2  \ar@{-}[d] \ar@{-}[dlll] \\
		111  \ar@{-}[d] \ar@{-}[drr] & & & & 21 \ar@{-}[d] \ar@{-}[dll] \ar@{-}[dr] & & & 12  \ar@{-}[d] \ar@{-}[dll] \\
		1111  \ar@{-}[d] \ar@{-}[dr]  & &211  \ar@{-}[d] \ar@{-}[dl] \ar@{-}[dr] & &121 \ar@{-}[d] \ar@{-}[dl]  &22 \ar@{-}[d] \ar@{-}[dll]|\hole \ar@{-}[dr]
		& &112  \ar@{-}[d] \ar@{-}[dl]  \\
		11111 & 2111 & 1211 & 221 & 1121 & 122 & 212 & 1112 }
	\end{equation*}
	\caption{The Young--Fibonacci lattice up to level $n=5$.}
	\label{fig:YF_lattice}
\end{figure}

\begin{definition}
	\label{def:harmonic_on_YF}
	A function $\varphi$ on $\mathbb{YF}$ is called \emph{harmonic} if it satisfies
	\begin{equation*}
		\varphi(w)=\sum_{w'\colon w' \searrow w} \varphi(w')\qquad
		\textnormal{for all } w\in \mathbb{YF}.
	\end{equation*}
	A harmonic function is called \emph{normalized} if $\varphi(\varnothing)=1$.
\end{definition}

For $w\in \mathbb{YF}$, denote by $\dim(w)$
the number of oriented paths
(also known as {\it saturated chains})
from $\varnothing$ to $w$ in the Young--Fibonacci lattice.
Let $I_2(w)$ be the sequence of all positions of the letter 2 in $w$, reading from left to right.
Then
\begin{equation}
	\label{eq:dimension_formula}
	\dim(w)
	=\prod_{i\in I_2(w)} \mathsf{d}_i(w),\qquad
	\textnormal{where } \mathsf{d}_i(w)=|v|+1 \textnormal{ if } w=u2v \textnormal{ is the splitting of $w$ at position $i$}.
\end{equation}
Equivalently, $\dim(w)$ obeys the following recursion:
\begin{equation}
	\label{eq:dimension_recursion}
	\dim(w)
	\, = \,
	\begin{cases}
		1, & \text{if } w = \varnothing ;\\
		\dim(v), & \text{if } w = 1v \text{ for a Fibonacci word } v ;\\
		(|v| + 1) \dim(v), & \text{if } w = 2v \text{ for a Fibonacci word } v.
	\end{cases}
\end{equation}
For example,
if $w=22121$, then $I_2(w)=( 1,2,4 )$, and
$\dim w=70$.
Since $\mathbb{YF}$ is a $1$-differential poset, we have
\cite[Corollary~3.9]{stanley1988differential},
(see also \cite{fomin1994duality}):
\begin{equation}
	\label{eq:dim_sum_squares}
	\sum_{w\in \mathbb{YF}_n} \dim^2 (w) = n!\ssp .
\end{equation}

With any \emph{nonnegative} normalized harmonic function we can associate a family of probability measures
$M_n$ on $\mathbb{YF}_n$ as follows:
\begin{equation}
	\label{eq:coherent_measure_for_harmonic_function}
	M_n(w)\coloneqq \dim(w) \cdot \varphi(w), \qquad  w\in \mathbb{YF}_n.
\end{equation}
The fact that $\sum_{w\in \mathbb{YF}_n} M_n(w)=1$ follows from the normalization of $\varphi$,
and the harmonicity of $\varphi$
translates into the \emph{coherence property} of the measures $M_n$:
\begin{equation}
	\label{eq:coherence_property_standard_down_transitions}
	M_n(w)=\sum_{w'\colon w'\searrow w} M_{n+1}(w')\ssp \frac{\dim (w)}{\dim (w')}, \qquad w\in \mathbb{YF}_n.
\end{equation}

The set of all nonnegative normalized harmonic functions on
$\mathbb{YF}$ forms a simplex $\Upsilon(\mathbb{YF})$.
The set of extreme points of this
simplex (the ones
not expressible as a nontrivial convex combination of
other points) is denoted by $\Upsilon_{\mathrm{ext}}(\mathbb{YF})$.
In general, $\Upsilon_{\mathrm{ext}}(\mathbb{YF})$ is
a subset of the \emph{Martin boundary}, denoted by $\Upsilon_{\mathrm{Martin}}(\mathbb{YF})$.
The latter consists of harmonic functions which can be obtained by finite rank approximation.
The Martin boundary of the Young--Fibonacci lattice is described in \cite{KerovGoodman1997}.
Recently, it was shown in the preprints \cite{BochkovEvtushevsky2020}, \cite{Evtushevsky2020PartII} that
the Martin boundary coincides with the set of extreme points $\Upsilon_{\mathrm{ext}}(\mathbb{YF})$.

For any coherent family of measures $M_n$ on $\mathbb{YF}_n$, $n=0,1,2,\ldots $,
there exists a unique probability measure
$\mu$
on $\Upsilon_{\mathrm{ext}}(\mathbb{YF})$
such that
\begin{equation}
	\label{eq:harmonic_function_as_integral_Choquet}
	M_n(w)=\int_{\Upsilon_{\mathrm{ext}}(\mathbb{YF})}
	\dim(w)\ssp \varphi_\omega(w)\ssp \mu(d\ssp \omega),\qquad w\in \mathbb{YF}_n.
\end{equation}
Here $\varphi_\omega$ is the extremal harmonic function
corresponding to $\omega\in
\Upsilon_{\mathrm{ext}}(\mathbb{YF})$.

\subsection{Plancherel measure and its scaling limit}
\label{sub:Plancherel_measure_and_scaling}

An important example of a harmonic function on $\mathbb{YF}$ is the \emph{Plancherel} function
defined as
\begin{equation}
	\label{eq:Plancherel_harmonic_function}
	\varphi_{_\mathrm{PL}}(w)\coloneqq \frac{\dim (w)}{n!},\qquad  w\in \mathbb{YF}_n.
\end{equation}
In \cite{gnedin2000plancherel} it is shown that
$\varphi_{_\mathrm{PL}}$ belongs to $\Upsilon_{\mathrm{ext}}(\mathbb{YF})$.
Moreover, for the Plancherel measure $M_n(w)=\dim^2 (w)/n!$ corresponding to $\varphi_{_\mathrm{PL}}$ as in \eqref{eq:coherent_measure_for_harmonic_function},
\cite{gnedin2000plancherel} establishes a $n\to\infty$ scaling limit theorem for the
positions of the $1$'s in the random Fibonacci word $w$ which we now describe.

Represent $w\in \mathbb{YF}$ as a
sequence of contiguous blocks of letters 2 separated by 1's.
For example, $w=122112=(1)(221)(1)(2)$.
Each block except possibly the rightmost one contains exactly one 1, which is its terminating letter.
Denote by $\tilde{h}_1, \tilde{h}_2, \ldots$ the sequence of weights of the blocks,
reading from left to right.
For the example above, $\tilde{h}_1=1, \tilde{h}_2=5, \tilde{h}_3=1, \tilde{h}_4=2$, and $\tilde{h}_j=0$ for $j\ge5$.
We have $\tilde{h}_1+\tilde{h}_2+\ldots =n$.

\begin{definition}
	\label{def:GEM}
	The \emph{GEM}
	(\emph{Griffiths--Engen--McCloskey})
	\emph{distribution}
	with parameter $\theta>0$
	(denoted $\mathrm{GEM}(\theta)$)
	is a
	probability measure on the infinite-dimensional simplex
	\begin{equation}
		\label{eq:Delta_infinite_dimensional_simplex}
		\Delta\coloneqq \Bigl\{ (x_1,x_2,\ldots)\colon x_j\ge0,\quad \sum_{j=1}^{\infty} x_j\le 1 \Bigr\}
	\end{equation}
	obtained
	from the residual allocation model (also called the stick-breaking construction) as follows.
	By definition, a random point $X=(X_1,X_2,\ldots )\in \Delta$ under $\mathrm{GEM}(\theta)$ is distributed as
	\begin{equation*}
		X_1=U_1,\qquad X_n=(1-U_1)(1-U_2)\cdots(1-U_{n-1})U_n,\quad n=2,3,\ldots,
	\end{equation*}
	where $U_1,U_2,\ldots$ are independent $\mathrm{beta(1,\theta)}$
	random variables
	(i.e.,
	with
	density $\theta(1-u)^{\theta-1}$ on the unit segment $[0,1]$).
	We refer to
	\cite[Chapter~41]{Johnson1997}
	for further discussion and applications of GEM distributions.
\end{definition}

Theorem~5.1 in
\cite{gnedin2000plancherel} establishes the convergence in distribution as $n\to\infty$:
\begin{equation}
	\label{eq:GEM_convergence}
	\Bigl( \frac{\tilde{h}_1(w)}{n}, \frac{\tilde{h}_2(w)}{n}, \ldots \Bigr)
	\longrightarrow X=(X_1,X_2,\ldots ),
	\qquad
	X\sim \mathrm{GEM}(1/2),
\end{equation}
where $\tilde{h}_j(w)$ are the block sizes (described above) of the random Fibonacci word $w$
distributed according to the Plancherel measure on $\mathbb{YF}_n$.
One of the aims of the present paper is to extend this scaling limit result to a wider class of harmonic functions on $\mathbb{YF}$ arising from the constructions in \Cref{sub:clone_Schur} below.

\subsection{Harmonic functions from clone Schur functions}
\label{sub:clone_Schur}

A rich family of non-extremal harmonic functions on $\mathbb{YF}$ comes from
clone Schur functions
\cite{okada1994algebras}
which we now describe.
Let $\vec x=(x_1,x_2,\ldots )$ and $\vec y=(y_1,y_2,\ldots)$ be two families of indeterminates.
Define two sequences of tridiagonal determinants as follows:
\begin{equation}
	\label{eq:A_B_dets}
	A_\ell (\vec{x} \mid \vec{y}\ssp) \coloneqq
	\det
	\underbrace{\begin{pmatrix}
	x_1 & y_1 & 0 & \cdots\\
	1 & x_2 & y_2 &\\
	0 & 1 & x_3  & \\
	\vdots & & & \ddots
	\end{pmatrix}}_{\text{$\ell \times \ell $ tridiagonal matrix}} ,
	\qquad
	B_{\ell -1} (\vec{x} \mid \vec{y}\ssp) \coloneqq
	\det
	\underbrace{\begin{pmatrix}
	y_1 & x_1y_2 & 0 & \cdots\\
	1 & x_3 & y_3 &\\
	0 & 1 & x_4 & \\
	\vdots & & & \ddots
	\end{pmatrix}}_{\text{$\ell \times \ell$ tridiagonal matrix}} .
\end{equation}
Here $\ell\ge0$.
For a sequence $\vec u=(u_1,u_2,\ldots)$,
denote its
\emph{shift} by $\vec u+ \ell = (u_{1+\ell} \, ,u_{2+\ell} \, ,\ldots)$, where $\ell\in \mathbb{Z}_{\ge0}$.

\begin{remark}
	\label{rmk:notation_Al_Blm}
	When there is no
	risk of ambiguity, we'll abbreviate
	$A_\ell (\vec{x} \mid \vec{y}\ssp)$ and
	$B_\ell (\vec{x} \mid \vec{y}\ssp)$
	as $A_\ell$ and $B_\ell$, respectively.
	Moreover, we will use the shorthand notation
	$A_{\ell}(m)\coloneqq A_{\ell}(\vec x+m\mid \vec y+m)$
	and
	$B_{\ell-1}(m)\coloneqq B_{\ell-1}(\vec x+m\mid \vec y+m)$
	for the shifted determinants.
\end{remark}

\begin{definition}
	\label{def:clone_Schur}
	For any Fibonacci word $w$,
	define the
	(\emph{biserial})
	\emph{clone Schur function} $s_w(\vec{x} \mid \vec{y}\ssp)$ through the following recurrence:
	\begin{equation}
		\label{eq:clone_Schur_recurrence_def}
		s_w (\vec{x} \mid \vec{y}\ssp)
		\coloneqq
		\begin{cases}
			A_k(\vec{x} \mid \vec{y}\ssp), & \textnormal{if $w=1^{k}$ for some $k\ge 0$},\\
			B_k\bigl( \vec{x}+|u| \mid \vec{y}+|u| \bigr)\cdot s_u (\vec{x} \mid \vec{y}\ssp)
			, & \textnormal{if $w=1^k2u$ for some $k\ge 0$}.
		\end{cases}
	\end{equation}
\end{definition}

Note that these functions are not symmetric in the variables,
and the order in the sequences $(x_1,x_2,\ldots)$ and $(y_1,y_2,\ldots)$ is important.
The clone Schur functions satisfy
a $\mathbb{YF}$-version of the Littlewood--Richardson identity, whose simplest form
is the following clone Pieri rule
established in \cite{okada1994algebras}:
\begin{equation}
\label{def:clone_pieri}
	x_{|w|+1} \cdot s_w (\vec x\mid \vec{y}\ssp)
\, = \, \sum_{w'\colon w'\searrow w}
s_{w'}(\vec x\mid \vec{y}\ssp)
\end{equation}

\subsection{Background on clone Schur functions}

Let us briefly mention the background
(developed in \cite{okada1994algebras})
behind the
clone Schur functions.
The biserial clone Schur functions
$s_w (\vec{x} \mid \vec{y}\ssp)$ arise
as evaluations (depending on $\vec x$ and $\vec y\, $) of Okada's
clone Schur functions $s_w( \mathbf{x} \mid \mathbf{y})$ which are
noncommutative polynomials
in the free algebra generated by two
symbols
$\mathbf{x}$ and $\mathbf{y}$.
Both the
clone and biserial clone Schur functions play a vital role
vis-\`a-vie the representation
theory of the Okada
algebra(s): the multiplicative
structure of
the noncommutative
clone Schur functions models the
{\it induction product} for
irreducible representations of Okada
algebras, while the biserial
clone Schur functions are
matrix entries for the action
of the generators in these
representations. This amplifies
the parallel with usual
Schur functions where the
Littlewood--Richardson rule
for multiplying Schur functions
describes the induction
product of representations of the
symmetric group.

To summarize, the usual Young lattice $\mathbb{Y}$
(or partitions ordered by inclusion)
is simultaneously responsible for the branching
of the representations of the symmetric groups
$S_n$, and for the Pieri rule for Schur functions (the simplest of the
Littlewood--Richardson rules).
Similarly, the Young--Fibonacci lattice $\mathbb{YF}$
is simultaneously the branching
lattice for Okada algebra representations,
and is responsible for the clone Pieri rule
\eqref{def:clone_pieri} for the biserial clone Schur functions.

\subsection{Properties of biserial clone Schur functions}

Let us proceed with a number of straightforward properties of the
biserial clone Schur functions.
For a complex-valued
sequence $\vec \gamma = (\gamma_1, \gamma_2, \gamma_3, \dots)$
one readily sees from \Cref{def:clone_Schur} that
\begin{equation}
	\label{eq:clone_Schur_action}
	s_w\big(\vec \gamma \cdot \vec x \, \ssp
 \big| \, \vec \gamma \cdot (\vec \gamma + 1) \cdot \vec{y} \ssp \big) \, = \, (\gamma_1 \cdots \gamma_{\ssp |w|} \big) \ssp s_w(\vec x\mid \vec{y}\ssp),
\end{equation}
where
$\vec \gamma \cdot \vec x = (\gamma_1 x_1, \, \gamma_2 x_2, \, \gamma_3 x_3, \, \dots )$ and
$\vec \gamma + 1 = (\gamma_2, \gamma_3, \gamma_4, \dots)$. In particular,
the biserial clone Schur functions
with the variables $(\vec x,\vec{y}\ssp)$
scale as follows:
\begin{equation}
	\label{eq:clone_Schur_scaling}
	s_w(\gamma \ssp \vec x\mid \gamma^2 \ssp \vec{y}\ssp)=\gamma^{|w|}\ssp s_w(\vec x\mid \vec{y}\ssp),
\end{equation}
where $\gamma \ssp \vec x$ means that we multiplied all the variables $x_i$ by $\gamma \in \mathbb{C}$,
and similarly for $\gamma^2 \ssp \vec y$.

Assume that $x_i\ne 0$ for all $i$, and define the following normalization:
\begin{equation}
	\label{eq:clone_Schur_normalization}
	\varphi_{\vec x, \vec y} \, (w)\coloneqq \frac{s_w(\vec x\mid \vec{y}\ssp)}{x_1\cdots x_{|w|} },
	\qquad w\in \mathbb{YF}.
\end{equation}
Formula \eqref{def:clone_pieri}
implies that these normalized clone Schur functions define a harmonic function
on $\mathbb{YF}$ (see \Cref{def:harmonic_on_YF}):
\begin{proposition}[\cite{okada1994algebras}]
	\label{prop:YF_clone_Schur_branching}
	Let the variables $\vec x$ and $\vec y$ be such that $x_i\ne 0$ for all $i$.
	Then
	\begin{equation}
		\label{eq:YF_clone_Schur_branching}
		\varphi_{\vec x, \vec y} \,(w) =
		\sum_{w'\colon w'\searrow w}
		\varphi_{\vec x, \vec y} \, (w') \qquad \textnormal{for all $w\in \mathbb{YF}$}.
	\end{equation}
\end{proposition}
We call $\varphi_{\vec x, \vec y}$ the \emph{clone harmonic function},
and the corresponding coherent probability measures \eqref{eq:coherent_measure_for_harmonic_function}
the \emph{clone measures}. At this point, we treat the measures as \emph{formal}
and do not require them to be nonnegative
(just need their individual ``probability'' weights to sum to $1$).
In \Cref{sec:Fibonacci_positivity} below we
characterize
specializations $(\vec x,\vec{y}\ssp)$ for which
the corresponding clone harmonic function is positive on the whole $\mathbb{YF}$.

\begin{example}
	\label{ex:Plancherel_clone}
	For the particular choice $x_k=y_k=k$, $k\ge1$, the clone harmonic function
	$\varphi_{\vec x, \vec y}$ turns into the
	Plancherel harmonic function $\varphi_{\scriptscriptstyle\mathrm{PL}}$ \eqref{eq:Plancherel_harmonic_function}.
	Indeed, this follows from
	\begin{equation}
		\label{eq:Plancherel_clone_determinants}
		A_{\ell}(\vec x\mid \vec{y}\ssp) = 1,
		\qquad
		B_{\ell-1}(\vec x+r\mid \vec y+r) =r+1,
	\end{equation}
	and so with these parameters we have $s_w(\vec x\mid \vec{y}\ssp)=\dim(w)$, see \eqref{eq:dimension_formula}.
	Denote this choice of parameters by
	$\Pi=(\vec x\mid \vec{y}\ssp)$
	and call it the \emph{Plancherel specialization}.
\end{example}

Note that the normalization in \eqref{eq:clone_Schur_normalization}
leads to the fact that multiple pairs of sequences $(\vec x,\vec{y}\ssp)$
correspond to the same clone harmonic function.
See
\Cref{rmk:connection_to_harmonic_functions_and_non_uniqueness_of_c_parameters}
below for more discussion.

\section{Clone Cauchy Identities}
\label{sec:clone-cauchy-identities}

In this section, we establish summation identities for the
clone Schur functions that parallel the classical summation
identities for the usual symmetric functions, including in
particular the celebrated Cauchy identity.

\subsection{Clone complete homogeneous functions and clone Kostka numbers}
\label{sub:clone-Kostka-numbers}

Throughout this subsection, we view $\vec{x} = (x_1, x_2, \ldots)$ and
$\vec{y} = (y_1, y_2, \ldots)$ as two families of indeterminates.

\begin{definition}[\cite{okada1994algebras}]
Given a Fibonacci word $w \in \mathbb{YF}$, the
biserial \emph{clone homogeneous function}
$h_w(\vec{x} \mid \vec{y}\ssp)$
is the monomial defined recursively by
\begin{equation}
	\label{eq:clone-homogeneous-function}
h_w(\vec{x} \mid \vec{y}\ssp)
\coloneqq
\begin{cases}
x_{|v|+1} \, h_v(\vec{x} \mid \vec{y}\ssp),
& \text{if $w = 1v$;} \\
y_{|v|+1} \,  h_v(\vec{x} \mid \vec{y}\ssp),
& \text{if $w = 2v$,}
\end{cases}
\end{equation}
starting with the base case
$h_\varnothing(\vec{x} \mid \vec{y}\ssp) \coloneqq 1$.
\end{definition}

The relationship between clone homogeneous and clone Schur functions
is explained by the following statement involving
a clone version of Kostka numbers:

\begin{proposition}[{\cite[Section~4]{okada1994algebras}}]
	\label{prop:clone-Kostka-numbers}
Given a Fibonacci word $v \in \mathbb{YF}$,
the clone homogeneous function
$h_v(\vec{x} \mid \vec{y} \ssp)$
has an expansion into clone Schur functions given by
\begin{equation}
\label{schur-expansion-h}
h_v(\vec{x} \mid \vec{y} \ssp)
= \sum_{|u| = |v|} K_{u,v} \,
s_u(\vec{x} \mid \vec{y} \ssp),
\end{equation}
where $K_{u,v}$ are nonnegative integers known as the \emph{clone Kostka numbers}.
They can
be calculated using the following four basic recursions:
\[
\begin{array}{c|c}
\displaystyle K_{2u, 2v} = K_{u, v} &
\displaystyle K_{2u, 1v} = \sum\nolimits_{u \nearrow w} K_{w, v} \\[10pt]
\hline
\rule{0pt}{3ex}
\displaystyle K_{1u, 2v} = 0 &
\displaystyle K_{1u, 1v} = K_{u, v}
\end{array}
\]
\smallskip
starting from the initial conditions $K_{\varnothing, \varnothing} = 1$ and $K_{1,1} = 1$.
\end{proposition}

We refer to \cite{okada1994algebras}
for a combinatorial
interpretation of these numbers
in terms of chains in the Young--Fibonacci lattice.

\begin{remark}
The recursions for \( K_{2u,1v} \) and \( K_{1u,1v} \)
imply that \( K_{w,1^n} = \dim(w) \)
for any Fibonacci word \( w \in \mathbb{YF}_n \).
This observation, together with the
expansion given in \eqref{schur-expansion-h},
allows us to get the following identity (familiar from the normalization
\eqref{eq:clone_Schur_normalization}
of the clone Schur functions):
\begin{equation}
\label{eq:familiar-identity}
h_{1^n} (\vec{x} \mid \vec{y} \ssp) :=
x_1 \cdots x_n =
\sum_{|w| = n} \dim(w) \ssp
s_w(\vec{x} \mid \vec{y} \ssp).
\end{equation}
\end{remark}

The next corollary allows us
to conveniently interpret formula
\eqref{eq:familiar-identity}.
\begin{corollary}
For any \( n \geq 0 \), we have
\begin{equation}
\label{eq:Important-Identity}
s_{1^n}(\vec{x} \mid \vec{y} \ssp)
+ \sum_{m = 0}^{n-2} (m+1) \ssp (x_1 \cdots x_m) \ssp
s_{1^{n-m-2}2}(\vec{x} + m \mid \vec{y} + m)
= x_1 \cdots x_n.
\end{equation}
\end{corollary}
\begin{proof}
Using the expansion
\[
x_1 \cdots x_m = \sum_{|w| = m} \dim(w) s_w(\vec{x} \mid \vec{y} \ssp ),
\]
we can rewrite
\begin{equation}
\label{eq:Important-Identity-step1}
\textnormal{LHS}\ssp\eqref{eq:Important-Identity}
=
s_{1^n}(\vec{x} \mid \vec{y} \ssp)
+ \sum_{m = 0}^{n-2} (m+1) \ssp
s_{1^{n-m-2}2}(\vec{x} + m \mid \vec{y} + m)
\sum_{|w| = m} \dim(w) s_w(\vec{x} \mid \vec{y} \ssp).
\end{equation}
By \Cref{def:clone_Schur}, we know that
\[
s_{1^{n-m-2}2w}(\vec{x} \mid \vec{y} \ssp)
= s_{1^{n-m-2}2}(\vec{x} + m \mid \vec{y} + m) s_w(\vec{x} \mid \vec{y} \ssp ),
\]
so substituting into \eqref{eq:Important-Identity-step1} gives
\begin{align*}
& s_{1^n}(\vec{x} \mid \vec{y} \ssp)
+ \sum_{m = 0}^{n-2}
\sum_{|w| = m} (m+1) \dim(w)
s_{1^{n-m-2}2w}(\vec{x} \mid \vec{y} \ssp) \\
&= s_{1^n}(\vec{x} \mid \vec{y} \ssp)
+ \sum_{m = 0}^{n-2}
\sum_{|w| = m} \dim\big(1^{n-m-2}2w\big)
s_{1^{n-m-2}2w}(\vec{x} \mid \vec{y} \ssp) \\
&= \sum_{|w| = n} \dim(w) s_w(\vec{x} \mid \vec{y} \ssp) \\
&= x_1 \cdots x_n,
\end{align*}
as desired.
\end{proof}

\begin{remark}
\label{rem:finite-tversion-Important-Identity}
Let us set $x_k = 1 + t_{k-1}$ and $y_k = t_k$
for all $k \geq 1$, where
$\vec t=(t_1,t_2,t_3,\ldots )$
is a sequence of auxiliary indeterminates
(with the agreement that $t_0 =0$).
Under this parametrization,
\eqref{eq:Important-Identity}
becomes
\begin{equation}
\label{eq:finite-tversion-Important-Identity}
1 \, + \,
\sum_{m \ssp = \ssp 0}^{n-2}
\ssp (m+1) \ssp B_{n-m-2}(m)
\prod_{k \ssp = \ssp 0}^{m-1}
\ssp (1 + t_k)
\ = \
\prod_{k \ssp = \, 0}^{n-1}
\ssp (1 + t_k),
\end{equation}
where the $B_\ell(m)$'s are the determinants defined in \Cref{sub:clone_Schur}
above.

\end{remark}
\begin{remark}
Continuing from the previous \Cref{rem:finite-tversion-Important-Identity},
if we introduce a regulating parameter \( z \)
and take the formal limit as \( n \to \infty \),
we obtain the following identity in the ring
\( \mathbb{C}[t_k \colon k \geq 1][\![z]\!] \)
of formal power series in \( z \) with
coefficients which are polynomials in
the \( t_k \)'s:
\begin{equation}
\label{eq:infinite-tversion-Important-Identity}
1 + \sum_{m \geq 0} (m+1) B_{\infty}(m; z)
\prod_{k = 0}^{m-1} (1 + t_k z^k)
= \prod_{k = 0}^\infty (1 + t_k z^k),
\end{equation}
where
\begin{equation*}
B_{\infty}(m; z) \coloneqq t_{m+1} z^{m+1}
+ \big(t_{m+1} z^{m+1} - t_m z^m - 1\big) t_{m+2} z^{m+2} A_{\infty}(m+3; z),
\end{equation*}
and
\begin{equation*}
A_{\infty}(m; z) \coloneqq 1 + \sum_{r=1}^\infty t_m t_{m+1} \cdots t_{m+r-1} z^{rm + \binom{r}{2}}.
\end{equation*}
\end{remark}

Let us now obtain a recursion for the (inverse) clone Kostka numbers.
Let $\mathrm{\bf K}_n \coloneqq (K_{u,v})_{u,v \in \mathbb{YF}_n}$ be the matrix whose entries are the clone Kostka numbers.
This matrix admits the following block decomposition, which follows from the four recursions in \Cref{prop:clone-Kostka-numbers}:
\begin{equation}
	\label{eq:Kostka-block-recursion}
\mathrm{\bf K}_n  =
\left( \,\, \begin{matrix}
\raisebox{0.3ex}{$\mathrm{\bf K}_{n-2}$}
& \rvline &  \raisebox{0.3ex}{$\mathrm{\bf D}_{n-1} \mathrm{\bf K}_{n-1}$} \\
\hline \raisebox{-0.3ex}{$\bigzero$} & \rvline &
\raisebox{-0.3ex}{$\mathrm{\bf K}_{n-1}$}
\end{matrix} \,\,  \right),
\end{equation}
\noindent
where $\mathrm{\bf D}_k$ is the $\mathbb{YF}_{k-1} \times
\mathbb{YF}_k$ matrix of the $k$-th \emph{down operator}
for the Young--Fibonacci lattice,
defined by
\begin{equation*}
	\boldsymbol{\mathcal{D}} \ssp\delta_v \, = \,
	\sum_{\scriptstyle u \ssp \nearrow \ssp v} \delta_u.
\end{equation*}
In \eqref{eq:Kostka-block-recursion},
we
order the Fibonacci words in \( \mathbb{YF}_n \)
lexicographically. For example,
the Fibonacci words for $n=5$ are ordered as follows:
\[
\begin{array}{|c|cccccccc|}
\hline
w & \mathrm{221} & \mathrm{212} & \mathrm{2111} & \mathrm{122} & \mathrm{1211} & \mathrm{1121}
& \mathrm{1112} & \mathrm{11111} \\
\hline
\text{position} & \mathrm{1} & \mathrm{2} & \mathrm{3} & \mathrm{4} & \mathrm{5}
& \mathrm{6} & \mathrm{7} & \mathrm{8} \\
\hline
\end{array}
\]
\begin{remark}
	In \Cref{sec:random_Fibonacci_words_to_random_permutations} below, we provide
	the necessary references and
	discussion around the operator~$\boldsymbol{\mathcal{D}}$
	(and its adjoint
	$\boldsymbol{\mathcal{U}}$)
	in connection with the
	Robinson--Schensted-like correspondence for the Young--Fibonacci lattice.
\end{remark}

The matrix $\mathrm{\bf K}_n$ is invertible,
and has an inverse given by the following recursion,
which is straightforward from
\eqref{eq:Kostka-block-recursion}:

\begin{lemma}[Recursion for inverse clone Kostka matrices]
\label{inverse-clone-kostka-block-recursion}
The inverse clone Kostka matrices \( \mathrm{\bf K}_n^{-1} = (K^{u,v}) \)
satisfy a three-step recursion with initial conditions
\[
\mathrm{\bf K}_0^{-1} = \mathrm{\bf K}_1^{-1} = (1) \quad \text{and} \quad
\mathrm{\bf K}_2^{-1} =
\begin{pmatrix}
1 & -1 \\
0 & 1
\end{pmatrix},
\]
and
\begin{equation*}
    \mathrm{\bf K}_n^{-1} \ = \
        \begin{Blockmatrix}
            \block{2}{1}{6}{5}{\mathrm{\bf K}_{n-2}^{-1}}
            \block{2}{6}{3}{8}{\bigzero}
            \block{4}{6}{6}{8}{-\mathrm{\bf K}_{n-3}^{-1}}
						\block{2}{9}{6}{13}{-\mathrm{\bf K}_{n-2}^{-1}}
            \block{7}{1}{14}{5}{\bigzero}
            \block{7}{6}{14}{13}{\mathrm{\bf K}_{n-1}^{-1}}
					\end{Blockmatrix}
\end{equation*}
\end{lemma}

\subsection{Clone Cauchy identities}
\label{sub:clone-cauchy}

We establish two fundamental summation identities for
biserial clone symmetric functions. The second provides a
clone analogue of the classical Cauchy identity for ordinary
Schur symmetric functions. These results are obtained as
biserial specializations of formulas arising in the broader,
noncommutative theory of clone symmetric functions. Although
the treatment in this subsection is self-contained, certain
features of the noncommutative theory are lost under
specialization. An investigation of the noncommutative
aspects will appear in future work.

We need
four families of indeterminates:
$\vec x=(x_1,x_2,\ldots )$ and $\vec y=(y_1,y_2,\ldots)$ together with
$\vec p=(p_1,p_2,\ldots )$ and $\vec q=(q_1,q_2,\ldots)$.
Recall the clone homogeneous functions
\eqref{eq:clone-homogeneous-function}
and the clone Schur functions
from \Cref{def:clone_Schur}.

\begin{proposition}[First clone Cauchy identity]
\label{prop:first-clone-cauchy}
We have
\begin{equation}
\label{eq:first-clone-cauchy}
H_n\big( \vec{x}, \ssp \vec{y} \, ; \vec{p}, \ssp \vec{q} \ssp \big)
\coloneqq \sum_{|w| \, = \, n} h_w (\vec{p} \mid \vec{q}\ssp)
\ssp s_w (\vec{x} \mid \vec{y}\ssp)
= \det
\underbrace{\begin{pmatrix}
\mathrm{A}'_1 & \mathrm{B}'_1 & - \mathrm{C}'_1 & 0 & \cdots \\
1 & \mathrm{A}'_2 & \mathrm{B}'_2 & - \mathrm{C}'_2 & \\
0 & 1 & \mathrm{A}'_3 & \mathrm{B}'_3 & \\
0 & 0 & 1 & \mathrm{A}'_4 & \\
\vdots & & & & \ddots
\end{pmatrix}}_{n \times n \ \mathrm{quadridiagonal \, matrix}},
\end{equation}
where $\mathrm{A}'_k = p_k x_k$,
$\mathrm{B}'_k = y_k (p_k p_{k+1} - q_k)$,
and $\mathrm{C}'_k = q_k x_k y_{k+1} p_{k+2}$ for all $k \geq 1$.
\end{proposition}

\begin{proof}
For simplicity, let us use the shorthand
$s_w$, $h_w$, and
$h_w'$
for $s_w(\vec{x} \mid \vec{y}\ssp) $,
$h_w(\vec{x} \mid \vec{y}\ssp) $,
and $h_w(\vec{p} \mid \vec{q}\ssp) $,
respectively.
Begin by noticing that
\begin{equation*}
		H_0 =  1, \qquad
		H_1 =  p_1 x_1,  \qquad
		H_2 =   (x_1 x_2 - y_1) p_1 p_2 + q_1y_1.
\end{equation*}
The expansion
$
s_v = \sum_{|u| = |v|} K^{u,v} \, h_u
$,
where \( K^{u,v} \) is the \( u \times v \) entry of the inverse clone Kostka matrix
\( \mathbf{K}_n^{-1} \), leads to
\[
H_n = \sum_{u, v \in \mathbb{YF}_n} K^{u,v} \, h_u  h_v'.
\]
The recursive block-matrix decomposition of \( \mathbf{K}_n^{-1} \)
from \Cref{inverse-clone-kostka-block-recursion}, along with the following identities:
\begin{equation}
\label{eq:clone-homogeneous-clone-cauchy-recursion-for-h}
\begin{array}{ll}
h_{1u} &= x_{n+1} \, h_u, \\
h_{2u} &= y_{n+1} \, h_u, \\
h_{11u} &= x_{n+1} x_{n+2} \, h_u, \\
h_{21u} &= x_{n+1} y_{n+2} \, h_u,
\end{array}
\quad
\begin{array}{ll}
h_{1v}' &= p_{n+1} \, h_v', \\
h_{2v}' &= q_{n+1} \, h_v', \\
h_{11v}' &= p_{n+1} p_{n+2} \, h_v', \\
h_{12v}' &= p_{n+3} q_{n+1} \, h_v',
\end{array}
\end{equation}
with $|u| = |v| = n$,
imply that $H_n$ satisfies the following recursion:
\begin{equation}
\label{H-recurrence}
H_n \, = \, x_n p_n \, H_{n-1} \, + \,
y_{n-1} \big(q_{n-1} - x_{n-1}x_n \big)
H_{n-2} \, - \,
p_n q_{n-2} x_{n-2} y_{n-1}
\,
H_{n-3},
\end{equation}
for all $n\ge3$.
Equivalently, the $n$-th kernel $H_n$ can be
expressed as the quadridiagonal determinant
given by \eqref{eq:first-clone-cauchy}.
This completes the proof.
\end{proof}

\begin{proposition}[Second clone Cauchy identity]
\label{prop:second-clone-cauchy}
We have
\begin{equation}
\label{eq:second-clone-cauchy}
S_n
\big( \vec{x}, \ssp \vec{y} \, ;
\vec{p}, \ssp \vec{q} \ssp \big)
\coloneqq \sum_{|w| = n}
s_w (\vec{p} \mid \vec{q} \ssp)
\ssp s_w (\vec{x} \mid \vec{y} \ssp)
=
\det \underbrace{\begin{pmatrix}
\mathrm{A}_1 & \mathrm{B}_1 & \mathrm{C}_1 & 0 & \cdots \\
1 & \mathrm{A}_2 & \mathrm{B}_2 & \mathrm{C}_2 & \\
0 & 1 & \mathrm{A}_3 & \mathrm{B}_3 & \\
0 & 0 & 1 & \mathrm{A}_4 & \\
\vdots & & & & \ddots
\end{pmatrix}}_{n \times n \ \mathrm{quadridiagonal \, matrix}},
\end{equation}
where
\[
\mathrm{A}_k = p_k x_k, \quad
\mathrm{B}_k = q_k(x_k x_{k+1} - y_k) + y_k(p_k p_{k+1} - q_k), \quad
\mathrm{C}_k = p_k x_k q_{k+1} y_{k+1}.
\]
\end{proposition}
Note that $\mathrm{A}_k,\mathrm{B}_k,\mathrm{C}_k$
differ from $\mathrm{A}'_k,\mathrm{B}'_k,\mathrm{C}'_k$
in \Cref{prop:first-clone-cauchy}, hence we use different notation.

\begin{proof}[Proof of \Cref{prop:second-clone-cauchy}]
We apply inverse clone Kostka expansion twice:
\begin{equation*}
\begin{split}
S_n
\big( \vec{x}, \ssp \vec{y} \, ;
\vec{p}, \ssp \vec{q} \ssp \big)
&= \sum_{w \, \in \, \mathbb{YF}_n}
s_w(\vec{x} \mid \vec{y}\ssp)
\ssp s_w(\vec{p} \mid \vec{q}\ssp) \\
&= \sum_{w \, \in \, \mathbb{YF}_n}
\sum_{u \, \in \, \mathbb{YF}_n}
\sum_{v \, \in \, \mathbb{YF}_n}
K^{u,w} K^{v,w} \,
h_u (\vec{x} \mid \vec{y}\ssp)
\ssp h_v (\vec{p} \mid \vec{q}\ssp)
= \mathbf{h}_n \,
\mathrm{\bf K}_n^{-1} \,
\mathrm{\bf K}_n^{-\mathrm{\scriptscriptstyle T}}.
\end{split}
\end{equation*}
Here $\mathrm{\bf K}_n^{\mathrm{\scriptscriptstyle -T}}$ is the inverse transpose of $\mathrm{\bf K}_n$, and
$\mathrm{\bf h}_n$ is the row vector with entries
$h_w (\vec{x} \mid \vec{y}\ssp)
\ssp h_w (\vec{p} \mid \vec{q}\ssp)$
indexed by Fibonacci words $w \in \mathbb{YF}_n$, listed in increasing lexicographic order.
Define $\mathrm{\bf L}_n \coloneqq \mathrm{\bf K}_n^{-1} \, \mathrm{\bf K}_n^{-\mathrm{\scriptscriptstyle T}}$.
Observe that:
\begin{equation*}
\mathrm{\bf L}_0 = \mathrm{\bf L}_1 = (1) \quad \text{and}
\quad \mathrm{\bf L}_2 =
\begin{pmatrix} 2 & -1 \\ -1 & 1
\end{pmatrix}.
\end{equation*}
For $n \geq 3$, we have the following recursive block-matrix decomposition
\begin{equation*}
\mathrm{\bf L}_n \, = \
\begin{Blockmatrix}
\block[]{1}{1}{2}{3}{2\mathrm{\bf L}_{n-2}}
\block{1}{1}{5}{5}{\,}
\block{3}{3}{5}{5}{3\mathrm{\bf L}_{n-3}}
\block{1}{6}{5}{8}{\bigzero}
\block{1}{9}{5}{13}{-\mathrm{\bf L}_{n-2}}
\block{6}{1}{8}{5}{\bigzero}
\block{9}{1}{13}{5}{-\mathrm{\bf L}_{n-2}}
\block{6}{6}{13}{13}{\mathrm{\bf L}_{n-1}}
\end{Blockmatrix}
\end{equation*}
It is important to emphasize that the rows and columns of $\mathrm{\bf L}_n$
correspond to Fibonacci words $w \in \mathbb{YF}_n$ which are ordered lexicographically. For example, the hooked-shaped region
labeled by $2\mathrm{\bf L}_{n-2}$ in the upper left-hand corner
corresponds to pairs of
Fibonacci words $u \times v \in \mathbb{YF}_n \times \mathbb{YF}_n$
of the form $u = 2u'$ and $v = 2v'$,
where $u', v' \in \mathbb{YF}_{n-2}$ and the prefixes
of both $u'$ and $v'$ are not simultaneously equal to $1$.
Using
\eqref{eq:clone-homogeneous-clone-cauchy-recursion-for-h}
together with the block-decomposition of $\mathrm{\bf L}_n$, we get the required three-step recurrence:
\begin{equation*}
S_n \big( \vec{x}, \ssp \vec{y} \, ;
\vec{p}, \ssp \vec{q} \ssp \big)
= \mathrm{A}_n \, S_{n-1}
\big( \vec{x}, \ssp \vec{y} \, ;
\vec{p}, \ssp \vec{q} \ssp \big)
+ \mathrm{B}_{n-1}
S_{n-2}
\big( \vec{x}, \ssp \vec{y} \, ;
\vec{p}, \ssp \vec{q} \ssp \big)
+ \mathrm{C}_{n-2} \,
S_{n-3}
\big( \vec{x}, \ssp \vec{y} \, ;
\vec{p}, \ssp \vec{q} \ssp \big)
\end{equation*}
for $n \geq 3$, where the initial values of
$S_n$ are given by:
\begin{equation*}
\begin{array}{ll}
S_0\big( \vec{x}, \ssp \vec{y} \, ;
\vec{p}, \ssp \vec{q} \ssp \big) &= 1, \\
S_1\big( \vec{x}, \ssp \vec{y} \, ;
\vec{p}, \ssp \vec{q} \ssp \big) &= p_1 x_1, \\
S_2 \big( \vec{x}, \ssp \vec{y} \, ;
\vec{p}, \ssp \vec{q} \ssp \big) &=
(p_1 p_2 - q_1) (x_1 x_2 - y_1)
+ q_1 y_1.
\end{array}
\end{equation*}
The results of the proposition are consequences of this recurrence formula.
\end{proof}

\section{Characterization of Fibonacci Positivity}
\label{sec:Fibonacci_positivity}

In this section, we characterize the specializations
$(\vec{x}, \vec{y}\ssp)$ under which the clone Schur functions
$s_w(\vec{x} \mid \vec{y}\ssp)$ are positive for all
$w \in \mathbb{YF}$ (referred to as \emph{Fibonacci positive specializations}).
This proves \Cref{thm:positivity_introduction} from the Introduction.

\subsection{Reduction to a single sequence parametrization}
\label{sub:reduction_to_one_sequence}

\begin{definition}
	\label{def:fibonacci_nonnegative_biserial}
	A specialization $(\vec x,\vec{y}\ssp)$,
	where $\vec x=(x_1,x_2,\ldots)$, $\vec y=(y_1,y_2,\ldots)$, and $x_i,y_j\in \mathbb{C}$,
	is called \emph{Fibonacci nonnegative} if the clone Schur functions
	$s_w(\vec x\mid \vec{y}\ssp)$ are nonnegative for all Fibonacci words $w\in \mathbb{YF}$.
	If $s_w(\vec x\mid \vec{y}\ssp)>0$ for all $w\in \mathbb{YF}$,
	we say that $(\vec x,\vec{y}\ssp)$ is \emph{Fibonacci positive}.
\end{definition}

One readily sees that Fibonacci positivity
is equivalent to the positivity of
the determinants
$A_{\ell}( \vec x\mid \vec{y}\ssp)$ and $B_{\ell}( \vec x+r\mid \vec y+r)$
for all $\ell,r\in \mathbb{Z}_{\ge0}$.
This, in turn, is equivalent to the \emph{total positivity} of the
following family of
semi-infinite tridiagonal matrices,
where $r\ge 0$:
\begin{equation}
	\label{eq:tridiagonal_matrices_A_Br}
		\mathcal{A} \big(\, \vec{x} \, \big| \, \vec{y} \, \big) \, \coloneqq \
		\begin{pmatrix}
		x_1 & y_1 & 0 & \cdots\\
		1 & x_2 & y_2 &\\
		0 & 1 & x_3  & \\
		\vdots & & & \ddots
	\end{pmatrix},\qquad
		\mathcal{B}_r \big( \, \vec{x} \, \big| \, \vec{y} \, \big) \, \coloneqq \
		\begin{pmatrix}
		y_{r+1} & x_{r+1} y_{r+2} & 0 & \cdots\\
		1 & x_{r+3} & y_{r+3} &\\
		0 & 1 & x_{r+4}  & \\
		\vdots & & & \ddots
	\end{pmatrix}.
\end{equation}
Indeed, it is known
(for example, see \cite{FominZelevinsky1999})
that
the total positivity of a tridiagonal matrix
is equivalent to the positivity of its leading principal minors, namely those
formed by several initial and
consecutive rows and columns.
The list of additional
references on total positivity is vast, and we mention only a few
sources here:
\cite{Edrei53}, \cite{Karlin1968}, \cite{Schoenberg1988}, \cite{FZTP2000}.

\begin{remark}
	We use the convention that a tridiagonal matrix is called
	\emph{totally positive} provided that all its minors are
	strictly positive except those forced to vanish by the
	tridiagonal structure. In the literature, the phrase
	(\emph{strictly}) \emph{totally positive} is sometimes
	used for matrices all of whose minors are positive. See,
	e.g., the first footnote in \cite{fallat2017total} for a
	comparison of terminology in references.  In the present
	paper, however, we need to adapt the terminology to the
	tridiagonal structure of the matrix, and require that all
	minors which are not identically vanishing are strictly
	positive.
\end{remark}

Since total positivity of $\mathcal{A} ( \vec{x} \, \big| \, \vec{y} \ssp)$
\eqref{eq:tridiagonal_matrices_A_Br}
is a necessary condition for a specialization $(\vec x,\vec y\ssp )$ to
be Fibonacci positive, we may restrict our attention to
pairs of sequences $( \vec x , \vec y \ssp)$ for which
 $\mathcal{A} ( \vec{x} \, \big| \, \vec{y} \ssp)$ is totally positive.
Using a general factorizaton ansatz
introduced in \cite{FominZelevinsky1999}
for elements in double Bruhat cells,
we know that
the matrix $\mathcal{A} ( \vec{x} \, \big| \, \vec{y} \ssp)$ is totally positive if
and only if there exist auxiliary real parameters
$c_k,d_k>0$, $k\ge1$,
such that
\begin{equation}
	\label{eq:xk_yk_from_cd}
	x_k = c_k + d_{k-1} \quad \text{and} \quad y_k = c_k \ssp d_k
	\quad \text{for all} \ k \geq 1,
\end{equation}
with the condition that $x_1 = c_1$.
Moreover, $\{c_k\}, \{d_k\}$ are uniquely
determined by $(\vec x,\vec y)$.

Notice that formula \eqref{eq:clone_Schur_action}
implies that
for $(\vec x,\vec y\ssp)$ depending on $c_k,d_k$ as above, we have
\begin{equation}
	\label{eq:AB_from_cd}
	s_w ( \vec x \, \big| \, \vec y \ssp )
	=
	\big( c_1 \cdots c_n \big) \ssp s_w ( \vec u \, \big| \, \vec t\ssp \ssp )
	\quad \textnormal{for all $w\in \mathbb{YF}$},\quad
	|w|=n,
\end{equation}
where $t_k\coloneqq \frac{d_k}{c_{k+1}}$,
$k\ge1$, with the agreement that $t_0=0$, and we denote, for short,
\begin{equation*}
	u_k \coloneqq 1 + t_{k-1}, \qquad k\geq 1.
\end{equation*}
Clearly, the positivity of the left- and right-hand sides of
\eqref{eq:AB_from_cd} for all $w\in \mathbb{YF}$ are equivalent to each other, and so
the problem of characterizing Fibonacci positive
specializations $(\vec x,\vec y\ssp)$ can be reduced to the
problem of identifying necessary and sufficient conditions
under which the sequence
$\vec t$
makes the tridiagonal matrices
$\mathcal{A} ( \vec{u} \, \big| \, \vec{t} \ssp\ssp)$
and $\mathcal{B}_r ( \vec{u} \, \big| \, \vec{t} \ssp\ssp)$
totally positive (for all $r \geq 1$).
In the next \Cref{sub:characterization},
we will classify such sequences $\vec t$, which leads to a complete characterization of
Fibonacci positivity.

\begin{remark}
	The $\vec{t}$-sequences (where $t_k > 0$ for $k \geq 1$ and
	$t_0 = 0$) parametrize a fundamental domain
	\begin{equation*}
	\mathcal{D} = \left\{ (\vec{x}, \vec{y}) \colon x_k = 1 + t_{k-1}, \, y_k = t_k \right\}
	\end{equation*}
	within the overall set of totally positive
	(not necessarily Fibonacci positive)
	tridiagonal
	matrices. Here, the fundamental domain is
	understood
	with respect to the action of the
	multiplicative group $\mathbb{R}_{\scriptscriptstyle >0}^\infty$
	which rescales by the $\vec{c}$-parameters as in
	\eqref{eq:xk_yk_from_cd}--\eqref{eq:AB_from_cd}.
	Our goal in characterizing Fibonacci
	positive specializations is to identify
	the subset $\mathcal{D}^{\mathrm{Fib}}\subset\mathcal{D}$ which
	is also a fundamental domain for
	the set of all Fibonacci positive
	specializations under the
	action of $\mathbb{R}_{\scriptscriptstyle > 0}^\infty$.
\end{remark}

Let us make the passage from $\vec x$ and $\vec y$ to
$\vec t$ and $\vec c$ more explicit:
\begin{proposition}
	\label{prop:c_t_from_x_y}
	We have
	\begin{equation*}
		c_k=\frac{A_{k}(\vec x\mid \vec{y}\ssp)}{A_{k-1}(\vec x\mid \vec{y}\ssp)},\qquad
		t_k=y_k\frac{A_{k-1}(\vec x\mid \vec{y}\ssp)}{A_{k+1}(\vec x\mid \vec{y}\ssp)}
		=x_{k+1}\frac{A_k(\vec x\mid \vec{y}\ssp)}{A_{k+1}(\vec x\mid \vec{y}\ssp)}-1
		,\qquad k\ge1,
	\end{equation*}
	with the agreement that $A_0(\vec x\mid \vec{y}\ssp)=1$ and $t_0=0$.
\end{proposition}
\begin{proof}
	Consider identity \eqref{eq:AB_from_cd}
	for $w=1^n$. We have by \Cref{def:clone_Schur}:
	\begin{equation*}
		A_n(\vec x\mid \vec{y}\ssp) = c_1 \cdots c_n \cdot A_n(\ssp \vec u\mid \vec t \, \, )= c_1 \cdots c_n,
	\end{equation*}
	since the tridiagonal determinant $A_n(\ssp \vec u\mid \vec t \ssp \ssp )$ is equal to $1$. This implies the statement about
	the~$c_k$'s.
	From the fact that $y_k=c_kc_{k+1}t_k$ we get the first formula for $t_k$.
	The second formula follows from the first one by means of the
	three-term recurrence for the tridiagonal determinants~$A_k(\vec x\mid \vec{y}\ssp)$.
	This completes the proof.
\end{proof}

\subsection{Characterization}
\label{sub:characterization}

To address the question of Fibonacci positivity (\Cref{def:fibonacci_nonnegative_biserial}),
it now suffices to consider only positive sequences $\vec t=(t_0,t_1,t_2,\ldots)$ with $t_0=0$.
We will denote $u_k\coloneqq 1+t_{k-1}$, for short.
Let us define for all $m\ge0$:
\begin{equation}
	\label{eq:A_infty_B_infty_series_definitions}
	A_\infty(m)\coloneqq 1\, + \, \sum_{r=1}^{\infty} \, t_m t_{m+1}\cdots t_{m+r-1},\qquad
	B_{\infty}(m)\coloneqq
	t_{m+1}+(t_{m+1}-t_{m}-1)t_{m+2}\ssp A_{\infty}(m+3).
\end{equation}
Note that
$A_\infty(m)$
and $B_\infty(m)$ are the
respective expansions of
$\det \mathcal{A} \big(\ssp \vec{u} + m \, \big| \, \vec{t} +m \ssp \big)$
and $\det \mathcal{B}_m \big(\ssp \vec{u} \, \big| \, \vec{t} \, \, \big)$
in the parameters $t_k$ for $k \geq 1$
when treated as formal variables.
Note that we have $A_\infty(0) = 1$.
\begin{lemma}
	\label{lemma:mathcal_S}
	The sum
	$A_\infty(m)$ \eqref{eq:A_infty_B_infty_series_definitions} is
	convergent (resp., divergent) for some
    $m \ge 1$
	if and only if
	it is convergent (resp., divergent)
	for all $m \geq 1$.
\end{lemma}
\begin{proof}
	We have
	\begin{equation*}
		t_m^{-1} \, \sum_{r=1}^K \, t_m t_{m+1}\cdots t_{m+r-1}
		\ = \ 1 \, - \, t_{m+1}\cdots t_{m+K} \, + \, \sum_{r=1}^K \, t_{m+1}t_{m+2}\cdots t_{m+r}.
	\end{equation*}
	If the product $t_{m+1}\cdots t_{m+K}$ does not go to zero, then
	$A_\infty(m)$ diverges for all
    $m \geq 1$.
	Otherwise, we see that the partial sums of $A_\infty(m)$ and
	$A_\infty(m+1)$ diverge or converge simultaneously.
\end{proof}

\begin{definition}
	\label{def:finite_infinite_type_specializations}
	We introduce two types of positive real
	sequences $\vec t$ based on the convergence of the $A_\infty(m)$'s:
	\begin{enumerate}[\bf1.\/]
	\item
			A sequence $\vec t$
			has \emph{convergent type} if the series
			$A_\infty(m)$ is convergent and
			$B_\infty(m)\ge0$
			for all $m \geq 0$ (with the agreement that $t_0=0$).
		\item
			A sequence $\vec t$ has \emph{divergent type} if
			$t_{m+1}\ge 1+t_m$ for all $m\ge0$.
			\par\noindent
			Note that for a sequence of divergent type,
			we have $t_m\ge m$, and so the series $A_\infty(m)$ diverge for all
			$m$.
	\end{enumerate}
	We will also refer to the corresponding specialization
	$(\vec u,\vec t\ssp)$
	as
	having convergent or divergent type.
\end{definition}

We now present two general criteria
for Fibonacci positivity
based on the convergence or divergence of the series
$A_\infty(m)$ \eqref{eq:A_infty_B_infty_series_definitions}.

\begin{proposition}
	\label{prop:finite_type_positive}
	Assume that the $A_\infty(m)$'s are convergent
	for some (all) $m\ge 1$.
	The specialization determined by $\vec t$
	is Fibonacci positive if and only if $\vec t$ is a sequence
	of convergent type.
\end{proposition}
\begin{proof}
	Throughout
	the proof, we will use the notation of
	\Cref{rmk:notation_Al_Blm}.
	First, let $\vec t$ be a sequence of convergent type.
	One readily sees that
	$A_{\ell}(m)=1+t_{m}A_{\ell-1}(m+1)$ for all $\ell\ge2$, so
	\begin{equation}
		\label{eq:A_ell_m_as_sum}
		A_\ell(m)= 1 \, + \, \sum_{r=1}^\ell
		\, t_m t_{m+1}\cdots t_{m+r-1}
		>0,\qquad m\ge1.
	\end{equation}
	Note that the right-hand side of \eqref{eq:A_ell_m_as_sum} is a partial sum of
	$A_\infty(m)$.

	Next, let us consider the determinants
	$B_\ell(m)$.
	We have
	\begin{equation*}
		\begin{split}
			B_0(m)&=t_{m+1}>0,\\
			B_1(m)&=
			t_{m+1}-
			(1+t_m-t_{m+1})\ssp t_{m+2}
		\end{split}
	\end{equation*}
	for all $m\ge 0$.
	If $1+t_m-t_{m+1}\le 0$, then this is already positive.
	Otherwise,
	we have
	\begin{equation*}
		B_1(m) > B_\infty(m)=
		t_{m+1}-(1+t_m-t_{m+1})\ssp t_{m+2}\ssp A_\infty(m+3)\ge0,
	\end{equation*}
	where the last inequality holds thanks to the convergent type assumption.
	The first strict inequality holds because the partial sums of
	$A_\infty(m+3)$ monotically increase to the infinite sum.
	Thus,
	$B_1(m)>0$.

	For larger determinants with $\ell\ge3$, we have
	\begin{equation}
		\label{eq:B_ell_m_finite_type_proof}
		\begin{split}
			B_{\ell-1}(m)
			&=
			v_m A_{\ell-1}(m+2)-u_m v_{m+1}A_{\ell-2}(m+3)
			\\
			&=
			v_m\left( 1+t_{m+2}A_{\ell-2}(m+3) \right)-
			u_m v_{m+1}A_{\ell-2}(m+3)
			\\
			&=
			t_{m+1}-(1+t_m-t_{m+1})t_{m+2}\ssp A_{\ell-2}(m+3).
		\end{split}
	\end{equation}
	Similarly, if $1+t_m-t_{m+1}\le 0$, then this is already positive.
	Otherwise, we have
	$B_{\ell-1}(m)
	>
	B_\infty(m)\ge 0$,
	since $A_{\ell-2}(m+2) < A_\infty(m+3)$.
	The last nonnegativity again follows from the convergent type assumption.
	This implies that for a sequence $\vec t$ of convergent type,
	all clone Schur functions
	$s_w(\vec u\mid \vec{t}\ssp)$
	are positive.

	\smallskip

	Let us now consider the converse statement and assume that the
	specialization
	determined by $\vec t$
	is Fibonacci positive.
	The positivity of the $t_k$'s implies
	that $A_\ell(m)$ is positive for all $\ell\ge1$, $m \geq 0$,
	see \eqref{eq:A_ell_m_as_sum}.
	Assume that $\vec{t}$ is
	not of convergent type,
	that is,
	$t_{m_0+1}< (1+t_{m_0}-t_{m_0+1})\ssp t_{m_0+2}\ssp A_\infty(m_0+3)$
	for some $m_0 \geq 0$
	(this automatically implies that $1+t_{m_0}-t_{m_0+1}>0$).
	Since
	\begin{equation*}
		A_\infty(m_0+3) = \lim_{\ell \rightarrow \infty} A_{\ell-2}(m_0+2),
	\end{equation*}
	there exists $\ell_0 \gg 1$ (depending on $m_0$) such that
	$t_{m_0+1}< (1+t_{m_0}-t_{m_0+1})\ssp t_{m_0+2}\ssp A_{\ell_0-2}(m_0+2)$.
	By \eqref{eq:B_ell_m_finite_type_proof}, this shows that $B_{\ell_0-1}(m_0)<0$,
	which
	violates the Fibonacci positivity.
\end{proof}

\begin{proposition}
	\label{prop:infinite_type_positive}
	Assume that $A_\infty(m)$ is divergent for
	some (all) $m\ge 1$.
	The specialization determined by $\vec t$
	is Fibonacci positive if and only if $\vec t$ is
	a sequence of divergent type.
\end{proposition}
\begin{proof}
	Here we use the notation of \Cref{rmk:notation_Al_Blm}.
	Assume that $\vec t$ is a sequence of divergent type.
	Similarly to the proof of \Cref{prop:finite_type_positive},
	we see that $A_\ell(m)>0$ for all $\ell\ge1$, $m \geq 0$.
	We have
	\begin{equation*}
		B_0(m)=t_{m+1}>0,\qquad
		B_1(m)=t_{m+1}+(t_{m+1}-t_m-1)t_{m+2},
	\end{equation*}
	and $t_{m+1}-t_m-1\ge 0$ for all $m\ge0$ by the assumption.
	Thus, $B_1(m)>0$ for all $m\ge0$.
	Next, using \eqref{eq:B_ell_m_finite_type_proof}, we similarly see that
	$B_{\ell-1}(m)>0$ for all $\ell\ge3$ and $m\ge0$.

	\smallskip

	Let us now consider the converse statement and assume that the
	specialization
	determined by $\vec t$
	is Fibonacci positive.
	We still have $A_\ell(m)>0$ for all $\ell\ge0$, $m \geq 0$.
	Assume that $\vec{t}$ is
	not of divergent type,
	that is, there exists $m_0 \geq 0$ such that
	$t_{m_0+1}< 1+t_{m_0}$.
	We have
	\begin{equation*}
		B_{\ell-1}(m_0)=
		t_{m_0+1}+(\underbrace{t_{m_0+1}-t_{m_0}-1}_{<0})t_{m_0+2}A_{\ell-2}(m_0+3).
	\end{equation*}
	Since $A_{\ell-2}(m_0+3)$ is positive and unbounded as $\ell \to \infty$,
	we see that $B_{\ell_0-1}(m_0)<0$ for some $\ell_0 \gg 1$ (depending on $m_0$).
	This violates the Fibonacci positivity,
	and completes the proof.
\end{proof}

Sequences of divergent type can be treated formally.
Introduce variables $\epsilon_k$ for $k \geq 1$, and let
$\epsilon^{\ssp \mathbf{i}} \coloneqq \epsilon_1^{i_1} \cdots
\epsilon_k^{i_k}$ be the monomial corresponding to an integer
composition $\mathbf{i} = (i_1, \dots, i_k) \in
\mathbb{Z}_{\geq 0}^k$. Define
\begin{equation}
    \label{eq:Fib_via_epsilons}
		t_k \coloneqq k + \epsilon_1 + \cdots + \epsilon_k.
\end{equation}

\begin{corollary}
\label{corollary:coefficientwise-total-positivity}
Let $\vec{t}$ be given by \eqref{eq:Fib_via_epsilons}. Then,
the semi-infinite, tridiagonal matrices
$\mathcal{A} \big(\, \vec{u} \, \big| \, \vec{t} \ssp\ssp \big)$ and
$\mathcal{B}_r \big( \, \vec{u} \, \big| \, \vec{t} \ssp\ssp \big)$
\eqref{eq:tridiagonal_matrices_A_Br}
for $r \geq 0$
are \emph{coefficientwise} totally positive: Each minor (which does not
identically vanish on the space of all semi-infinite, tridiagonal matrices)
is a polynomial in $\mathbb{Z}[\epsilon_1, \epsilon_2, \dots]$
with nonnegative coefficients, at least one of which is positive.

Consequently,
the clone Schur function $s_w(\vec{u} \, | \, \vec{t}\ssp \ssp)$ expands
as a polynomial in
$\epsilon_1, \epsilon_2, \ldots$ with nonnegative integer coefficients.
\end{corollary}
\begin{proof}
The statements readily follow from the expansions
\eqref{eq:A_ell_m_as_sum}--\eqref{eq:B_ell_m_finite_type_proof}
and the recursion for the clone
Schur functions \eqref{eq:clone_Schur_recurrence_def}.
\end{proof}

\begin{problem}
	\label{problem:combinatorial_interpretation_coefficients_Fibonacci_positive_via_epsilons}
	How to combinatorially interpret the coefficients of the
	monomials in the expansion of a clone Schur function
	$s_w(\vec{u} \, | \, \vec{t} \ssp)$
	in terms of the $\epsilon$-variables?
\end{problem}

The problem of identifying matrices (with polynomial entries)
that are coefficientwise totally positive has been the subject
of recent activity. We refer the reader to
\cite{sokal2014coefficientwise},
\cite{petreolle2023lattice}, \cite{chen2021coefficientwise},
and \cite{deb2023coefficientwise}. The formal specialization
given in \eqref{eq:Fib_via_epsilons} is universal
for sequences of divergent type
in the following sense:

\begin{corollary}
\label{corollary:universal-divergent-type}
Any Fibonacci positive sequence $\vec{t}$ of divergent type
can be obtained by specializing the $\epsilon$-variables in
\eqref{eq:Fib_via_epsilons} to arbitrary positive real numbers.
Moreover, the values of the $\epsilon_j$'s are uniquely determined
by $\vec{t}$.
\end{corollary}

Summarizing \Cref{sub:reduction_to_one_sequence}
and the results of \Cref{prop:finite_type_positive,prop:infinite_type_positive},
we have:

\begin{theorem}[Characterization of Fibonacci positive specializations]
	\label{thm:Fibonacci_positivity}
	All Fibonacci positive specializations
	$(\vec x,\vec{y}\ssp)$
	have the form
	\begin{equation}
		\label{eq:Fibonacci_positive_specialization_in_theorem}
		x_k=c_k\ssp (1+t_{k-1}),\qquad y_k=c_k\ssp c_{k+1}\ssp t_k, \qquad k\ge1
	\end{equation}
	(with $t_0=0$ by agreement),
	where $\vec t$ is a sequence of convergent or divergent type
	as in \Cref{def:finite_infinite_type_specializations},
	and $\vec c$ is an arbitrary positive real sequence.
	The sequences $\vec c$ and $\vec t$ are determined by $(\vec x,\vec{y}\ssp)$ uniquely
	via \Cref{prop:c_t_from_x_y}.
\end{theorem}

\begin{remark}
	\label{rmk:connection_to_harmonic_functions_and_non_uniqueness_of_c_parameters}
	The characterization result (\Cref{thm:Fibonacci_positivity})
	includes two sequences, $\vec t$ and $\vec c$.
	On the other hand, from \eqref{eq:clone_Schur_action}
	we readily see that the corresponding clone harmonic function
	$\varphi_{\vec x,\vec{y}}$ depends only on the sequence $\vec t$.
	Thus, while multiple Fibonacci positive specializations give rise to the same
	clone harmonic function, the type of the specialization (convergent or divergent)
	is uniquely determined by the harmonic function via the sequence $\vec t$.
\end{remark}

\section{Properties of Fibonacci Positive Specializations}
\label{sec:Fibonacci_positive_properties_part1}

\subsection{Necessary conditions}

Here, we list several necessary conditions on the sequence $\vec t$ to be Fibonacci positive, following from \Cref{def:finite_infinite_type_specializations}.

\begin{proposition}
	\label{proposition:valleys}
	For any $m \geq 1$, none of the
    inequalities
    $t_m \ge t_{m+1} \le t_{m+2}$ hold whenever $\vec t$ is
    a Fibonacci positive specialization.
\end{proposition}
\begin{proof}
It is sufficient to examine the determinants
\begin{equation*}
    B_1(m) = t_{m+1} - \left(1 + t_m - t_{m+1}\right)t_{m+2},\qquad m \geq 1,
\end{equation*}
and verify for each $m \geq 0$ that
$t_m \ge t_{m+1} \le t_{m+2}$
is
inconsistent with
the positivity of $B_1(m)$.
\end{proof}

\Cref{proposition:valleys}
implies that a Fibonacci positive
sequence $\vec{t}$ can exhibit one of three behaviors
(bearing in mind our convention $t_0 = 0$):
\begin{enumerate}[$\bullet$]
    \item The sequence $\vec{t}$ strictly increases, i.e., $t_k > t_{k-1}$ for all $k \geq 1$.
    \item There exists an $\ell \geq 1$ such that $t_k > t_{k-1}$ for $1 \leq k \leq \ell$,
    and thereafter $t_k < t_{k-1}$ for $k \geq \ell + 1$.
    \item There exists an $\ell \geq 1$ such that $t_k > t_{k-1}$ for $0 \leq k \leq \ell$,
    the sequence forms a plateau with $t_\ell = t_{\ell+1}$, and subsequently $t_k < t_{k-1}$ for all $k \geq \ell + 2$.
\end{enumerate}
In particular, a Fibonacci positive
sequence $\vec{t}$ must eventually either strictly
increase or strictly decrease.

\begin{lemma}
\label{lemma:monotonic-decreasing-Astuff}
If $\vec{t}$ is a sequence of convergent type, then
\begin{equation*}
\label{eq:monotonic-decreasing-Astuff}
	A_\infty(1)\ge A_\infty(2)> A_\infty(3)> \ldots .
\end{equation*}
Furthermore,
$A_\infty(1) = A_\infty(2)$
if and only if
$A_\infty(2) = (1-t_1)^{-1}$ and
$t_1 \in (0,1)$.
\end{lemma}
\begin{proof}
	For $m \geq 0$, observe that
\[
B_\infty(m)
= A_\infty(m+1)
- A_\infty(m+2)
- t_m \, t_{m+2} \, A_\infty(m+3).
\]
Consequently,
$
B_\infty(m) \geq 0
$
if and only if
\[
A_\infty(m+1) \geq
A_\infty(m+2) +
t_m \, t_{m+2} \, A_\infty(m+3) \geq
A_\infty(m+2),
\]
the latter inequality being strict whenever $m \geq 1$.
Note that $A_\infty(1) = 1 + t_1 A_\infty(2)$,
so $A_\infty(1) = A_\infty(2)$
holds if and only if $t_1 \in (0, 1)$
and $A_\infty(2) = (1-t_1)^{-1}$.
This completes the proof.
\end{proof}

\begin{proposition}
	\label{proposition:cannot_nondecrease}
If $\vec{t}$ is a sequence of convergent type, then it cannot eventually weakly increase.
In other words, there is no $m_0 \geq 1$ such that $t_m \leq t_{m+1}$ for all
$m \geq m_0$.
\end{proposition}
\begin{proof}
	If such an $m_0$ exists, then for all $m \geq m_0$, we have
	$
	A_\infty(m) \leq A_\infty(m+1),
	$
	which contradicts the conclusion of
	\Cref{lemma:monotonic-decreasing-Astuff}. This completes the proof.
\end{proof}

\Cref{proposition:cannot_nondecrease} shows that the sequence $\vec t$ of convergent type
must have a limit. In fact, this limit is always zero:

\begin{proposition}
	\label{proposition:limit_zero_of_convergent_type}
	Let $\vec{t}$ be a sequence
	of convergent type.
	Then $\lim_{m \rightarrow \infty} t_m = 0$.
\end{proposition}
\begin{proof}
	Denote $\gamma \coloneqq \lim_{m \rightarrow \infty} t_m$,
	which
	exists since the sequence eventually weakly decreases.
	Using the fact that $A_\infty(m+3)\le (1-t_{m+3})^{-1}$ for $m\ge m_0$,
	we see that $\gamma$  must be between $0$ and $1$.
	By \Cref{def:finite_infinite_type_specializations}, we can write
	for all $m\ge m_0$:
	\begin{equation*}
		t_{m+1}\ge (1+t_m-t_{m+1})\ssp t_{m+2}\ssp
         A_\infty(m+3)
		\ge (1+t_m-t_{m+1})\ssp t_{m+2}\ssp (1-\gamma)^{-1}.
	\end{equation*}
	Taking the limit as $m\to\infty$, we get the inequality
	$\gamma\ge \gamma(1-\gamma)^{-1}$, which implies that $\gamma=0$.
\end{proof}

\begin{proposition}
\label{proposition:bounded_limit__of_convergent_type}
	Let $\vec{t}$ be a sequence
	of convergent type.
	Then
	$\limsup_{m \rightarrow \infty} m \ssp t_m \in [0,1]$,
	and similarly $\liminf_{m \rightarrow \infty} m \ssp t_m \in [0,1]$.
\end{proposition}
\begin{proof}
Harmonicity
 (\Cref{def:harmonic_on_YF})
 implies that
	\begin{align*}
	1 &=
	\sum_{|w| = m+1} \dim(w) \ssp\varphi_{\vec u, \vec t\ssp}(w) \\
	&=
	\sum_{|w| = m}
	\dim(1w)\ssp \varphi_{\vec u, \vec t\ssp}(1w) +
	\sum_{|w| = m-1}
	\dim(2w)\ssp \varphi_{\vec u, \vec t\ssp}(2w) \\
	&=
	\sum_{|w| = m}
	\dim(w) \varphi_{\vec u, \vec t\ssp}(1w) +
	\sum_{|w| = m-1}
	\frac{m\ssp t_m}{(1+ t_{m-1})(1 + t_m)}\ssp
	\dim(w)\ssp \varphi_{\vec u, \vec t\ssp}(w) \\
	&=
	\sum_{|w| = m}
	\dim(w)\ssp \varphi_{\vec u, \vec t\ssp}(1w) +
	\frac{m\ssp t_m}{(1+ t_{m-1})(1 + t_m)}.
	\end{align*}
	Both
$\sum_{|w|=m} \dim(w) \varphi_{\vec u, \vec t\ssp}(1w)$
and
$m\ssp t_m (1 + t_{m-1})^{-1}(1 + t_m)^{-1}$
are nonzero. Thus, we may conclude that
$m\ssp t_m (1 + t_{m-1})^{-1}(1 + t_m)^{-1} \in (0,1)$
for all $m \geq 1$.
By Proposition~\ref{proposition:limit_zero_of_convergent_type},
$t_m \to 0$ as $m \to \infty$, and consequently,
\begin{equation*}
1 \geq
\limsup_{m \to \infty}
\frac{m\ssp t_m}{(1+ t_{m-1})(1 + t_m)}
=
\limsup_{m \to \infty} m\ssp t_m
\geq 0.
\end{equation*}
Similarly, $\liminf_{m \rightarrow \infty} m\ssp t_m \in [0,1]$.
This completes the proof.
\end{proof}

\begin{remark}[Non-example of convergent type specializations]
	\label{rmk:nonexample_convergent_type}
	Let $0<\upalpha<1$.
	By \Cref{proposition:bounded_limit__of_convergent_type},
	the sequence
	$
	t_k = \varkappa\ssp k^{-\upalpha}
	$,
	$k\ge 1$,
	is never of convergent type for any value $\varkappa > 0$,
	despite the fact that $t_m \to 0$ as $m \to \infty$.
\end{remark}

\begin{remark}
	\label{rmk:might_not_be_convergent}
	The sequence $\{m\ssp t_m\}$ itself might not have a limit
	under the assumptions that
	\begin{enumerate}[$\bullet$]
		\item
			$t_m$ is eventually decreasing
			to zero, and
		\item
			the terms $m\ssp t_m$ are bounded by, say, $1$.
	\end{enumerate}
	Indeed, denote $f_n=n\ssp t_n$, then $t_n\ge t_{n+1}$ implies
	that $f_n-f_{n+1}\ge -1/n$.
	Thus, $f_n$ may make steps in the interval $[0,1]$ of size at most
	$1/n$ in any direction. Since the series $\sum 1/n$ diverges,
	we can organize the steps in such a way that $f_n$
	has at least two subsequential limits.

This observation indicates that we cannot, in general, strengthen
\Cref{proposition:bounded_limit__of_convergent_type}
to the full convergence of the sequence $m\ssp t_m$.
Nevertheless, more subtle consequences of the assumption that~$\vec t$
is a Fibonacci positive specialization of convergent type
might still enforce the convergence of $m\ssp t_m$.
We do not address this question here.
\end{remark}

\subsection{Operations preserving Fibonacci positivity}
\label{sub:operations_preserving_Fibonacci_positivity}

We describe a number of operations
which preserve Fibonacci positivity.
The first is straightforward:

\begin{proposition}[Left shift]
\label{propositon:shifting-fibonacci-positive-specializations}
For any integer $r \geq 0$,
the shifted pair of sequences
$(\vec{x} + r, \, \vec{y} + r)=
(x_{1+r}, x_{2+r}, \ldots, \, y_{1+r}, y_{2+r}, \ldots)$
is a Fibonacci positive specialization whenever $(\vec{x}, \vec{y} \ssp)$
is Fibonacci positive. Moreover, for any $r\ge1$, the shifted pair
of sequences define a specialization of convergent type.
\end{proposition}
\begin{proof}
	It suffices to consider the case $r=1$.
	The new $\vec t$-sequence of the shifted specialization
	(denote it by $\vec T$)
	is expressed in terms of the original $\vec t$-sequence as
	\begin{equation*}
		T_k=\frac{t_{k+1}\tilde A_{k-1}}{\tilde A_{k+1}},\qquad \tilde A_k\coloneqq 1+t_1+t_1t_2+ \ssp \cdots \ssp + t_1\cdots t_{k}.
	\end{equation*}
	Therefore (since $\tilde A_0=1$), we have
	\begin{equation}
		\label{eq:shifted_fibonacci_positive_specialization_proof}
		T_1 \cdots T_k \, = \, t_2\cdots t_{k+1} \, \frac{\tilde A_1}{\tilde A_{k}\tilde A_{k+1}}.
	\end{equation}
	If the original $\vec t$-sequence is of convergent type, then $\tilde A_k \tilde A_{k+1}$ stays bounded as $k\to \infty$, and the series
	with the summands \eqref{eq:shifted_fibonacci_positive_specialization_proof} converges.
	If the original $\vec t$-sequence is of divergent type, then
	$t_k\ge k$, we may use the inequalities
	\begin{equation*}
		\tilde A_{k+1}\ge t_1\cdots t_{k+1},\qquad  \tilde A_k\ge t_1 \cdots t_k\ge k!
	\end{equation*}
	to conclude that the series
	with the summands \eqref{eq:shifted_fibonacci_positive_specialization_proof}
	also converges.
\end{proof}

There are also two right shifts, which
can be applied to divergent and
convergent Fibonacci positive sequences,
both of which introduce a new parameter:

\begin{proposition}[Infiltration right shift for convergent type]
\label{proposition:infiltration_shift_fibonacci_positive_specializations}
Let $\vec t$ be a Fibonacci positive sequence of
convergent type and suppose that
\begin{equation}
\label{eq:infiltration-inequality}
L:= \,
{t_1 A_\infty(2) \over {t_1 A_\infty(2) + 1}}
\ \,  \leq  \ \,
{t_1 A_\infty(2) \over {A_\infty(2) -1}}
\ssp - \ssp 1
\ =:R
\end{equation}
Set
$T_k\coloneqq t_{k-1}$ for $k\ge2$
and let $T_1 \in [L,R]$ be a new parameter,
then $\vec T$ is Fibonacci positive of
convergent type.
\end{proposition}

\begin{proof}
For $m \geq 0$ let
\[
\begin{split}
A'_\infty(m) &:= \, 1 \ssp + \ssp T_m
\ssp + \ssp T_m T_{m+1} \ssp + \ssp T_m T_{m+1} T_{m+2} \ssp + \ssp \cdots, \\
B'_\infty(m) &:= \, T_{m+1} \ssp + \ssp
(T_{m+1} - T_m - 1) T_{m+2} A'_\infty(m+3).
\end{split}
\]
with the understanding that $T_0 = 0$.
Notice that $A'_\infty(m) = A_\infty(m-1)$
for $m \geq 2$ and  $A'_\infty(1) =
1 + T_1 A_\infty(1)$ which are
convergent. Furthermore
$B'_\infty(m) = B_\infty(m-1) \geq 0$ for
all $m \geq 2$. So we only need to check
that
\begin{equation}
\label{eq:pair-B-inequalities}
\begin{split}
B'_\infty(0)
&= T_1 \ssp + \ssp (T_1 - 1)t_1 A_\infty(2)
\, \geq 0, \\
B'_\infty(1)
&= t_1 \ssp + \ssp (t_1 - T_1 - 1)t_2 A_\infty(3) \\[2pt]
&= t_1 \ssp + \ssp (t_1 - T_1 - 1)
(A_\infty(2) - 1) \, \geq 0.
\end{split}
\end{equation}
The pair of
inequalities in \eqref{eq:pair-B-inequalities} is
clearly equivalent to $T_1 \in [L,R]$.
\end{proof}

\begin{remark}
It is always the case that $L  > 0$ and so
\eqref{eq:infiltration-inequality}
forces $ R > 0$ as well.
Since

\[
\begin{split}
B_\infty(0)
&= \, t_1 \ssp + \ssp (t_1 - 1)t_2 A_\infty(3)\\
&= \, 1 + (t_1 -1) A_\infty(2)\\
&= \, (A_\infty(2) -1) R
\end{split}
\]
it follows that $B_\infty(0) > 0$
whenever \eqref{eq:infiltration-inequality}
holds.
\end{remark}

\begin{proposition}[Infiltration right shift for divergent type]
\label{proposition:right_shifting_fibonacci_positive_specializations}
Let $\vec t$ be a Fibonacci positive sequence of divergent type, and $\sigma\ge 0$.
Set
$T_k\coloneqq \sigma+t_{k-1}$, $k\ge2$,
and let $T_1$ be a new parameter.
If $0<T_1\le t_1+\sigma-1$,
	then $\vec T$ is Fibonacci positive of divergent type.
\end{proposition}
\begin{proof}s
    Since $t_k\ge 1+t_{k-1}$, we have $T_k\ge 1+T_{k-1}$ for
	$k\ge 3$. For $k=2$, we need $T_2=\sigma+t_1\ge 1+T_1$, which
	explains the condition $T_1\le t_1+\sigma-1$.
	This completes the proof.
\end{proof}

Fibonacci positivity can also be seen as a ``snake'' that eats
its own tail because $\mathcal{B}_r(\ssp \vec{x} \mid
\vec{y} \ssp)$ provides a new Fibonacci positive
specialization for each $r \geq 0$, whenever the pair
$(\vec{x}, \vec{y} \ssp)$ satisfies Fibonacci positivity.
The following result introduces a form of plethystic
substitution for clone Schur functions
that preserves Fibonacci positivity:

\begin{proposition}[Ouroboric shift]
Let $(\vec{x}, \vec{y} \ssp)$
be a Fibonacci positive specialization. Then, for any $r \geq 0$, the specialization
$(\vec{X}, \vec{Y})$ is also Fibonacci positive, where
\[
X_k \ssp :=
\begin{cases}
y_{r+1}, & \text{if $k = 1$,} \\
x_{k+r+1}, & \text{if $k \geq 2$,}
\end{cases}
\qquad
Y_k \ssp :=
\begin{cases}
x_{r+1} \ssp y_{r+2}, & \text{if $k = 1$,} \\
y_{k+r+1}, & \text{if $k \geq 2$.}
\end{cases}
\]
\end{proposition}

\begin{proof}
We need only to prove that the
semi-infinite matrix
\begin{equation}
\mathcal{B}_0 \big( \, \vec{X} \, \big| \, \vec{Y} \, \big) \, = \,
\begin{pmatrix}
		Y_{1} & X_{1} Y_{2} & 0 & \cdots\\
		1 & X_{3} & Y_{3} &\\
		0 & 1 & X_{4}  & \\
		\vdots & & & \ddots
	\end{pmatrix}
    \, = \,
\begin{pmatrix}
		x_{r+1} \ssp y_{r+2}
        & y_{r+1} \ssp y_{r+3} & 0 & \cdots\\
		1 & x_{r+4} & y_{r+4} &\\
		0 & 1 & x_{r+5}  & \\
		\vdots & & & \ddots
	\end{pmatrix}
\end{equation}
is totally positive for all $r \geq 0$.
Let us show that the corresponding
determinants satisfy
$B_\ell\big( \, \vec{X} \, \big| \, \vec{Y} \, \big) > 0$ for all $\ell \geq 0$.
There is no issue when $\ell = 0$,
since
$B_0\big( \, \vec{X} \, \big| \, \vec{Y} \, \big) = x_{r+1} \ssp y_{r+1} > 0$.
We can assume that $\ell \geq 1$. We have
\begin{equation}
\label{eq:ouroboric-expansion}
\begin{split}
B_\ell (  \vec{X} \, \mid \, \vec{Y} \, )
&=
x_{r+1} \ssp y_{r+2} \ssp
s_{1^{\ell}}
 (  \vec{x} + r + 3\, \mid \, \vec{y} + r + 3  )
-
y_{r+1} \ssp y_{r+3} \ssp
s_{1^{\ell-1}}
 (  \vec{x} + r + 4\, \mid \, \vec{y} + r + 4  ) \\
&=
s_{1^{\ell}}
 (  \vec{x} + r + 3\, \mid \, \vec{y} + r + 3  )
\ssp s_{21}
(  \vec{x} + r \, \mid \, \vec{y} + r \, ) \\
&\hspace{120pt} -
s_{1^{\ell-1}}
(  \vec{x} + r + 4 \, \mid \, \vec{y} + r + 4  ) \ssp
s_{22}
(\vec{x} + r \, \mid \, \vec{y} + r  ).
\end{split}
\end{equation}
By \Cref{propositon:shifting-fibonacci-positive-specializations},
it will be sufficient to restrict our
analysis of formula \eqref{eq:ouroboric-expansion} to the case where
$r = 0$, since
the four clone Schur functions
that occur are each
shifted by $r \geq 0$.
We have
\begin{equation}
\label{clone-identity-ouroboric}
\begin{split}
&s_{1^{\ell}}
(\vec{x} + 3 \mid \vec{y} + 3)
\ssp s_{21}
(\vec{x} \mid \vec{y} \ssp)
-
s_{1^{\ell-1}}
(\vec{x} + 4 \mid \vec{y} + 4) \ssp
s_{22}
(\vec{x} \mid \vec{y} \ssp) \\
&\hspace{140pt}=
s_{1^\ell 21}
(\vec{x} \mid \vec{y} \ssp)
+
s_{1^{\ell-1}}
(\vec{x} + 4 \mid \vec{y} + 4)
\ssp s_{211}
(\vec{x} \mid \vec{y} \ssp).
\end{split}
\end{equation}
for all $\ell \geq 1$.
Formula \eqref{clone-identity-ouroboric} follows as a direct
consequence of the clone Littlewood-Richardson identity
\cite{okada1994algebras}. Alternatively, it can be verified
directly by induction. The base case, $\ell = 1$, reduces to
a restatement of the clone Pieri rule for $s_{21}(\vec{x}
\mid \vec{y} \ssp)$:
\[
s_1(\vec{x} + 3 \mid \vec{y} + 3 )
\ssp s_{21}(\vec{x} \mid \vec{y} \ssp)
-
s_{22}(\vec{x} \mid \vec{y} \ssp)
=
s_{211}(\vec{x} \mid \vec{y} \ssp)
+
\ssp s_{121}(\vec{x} \mid \vec{y} \ssp).
\]
As a result, we conclude that
\eqref{clone-identity-ouroboric} is positive, as its
right-hand side
involves only sums and products of clone Schur functions. These
functions, by definition, are positive since $(\vec{x},
\vec{y} \ssp)$ is a Fibonacci positive specialization.
This completes the proof.
\end{proof}

\begin{proposition}
	\label{proposition:infinite_type_scaling}
	Let $\vec{t}$ be a divergent type sequence.
	Let $\vec \alpha$ be a positive real sequence such that
	\begin{equation}
		\label{eq:sequence_alpha_condition}
		\alpha_k t_{k} - \alpha_{k+1} t_{k-1} \geq \alpha_k \alpha_{k+1}, \qquad k \geq 1.
	\end{equation}
	Then the specialization $(\vec x,\vec{y}\ssp)$ defined
	by
	$x_k = \alpha_k + t_{k-1}$ and $y_k = \alpha_k t_{k}$
	is Fibonacci positive.
\end{proposition}
A particular case is when
$\alpha_k=\rho\in(0,1]$ for all $k$. Then
\eqref{eq:sequence_alpha_condition} clearly holds, and
for a sequence~$\vec t$ of divergent type, the specialization
$x_k=\rho+t_{k-1}$ and $y_k=\rho\ssp t_{k}$ is Fibonacci positive.
\begin{proof}[Proof of \Cref{proposition:infinite_type_scaling}]
	Denote $r_k\coloneqq t_k/\alpha_{k+1}$, $k\ge1$ (with $r_0=0$).
	For all Fibonacci words~$w$,
	we have
	$s_w(\vec x\mid \vec{y}\ssp) = (\alpha_1 \cdots \alpha_{|w|})
	\ssp s_w(\vec{x} \, ' \mid \vec{y} \, ')$,
	where $x_k' = 1 + r_{k-1}$ and $y_k' = r_{k}$.
	Note that \eqref{eq:sequence_alpha_condition}
	implies that $r_{k} \geq 1 + r_{k-1}$ for all $k \geq 1$.
	Thus, the specialization $(\vec x',\vec{y}\ssp')$ is Fibonacci positive
	of divergent type, and so is $(\vec x,\vec{y}\ssp)$.
\end{proof}

\begin{proposition}
Let $\vec{t} = (t_1, t_2, t_3, \dots )$
be a strictly decreasing sequence of
convergent type. Then the sequence
$\gamma \ssp \vec{t} \coloneqq ( \gamma t_1,
\gamma t_2, \gamma t_3, \dots )$
is of convergent type whenever
$0 < \gamma \leq 1$.
\end{proposition}
\begin{proof}
For $m \geq 0$, let
\begin{align*}
A_\infty(m,\gamma) &\coloneqq
1 + \gamma t_m + \gamma^2 t_m t_{m+1}
+ \gamma^3 t_m t_{m+1} t_{m+2} + \ldots, \\
B_\infty(m,\gamma) &\coloneqq
\gamma t_{m+1} -
\left( 1 + \gamma t_m -
\gamma t_{m+1} \right) \gamma t_{m+2}
A_\infty(m+3, \gamma), \\
\phi_0(\gamma) &\coloneqq
t_1 -
\left( 1 -
\gamma t_1 \right) t_2
A_\infty(3, \gamma).
\end{align*}
Clearly, $A_\infty(m, \gamma)$ is convergent
for any $\gamma \geq 0$, so we only have
to address the nonnegativity
of $B_\infty(m, \gamma)$ whenever
$0 < \gamma \leq 1$ and
$m \geq 0$.

Consider two cases, $t_1 \leq 1$ and $t_1 > 1$.
If $t_1 \leq 1$, we have
\begin{equation*}
    \phi_0(\gamma) \geq R(\gamma) \quad \text{for all} \quad \gamma \geq 0,
		\qquad \textnormal{where}\qquad
    R(\gamma) \coloneqq t_1 -
    \frac{(1 - \gamma t_1) t_2}{1 - \gamma t_3}.
\end{equation*}
Furthermore, $R(0) = t_1 - t_2 > 0$, and $R(\gamma)$ only vanishes at
\begin{equation*}
    \gamma_0 =
    \frac{t_1 - t_2}{t_1 ( t_3 - t_1 )} < 0.
\end{equation*}
It follows that $R(\gamma) > 0$ for all $\gamma > 0$, which forces
\begin{equation*}
    \phi_0(\gamma) > 0 \quad \text{and} \quad B_\infty(0,\gamma) > 0
    \quad \text{for all} \quad \gamma > 0.
\end{equation*}

If $t_1 > 1$, then
$\phi_0(\gamma) \geq 0$ whenever $\frac{1}{t_1} \leq \gamma \leq 1$.
For $\gamma$ within the range $0 \leq \gamma < \frac{1}{t_1}$, the inequality
$
    \phi_0(\gamma) \geq R(\gamma)
$
is valid, and we may again conclude that $\phi_0(\gamma) > 0$ whenever $0 \leq \gamma < \frac{1}{t_1}$.

The nonnegativity of $B_\infty(m,\gamma)$ for $m \geq 1$ follows from
\begin{align*}
    t_{m+1} &\geq \left( 1 + t_m - t_{m+1} \right)
    t_{m+2} A_\infty(m+3),
\end{align*}
together with
$
    A_\infty(m) > A_\infty(m, \gamma)
		$
and
$t_m - t_{m+1} \geq \gamma t_m - \gamma t_{m+1}$ whenever $0 \leq \gamma \leq 1$.
This completes the proof.
\end{proof}

\section{Examples of Fibonacci Positivity}
\label{sec:examples_of_Fibonacci_positivity}

Here
we present several Fibonacci positive
specializations, starting with divergent type.

\subsection{Divergent type examples}
\label{sub:divergent_type_examples_new}

We use the standard $q$-integer notation
$[k]_q=(1-q^k)/(1-q)$.

\begin{definition}[Examples of divergent type]
	\label{def:positive_specializations_class_II}
	We introduce a list of Fibonacci positive specializations
	$(\vec x,\vec{y}\ssp)$
	related to sequences of divergent type.
	The naming of some of the specializations is motivated by
	connections with Stieltjes moment sequences
	and the Askey scheme developed in \Cref{part:2}.
	We consider the following specializations:
	\begin{enumerate}[$\bullet$]
		\item \textbf{Charlier (deformed Plancherel).}
			$x_k = k + \rho -1$, $y_k = k \rho$ for $\rho \in (0 , 1]$.
			It reduces to the Plancherel specialization for $\rho=1$.
		\item \textbf{Type-I Al-Salam--Carlitz.}
			$x_k = \rho q^{k-1}+[k-1]_q$, $y_k = \rho q^{k-1}\ssp [k]_q$ for $\rho \in (0,1]$
			and $q \in (0 , 1)$.
		\item \textbf{Al-Salam--Chihara.}
			$x_k = \rho + [k-1]_q$, $y_k = \rho [k]_q$ for $\rho \in (0,1]$
			and $q \in [1 , \infty)$.
		\item \textbf{$q$-Charlier.}
			$x_k = \rho q^{2k-2} + [k-1]_q\big(1 + \rho (q-1) q^{k-2}\big)$,
			$y_k = \rho q^{2k-2} [k]_q\big(1 + \rho (q-1) q^{k-1}\big)$
			for $\rho,q \in (0,1]$.
		\item \textbf{Cigler--Zeng.}
			$x_k = q^{k-1}$, $y_k = q^k -1$ for
			$q\in [q_0 , \infty)$ where $q_0 \approx 1.4656$ is the unique real root of $z^3 = z^2+1$.
			This specialization can also be deformed to
			$x_k = q^{k-1}+\rho-1$, $y_k = \rho(q^k -1)$ with $\rho\in(0,1]$.
			The name comes from a family of orthogonal
			polynomials introduced in
			\cite{cigler2011curious}.
	\end{enumerate}
\end{definition}

\begin{proposition}
	\label{prop:class_II_positive_specializations}
	The specializations in \Cref{def:positive_specializations_class_II}
	are Fibonacci positive and are of divergent type.
\end{proposition}
\begin{proof}
    Let us examine the Fibonacci positivity in each case using the classification (\Cref{thm:Fibonacci_positivity}) and the operations preserving Fibonacci positivity (\Cref{thm:Fibonacci_positivity} or \Cref{proposition:infinite_type_scaling}).

	The Charlier case follows from
	\Cref{thm:Fibonacci_positivity}
	 with $c_k=\rho$ and $t_k=k \rho^{-1}$ for all $k$.

	For the Type-I Al-Salam--Carlitz case,
	take $t_k=(\rho q^k)^{-1}[k]_q$, so the $A_\infty(m)$'s clearly diverge.
	We have $t_{k+1}-t_k=(\rho q^{k+1})^{-1}\ge 1$, so $\vec t$ is of divergent type.
	To get the desired specialization, we use \Cref{thm:Fibonacci_positivity} with
	$c_k=\rho q^{k-1}$.

	For the Al-Salam--Chihara case,
	take $t_k=\rho^{-1}[k]_q$, so the $A_\infty(m)$'s clearly diverge.
	We have $t_{k+1}-t_k=\rho^{-1}q^k\ge 1$, so $\vec t$ is of divergent type.
	To get the specialization, apply
  \Cref{thm:Fibonacci_positivity} with
  $c_k = \rho^{-1}$ for all $k$.

	For the $q$-Charlier case, take
	$t_k=[k]_q \ssp(1+\rho(q-1)q^{k-1})/(\rho q^{2k})$
	and $c_k=\rho q^{2k-2}$ in \Cref{thm:Fibonacci_positivity}.
	The series $A_\infty(m)$ diverges for all $m$.
	Moreover,
	\begin{equation*}
		t_{k+1}-t_k=
		\rho^{-1}q^{-2 k-2} \bigl((\rho - 1) q^{k+1} - \rho q^k + q + 1\bigl).
	\end{equation*}
	One can check that this expression is $\ge1$ for $0<\rho\le 1$, $0<q\le 1$.
	Thus, $\vec t$ is of divergent type.

  Finally, for the Cigler--Zeng case, we have $c_k=1$ and
	$t_k=q^k-1$, so the
	$A_\infty(m)$'s clearly diverge.
	The difference $t_{k+1}-t_k=q^k(q-1) \geq 1$ for all
	$q\ge q_0$ where $q_0 \approx 1.4656$ is the unique real root of
	the cubic equation $z^3=z^2+1$.

	This completes the proof.
\end{proof}

Inspired by another member of the Askey scheme, we conjecture that
the corresponding specialization is Fibonacci positive:
\begin{conjecture}
	\label{conj:alternative_q_Charlier}
	There exist values of $\rho$ and $q$ for which the following specialization
	(which we call the \textbf{alternative $q$-Charlier})
	is Fibonacci positive:
	\begin{equation*}
				x_k = \frac{q^{\,k-1}\bigl(1+\rho q^{k-2}+\rho q^{k-1}-\rho q^{2k-2}\bigr)}
										 {\bigl(1+\rho q^{2k-3}\bigr)\bigl(1+\rho q^{2k-1}\bigr)},\qquad
				y_k = \frac{\rho\,q^{3k-2}\bigl(1-q^{k}\bigr)\bigl(1+\rho q^{k-1}\bigr)}
										 {\bigl(1+\rho q^{2k}\bigr)\bigl(1+\rho q^{2k-1}\bigr)^{2}\bigl(1+\rho q^{2k-2}\bigr)},
	\end{equation*}
	where $k\ge1$.
	The sequences $\vec c$ and $\vec t$ are computed for this specialization
	by \Cref{prop:c_t_from_x_y}:
	\begin{equation*}
		c_k=
		\frac{q^{k-1} \left(1+\rho  q^{k-1}\right)}{\left(1+\rho
		q^{2 k-1}\right) \left(1+\rho  q^{2 k-2}\right)},
		\qquad
		t_k=
		\frac{\rho\,q^{k-1}\left(1-q^{k}\right)\left(1+\rho\,q^{2k+1}\right)}
		{\left(1+\rho\,q^{k}\right)\left(1+\rho\,q^{2k-1}\right)}
		,
	\end{equation*}
	where also $k\ge1$.
\end{conjecture}

For future use in \Cref{sub:shifted_Charlier_defn} below, we
introduce a family of divergent specializations that deform
the Plancherel specialization:
\begin{definition}
	\label{def:fake_shifted_Charlier}
	Let $\rho\in(0,1]$ and $\sigma\in [1,\infty)$
	be two parameters.
	The \textbf{fake shifted Charlier}
	specialization
	(see
	\Cref{rmk:why_not_fake_shifted_Charlier}
	below
	for a discussion of the name)
	is defined as
	\begin{equation}
		\label{eq:fake_shifted_Charlier}
		x_k=\begin{cases}
			\rho, & k=1,\\
			k+\rho+\sigma-2, & k\ge2,
		\end{cases}
		\qquad  y_k=\rho(k+\sigma-1), \quad k\ge1.
	\end{equation}
	It corresponds to setting $t_k=\rho^{-1}(\sigma+k-1)$ and
	$c_k=\rho$, where $k\ge1$. The convention $t_0=0$ implies that $x_k=1$.
	When $\sigma=\rho=1$, the specialization becomes the Plancherel specialization
	$x_k=y_k=k$. One can check directly by \Cref{def:finite_infinite_type_specializations}
	that the $\vec t$-sequence $t_k=\rho^{-1}(\sigma+k-1)$ is of divergent type.
\end{definition}

\begin{remark}
	When $\rho=1$, the specialization
	\eqref{eq:fake_shifted_Charlier} can be obtained from the
	Plancherel specialization by applying a right shift
	(\Cref{proposition:right_shifting_fibonacci_positive_specializations})
	with $T_1=\sigma$. The extra parameter $\rho$ can
	subsequently be incorporated via
	\Cref{proposition:infinite_type_scaling} with
	$\alpha_k=\rho$ for all $k$.
	This implies that the fake shifted Charlier specialization
	is Fibonacci positive.
\end{remark}

\subsection{Convergent type example: Shifted Charlier specialization}
\label{sub:shifted_Charlier_defn}

The next example is a
deformation of the Charlier specialization by means of a new parameter $\sigma$:
\begin{definition}
	\label{def:shifted_Plancherel_and_Charlier}
	Let $\rho\in(0,1]$ and $\sigma\in(1,\infty)$ be two parameters.
	The \textbf{shifted Charlier} specialization
	is defined as
	\begin{equation}
		\label{eq:shifted_Charlier_definition}
		x_k=k+\rho+\sigma-2,\qquad
		y_k=(k+\sigma-1)\rho.
	\end{equation}
	It reduces to the Charlier specialization in the limit as $\sigma\to 1$.
	When $\rho=1$, we call \eqref{eq:shifted_Charlier_definition} the \textbf{shifted Plancherel} specialization.
	The latter reduces to the usual Plancherel specialization in the limit as $\sigma\to 1$.
\end{definition}

\begin{proposition}
	\label{prop:shifted_Charlier_and_Plancherel}
	For all $\sigma>1$ and $0<\rho\le1$, the shifted Charlier specialization is Fibonacci positive
	and is of convergent type.
\end{proposition}
In the limit as $\sigma\to1$,
the convergent type of the specialization
becomes divergent
(see \Cref{prop:class_II_positive_specializations}).
Note that for $\rho=1$ and integer
$\sigma \geq 2$, the statement of \Cref{prop:shifted_Charlier_and_Plancherel}
follows in a simpler way from
\Cref{propositon:shifting-fibonacci-positive-specializations} by applying left shifts to the
Plancherel specialization.
\begin{proof}[Proof of \Cref{prop:shifted_Charlier_and_Plancherel}]
	We thus first deal with the case $\rho=1$.
	Denote $\gamma\coloneqq \sigma-1>0$.
	One can check (for example,
	using the tridiagonal recurrence)
	that
	for $\rho=1$ and $x_k=y_k=k+\gamma$,
	the $A$-determinants \eqref{eq:A_B_dets} are
	sums of the Pochhammer symbols:
	\begin{equation}
		\label{eq:A_for_shifted_Plancherel_explicit}
		A_k=\sum_{j=0}^{k} \, (\gamma)_j, \qquad (\gamma)_j=\gamma(\gamma+1)\cdots(\gamma+j-1).
	\end{equation}
	This immediately implies
	that these determinants are positive.
	Moreover, on can check that
	asymptotically, we have
	$t_k=y_k A_{k-1}/A_{k+1}\sim 1/k$ as $k\to\infty$.
	See \Cref{fig:t_k_for_shifted_Plancherel} for a plot showing unimodality
	and eventual decay of the sequence $t_k$.
	Therefore, the series $A_\infty(m)$
	\eqref{eq:A_infty_B_infty_series_definitions}
	converges for all $m$.\footnote{Note that since the shifted Plancherel specialization
	does not have the form $x_k=1+y_{k-1}$ (in particular, $x_1=1+\gamma$ instead of $1$),
	the series $A_\infty(0)$ is \emph{not} the limit of the expressions $A_k$ \eqref{eq:A_for_shifted_Plancherel_explicit}.
	In particular, the series $A_\infty(0)$ converges, while clearly $\lim_{k\to\infty}A_k=+\infty$.}
	\begin{figure}[htpb]
		\centering
		\includegraphics[width=0.5\textwidth]{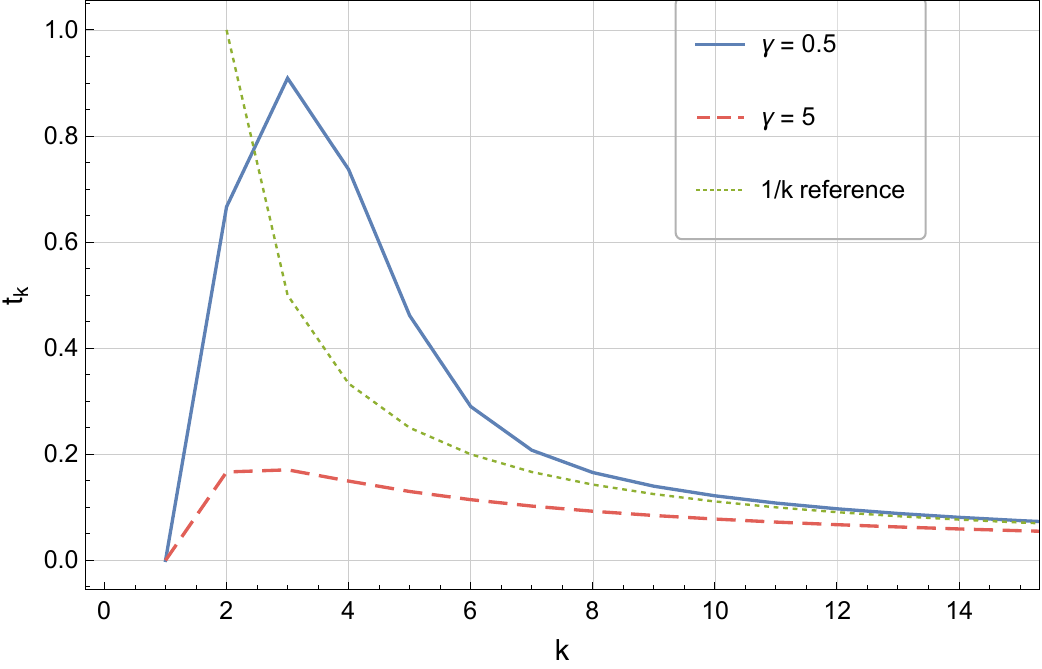}
		\caption{
			Plot of the sequences $t_k$ for the shifted Plancherel specialization
			with $\gamma=0.5$ and $\gamma=5$, together with the reference asymptotic behavior $1/k$.
			We see that the $t_k$'s are unimodal.
			For some $\gamma$, two of the neighboring $t_k$'s may become equal
			(cf.~\Cref{proposition:valleys} which prohibits consecutive
			triplets of equal $t$'s). For example, $t_2=t_3$ for $\gamma\approx
			0.147$.}
		\label{fig:t_k_for_shifted_Plancherel}
	\end{figure}

	Therefore, if the shifted Plancherel specialization is Fibonacci positive,
	it must be of convergent type.
	To conclude the positivity in the case $\rho=1$, observe that the
	B-determinants are explicit,
	\begin{equation}
		\label{eq:B_for_shifted_Plancherel_explicit}
		B_k(m)=m+1+\gamma, \qquad  k,m\ge0.
	\end{equation}
	Indeed, one can deduce this by expanding the tridiagonal
	determinants and observing a three-term recurrence relation.
	The expressions \eqref{eq:B_for_shifted_Plancherel_explicit} are evidently positive.

	\medskip

	Let us now consider the remaining case $0< \rho< 1$.
	Observe that if
	a sequence
	$\vec t$ defines a Fibonacci positive specialization
	via \eqref{eq:Fibonacci_positive_specialization_in_theorem},
	and we set
	\begin{equation}
		\label{eq:rho_deformed_x_y_parameters}
		x_k'=c_k(\rho+t_{k-1}),\qquad  y_k'=\rho c_kc_{k+1}t_k,
	\end{equation}
	then
	the new $\vec t$-sequence is given by $t_k'=t_k/\rho$.
	This implies that for general $\rho\in(0,1)$, the series $A_m(\infty)$
	also converges.
	It remains to check the positivity of the
	determinants $A_k$ and $B_k(m)$~\eqref{eq:A_B_dets}.

	Consider first the determinants $A_k$ and $B_k(m)$ with $m\ge 1$.
	Recall the fake shifted Charlier specialization
	(\Cref{def:fake_shifted_Charlier}),
	and let $x_k^\bullet$ and $y_k^\bullet$ denote its parameters.
	\eqref{eq:fake_shifted_Charlier}.
	It is Fibonacci positive and of divergent type.
	Our shifted Charlier specialization
	\eqref{eq:shifted_Charlier_definition}
	differs from its fake counterpart
	only insofar as $x_1^\bullet=\rho$ is replaced by $x_1=\rho+\gamma$.
	In particular, the determinants $B_k(m)$ for all $m\ge1$
	are the same in both specializations, and hence are positive.
	By linearity in the first row, the determinants $A_k$ and $B_k(0)$
	for two the specializations
	are related as follows:
	\begin{equation}
		\label{eq:shifted_Charlier_determinants_fake_relation}
		A_k=A_k^{\bullet}+\gamma A_{k-1}^{\bullet}(1),
		\qquad
		B_k(0)=B_k^{\bullet}(0)-\gamma y_2 A_{k-1}^{\bullet}(3),
	\end{equation}
	where $A_k^\bullet(m),B_k^\bullet(m)$ are the determinants for the fake
	shifted Charlier specialization,
	with indices shifted by $m$.
	From
	\eqref{eq:shifted_Charlier_determinants_fake_relation},
	we immediately see that $A_k>0$ for all $k$.

	The determinants $B_k(0)$ require a different treatment.
	Fix $k$ and denote by $f(\rho)\coloneqq B_k^{(\rho)}(0)$ the determinants constructed from the
	general $\rho$-dependent specialization
	\eqref{eq:rho_deformed_x_y_parameters}.
	We already showed that $f(1)>0$. Clearly, $f(0)=0$.
	Therefore, $f(\rho)$ has an even number of roots on $(0,1)$,
	counted with multiplicity.
	One can check that this polynomial is explicitly given by
	\begin{equation}
		\label{eq:shifted_Charlier_determinants_explicit_f_rho}
		f(\rho)=
		\rho\ssp c_1\cdots c_{k+2} \,
		\sum_{j=0}^k \rho^{k - j} \left(t_1 -
		\mathbf{1}_{j < k} \cdot t_{j + 2}\right)
		t_2 t_3\cdots t_{j+1},
	\end{equation}
	where $t_i=(i+\gamma)A_{i-1}/A_{i+1}$ (see \Cref{prop:c_t_from_x_y}),
	and where $A_i$ is the sum \eqref{eq:A_for_shifted_Plancherel_explicit}.
	This implies that
	$f'(0)=t_1\cdots t_{k+1} >0$.
	We claim that the
	coefficient sequence of \eqref{eq:shifted_Charlier_determinants_explicit_f_rho}
	has exactly one sign change. This would imply that $f(\rho)$ has at most one root in $(0,1)$,
	and hence is positive for all $\rho\in(0,1)$. This would finish the proof.

	Let us now check the claim about the sign changes.
	The first coefficient is
	\begin{equation*}
		t_1-t_2=-1 / (\gamma ^3+4 \gamma ^2+4 \gamma +1)<0,
	\end{equation*}
	and the last coefficient
	is $t_1t_2\cdots t_{k+1} >0$.
	We will show that if
	$\gamma$ is such that
	$t_1-t_m\ge0$ for some $m\ge 3$, then
	$t_1-t_{m+1}\ge0$.
	The quantity $t_1-t_m$ is a rational function with a positive denominator,
	and its numerator is equat to
	\begin{equation*}
		(\gamma+2)\sum_{j=0}^{m+1}(\gamma)_j-(\gamma+1)(\gamma+m+1)\sum_{j=0}^m(\gamma)_j=
		\gamma \ssp  Q_m(\gamma)-(m-1),
	\end{equation*}
	where the $Q_m$'s are certain polynomials. One can check that both polynomials
	$Q_m(\gamma)$ and $Q_m(\gamma)-m\ssp Q_{m-1}(\gamma)$ have nonnegative integer coefficients for all $m\ge 3$.
	The condition $t_1-t_m\ge0$ is equivalent to to $Q_m(\gamma)\ge (m-1)/\gamma$.
	We have
	\begin{equation*}
		Q_{m+1}-m/\gamma =
		\left( Q_{m+1}-Q_m-1/\gamma \right)+
		\underbrace{\left( Q_m-(m-1)/\gamma \right)}_{\ge0}.
	\end{equation*}
	We also see that
	\begin{equation*}
		\gamma(Q_{m+1}-Q_m-1) - \left( \gamma Q_m-(m-1) \right)=
		\gamma\left( Q_{m+1}-2Q_m \right)+m-2,
	\end{equation*}
	and for $m\ge 3$ this polynomial has nonnegative integer coefficients.
	This implies that $Q_{m+1}-Q_m-1/\gamma\ge0$, and
	hence we get the claim about the sign changes. This completes the proof
	of the proposition.
\end{proof}

\begin{remark}
	\label{rmk:why_not_fake_shifted_Charlier}
	Throughout the paper, we focus mainly on the shifted Charlier specialization rather than
	the fake one from
	\Cref{def:fake_shifted_Charlier}.
	In the case $\rho=1$ (shifted Plancherel specializations),
	we compare the asymptotic behavior of the fake and the true
	shifted Plancherel specializations in
	\Cref{sec:asymptotics_shifted_Plancherel_specialization,sub:two_cycles}
	(in two different regimes). Surprisingly, we observe that they are the same
	for both regimes.
\end{remark}

\subsection{Convergent type example: Power specializations}
\label{sub:power_spec_defn}

Let us consider two more examples
of convergent type of the form
\begin{equation}
	\label{eq:convergent_spec_definition}
	t_k=\frac{\varkappa}{k^\upalpha},\qquad
	\upalpha=1,2,\qquad k\ge1,
\end{equation}
where $\varkappa$ is a positive real parameter.
We call these the \emph{power specializations}.
Note that we must have $\upalpha\ge 1$,
see \Cref{rmk:nonexample_convergent_type}.

\begin{proposition}
	\label{prop:convergent_type_two_examples}
	There exist
	upper bounds
	$\varkappa_1^{(\upalpha)}$, $\upalpha=1,2$,
	with
	$\varkappa_1^{(1)} \approx 0.844637$ and $\varkappa_1^{(2)} \approx 1.41056$,
	such that for all
	$0<\varkappa<\varkappa_1^{(\upalpha)}$,
	the specialization \eqref{eq:convergent_spec_definition}
	is Fibonacci positive and of convergent type.
\end{proposition}

In the proof and throughout the rest of the paper, we
use the standard notation
for the hypergeometric functions
and Pochhammer symbols:
\begin{equation}
	\label{eq:hypergeometric_function}
	{}_r F_s
	\left(
	\begin{array}{c}
		a_1, \dots, a_r \\
		b_1, \dots, b_s
	\end{array}
	\ssp\middle| \ssp z
	\right)
	=
	\sum_{k=0}^{\infty} \frac{(a_1)_k \cdots (a_r)_k}{(b_1)_k
	\cdots (b_s)_k} \frac{z^k}{k!},
	\qquad (a)_k = a(a+1)\cdots(a+k-1).
\end{equation}

\begin{proof}[Proof of \Cref{prop:convergent_type_two_examples}]

We have for integer $\upalpha$:
\begin{align*}
A_\infty(m)
&= \ 1  +  \frac{\varkappa}{m^\upalpha}
 +  \frac{\varkappa^2}{m^\upalpha (m+1)^\upalpha}
 +  \frac{\varkappa^3}{m^\upalpha (m+1)^\upalpha (m+2)^\upalpha}
 +  \cdots \\ \\
 &= \ \sum_{r = 0}^\infty \left(\frac{ \Gamma(m)}{\Gamma(m+r)} \right)^\upalpha
\varkappa^r
= {}_1F_\upalpha
\bigl( 1 \, ; \, \underbrace{m, \dots, m}_{\upalpha \text{ times}}
\, ; \, \varkappa \bigr).
\end{align*}
The desired inequality
\[
B_\infty(m) = \, t_{m+1}  -
\big(1 + t_m - t_{m+1} \big)
 t_{m+2} \, A_\infty(m+3)
 \geq  0
\]
can be rewritten as
\[
\frac{\varkappa}{(m+1)^\upalpha}
- \left( 1 + \frac{\varkappa}{m^\upalpha}\ssp\mathbf{1}_{m>0} - \frac{\varkappa}{(m+1)^\upalpha} \right)
\, \frac{\varkappa}{(m+2)^\upalpha}
\, _1F_\upalpha
\bigl( 1 \, ; \, \underbrace{m+3, \dots, m+3}_{\upalpha \text{ times}}
\, ; \, \varkappa \bigr)
\geq  0.
\]
As a function of $\varkappa$, one can check that $B_\infty(m)$ vanishes only at $\varkappa = 0$ and at a value
$\varkappa_m^{(\upalpha)} \in (0, \infty)$ for each $m \geq 0$. Furthermore, the sequence
$\{ \varkappa_m^{(\upalpha)} \, : \, m \geq 0\}$ is strictly increasing, and consequently, the $\vec{t}$-sequence will be of convergent type if and only if
$0 < \varkappa \leq \varkappa_1^{(\upalpha)}$. The
bounds are numerically found to be
$\varkappa_1^{(1)} \approx 0.844637$ and $\varkappa_1^{(2)} \approx 1.41056$.
\end{proof}

\part{Fibonacci Positivity and Stieltjes Moment Sequences}
\label{part:2}

In this part we examine Fibonacci positivity
in light of the well-known correspondence
(due to
\cite{flajolet1980combinatorial},
\cite{viennot1983theorie},
\cite{CorteelKimStanton2016},
\cite{sokal2020euler},
\cite{petreolle2023lattice})
between semi-infinite, totally positive, tridiagonal
matrices and Stieltjes moment sequences.
We
recall the general setup related to
Stieltjes moment sequences,
continued fractions,
tridiagonal matrices,
orthogonal polynomials,
Motzkin polynomials, and Toda flow in \Cref{sec:Stieltjes_moment_sequences}.
In \Cref{sec:Stieltjes_moments_general}
we derive general formulas for
Stieltjes moment sequences in terms of the
$\vec c$- and $\vec t$-sequences associated with
a totally positive tridiagonal matrix. These formulas involve sums over non-crossing set partitions, and we also provide an expression involving a sum over compositions.
In \Cref{sec:Fibonacci_positive_moments}
we obtain a different formula for
Stieltjes moment sequences
arising from Fibonacci positive specializations; this expression, again written in terms of the corresponding $\vec c$- and $\vec t$-sequences, is given by a sum over all set partitions.
In \Cref{sec:Fibonacci_and_Stieltjes_examples},
we connect some examples of divergent type
defined in \Cref{sub:divergent_type_examples_new}
to the Askey scheme of orthogonal polynomials.
Finally, in \Cref{sec:shifted_Charlier_Stieltjes},
we treat the shifted Charlier specialization (\Cref{sub:shifted_Charlier_defn})
from the moment sequence perspective.

\section{Stieltjes Moment Sequences and Total Positivity}
\label{sec:Stieltjes_moment_sequences}

\subsection{Stieltjes moment sequences and Jacobi continued fractions}

Recall that a sequence
$\vec{a}=(a_0, a_1, a_2, \dots)$ of real numbers is called a
\emph{strong Stieltjes moment
sequence} if there exists a nonnegative Borel measure $\upnu(dt)$ on $[0,\infty)$
with \emph{infinite support} such that
$a_n = \int_0^\infty t^n  \upnu(dt)$ for each $n \geq 0$.
The following result may be found, e.g.,
in
\cite{sokal2020euler}:

\begin{theorem}
\label{thm:stieltjes-moments-theorem}
A sequence of real numbers $\vec{a} = (a_0, a_1, a_2, \dots)$ is a strong Stieltjes moment sequence if and only if there exist two real number sequences, $\vec{x}$ and $\vec{y}$, such that the matrix $\mathcal{A} \big( \, \vec{x} \, | \, \vec{y} \, \big)$ defined in \eqref{eq:tridiagonal_matrices_A_Br} is totally positive, and the (normalized) ordinary moment generating function of $\vec a$,
\begin{equation}
	\label{eq:moment-generating-function_theorem}
	M(z) = \sum_{n \geq 0} \frac{a_n}{a_0} z^n,
\end{equation}
is expressed by the Jacobi continued fraction
depending on $( \, \vec{x} \, | \, \vec{y} \, )$ as
\begin{equation}
	\label{eq:Jacobi-continued-fraction_theorem}
	M(z)=
	J_{\, \vec{x}, \vec{y}} \,(z) \coloneqq
{1 \over {1 - x_1z - {\displaystyle y_1z^2 \over {\displaystyle 1 - x_2 z - {\displaystyle y_2z^2 \over
{\displaystyle 1 - x_3 z - {\displaystyle y_3z^2 \over {\ddots } } } } } } } }
\end{equation}

Moreover,
the equality between the
generating function $M(z)$ \eqref{eq:moment-generating-function_theorem}
and the continued fraction
$J_{\, \vec{x}, \vec{y}} \,(z)$ \eqref{eq:Jacobi-continued-fraction_theorem}
is witnessed by the recursion
\begin{equation*}
	P_{n+1}(t) = (t - x_{n+1})P_n(t) - y_n P_{n-1}(t),
	\quad n \geq 1,
	\qquad
	P_0(t) = 1, \quad P_1(t) = t - x_1.
\end{equation*}
responsible for generating the polynomials
$P_n(t)$ which are orthogonal with respect to the nonnegative
Borel measure $\upnu(dt)$
on $[0,\infty)$ whose moment sequence is
$\vec{a}$.
\end{theorem}

A putative or "formal" moment sequence $\vec{a}$ can always be
combinatorially determined from any pair of sequences $\vec{x}$
and $\vec{y}$ by calculating the associated \textit{Motzkin polynomials}.
Specifically, the ratio ${a_n/a_0}$ can be expressed as the generating
function of all length-$n$ Motzkin paths, where each up-step $\nearrow$
at height $k$ is weighted by $y_k$, and each horizontal step $\rightarrow$
at height $k$ is weighted by $x_{k+1}$. Figure \ref{motzkin-path}
illustrates an example of a weighted Motzkin path of length seven.

\begin{figure}[h]
\centering
\begin{tikzpicture}[scale=1]
    \foreach \x in {0,1,2,3,4,5,6,7} {
        \foreach \y in {0,1,2,3} {
            \fill (\x,\y) circle (2pt);
        }
    }

    \draw[thick] (0,0) -- (1,1) node[midway,left,red]  {\small$y_1$};
    \draw[thick] (1,1) -- (2,1) node[midway,above,red] {\small$x_2$};
    \draw[thick] (2,1) -- (3,2) node[midway,left,red] {\small$y_2$};
    \draw[thick] (3,2) -- (4,1) node[midway,right,red] {\small$1$};
    \draw[thick] (4,1) -- (5,1) node[midway,above,red] {\small$x_2$};
    \draw[thick] (5,1) -- (6,0) node[midway,right,red] {\small1};
    \draw[thick] (6,0) -- (7,0) node[midway,above,red] {\small$x_1$};
\end{tikzpicture}
\caption{An example of a Motzkin path of weight $x_1x_2^2 y_1 y_2$.}
\label{motzkin-path}
\end{figure}

Below
we list the first four (normalized) formal
moments which are the Motzkin polynomials:
\begin{equation}
	\label{eq:motzkin-polynomials}
	\begin{split}
    a_1/a_0 &= x_1, \\
    a_2/a_0 &= x_1^2 + y_1, \\
    a_3/a_0 &= x_1^3 + 2x_1y_1 + x_2y_1, \\
    a_4/a_0 &= x_1^4 + 3x_1^2y_1 + y_1^2 + 2x_1x_2y_1 + x_2^2y_1 + y_1y_2.
	\end{split}
\end{equation}

\begin{remark}[Permutation statistics and Jacobi continued fractions]
	\label{rmk:14_permutation_statistics}
	Jacobi continued fractions
	and their associated moments are
	connected not only to Motzkin paths (and set partitions,
	as we explain below in \Cref{sec:Stieltjes_moments_general}),
	but also to
	permutation
	statistics. One of the most recent examples of such
	connections is \cite{blitvic2021permutations},
	which relates a 14-parameter Jacobi continued fraction
	to permutation enumeration.

	It would be very interesting to combine the random permutations
	arising from this 14-parameter enumeration with the
	Young--Fibonacci RS correspondence, which we describe
	in \Cref{sec:random_Fibonacci_words_to_random_permutations}
	below. The resulting measures on Fibonacci words
	may coincide with certain clone Schur measures.
	We do not develop this direction further in the present work.
\end{remark}

\subsection{Toda flow}

By \Cref{thm:stieltjes-moments-theorem}, the fact that the sequence
$\vec{a}$, as in \eqref{eq:motzkin-polynomials}, is realized by an
infinitely supported, nonnegative Borel measure is equivalent to the
total positivity of $\mathcal{A} \big( \, \vec{x} \, | \, \vec{y} \, \big)$
\eqref{eq:tridiagonal_matrices_A_Br}.
Conversely, sequences $\vec{x}$ and $\vec{y}$ can be constructed
from a Borel measure $\upnu(dt)$ using the \textit{Toda flow}
\cite{GekhtmanShapiro1997, NakamuraZhedanov2004},
which we now recall.

Having $\upnu(dt)$, consider its exponential reweighting
$e^{\varrho t} \upnu(dt)$. The moments of the reweighted measure
satisfy
\begin{equation*}
    a_n(\varrho) = \frac{d^n}{d\varrho^n} a_0(\varrho),
		\quad \text{where} \quad
    a_0(\varrho) = \int_{-\infty}^{\infty} e^{\varrho t} \ssp \upnu(dt)
    = \sum_{n \geq 0} \frac{a_n}{n!}\ssp \varrho^n.
\end{equation*}
The sum
on the far right is the exponential moment generating function of
$\upnu(dt)$. As functions of $\varrho$, the
associated
tridiagonal parameters $x_n(\varrho)$ and
$y_n(\varrho)$ for $n \geq 1$ must obey
the {\it Toda chain equations}, namely,
\begin{equation}
\label{TodaChain}
\begin{split}
{d \over {d\varrho}} \, x_n(\varrho)
&=  y_n(\varrho)  -  y_{n-1}(\varrho); \\
{d \over {d\varrho}} \, y_n(\varrho)
&=  y_n(\varrho) \big( x_{n+1}(\varrho)  -
x_n(\varrho)\big).
\end{split}
\end{equation}
Their solutions are given by
\begin{equation}
\label{TodaSolution}
\begin{split}
x_n(\varrho)
&= \frac{d}{d\varrho}
\log \left( \frac{\Delta_n(\varrho)}{\Delta_{n-1}(\varrho)} \right) \\
&= \mathrm{Tr} \left(
\mathrm{H}_n^{-1}(\varrho)
\, \mathrm{H}_n^{(1)}(\varrho)
\right) -
\mathrm{Tr} \left(
\mathrm{H}_{n-1}^{-1}(\varrho) \,
\mathrm{H}_{n-1}^{(1)}(\varrho)
\right), \\
y_n(\varrho)
&= \frac{\Delta_{n-1}(\varrho) \, \Delta_{n+1}(\varrho)}
{\Delta_n(\varrho)^2}.
\end{split}
\end{equation}
Here, $\Delta_n(\varrho) = \det \mathrm{H}_n(\varrho)$, and
$\mathrm{H}_n(\varrho)$ and
$\mathrm{H}_n^{(1)}(\varrho)$
are the Hankel matrices
\begin{equation*}
  \mathrm{H}_n(\varrho) \coloneqq
  \underbrace{
    \begin{pmatrix}
      a_0(\varrho) & a_1(\varrho) & a_2(\varrho) & \cdots \\
      a_1(\varrho) & a_2(\varrho) & a_3(\varrho) & \cdots \\
      a_2(\varrho) & a_3(\varrho) & a_4(\varrho) & \cdots \\
      \vdots & \vdots & \vdots & \ddots
    \end{pmatrix}
  }_{\text{$n \times n$ Hankel matrix}},
	\quad
  \mathrm{H}_n^{(1)}(\varrho) \coloneqq
  \frac{d}{d \varrho} \mathrm{H}_n(\varrho)
  =
  \underbrace{
    \begin{pmatrix}
      a_1(\varrho) & a_2(\varrho) & a_3(\varrho) & \cdots \\
      a_2(\varrho) & a_3(\varrho) & a_4(\varrho) & \cdots \\
      a_3(\varrho) & a_4(\varrho) & a_5(\varrho) & \cdots \\
      \vdots & \vdots & \vdots & \ddots
    \end{pmatrix}
	}_{\text{$n \times n$ Hankel matrix}}.
\end{equation*}
In $\mathrm{H}_n^{(1)}(\varrho)$,
we used the fact that
${d \over {d\varrho}} a_k(\varrho) = a_{k+1}(\varrho)$.
For example, the solutions
for $x_k(\varrho)$ and $y_k(\varrho)$
for $k= 1,2$ are
\begin{equation*}
\begin{split}
	x_1(\varrho) &= \frac{a_1(\varrho)}{a_0(\varrho)},\qquad
x_2(\varrho)= \frac{a_1^3(\varrho) - 2a_0(\varrho)a_1(\varrho)a_2(\varrho) + a_0^2(\varrho)a_3(\varrho)}
{a_0 \big( a_0(\varrho) a_2(\varrho) - a_1^2(\varrho) \big)}, \\
y_1(\varrho) &= \frac{a_0(\varrho)a_2(\varrho) - a_1^2(\varrho)}{a_0^2(\varrho)}, \\
y_2(\varrho) &= \frac{a_0(\varrho)\big( a_0(\varrho)a_2(\varrho)a_4(\varrho)
+ 2a_1(\varrho)a_2(\varrho)a_3(\varrho) - a_1^2(\varrho)a_4(\varrho)
- a_0(\varrho)a_3^2(\varrho) - a_2^3(\varrho) \big)}
{\big( a_0(\varrho) a_2(\varrho) - a_1^2(\varrho) \big)^2}.
\end{split}
\end{equation*}
The sequences $\vec{x}$ and $\vec{y}$ for the original
measure $\upnu(dt)$ can be obtained by setting $\varrho = 0$
in \eqref{TodaSolution}. We emphasize that the Toda flow
preserves total positivity: Given two initial sequences
$\vec{x}$ and $\vec{y}$ for which the matrix
$\mathcal{A}(\vec{x} \mid \vec{y} \ssp )$ is totally positive, the
matrix $\mathcal{A}(\vec{x}(\varrho) \mid \vec{y}(\varrho))$
remains totally positive for any $\varrho \leq 0$. Here,
$\vec{x}(\varrho) = (x_1(\varrho), x_2(\varrho), \dots)$ and
$\vec{y}(\varrho) = (y_1(\varrho), y_2(\varrho), \dots)$ are
solutions of the Toda chain equations given by
\eqref{TodaSolution}.

\begin{example}
	\label{ex:Poisson-measure}
	Consider the Poisson distribution
	\begin{equation}
			\label{eq:Poisson-measure-definition}
			\upnu_{\scriptscriptstyle \mathrm{Pois}}^{\scriptscriptstyle (\rho)} (dt)
			\coloneqq
			e^{-\rho} \sum_{k \geq 0} \,\frac{\rho^k}{k!} \, \delta_k(dt),
	\end{equation}
	where $\delta_k$ is the Dirac delta mass at $k$.
	This distribution
	is
	obtained by applying the Toda flow, with ``time''
	$\varrho = \log(\rho)$, to the Poisson distribution
	$\upnu_{\scriptscriptstyle \mathrm{Pois}}^{\scriptscriptstyle (1)} (dt)$,
	and then
	renormalizing by $e^{1-\rho}$.
	Indeed,
	the associated
	tridiagonal parameters have the form
	\[
	x_n(\varrho) = n + \rho - 1 = n + e^\varrho - 1
	\quad \text{and} \quad
	y_n(\varrho) = \rho \, n = e^\varrho \, n.
	\]
	and satisfy the Toda chain equations \eqref{TodaChain}.
	Note that for all $\rho \in (0,1]$
    these tridiagonal parameters
    are Fibonacci positive;
	equivalently, $(\vec{x}(\varrho), \vec{y}(\varrho))$ is Fibonacci positive
    when the Toda flow parameter satisfies
    $\varrho \in (-\infty, 0]$.
\end{example}

Fibonacci positivity is stronger than total positivity.
This presents two natural questions:
\begin{problem}
	\label{problem:moment-sequences-tridiagonal-matrices-Fibonacci}
	What are the properties of moment sequences and nonnegative Borel
	measures $\upnu_{\vec{x},\vec{y}} \ssp (dt)$ associated with
	Fibonacci positive specializations $(\vec{x}, \vec{y}\ssp)$
	by \Cref{thm:stieltjes-moments-theorem}? Can these moment
	sequences and measures be characterized in a meaningful way?
\end{problem}
\begin{problem}
	\label{problem:Toda-flow}
	Does the Toda flow preserve
	the space of Fibonacci positive
	specializations
	$(\vec x,\vec{y}\ssp)$
	for values of the
	deformation parameter $\varrho$
	within some interval $(-R, 0]$
	with $R > 0$?
\end{problem}

We do not address these problems in full generality here.

\begin{remark}[Toda flow]
	Along with the Poisson measure
	(\Cref{ex:Poisson-measure}), the
	shifted Charlier specialization
	$x_k=\rho+\sigma+k-2$,
	$y_k=\rho(\sigma+k-1)$, $k\ge1$, also
	satisfies the Toda chain equations \eqref{TodaChain},
	after the same change of variables
	$\rho=e^{\varrho}$.
	Thus, in the shifted Charlier case,
	the Toda flow preserves the
	Fibonacci positivity when $\varrho\in(-\infty,0]$.

	In contrast, the
	Type-I Al-Salam--Carlitz, Al-Salam--Chihara, and
	$q$-Charlier specializations
	we consider in \Cref{sec:Fibonacci_and_Stieltjes_examples} below
	\emph{do not} satisfy the Toda chain equations
	with the natural change of variables $\rho = \exp(\varrho)$.
	This is not evidence against a positive answer to
	\Cref{problem:Toda-flow},
	but indicates that the associated Toda flow
	may require a different, more intricate
	parametrization.
\end{remark}

\section{Combinatorics of Stieltjes Moment Sequences}
\label{sec:Stieltjes_moments_general}

In this section, we present general results
on Stieltjes moment sequences
associated with a totally positive
specialization
$(\vec{x}, \vec{y} \ssp)$,
without assuming that it is Fibonacci positive.


\subsection{Statistics on non-crossing partitions}

We begin describing the moments by connecting
set partitions and
Motzkin paths.
Recall that $\pi \in \Pi(n)$ denotes an arbitrary set partition of
$\{1, \dots, n\}$.
It is always presented in canonical form, i.e.,
\[
\pi = B_1 \big| B_2 \big| \cdots \big| B_r,
\]
where the blocks $B_1, \dots, B_r$ are ordered such that
$\min B_1 < \cdots < \min B_r$.
We say that $i$ {\it precedes} $j$
whenever $i = b_s$ and $j = b_{s+1}$
for some $1 \leq s < \ell$
where $B = \{ b_1 < \cdots < b_\ell \}$
is a block of $\pi$ whose
elements are listed in increasing order.

An element of a block of size one is called a
\emph{singleton}.  An element is called an \emph{opener} if
it is not a singleton and it is the minimal element in its
block.  Likewise, an element is called a \emph{closer} if it
is not a singleton and it is the maximal element in its
block.  An element which is neither a singleton, an opener,
nor a closer is called a \emph{transient}. The sets of
singletons, openers, closers, and transients of a set
partition $\pi \in \Pi(n)$ are denoted, respectively, by
$\mathbcal{S}(\pi)$, $\mathbcal{O}(\pi)$,
$\mathbcal{C}(\pi)$,~and~$\mathbcal{T}(\pi)$.  Clearly,
these four sets are disjoint, and their union is $\{1,
\dots, n\}$.

A set partition $\pi \in \Pi(n)$ is often depicted by
arranging the numbers $\{1, \dots, n\}$ in increasing order
from left to right on the horizontal $x$-axis.  An arc is
drawn in the upper half-plane between $i, j \in \{ 1, \dots,
n\}$ whenever $i$ precedes $j$. The ensemble of arcs can
always be drawn so that any pair of arcs cross at most once
and no more than two arcs cross at any point in the upper
half-plane. A set partition $\pi \in \Pi(n)$ is said to be
{\it non-crossing} if its arcs are pairwise non-crossing;
the set of non-crossing set partitions of $n$ is denoted
$\mathrm{NC}(n)$.

\begin{figure}[htb]
\begin{minipage}[b]{0.45\linewidth}
\centering
\begin{tikzpicture}[scale=0.6]
\foreach \x in {1,...,9} {
    \fill (\x,0) circle (3pt);
    \node[below] at (\x,-0.3) {\x};
}
\draw[thick] (1,0) to[out=90,in=90] (3,0);
\draw[thick] (2,0) to[out=90,in=90] (9,0);
\draw[thick] (3,0) to[out=90,in=90] (5,0);
\draw[thick] (6,0) to[out=90,in=90] (7,0);
\draw[thick] (7,0) to[out=90,in=90] (8,0);
\end{tikzpicture}
\end{minipage}
\hfill
\begin{minipage}[b]{0.45\linewidth}
\centering
\begin{tikzpicture}[scale=0.6]
\foreach \x in {1,...,9} {
    \fill (\x,0) circle (3pt);
    \node[below] at (\x,-0.3) {\x};
}
\draw[thick] (1,0) to[out=90,in=90] (9,0);
\draw[thick] (2,0) to[out=90,in=90] (3,0);
\draw[thick] (3,0) to[out=90,in=90] (5,0);
\draw[thick] (6,0) to[out=90,in=90] (7,0);
\draw[thick] (7,0) to[out=90,in=90] (8,0);
\end{tikzpicture}
\end{minipage}
\caption{Examples of arc ensembles of set partitions $\pi = 135 \mid 29 \mid 4 \mid 678$ (left)
and $\pi' = 19 \mid 235 \mid 4 \mid 678$  (right). Note that
$\pi'$ is non-crossing, while $\pi$ is not. In $\pi$,
the openers are $1,2$, and $6$, the closers are $5,8$, and $9$,
the transients are $3$ and $7$. Finally, $4$ is the only singleton of $\pi$.}
\label{ex:arc-ensembles}
\label{fig:arc-ensembles}
\end{figure}

For $i \in \{1, \dots, n \}$, let $\Gamma_i(\pi)$ denote the
set of openers or transients $a < i$ such that $i \leq b$,
where $b$ is the element succeeding $a$ in the same block of
$\pi$.  For $i \in \mathbcal{C}(\pi) \cup
\mathbcal{T}(\pi)$, let $\gamma_i(\pi)$ be the position of
the opener or transient in $\Gamma_i(\pi)$ preceding~$i$,
where we list the elements $\Gamma_i(\pi) = \{a_1 < \dots <
a_k \}$ in increasing order.  Kasraoui and Zeng
\cite{kasraoui2006distribution} showed that a set partition
$\pi \in \Pi(n)$ is uniquely determined by the tuple
$(\mathbcal{O}, \mathbcal{C}, \mathbcal{S},\mathbcal{T})$
together with the integers $\gamma_i(\pi)$ for $i \in
\mathbcal{C}(\pi) \cup \mathbcal{T}(\pi)$.

\begin{remark}
\label{rem:noncrossing+gstatistic}
Note that $\pi \in \mathrm{NC}(n)$
if and only if
$\gamma_i(\pi) = \# \ssp \Gamma_i(\pi)$
for all $i \in \mathbcal{C}(\pi)
\cup \mathbcal{T}(\pi)$.
\end{remark}

\begin{example}
\label{example:set_partition}
Consider the
set partition $\pi = 135 \big| 29 \big| 4 \big| 678$ of $n=9$. In this case, we have
\[
\begin{array}{llll}
	\mathbcal{O}(\pi) = \{ 1, 2, 6 \};
&\qquad \Gamma_3(\pi) = \{1, 2 \},
&\qquad \gamma_3(\pi) = 1;\\
\mathbcal{C}(\pi) = \{ 5, 8, 9 \};
&\qquad \Gamma_5(\pi) = \{2, 3\},
&\qquad \gamma_5(\pi) = 2; \\
\mathbcal{S}(\pi) = \{ 4 \};
&\qquad \Gamma_7(\pi) = \{2, 6 \},
&\qquad \gamma_7(\pi) = 2; \\
\mathbcal{T}(\pi) = \{ 3, 7 \};
&\qquad \Gamma_8(\pi) = \{2, 7 \},
&\qquad \gamma_8(\pi) = 2; \\
&\qquad \Gamma_9(\pi) = \{2 \},
&\qquad \gamma_9(\pi) = 1,
\end{array}
\]
and $\Gamma_1(\pi)=\varnothing$,
$\Gamma_2(\pi)=\left\{ 1 \right\}$,
$\Gamma_4(\pi)=\left\{ 2,3 \right\}$,
$\Gamma_6(\pi)=\left\{ 2 \right\}$.
\end{example}

Denote also
\begin{equation}
\label{def:g-ell-statistics}
\begin{array}{rcll}
\displaystyle \ell_k(\pi)
&\displaystyle \coloneqq& \# \, \bigl\{\, 1\le i\le n
: \# \Gamma_i(\pi)=k \bigr\}, &\qquad k\ge 0, \\[4pt]
\displaystyle g_m(\pi)
&\displaystyle \coloneqq &\# \, \bigl\{\, i \in \mathbcal{C}(\pi)
\cup \mathbcal{T}(\pi)
: \gamma_i(\pi)=m \bigr\}, &\qquad m\ge 1.
\end{array}
\end{equation}
In \Cref{example:set_partition}, we have
\begin{equation*}
	\ell_0(\pi)=1,\qquad \ell_1(\pi)=3,\qquad
	\ell_2(\pi)=5,\qquad g_1(\pi)=2,\qquad g_2(\pi)=3,
\end{equation*}
and all other $\ell_k(\pi)$ and $g_k(\pi)$ are zero.

\begin{remark}
\label{remark:inequality-setpartitions}
Notice that $\ell_0(\pi) \geq 1$
and $\ell_k(\pi) \geq g_k(\pi)$
for any set partition $\pi \in \Pi(n)$.
\end{remark}

\begin{lemma}
\label{lemma:noncrossing+compositions}
Let $\pi \in \mathrm{NC}(n)$,
and let $p = \max \ssp \{ k  :  \ell_k(\pi) > 0 \} = \max \ssp \{ k  :  g_k(\pi) > 0 \}$.
Then for $p\ge1$, we have
$\ell_k(\pi)> g_k(\pi) \geq 1$
whenever $1 \leq k < p$
while $\ell_p(\pi) \geq
g_p(\pi) \geq 1$ (for $p=0$, we simply have $\ell_0=n$ and $g_1=0$).
In particular, this implies that
$\boldsymbol{g}(\pi) = (g_1(\pi), \dots, g_p(\pi))$
is a composition of
$\# \ssp ( \mathbcal{C}(\pi) \cup \mathbcal{T}(\pi) )$,
while
$\boldsymbol{\ell}(\pi) := (\ell_0(\pi), \dots,
\ell_p(\pi))$ is a
composition of $n$.
\end{lemma}
Recall that a \emph{composition} of $m$ is a sequence
of \emph{positive} integers
summing to $m$.
\begin{proof}
First, observe the following changes in $\#\Gamma_i(\pi)$:
\begin{enumerate}[$\bullet$]
	\item
	If $i$ is an opener, then $\#\Gamma_{i+1}(\pi) = \#\Gamma_{i}(\pi) + 1$;
	\item
	If $i$ is a singleton or a transient, then $\#\Gamma_{i+1}(\pi) = \#\Gamma_{i}(\pi)$;
	\item
	If $i$ is a closer, then $\#\Gamma_{i+1}(\pi) = \#\Gamma_{i}(\pi) - 1$.
\end{enumerate}
In other words, the sequence
$(0=\#\Gamma_1(\pi), \#\Gamma_2(\pi), \dots, \#\Gamma_n(\pi),0)$
is a Motzkin path with $n$ steps.
The number of visits to level $k$ of this path is counted by
$\ell_k(\pi)$.
If we label the horizontal steps of the path
by either color $0$ (if $i$ is a singleton)
or color $k$, where $k$ is the height of the horizontal step
(if $i$ is a transient),
then $g_k(\pi)$ counts the number of down-steps
that start at level $k\ge1$, and also the number
of horizontal steps at level $k\ge1$ that have positive color.
Note that the latter count is only valid for non-crossing
set partitions, see \Cref{rem:noncrossing+gstatistic} (and also the extension
to all set partitions in \Cref{sub:Charlier_histoires} below).
The quantity $p$ is the maximal level visited by the path,
so clearly $\ell_p(\pi)\ge g_p(\pi) \ge 1$.
For any $k<p$, the path must eventually rise from level $k$ to level $k+1$,
so we get the strict inequality. This completes the proof.
\end{proof}

Denote by $\mathrm{nest}(\pi)$ the
\emph{nesting} statistic of a set partition $\pi \in \Pi(n)$,
which counts the number of pairs $(a,b)$
and $(c,d)$
where $a$ precedes $b$
and $c$ precedes $d$ in $\pi$,
while the {\it nesting} inequality
$a<c<d<b$ holds.

\begin{lemma}
\label{lemma:nesting-formula}
For any
$\pi \in \Pi(n)$, we have
\begin{equation}
\label{eq:nesting-formula}
\mathrm{nest}(\pi)
\, = \,
\sum_{k \ssp \geq \ssp 1}
\, (k-1) g_k(\pi).
\end{equation}
\end{lemma}
\begin{proof}
For each quadruple $a<c<d<b$ contributing to the nesting statistic,
note that $d$ is not an opener or a singleton (since $c$ precedes $d$),
hence $d\in \mathbcal{C}(\pi) \cup \mathbcal{T}(\pi)$.
Moreover, $a\in \Gamma_d(\pi)$, since $b>d$.
Because $c$ precedes $d$ in $\pi$, the element $a\in \Gamma_d(\pi)$
can be any of the first $\gamma_d(\pi)-1$ elements of $\Gamma_d(\pi)$.
Clearly, for fixed $d$, the $\gamma_d(\pi)-1$ choices of $a$ yield
all the nestings for which $(c,d)$ is the inner arc.

For each $k$, and for every $g_k(\pi)$ elements $d \in \mathbcal{C}(\pi) \cup \mathbcal{T}(\pi)$
with $\gamma_d(\pi) = k$, there are exactly $k-1$ choices for $a \in \Gamma_d(\pi)$.
This completes the proof.
\end{proof}

\subsection{Charlier histoires}
\label{sub:Charlier_histoires}

Let us extend the correspondence between
non-crossing set partitions and
Motzkin paths (given in the proof of \Cref{lemma:noncrossing+compositions})
to all set partitions.
A length-$n$ \emph{Charlier histoire} is a
colored Motzkin path of length $n$,
where each~$\rightarrow$ step at height $k$ is assigned a
nonnegative integer color $\in \{0, \dots, k \}$, while
each~$\searrow$ step from height $k$ to $k-1$ is assigned a positive integer color $
\in \{1, \dots, k\}$. Let $\mathfrak{H}_{n}$ denote the
set of length-$n$ Charlier histoires.

Set partitions of $n$
and Charlier histoires of length $n$
are well-known to be in bijective correspondence.\footnote{We are
grateful to Dennis Stanton for explaining
this relationship to us.}
We present a bijection $\Pi(n) \rightarrow
\frak{H}_{n}$
which is adapted from
\cite{josuat2011crossings} and
\cite{kasraoui2006distribution},
and extends the easier correspondence for non-crossing partitions
from the proof of \Cref{lemma:noncrossing+compositions}.
Specifically, a Charlier histoire
	$\frak{h}_\pi$ is constructed from left to right by
	converting, in order, each element $i \in \{1, \dots, n\}$
	of a set partition $\pi \in \Pi(n)$ into a (colored) step
	of type $\{ \nearrow,  \rightarrow, \searrow \}$ according
	to the following rules:
	\begin{enumerate}[$\bullet$]
		\item Each $i \in \mathbcal{O}(\pi)$
			is converted into an $\nearrow$ step at
            height $\# \ssp \Gamma_i(\pi) + 1$;
		\item Each $i \in \mathbcal{S}(\pi)$
            \ssp is converted
			into an $\rightarrow$ step at
            height $\# \ssp \Gamma_i(\pi)$
            with color $\chi = 0$;
		\item Each $i \in \mathbcal{T}(\pi)$
        is converted into an $\rightarrow$ step
        at height $\# \ssp \Gamma_i(\pi)$
        with color $\chi = \gamma_i(\pi)$;
		\item Each $i \in \mathbcal{C}(\pi)$
        \ssp is converted into an $\searrow$ step at
            height $\# \ssp \Gamma_i(\pi)$
		    with color $\chi = \gamma_i(\pi)$.
	\end{enumerate}

\begin{figure}[h]
\centering
\begin{tikzpicture}[scale=0.7]
  \draw[step=1cm,help lines,gray!30] (0,0) grid (9,3);

  \foreach \x/\y in {0/0,1/1,2/2,3/2,4/2,5/1,6/2,7/2,8/1,9/0}
    {\filldraw[black] (\x,\y) circle (1.6pt);}

  \draw[very thick] (0,0) -- (1,1);
  \draw[very thick] (1,1) -- (2,2);
  \draw[very thick] (5,1) -- (6,2);

  \draw[very thick] (2,2) -- (3,2);
  \draw[very thick] (3,2) -- (4,2);
  \draw[very thick] (6,2) -- (7,2);

  \draw[very thick] (4,2) -- (5,1);
  \draw[very thick] (7,2) -- (8,1);
  \draw[very thick] (8,1) -- (9,0);

  \node[above] at (2.5,2) {\textcolor{red}{$1$}};
  \node[above] at (3.5,2) {\textcolor{red}{$0$}};
  \node[above] at (4.5,1.5) {\textcolor{red}{$2$}};
  \node[above] at (6.5,2) {\textcolor{red}{$2$}};
  \node[above] at (7.5,1.5) {\textcolor{red}{$2$}};
  \node[above] at (8.5,0.5) {\textcolor{red}{$1$}};

  \node[left] at (0,0) {$0$};
  \node[left] at (0,1) {$1$};
  \node[left] at (0,2) {$2$};
\end{tikzpicture}
\caption{The Charlier histoire
$\frak{h}_\pi$ corresponding
to $\pi = 135 \big| 29 \big| 4 \big| 678$.}
\label{charlier-histoire}
\end{figure}

\begin{remark}
\label{rem:noncrossing+histoires}
	A set partition $\pi \in \Pi(n)$
	is non-crossing
	if and only if
	each $\rightarrow$ step
	of the corresponding Charlier histoire $\frak{h}_\pi$
	at height $k$
	has either color $0$ or $k$,
	and each $\searrow$ step
	of $\frak{h}_\pi$
	at height $k$
	has color $k$.
	Indeed, this is equivalent to \Cref{rem:noncrossing+gstatistic}.
\end{remark}

Given a Charlier histoire $\frak{h}_\pi$, let $\frak{m}$ be
the Motzkin path obtained by forgetting the
colours. The area statistic
$\mathrm{area}(\frak{m})$ is simply the Euclidean area lying below
$\frak{m}$ and above the horizontal axis. Note that
$\mathrm{area}(\frak{m})$ is always an integer.
When convenient
we abuse notation and write
$\mathrm{area}(\pi):=\mathrm{area}(\frak{m})$.
We will use the following statement in the case-by-case analysis of moments presented in \Cref{sec:Fibonacci_and_Stieltjes_examples} below.

\begin{lemma}
	\label{lemma:area-statistic+formula}
	For any set partition $\pi \in \Pi(n)$,
	we have
	\begin{equation}
	\label{eq:area-statistic+formula}
	\mathrm{area}(\frak{m}) \, = \,
	\sum_{k \ssp \geq \ssp 0} \, k \ssp \ell_k(\pi).
	\end{equation}
\end{lemma}
\begin{proof}
	The area under a Motzkin path is equal to the sum of the
	starting heights of all its steps. By the construction
	of $\mathfrak{h}_\pi$, these heights are equal to $\#\Gamma_i(\pi)$, and so
	we have
	\begin{equation*}
		\mathrm{area}(\pi)=\sum_{i=1}^{n} \# \Gamma_i(\pi).
	\end{equation*}
	Grouping by the values of $\Gamma_i(\pi)=k$, $0\le k\le n$,
	we get the desired identity
	\eqref{eq:area-statistic+formula}.
	This completes the proof.
\end{proof}

\subsection{Totally positive moment formula}
\label{sub:totally_positive_moment_formula}

Let $\frak{M}_n$ denote the set of
all Motzkin paths of length $n$,
and let $\mathrm{pr}_n: \frak{H}_n \rightarrow \frak{M}_n$ be the projection map from
Charlier histoires to Motzkin paths which
simply ``forgets'' the colors.
By \Cref{rem:noncrossing+histoires}, there are
exactly $2^{\ssp \# \ssp
\mathrm{hor}(\frak{m})}$
Charlier histoires
$\frak{h}_\pi$ with
$\pi \in \mathrm{NC}(n)$
and $\mathrm{pr}_n (\frak{h}_\pi) = \frak{m}$
for each Motzkin
path $\frak{m} \in \frak{M}_n$,
where $\# \ssp \mathrm{hor}(\frak{m})$
is the number of
$\rightarrow$ steps
in~$\frak{m}$ which are strictly above the $x$-axis.
Thus,
\begin{equation}
\label{catalan-motzkin-count}
 \# \ssp \mathrm{NC}(n) \, = \,
\sum_{\frak{m} \ssp \in \ssp \frak{M}_n}
2^{\# \ssp \mathrm{hor}(\frak{m})}.
\end{equation}

The next result is a multivariate
enhancement of \eqref{catalan-motzkin-count}
which incorporates the parameters~$\vec{c}$ and~$\vec{t}$ coming from
total positivity:

\begin{proposition}
\label{prop:tp-moment-formula}
Let $(\vec{x}, \vec{y} \ssp )$
be a totally positive specialization
expressed in terms of the sequences
$\vec t$ and $\vec c$
\eqref{eq:Fibonacci_positive_specialization_in_theorem},
see
\Cref{prop:c_t_from_x_y}. Then
the associated $n$-th Stieltjes moment
is given by
\begin{equation}
\label{eq:tp-moment-formula}
a_n = \sum_{\pi \ssp \in \ssp
\mathrm{NC}(n)} \,
\prod_{k \ssp \geq \ssp 1} \,
c_k^{\ssp \ell_{k-1}(\pi)} \ssp t_k^{\ssp g_k(\pi)},
\end{equation}
where the statistics $\ell_k(\pi)$ and $g_k(\pi)$
are defined in \eqref{def:g-ell-statistics}.
Note that we do not assume that the specialization
$(\vec{x}, \vec{y} \ssp )$
is Fibonacci positive.
\end{proposition}

\begin{remark}
\label{rem:odd-even-trick-after-proposition}
Our \Cref{prop:tp-moment-formula} can be viewed as a
bi\-variate refinement of the classical ``odd-even trick''
\cite[p.~40]{Chihara1978}, see also
\cite[Section~4]{CorteelKimStanton2016}.\footnote{We thank Dennis
Stanton for bringing the odd-even trick to our attention.}
Specifically, the refinement is obtained by the
specialization $\lambda_{2k-1}=c_k$ and
$\lambda_{2k}=c_{k+1}t_k$ for all $k\ge 1$.  In the odd-even
construction, an east step at height $k$ is weighted by
$\lambda_{2k}+\lambda_{2k+1}=x_{k+1}$, while a southeast
step at height $k$ is weighted by
$\lambda_{2k}\lambda_{2k-1}=y_k$.  Hence one can deduce 
\Cref{prop:tp-moment-formula} from the results of
\cite{CorteelKimStanton2016}.
We utilize an odd-even-like
construction in \Cref{def:splitting+composition,def:even-odd-composition} below.
For the reader's convenience,
we include a self-contained proof of \Cref{prop:tp-moment-formula} below.
\end{remark}

\begin{proof}[Proof of \Cref{prop:tp-moment-formula}]
The recipe in \Cref{sec:Stieltjes_moment_sequences}
for computing Stieltjes moments
from the sequences $\vec{x}$ and $\vec{y}$
assigns $x_{k+1}$ to each~$\rightarrow$ step
at height $k$ and $y_k$ to each~$\nearrow$ step
at height $k$ in a Motzkin path
$\frak{m} \in \frak{M}_n$.
Since
$y_k = c_k c_{k+1} t_k$
can be factored into $c_k$ and
$c_{k+1} t_k$,
we can modify this rule
and instead assign weights to steps as follows:
\begin{enumerate}[$\bullet$]
\item each $\nearrow$ step at height $k$ is
weighted by $c_k$;
\item each $\searrow$ step at height $k$
is weighted by $c_{k+1} t_k$;
\item each $\rightarrow$ step at height $k$ is
	weighted by $x_{k+1} = c_{k+1} + c_{k+1} t_k$.
\end{enumerate}
We define the weight $\mathrm{wt}(\frak{m})$
of a Motzkin path $\frak{m} \in \frak{M}_n$ to be the product of the weights of its steps. This weight
agrees with \Cref{sec:Stieltjes_moment_sequences},
and we recover the Stieltjes moments $a_n =
\sum_{\frak{m} \in \frak{M}_n} \mathrm{wt}(\frak{m})$,
as before.

Let us now introduce a weight $\omega(\frak{h}_\pi)$
for each
Charlier histoire $\frak{h}_\pi \in \frak{H}_n$
with $\pi \in \mathrm{NC}(n)$. Note that
we consider only non-crossing set partitions $\pi$,
in accordance with the desired formula
\eqref{eq:tp-moment-formula}.
By \Cref{rem:noncrossing+histoires},
we only have to consider the cases where the colors are
either $\chi = 0$ (for $\rightarrow$ steps)
or $\chi = k$ (for $\rightarrow$ and
$\searrow$ steps at height $k$).
Define the weight $\omega(\mathfrak{h}_\pi)$ of a Charlier histoire
$\mathfrak{h}_\pi \in \mathfrak{H}_{n}$
with $\pi \in \mathrm{NC}(n)$
as the product of the weights of
its (colored) steps, where the weights are given as follows:
\begin{enumerate}[$\bullet$]
				\item each $\nearrow$ step at height $k$ is weighted by $c_k$;
				\item each $\rightarrow$ step at height $k$ is weighted by $c_{k+1}$ if $\chi = 0$, or else by $c_{k+1}t_k$ if $\chi=k$;
				\item each $\searrow$ step at height $k$ is weighted by $c_{k+1}t_k$.
\end{enumerate}

This system of weights for
Charlier histoires $\frak{h}_\pi \in \frak{H}_n$
with $\pi \in \mathrm{NC}(n)$ is
consistent
with the weights
$\mathrm{wt}(\frak{m})$
for Motzkin paths $\frak{m} \in \frak{M}_n$
and the projection map
$\mathrm{pr}_n: \frak{H}_n
\rightarrow \frak{M}_n$ in the
sense that
\[\mathrm{wt}(\frak{m})
\, = \, \sum_{\stackrel{\scriptstyle \pi \ssp \in \ssp \mathrm{NC}(n)}{\mathrm{pr}_n(\frak{h}_\pi) \ssp = \ssp \frak{m}}} \omega(\frak{h}_\pi). \]
Each $\rightarrow$ step
and each $\searrow$ step
occurring at height $k$
and colored $\chi = k$
contributes a
factor of $t_k$ to
$\omega(\frak{h}_\pi)$.
Since
$\pi$ is non-crossing,
these steps
correspond precisely to
the
elements in
$\mathbcal{C}(\pi) \cup \mathbcal{T}(\pi)$.
Clearly, $t_k$ must occur
$g_k(\pi)$ many times. Finally,
each $\nearrow$ step at height $k$ carries a
weight of $c_k$ while
all other steps at height $k$
carry a weight of $c_{k+1}$.
Said differently, the step
corresponding to $i \in \{1, \dots, n \}$
is weighted $c_{k+1}$
if and only if $\# \ssp \Gamma_i(\pi) = k$.
Consequently, the total number of
steps with weight $c_k$ is $\ell_{k-1}(\pi)$.
Hence,
\[ \omega(\frak{h}_\pi)
= \prod_{k \ssp \geq \ssp 1} \,
c_k^{\ssp \ell_{k-1}(\pi)} \ssp t_k^{\ssp g_k(\pi)}, \]
for any $\pi \in \mathrm{NC}(n)$, and we are
done.
\end{proof}

\subsection{Compositions and set partitions}

We examine connections between set partitions, integer
compositions, and Fibonacci words. Every set partition
determines an integer composition, and we introduce a
procedure to {\it split} a composition into a unique pair of
{\it Fibonacci compositions}, each naturally constructed
from a Fibonacci word. We present a more concise
reformulation of \Cref{eq:tp-moment-formula} for the $n$-th
Stieltjes moment in terms of compositions, which involves
this Fibonacci splitting.


\begin{definition}
\label{def:fibonacci-ribbon+composition}
Given a Fibonacci word $w = a_1 \cdots a_k$ with $a_i\in \{1,2\}$
and with rank $|w|=n$, let $\mathrm{\bf R}(w)$ denote the
connected {\it ribbon} consisting of $n$ boxes
$\Box$
arranged from left to right
where the $i$-th column contains
$a_i$ boxes.
See \Cref{fig:fibonacci-ribbon-example} for an illustration.
\begin{figure}[htpb]
\centering
\[ \begin{young}
, & , & , & , & , &,  &,  &,  &
\\ , & , & , & , & , & & & &
\\ , & , & , & , & & & , 1 & , 1 & , 2
\\ , & & & & & , 2
\\  &  & , 1 & , 1 & , 2
\\ , 1 & , 2
\end{young}\]
\caption{Fibonacci ribbon for $w= 121122112$
with $\boldsymbol{\varsigma}(w)=(2,4,2,4,1)\models 13$.}
\label{fig:fibonacci-ribbon-example}
\end{figure}

A composition of $n$ denoted by
$\boldsymbol{\varsigma}(w) = (\varsigma_1,\dots,\varsigma_p)$
is obtained from $w \in \Bbb{YF}_n$
by letting $\varsigma_i$ record the number of boxes in the
$i$-th row of $\mathrm{\bf R}(w)$,
counted from the bottom.
Alternatively, we have
$\varsigma_1 = 1 + r_1$, $\varsigma_p = 1 + r_p$,
and $\varsigma_k = 2 + r_k$ for
$1 < k < p$, where
$w = 1^{r_1} 2 1^{r_2} 2 \cdots
2 1^{r_p}$ is the decomposition
of $w$ into its {\it runs} (see \Cref{eq:runs}).
A composition
$\boldsymbol{\varsigma} = (\varsigma_1,\dots,\varsigma_p)$
is of the form $\boldsymbol{\varsigma}(w)$
for some $w \in \Bbb{YF}_n$ if and only if
$\boldsymbol{\varsigma} \models n$ and
$\varsigma_i > 1$ whenever $1 < i < p$.
Compositions of this kind are called
{\it Fibonacci compositions}.
The Fibonacci word in $\Bbb{YF}_n$
corresponding to a Fibonacci composition
$\boldsymbol{\varsigma} = (\varsigma_1,\dots,\varsigma_p)$ of $n$
is denoted $\mathrm{Fib}(\boldsymbol{\varsigma})$.
\end{definition}


The relationship between Fibonacci words and
Motzkin
constructions developed in \Cref{sub:Charlier_histoires,sub:totally_positive_moment_formula}
starts from the following observation.
\begin{lemma}
\label{lemma:fibonacci-composition+example-one}
For any $\pi \in \Pi(n)$, the sequence
\[
\boldsymbol{\ell}(\pi):=(\ell_0(\pi),\dots,\ell_p(\pi)),
\qquad
p \;=\; \max\{k : \ell_k(\pi) > 0\},
\]
is a Fibonacci composition of $n$.
\end{lemma}
\begin{proof}
	First, we observe that $\boldsymbol{\ell}(\pi)$ is a composition of $n$
	for an arbitrary (not necessarily non-crossing) set partition $\pi \in \Pi(n)$,
	due to the very construction. Next, we have
	$\ell_k(\pi)>1$ for all $1 \leq k < p$, since $\ell_k(\pi)$ counts the
	number of visits to level $k$ in the Motzkin path $\frak{m}$ corresponding to $\pi$
	as defined in \Cref{sub:Charlier_histoires}.
	A Motzkin path visits each of its non-maximal levels at least twice, and so we are done.
\end{proof}

\begin{definition}[Splitting of a composition]
\label{def:splitting+composition}
Let $\pmb{\varkappa}
= (\varkappa_1, \dots, \varkappa_m)$ be a composition of an integer $n \ge 1$.
We associate with $\pmb{\varkappa}$ two auxiliary Fibonacci
compositions of $n$, denoted
$\mathrm{A}(\pmb{\varkappa})
= (\alpha_1, \dots, \alpha_{p+1})$ and
$\mathrm{B}(\pmb{\varkappa})
= (\beta_1, \dots, \beta_{p+1})$,
where $p = \lfloor m/2 \rfloor $, defined as follows:
\begin{enumerate}[$\bullet$]
\item
If $m \le 2$, set
$\mathrm{A}(\pmb{\varkappa})
= \mathrm{B}(\pmb{\varkappa})
= \pmb{\varkappa}$.

\item If $m = 3$, set
$\mathrm{A}(\pmb{\varkappa})
= (\varkappa_1, \varkappa_2 + \varkappa_3)$
and
$\mathrm{B}(\pmb{\varkappa})
= (\varkappa_1 + \varkappa_3, \varkappa_2)$.

\item For $m \ge 4$, define
\[
\alpha_k =
\begin{cases}
\varkappa_1, & k = 1, \\
\varkappa_{2k-2} + \varkappa_{2k-1}, & 2 \le k \le p, \\
\varkappa_{2p}, & k = p+1,\ m = 2p, \\
\varkappa_{2p} + \varkappa_{2p+1}, & k = p+1,\ m = 2p+1,
\end{cases}
\qquad
\beta_k =
\begin{cases}
1 - p + \displaystyle\sum_{j=1}^p \varkappa_{2j-1}, & k = 1, \\
1 + \varkappa_{2k-2}, & 2 \le k \le p, \\
\varkappa_{2p}, & k = p+1 .
\end{cases}
\]
\end{enumerate}
\end{definition}

\begin{definition}
\label{def:fibwords+splitting+composition}
For a composition $\pmb{\varkappa}
=(\varkappa_1,\dots,\varkappa_p)$ of $n$,
we denote
\[
u(\pmb{\varkappa})
:=\mathrm{Fib}\bigl(\mathrm{A}(\pmb{\varkappa})\bigr),
\qquad
v(\pmb{\varkappa})
:=\mathrm{Fib}\bigl(\mathrm{B}(\pmb{\varkappa})\bigr).
\]
We call $u(\pmb{\varkappa}),v(\pmb{\varkappa})\in \mathbb{YF}_n$
the Fibonacci words associated with a composition
$\pmb{\varkappa}$.
\end{definition}

\begin{lemma}
\label{lemma:injection+compositions+splitting}
The map
\[
\pmb{\varkappa}\;\longmapsto\;
\bigl(u(\pmb{\varkappa}), v(\pmb{\varkappa})\bigr)
\]
is injective from $\mathrm{Comp}(n)$ to
$\Bbb{YF}_n \times \Bbb{YF}_n$.
Furthermore, the maps
$\pmb{\varkappa} \mapsto u(\pmb{\varkappa})$ and
$\pmb{\varkappa} \mapsto v(\pmb{\varkappa})$
define surjections from
$\mathrm{Comp}(n)$ to
$\Bbb{YF}_n$.
\end{lemma}

\begin{proof}
The injectivity assertion follows from the fact that the map
\[
\pmb{\varkappa}\;\longmapsto\;
\bigl(\mathrm{A}(\pmb{\varkappa}),\mathrm{B}(\pmb{\varkappa})\bigr)
\]
is injective from $\mathrm{Comp}(n)$ to
$\mathrm{Comp}(n)\times \mathrm{Comp}(n)$.
The fact that the image consists of pairs of Fibonacci compositions is irrelevant for injectivity.

For surjectivity, given $w \in \Bbb{YF}_n$
let $\boldsymbol{\varsigma}(w) = (\varsigma_1, \dots,
\varsigma_{p+1})$ where $p = \mathcal{h}(w) > 0$ is the number of $2$'s in $w$,
and define $\pmb{\varkappa} = (\varkappa_1,
\dots, \varkappa_{2p})$ by
setting $\varkappa_1 = \varsigma_1$
and $\varkappa_{2p}=\varsigma_{p+1}$
with
\[ \begin{array}{ll}
\varkappa_{2k} = \varsigma_{k+1} -1
&\text{for $1 \leq k < p $}, \\
\varkappa_{2k-1} = 1
& \text{for $1 < k \leq p$}.
\end{array} \]
When $\mathcal{h}(w) = 0$, i.e., $w = 1^n$, then
set $\pmb{\varkappa} = (n)$. By construction
$\mathrm{A}(\pmb{\varkappa}) = \mathrm{B}(\pmb{\varkappa})= \boldsymbol{\varsigma}(w)$. This settles the surjectivity claim.
\end{proof}

\begin{definition}
\label{def:composition+depletion}
Let $\boldsymbol{\varsigma}=(\varsigma_1,\dots,\varsigma_p)$ be a Fibonacci composition of $n$.
Its \emph{depletion} is the composition
\[
\mathrm{dep}(\boldsymbol{\varsigma})=(\delta_1,\dots,\delta_{p-1})
\]
of $n-\varsigma_1-p+1$ defined by
\[
\delta_k=\varsigma_{k+1}-1 \quad (1\le k\le p-2), \qquad
\delta_{p-1}=\varsigma_p.
\]
By convention, we set $\mathrm{dep}(1^n)= \pmb{\varnothing}$
where $\pmb{\varnothing}$ is the empty composition. Clearly,
$\mathrm{dep}(\boldsymbol\varsigma)=\mathrm{dep}(\boldsymbol\tau)$
if and only if $\boldsymbol{\alpha}=\boldsymbol{\tau}$ for any Fibonacci compositions $\boldsymbol{\varsigma},\boldsymbol{\tau}$ of $n$.
\end{definition}


For a composition $\pmb{\varkappa}=(\varkappa_1,\dots,\varkappa_p)$ we adopt the standard shorthand
\[
\boldsymbol{c}^{\pmb{\varkappa}}
:=c_1^{\varkappa_1}\cdots c_p^{\varkappa_p},
\qquad
\boldsymbol{t}^{\pmb{\varkappa}}
:=t_1^{\varkappa_1}\cdots t_p^{\varkappa_p},
\]
which will be used throughout.

\begin{definition}
\label{def:even-odd-composition}
Let $\pi \in \mathrm{NC}(n)$ be a non-crossing set partition.  Define
\[
\mathcal{z}_1(\pi)=\ell_0(\pi),
\qquad
\mathcal{z}_{2k}(\pi)=g_k(\pi),
\qquad
\mathcal{z}_{2k+1}(\pi)=\ell_k(\pi)-g_k(\pi),
\]
where $1\le k\le p$ and where
$p = \max \{ k : \ell_k(\pi) > 0 \}$.
Set
\[
\boldsymbol{\mathcal{z}}(\pi):=
\begin{cases}
\bigl(\mathcal{z}_1(\pi),\dots,\mathcal{z}_{2p+1}(\pi)\bigr), & \text{if } \ell_p(\pi) > g_p(\pi),\\
\bigl(\mathcal{z}_1(\pi),\dots,\mathcal{z}_{2p}(\pi)\bigr), & \text{if } \ell_p(\pi) = g_p(\pi).
\end{cases}
\]
Clearly, $\boldsymbol{\mathcal{z}}(\pi)$ is a composition of $n$.

Introduce a sequence of independent parameters
$\vec{\lambda}=(\lambda_1,\lambda_2,\lambda_3,\dots)$ and
apply the odd-even trick
\cite{Chihara1978}, \cite{CorteelKimStanton2016} (cf. \Cref{rem:odd-even-trick-after-proposition})
by specializing $c_k=\lambda_{2k-1}$ and $t_k=\lambda_{2k}/\lambda_{2k+1}$ for $k\ge 1$.
With this specialization, we have
$\boldsymbol{c}^{\ssp \boldsymbol{\ell}(\pi)}
\ssp \boldsymbol{t}^{\ssp \boldsymbol{g}(\pi)}
=\boldsymbol{\lambda}^{\boldsymbol{\mathcal{z}}(\pi)}$.
\end{definition}

\begin{lemma}
\label{lemma:splitting+set partitons}
For any non-crossing
set partition $\pi \in \Pi(n)$
we have
\[ \boldsymbol{\ell}(\pi) = \mathrm{A}(\boldsymbol{\mathcal{z}}(\pi))
\qquad \boldsymbol{g}(\pi) =
\mathrm{dep} \big(\mathrm{B}(\boldsymbol{\mathcal{z}}(\pi)) \big).\]
\end{lemma}
\begin{proof}
Set $\varkappa=\boldsymbol{\mathcal z}(\pi)$, and
apply
\Cref{def:splitting+composition} to compute
$\alpha=\mathrm{A}(\pmb{\varkappa})=(\alpha_{1},\dots ,\alpha_{p+1})$.
We have $\alpha_1=\varkappa_1=\ell_0$.
For $2\le k\le p$, we have
\begin{equation*}
\alpha_k=\varkappa_{2k-2}+\varkappa_{2k-1}=g_{k-1}+(\ell_{k-1}-g_{k-1})=\ell_{k-1}.
\end{equation*}
One can also check that $\alpha_{p+1}=\ell_p$ in both cases $m=2p$ and $m=2p+1$,
which completes the proof of the first equality.

Similarly for $\beta=\mathrm{B}(\pmb{\varkappa})=(\beta_{1},\dots ,\beta_{p+1})$, we get
its depletion $\delta=\operatorname{dep}(\beta)$:
\begin{equation*}
	\delta_k=\beta_{k+1}-1=(1+g_k)-1=g_k,\qquad 1\le k\le p-1,
\end{equation*}
and
$\delta_p=\beta_{p+1}=\varkappa_{2p}=g_p$.
This completes the proof of the second equality.
\end{proof}

\begin{proposition}
\label{prop:noncrossing+composition+surjection}
The map
$
\pi \mapsto \boldsymbol{\mathcal{z}}(\pi)$
is surjective from $\mathrm{NC}(n)$ onto the set $\mathrm{Comp}(n)$ of all compositions of $n$.
\end{proposition}
Observe that
the sizes of the sets agree with the statement:
\begin{equation*}
\#\mathrm{NC}(n)=\mathrm{Catalan}(n)\ge 2^{n-1}=\#\mathrm{Comp}(n)
,\qquad n\ge1,
\end{equation*}
where $\mathrm{Catalan}(n)$ is the $n$-th Catalan number.

\begin{proof}[Proof of \Cref{prop:noncrossing+composition+surjection}]
Fix a composition $\pmb{\varkappa}=(\varkappa_1,\dots,\varkappa_m)\models n$ and set
\[
\boldsymbol{\ell}:=\mathrm{A}(\pmb{\varkappa}), \qquad
\boldsymbol{g}:=\mathrm{dep}\bigl(\mathrm{B}(\pmb{\varkappa})\bigr).
\]

Suppose that there exists a Motzkin path (with labelled
horizontal steps) whose statistics are precisely
$\boldsymbol{\ell}$ and $\boldsymbol{g}$, interpreted as in
the proof of \Cref{lemma:noncrossing+compositions}. In this
case the desired statement follows. Indeed, starting from
such a Motzkin path one constructs a non-crossing set
partition in the standard way: up-steps and down-steps are
regarded as openers and closers, respectively, and are
matched in the usual nested order as in the classical
bijection between Dyck paths and non-crossing pairings. Each
horizontal step is treated as either a singleton or a
transient, according to whether its color is $0$ or the
height of the step, respectively; in the transient case the
element is placed into the block determined by the nearest
adjacent up- and down-steps.

Let $p=\max\{k : \ell_k > 0\}$ and assume $p>0$ (the case $p=0$ is trivial:
the horizontal path at height zero corresponds to the set partition consisting entirely of singletons).
By construction, the compositions
$\boldsymbol{\ell}=(\ell_0,\ell_1,\dots,\ell_p)$ and
$\boldsymbol{g}=(g_1,\dots,g_p)$ satisfy the properties
$\ell_k > g_k$ for $1 \le k < p$ and $\ell_p \ge g_p$, as in \Cref{lemma:noncrossing+compositions}.

We construct an inventory of steps for a labeled Motzkin
path $\frak{m}$ of length $n$. This inventory will consist
of the number of up-steps ($U_k$ from level $k-1$ to $k$), down-steps ($D_k$ from level $k+1$ to $k$),
horizontal steps corresponding to singletons ($H_{0,k}$ on level $k$),
and horizontal steps corresponding to transients ($H_{k,k}$ on level $k$)
at each height $k\ge0$.
This inventory satisfies
\begin{equation} \label{eq:lg_motzkin_relations}
\ell_k = U_k + D_k + H_{0,k} + H_{k,k} \qquad \text{and} \qquad g_k = D_{k-1} + H_{k,k},\qquad k\ge1,
\end{equation}
and $\ell_0=D_0 + H_{0,0}$.
Furthermore, we have the relations
$U_k=D_{k-1}$, and $U_0=D_p=0$.

We determine the number of steps of each type recursively in a certain deterministic way,
starting from the maximum height $p$ and working downwards. Note that
the choices we make are not unique, but they ensure that a given composition $\pmb{\varkappa}$ is represented by a unique
labeled Motzkin path $\frak{m}$ with the prescribed statistics $\boldsymbol{\ell}$ and $\boldsymbol{g}$.

First, set $D_p=0$, and consider $k$ from $p$ down to $1$.
We set
\begin{equation}
	\label{eq:U_k_choice}
	U_k=D_{k-1}=\min(g_k, \ell_{k-1} - g_{k-1})\ge1,
\end{equation}
then set
\begin{equation}
\label{eq:H_k_definitions}
	H_{k,k} = g_k - D_{k-1},\qquad H_{0,k} = \ell_k - g_k - D_k,
\end{equation}
where $D_k$ is determined in the previous step.
The choice of $D_{k-1}$ \eqref{eq:U_k_choice} is valid since
then both $H_{k,k}$ and $H_{0,k}$ defined by
\eqref{eq:H_k_definitions} are nonnegative,
and identities \eqref{eq:lg_motzkin_relations} hold.
Finally, for $k=0$, we set
\begin{equation*}
U_0=0,\qquad H_{0,0}=\ell_0-D_0,
\end{equation*}
where $D_0$ is determined in the previous step, and there is no $H_{k,k}$ for $k=0$.

Having an inventory of steps, we can arrange them in some prescribed
order, for example, first use the maximal number of
available up-steps, then all horizontal steps corresponding
to singletons, then horizontal steps corresponding to
transients (for the current height), then one down-step,
then repeat the process (use one up-step if available, and
so on). This produces a Motzkin path, and correspondingly, a non-crossing set partition $\pi$,
for which we have $\boldsymbol{\ell}(\pi) = \boldsymbol{\ell}$ and $\boldsymbol{g}(\pi) = \boldsymbol{g}$,
and $\boldsymbol{\mathcal{z}}(\pi) = \pmb{\varkappa}$.
An example of a Motzkin path is given in \Cref{fig:fibonacci-inverse-example}.

This partial inverse of the map $\pi \mapsto \boldsymbol{\mathcal{z}}(\pi)$
ensures its surjectivity, and so we are done.
\end{proof}

\begin{figure}[htpb]
\centering
\begin{tikzpicture}[scale=0.7]
  \draw[step=1cm,help lines,gray!30] (0,0) grid (9,3);

  \foreach \x/\y in {0/0,1/1,2/2,3/2,4/2,5/2,6/1,7/2,8/1,9/0}
    {\filldraw[black] (\x,\y) circle (1.6pt);}

  \draw[very thick] (0,0) -- (1,1);
  \draw[very thick] (1,1) -- (2,2);

  \draw[very thick] (2,2) -- (3,2);
  \draw[very thick] (3,2) -- (4,2);
  \draw[very thick] (4,2) -- (5,2);

  \draw[very thick] (5,2) -- (6,1);
  \draw[very thick] (6,1) -- (7,2);
  \draw[very thick] (7,2) -- (8,1);
  \draw[very thick] (8,1) -- (9,0);

  \node[above left, text=red] at (0.7,0.35) {$\scriptstyle U_{1}$};
  \node[above left, text=red] at (1.7,1.35) {$\scriptstyle U_{2}$};
  \node[above left, text=red] at (6.7,1.35) {$\scriptstyle U_{2}$};

  \node[below left, text=red] at (5.75,1.65) {$\scriptstyle D_{1}$};
  \node[below left, text=red] at (7.75,1.65) {$\scriptstyle D_{1}$};
  \node[below left, text=red] at (8.75,0.65) {$\scriptstyle D_{0}$};

  \node[above, text=red] at (2.5,2) {$\scriptstyle H_{0,2}$};
  \node[above, text=red] at (3.5,2) {$\scriptstyle H_{2,2}$};
  \node[above, text=red] at (4.5,2) {$\scriptstyle H_{2,2}$};
  \node[left] at (0,0) {$0$};
  \node[left] at (0,1) {$1$};
  \node[left] at (0,2) {$2$};
\end{tikzpicture}
\caption{A labeled Motzkin path for the composition $\boldsymbol{\ell}=(1,1,2,4,1)\models 9$ as in the 
proof of \Cref{prop:noncrossing+composition+surjection}.
The path has statistics
$\boldsymbol\ell=(1,3,5)$ and $\boldsymbol{g}=(1,4)$,
and yields the non-crossing set partition
$\pi=19\mid 2456\mid 3\mid 78$.
Note that the statistics are the same as for the non-crossing example in
\Cref{fig:arc-ensembles}, but the resulting set partition is different.}
\label{fig:fibonacci-inverse-example}
\end{figure}

\begin{remark}[Alternative proof via a canonical lift of the $\mathbcal{z}$-map]
Fix a composition $\pmb{\varkappa}=(\varkappa_1,\dots,\varkappa_m)\models n$.
We outline a recursive construction of a non-crossing set partition $\pi=\pi(\pmb{\varkappa})\in\mathrm{NC}(n)$ such that
$\boldsymbol{\ell}(\pi)=\mathrm{A}(\pmb{\varkappa})$ and
$\boldsymbol{g}(\pi)=\mathrm{dep}\big(\mathrm{B}(\pmb{\varkappa})\big)$.
Read $\pmb{\varkappa}$ from left to right and build $\pi$ from a previously constructed non-crossing set partition
$\widetilde{\pi}=\pi(\widetilde{\pmb{\varkappa}})$ corresponding to a subordinate composition
$\widetilde{\pmb{\varkappa}}\models\widetilde{n}<n$.
Blocks of $\widetilde{\pi}$ are denoted $B_1(\widetilde{\pi}),\dots,B_r(\widetilde{\pi})$.
Set $\pi(\pmb{\varnothing})=\varnothing$ for the empty composition $\pmb{\varnothing}$, and use the notation
$S+j=\{s+j:s\in S\}$.

Case 1: If $\varkappa_1>1$, let $\widetilde{\pmb{\varkappa}}=(\varkappa_1-1,\varkappa_2,\dots,\varkappa_m)$ (so $\widetilde{n}=n-1$), and define $\pi$ with blocks
$B_1(\pi)=\{1\}$ and $B_k(\pi)=B_{k-1}(\widetilde{\pi})+1$ for $k\ge2$.
Then $\ell_0(\pi)=\ell_0(\widetilde{\pi})+1$, $\ell_k(\pi)=\ell_k(\widetilde{\pi})$ for $k\ge1$, and $g_k(\pi)=g_k(\widetilde{\pi})$ for all $k\ge1$.

Case 2: If $\varkappa_1=1$, let $\widetilde{\pmb{\varkappa}}=(\varkappa_3,\dots,\varkappa_m)$
(so $\widetilde{n}=n-\varkappa_2-1$), and define $\pi$ with blocks
$B_1(\pi)=\{1,\dots,\varkappa_2\}\cup\{n\}$ and $B_k(\pi)=B_{k-1}(\widetilde{\pi})+\varkappa_2$ for $k\ge2$.
Then $\ell_0(\pi)=1$, $\ell_1(\pi)=\ell_0(\widetilde{\pi})+\varkappa_2$, $\ell_k(\pi)=\ell_{k-1}(\widetilde{\pi})$ for $k\ge2$, and
$g_1(\pi)=\varkappa_2$, $g_k(\pi)=g_{k-1}(\widetilde{\pi})$ for $k\ge2$.

In both cases, no crossings are introduced. Assuming inductively that
$\boldsymbol{\ell}(\widetilde{\pi})=\mathrm{A}(\widetilde{\pmb{\varkappa}})$ and
$\boldsymbol{g}(\widetilde{\pi})=\mathrm{dep}\big(\mathrm{B}(\widetilde{\pmb{\varkappa}})\big)$, one can check that
$\boldsymbol{\ell}(\pi)=\mathrm{A}(\pmb{\varkappa})$ and
$\boldsymbol{g}(\pi)=\mathrm{dep}\big(\mathrm{B}(\pmb{\varkappa})\big)$.
\end{remark}

We restate \Cref{prop:tp-moment-formula}
as a sum over compositions
using
\Cref{prop:noncrossing+composition+surjection}:

\begin{corollary}
\label{prop:tp-moment-formula-restatement}
Let $(\vec{x}, \vec{y} \ssp )$ be a totally positive
specialization expressed in terms of the sequences $\vec t$
and $\vec c$ as in
\eqref{eq:Fibonacci_positive_specialization_in_theorem}; see
\Cref{prop:c_t_from_x_y}. Then the associated $n$-th
Stieltjes moment is given by
\begin{equation}
\label{eq:tp-moment-formula-restatement}
a_n = \sum_{\pmb{\varkappa} \ssp \in \ssp \mathrm{Comp}(n)} \,
N(\pmb{\varkappa}) \ssp
\boldsymbol{c}^{\ssp \mathrm{A}(\pmb{\varkappa} ) } \ssp \boldsymbol{t}^{\ssp
\mathrm{dep} (\mathrm{B}(\pmb{\varkappa}))},
\end{equation}
where
$N(\pmb{\varkappa}) =
\# \ssp \{ \pi \in \mathrm{NC}(n) :
\boldsymbol{\mathcal{z}}(\pi) =
\pmb{\varkappa} \}$,
and where $\mathrm{A}(\pmb{\varkappa})$ and $\mathrm{dep}(\mathrm{B}(\pmb{\varkappa}))$
are gvien in
\Cref{def:splitting+composition,def:composition+depletion}.
\end{corollary}

\subsection{Multiplicity matrix}

Let us discuss the multiplicities $N(\pmb{\varkappa})$ appearing in \Cref{eq:tp-moment-formula-restatement}.
First,
by \Cref{prop:noncrossing+composition+surjection},
this multiplicity is nonzero for all $\pmb{\varkappa} \in \mathrm{Comp}(n)$.
Thus, there are $2^{n-1}$ summands in \eqref{eq:tp-moment-formula-restatement} all of which are positive.
Since
\begin{equation*}
\mathrm{Catalan}(n) = \# \ssp \mathrm{NC}(n) = \sum_{\pmb{\varkappa} \in \mathrm{Comp}(n)} N(\pmb{\varkappa}),
\end{equation*}
the Stieltjes moments $a_n$ may be thought of as defining a multi-variate
extension of the Catalan numbers
in the parameters $\vec{c}$ and
$\vec{t}$.
If we specialize $c_k=1$ (or $t_k=1$) for all $k \geq 1$, then the polynomial expression for $a_n$ in \eqref{eq:tp-moment-formula-restatement} collapses to a sum of $F_n$ distinct monomials, each with a positive integer coefficient, where $F_n$ is the $n$-th Fibonacci number. This observation leads us to the
following construction.

\begin{definition}[Multiplicity matrix]
\label{def:matrix+multiplicities}
For a pair of Fibonacci words
$u, v \in \Bbb{YF}_n$, set
$N_{u,v} = N(\pmb{\varkappa})$
if $u = u(\pmb{\varkappa})$
and $v = v(\pmb{\varkappa})$
for some
composition $\pmb{\varkappa} \in \mathrm{Comp}(n)$. By \Cref{lemma:injection+compositions+splitting},
this composition is unique if it exists.
Otherwise, set $N_{u,v} = 0$.
Order the Fibonacci words
in $\Bbb{YF}_n$ lexicographically,
and arrange the multiplicities
$N_{u,v}$ into an
$\Bbb{YF}_n \times \Bbb{YF}_n$
matrix denoted by $\mathrm{\bf N}_n$.
By
\Cref{lemma:injection+compositions+splitting},
the diagonal
entry $N_{u,u}$ is nonzero
for each $u \in \Bbb{YF}_n$.
See \Cref{ex:compressed-moment-formula}
for an illustration.
\end{definition}

By \Cref{prop:noncrossing+composition+surjection},
the matrix $\mathrm{\bf N}_n$ has $2^{n-1}$
nonzero entries whose sum is
$\mathrm{Catalan}(n)=\# \ssp \mathrm{NC}(n)$.
The following
conjecture makes use of the {\it dominance order}
$\succeq$
on Fibonacci words $u,v \in \Bbb{YF}_n$.
Recall from \cite{okada1994algebras} that $u \succeq v$ if $u=  u_1 \cdots u_p$
and $v = v_1 \cdots v_q$ in $\Bbb{YF}_n$,
and $u_1 + \cdots + u_k \geq v_1 + \cdots + v_k$
for all $1 \leq k \leq \min(p,q)$.

\begin{conjecture}
\label{conj:splitting+dominance+upper-triangular}
The matrix $\mathrm{\bf N}_n$ is upper triangular. Furthermore, if
$N_{u,v}$ is nonzero then $u \succeq v$ and
$\mathcal{h}(u)=\mathcal{h}(v)$, where
$\mathcal{h}(w)$
denotes the {\bf total hike},
i.e., the number of $2$'s appearing in
a Fibonacci word $w$.
\end{conjecture}

\begin{problem}
\label{prob:counting-N}
Calculate $N(\pmb{\varkappa})$
for any
$\pmb{\varkappa}
\in \mathrm{Comp}(n)$ directly
in terms of the composition itself.
Alternatively, calculate $N_{u,v}$
in terms of the pair
of Fibonacci words $u, v \in \Bbb{YF}_n$.
In particular find necessary and
sufficient conditions for when
$N_{u,v}$ vanishes.
\end{problem}

\begin{figure}[htpb]
\centering
	\[
	\begin{array}{|c|cccccccc|}
	\hline
	N_{u,v} & \mathrm{221} & \mathrm{212} & \mathrm{2111} & \mathrm{122} & \mathrm{1211} & \mathrm{1121}
	& \mathrm{1112} & \mathrm{11111} \\
	\hline
	\mathrm{221} & \boldsymbol{\textcolor{blue}{1}} & \mathrm{0} & \mathrm{0}
	&\boldsymbol{\textcolor{blue}{1}}& \mathrm{0}
	& \mathrm{0} & \mathrm{0} & \mathrm{0} \\
	\mathrm{212} & \mathrm{0} & \boldsymbol{\textcolor{blue}{2}} & \mathrm{0}
	& \boldsymbol{\textcolor{blue}{2}}
	& \mathrm{0}
	& \mathrm{0} & \mathrm{0} & \mathrm{0} \\
	\mathrm{2111} & \mathrm{0} & \mathrm{0} & \boldsymbol{\textcolor{blue}{1}} & \mathrm{0}
	& \boldsymbol{\textcolor{blue}{3}}
	& \boldsymbol{\textcolor{blue}{3}}
	& \boldsymbol{\textcolor{blue}{1}}
	& \mathrm{0} \\
	\mathrm{122} & \mathrm{0} & \mathrm{0}
	& \mathrm{0}
	& \boldsymbol{\textcolor{blue}{2}} & \mathrm{0}
	& \mathrm{0} & \mathrm{0} & \mathrm{0} \\
	\mathrm{1211} & \mathrm{0} & \mathrm{0} & \mathrm{0} & \mathrm{0} & \boldsymbol{\textcolor{blue}{4}}
	& \boldsymbol{\textcolor{blue}{6}}
	& \boldsymbol{\textcolor{blue}{2}}
	& \mathrm{0} \\
	\mathrm{1121} & \mathrm{0} & \mathrm{0}
	& \mathrm{0} & \mathrm{0} & \mathrm{0}
	& \boldsymbol{\textcolor{blue}{6}}
	& \boldsymbol{\textcolor{blue}{3}}
	& \mathrm{0} \\
	\mathrm{1112} & \mathrm{0} & \mathrm{0}
	& \mathrm{0}  & \mathrm{0} & \mathrm{0}
	& \mathrm{0} & \boldsymbol{\textcolor{blue}{4}} & \mathrm{0} \\
	\mathrm{11111} & \mathrm{0} & \mathrm{0} & \mathrm{0} & \mathrm{0} & \mathrm{0}
	& \mathrm{0} & \mathrm{0} & \boldsymbol{\textcolor{blue}{1}} \\
	\hline
	\end{array}
	\]
	\caption{The coefficients $N_{u,v}$ from
    \Cref{def:matrix+multiplicities}
    when $n=5$. We have $\# \ssp \mathrm{NC}(5) = 42 = 5 \cdot
	\boldsymbol{\textcolor{blue}1} + 4 \cdot
	\boldsymbol{\textcolor{blue}2} + 3 \cdot
	\boldsymbol{\textcolor{blue} 3} + 2 \cdot
	\boldsymbol{\textcolor{blue} 4} + 2 \cdot
	\boldsymbol{\textcolor{blue} 6}$ and $2^4 = 16 = 5 + 4 + 3
	+ 2 + 2$. Upper triangularity is a manifestation of \Cref{conj:splitting+dominance+upper-triangular}.}
\label{ex:compressed-moment-formula}
\end{figure}

\section{Moments for Fibonacci Positive Specializations}
\label{sec:Fibonacci_positive_moments}

Let us now consider the Stieltjes moments in the case when
the specialization $(\vec{x}, \vec{y} \ssp )$ is Fibonacci
positive and of divergent type (see \Cref{def:finite_infinite_type_specializations}).
We continue to use the definitions and notation from the previous
\Cref{sec:Stieltjes_moments_general}.

In this case, we can further
express the $n$-th Stieltjes moment $a_n$ as a sum over all
set partitions of $n$. The next result is a multivariate
version of the enumerative formula
\begin{equation*}
 \# \ssp \Pi(n) \, = \,
\sum_{\frak{m} \ssp \in \ssp \frak{M}_n}
\ssp \prod_{k \ssp \geq \ssp 1} \,
k^{\ssp
\# \ssp \mathrm{down}_k(\frak{m})}
\ssp \prod_{k \ssp \geq \ssp 0} \,
(1+k)^{\ssp \# \ssp \mathrm{hor}_k
(\frak{m})},
\end{equation*}
where $\mathrm{hor}_k(\frak{m})$
and $\mathrm{down}_k(\frak{m})$,
respectively, count the number of
$\rightarrow$ and $\searrow$
steps at height $k$
taken by a length-$n$ Motzkin path
$\frak{m} \in \frak{M}_n$.

\begin{proposition}
\label{prop:fp-divergent-moment-formula}
Let $(\vec{x}, \vec{y} \ssp )$
be a Fibonacci positive specialization
of divergent type
\begin{equation}
		x_k=c_k\ssp (k +\epsilon_1 + \cdots + \epsilon_{k-1}),\qquad y_k=c_k\ssp c_{k+1}\ssp (k + \epsilon_1 +
        \cdots + \epsilon_k), \qquad k\ge1,
\end{equation}
where $\vec{c}$ and $\vec{\epsilon}$
are sequences of positive real
numbers uniquely determined by \eqref{eq:Fib_via_epsilons}
and \Cref{corollary:coefficientwise-total-positivity}.
Then
the associated $n$-th Stieltjes moment
is given by
\begin{equation*}
a_n = \sum_{\pi \ssp \in \ssp
\Pi(n)} \,
\prod_{k \ssp \geq \ssp 1} \,
c_k^{\ssp \ell_{k-1}(\pi)} \ssp (1 + \epsilon_k)^{\ssp g_k(\pi)}.
\end{equation*}
\end{proposition}
\Cref{prop:fp-divergent-moment-formula} is not
a corollary of \Cref{prop:tp-moment-formula} or \Cref{prop:tp-moment-formula-restatement},
although its proof is similar in spirit.
Note that here the sum ranges over all set partitions,
whereas in \Cref{prop:tp-moment-formula}
the sum ranges only over non-crossing set partitions.
\begin{proof}[Proof of \Cref{prop:fp-divergent-moment-formula}]
Since
$y_k = c_k c_{k+1} (k + \epsilon_1
+ \cdots + \epsilon_k)$
can be factored as $c_k$ and
$c_{k+1} (k + \epsilon_1 + \cdots
+ \epsilon_k)$,
we modify the rule
described in
\Cref{sec:Stieltjes_moments_general}
and assign weights to steps
of a Motzkin path $\frak{m} \in \frak{M}_n$
as follows:
\begin{enumerate}[$\bullet$]
\item each $\nearrow$ step at height $k$ is
assigned weight $c_k$;
\item each $\searrow$ step at height $k$
is assigned weight $c_{k+1} (k + \epsilon_1
+ \cdots + \epsilon_k)$;
\item each $\rightarrow$ step at height $k$ is
assigned weight $x_{k+1} = c_{k+1} + c_{k+1} (k +
    \epsilon_1 + \cdots + \epsilon_k)$.
\end{enumerate}
Likewise, we introduce weights $\omega(\frak{h}_\pi)$
for each Charlier histoire $\frak{h}_\pi \in \frak{H}_n$,
where $\pi \in \Pi(n)$ is an
arbitrary set partition.
Specifically, define $\omega(\mathfrak{h}_\pi)$
as the product of the weights of
its (colored) steps, where the weights are given as follows:
\begin{enumerate}[$\bullet$]
    \item each $\nearrow$ step at height $k$ is assigned weight $c_k$;
    \item each $\rightarrow$ step at height $k$ is
        assigned weight $c_{k+1}$ if $\chi = 0$, and $c_{k+1}(1 + \epsilon_\chi)$ if
        $\chi \in \{1, \dots, k \}$;
    \item each $\searrow$ step at height $k$
    with color $\chi \in \{1, \dots, k \}$ is
			assigned weight $c_{k+1}(1 + \epsilon_\chi)$.
\end{enumerate}
As in the proof of \Cref{prop:tp-moment-formula},
this system of weights for
Charlier histoires is designed to
be consistent with the Motzkin weights
and the projection map
$\mathrm{pr}_n: \frak{H}_n
\rightarrow \frak{M}_n$, i.e.,
\[
	\mathrm{wt}(\frak{m})
	\, = \, \sum_{\stackrel{\scriptstyle \pi \ssp \in \ssp
	\Pi(n)}{\mathrm{pr}_n(\frak{h}_\pi) \ssp = \ssp \frak{m}}}
	\omega(\frak{h}_\pi).
\]
The same reasoning
as in \Cref{prop:tp-moment-formula}
shows that
\[ \omega(\frak{h}_\pi)
= \prod_{k \ssp \geq \ssp 1} \,
c_k^{\ssp \ell_{k-1}(\pi)} \ssp (1+ \epsilon_k)^{\ssp g_k(\pi)} \]
for any $\pi \in \Pi(n)$.
\end{proof}

\begin{remark}
It would be interesting to find an analogue
of \Cref{prop:fp-divergent-moment-formula} for the convergent case of Fibonacci positive specializations.
\end{remark}

The moment polynomials in
\Cref{prop:tp-moment-formula,prop:fp-divergent-moment-formula}
(depending on the variables $c_k,t_k$ or $c_k,\epsilon_k$)
should not be confused with the well-known multivariate Bell
polynomials \cite{bell1934exponential}, even though the
univariate Bell polynomials arise as Stieltjes moments for the Charlier Fibonacci positive specialization. We discuss the latter
fact, along with its variants for other families of
orthogonal polynomials, in
\Cref{sec:Fibonacci_and_Stieltjes_examples} below.

\section{Fibonacci Positive Examples from the Askey Scheme}
\label{sec:Fibonacci_and_Stieltjes_examples}

Here we consider a number of examples of Fibonacci positive
specializations introduced in \Cref{sec:examples_of_Fibonacci_positivity},
and connect them to
orthogonal polynomials from the ($q$-)Askey scheme.

\subsection{Charlier specialization}

\subsubsection{Orthogonal polynomials}
For $\rho \in (0,1]$, set $x_k = \rho + k - 1$ and $y_k = \rho k$ for all $k \geq 1$. In this case, the orthogonal polynomials satisfy the three-term recurrence
\begin{equation*}
	P_{n+1}(t) = \big(t - \rho - n \big) P_n(t) - \rho\ssp n\ssp P_{n-1}(t).
\end{equation*}
These are readily recognized as
the classical Charlier polynomials. The associated
orthogonality measure is the Poisson distribution
$\upnu_{\scriptscriptstyle \mathrm{Pois}}^{\scriptscriptstyle (\rho)}$
\eqref{eq:Poisson-measure-definition} with the
parameter $\rho$. This measure is
supported on $\mathbb{Z}_{\geq 0}$.
For more details on Charlier polynomials we refer to
\cite[Chapter~1.12]{Koekoek1996}.

\subsubsection{Stieltjes moments}
The moments
$a_n$
of $\upnu_{\scriptscriptstyle \mathrm{Pois}}^{\scriptscriptstyle (\rho)}$
are the Bell polynomials (sometimes called Touchard
polynomials), which have the combinatorial interpretation
\begin{equation}
	\label{eq:Charlier-moments}
	a_n=B_n(\rho) \coloneqq \sum\nolimits_{\pi \in \Pi(n)} \rho^{\# \mathrm{blocks}(\pi)}.
\end{equation}
where $\# \mathrm{blocks}(\pi)$ counts the number of blocks in a set partition $\pi \in \Pi(n)$.

Observe that the
moment generating function $M(z)$
\eqref{eq:moment-generating-function_theorem}
is the confluent hypergeometric function
$_1F_1\left(1; 1 - \frac{1}{z}; -\rho \right)$
(see \eqref{eq:hypergeometric_function} for the notation).

\begin{remark}
The moments \eqref{eq:Charlier-moments} can be recovered
from \Cref{prop:fp-divergent-moment-formula}. To do so,
observe that $c_k = \rho$ and $1+\epsilon_k = 1/\rho$ for all $k \geq 1$ and use the identities
$\sum_{k \ssp \geq \ssp 1} \ell_{k-1}(\pi) = n$
and $\sum_{k \ssp \geq \ssp 1}
g_k(\pi) = \# \big( \mathcal{C}(\pi)
\cup \mathcal{T}(\pi) \big)$
along with the fact that $n - \# \big( \mathcal{C}(\pi) \cup \mathcal{T}(\pi) \big) = \# \mathrm{blocks}(\pi)$
for any set partition $\pi \in \Pi(n)$.
\end{remark}

\subsection{Type-I Al-Salam--Carlitz specialization}

For $\rho, q \in (0,1]$, define $x_k = \rho \, q^{k-1} + [k-1]_q$ and
$y_k = \rho \, q^{k-1} [k]_q$ for $k \geq 1$.

\subsubsection{Orthogonal polynomials}
The corresponding orthogonal polynomials satisfy
\begin{equation*}
P_{n+1}(t) = \big(t - \rho \ssp q^n - [n]_q \big) P_n(t)
- \rho\ssp q^{n-1}[n]_q P_{n-1}(t),
\end{equation*}
and appear in \cite{deMedicisStantonWhite1995}, where they are denoted
$P_n(t) = C^{(\rho)}_n(t ; q)$. These polynomials can be identified with
the Type-I Al-Salam--Carlitz polynomials $U^{(a)}_n(x;q)=U_n(x, a ; q)$
via a change of variables and parameters, namely,
\begin{equation}
\label{eq:Al-Salam-Carlitz-specialization}
P_n(t) = C^{(\rho)}_n(t ; q) = \rho^n \, U_n
\bigg( \frac{t}{\rho} - \frac{1}{\rho(1-q)}
, \frac{-1}{\rho(1-q)} ; q \bigg).
\end{equation}
See \cite[Chapter~3.24]{Koekoek1996} for more details on the
Al-Salam--Carlitz polynomials $U_n(x, a ; q)$.
In particular, they satisfy the orthogonality relation
\[ \int_a^1 w_{\scriptscriptstyle U}(x, a; q)
\ssp U_m(x, a ; q)
\ssp U_n(x, a ; q) \, d_qx
\, = \, (1-q)(-a)^n q^{n(n-1)/2} (q;q)_n \ssp \mathbf{1}_{m=n} \]
where
\[\displaystyle  w_{\scriptscriptstyle U}(x, a; q)
\, := \,
{(qx;q)_\infty \ssp (qx/a;q)_\infty
\over{(q;q)_\infty \ssp (a;q)_\infty \ssp (q/a;q)_\infty}}
\]
and where
\[ \int_a^1 f(x) \, d_qx \, := \,
(1-q) \Bigg( \sum_{n \ssp \geq \ssp 0} f\big(q^n \big)q^n
\, - \, a \sum_{n \ssp \geq \ssp 0}
f\big(aq^n\big)q^n \Bigg)\]
is the Jackson integral. As usual,
\[
(z;q)_k \coloneqq (1-z)(1-zq)\cdots(1-zq^{k-1}),
\quad \textnormal{and} \quad
(z;q)_\infty \coloneqq \prod\nolimits_{k \ge 0} (1-zq^k)
\]
are $q$-Pochhammer symbols.

In \eqref{eq:Al-Salam-Carlitz-specialization}, we have $q \in (0,1)$, and the parameter $a$ must be negative.
This is consistent with
the change of variables used in
\cite{deMedicisStantonWhite1995}, and with the range of
values $\rho, q \in (0,1]$ we consider.
In terms of the variable $t$, the nonnegative Borel measure
corresponding to the Type-I Al-Salam--Carlitz
Fibonacci positive specialization
is
supported by the discrete subset
\begin{equation*}
\bigl\{ [k]_q \ssp, \rho \ssp q^k + \frac{1}{1-q} \bigr\}_{k\ge0} \subset \mathbb{R}_{\scriptscriptstyle \geq 0}.
\end{equation*}

\subsubsection{Stieltjes moments}

The $n$-th moment $a_n$ of the orthogonality measure for
$P_n(t)$
is given by a $q$-variant of the Bell polynomial, and
can also be expressed as a generating function for set partitions:
\begin{equation*}
a_n=
B_{q,n}(\rho) \coloneqq \sum\nolimits_{\pi \in \Pi(n)} \rho^{\# \text{blocks}(\pi)}
q^{\text{inv}(\pi)},
\end{equation*}
which incorporates an additional $q$-statistic $\text{inv}(\pi)$
counting inversions in the set partition $\pi \in \Pi(n)$.
We refer to \cite{wachs1991pq}
and \cite{Zeng1993} for details.

\begin{remark}
Alternatively, we may apply
\Cref{prop:fp-divergent-moment-formula} using
$c_k = \rho \, q^{k-1}$ and $1 + \epsilon_k = \rho^{-1} q^{-k}$
for all $k \geq 1$. \Cref{prop:fp-divergent-moment-formula}, together with
\eqref{eq:nesting-formula}--\eqref{eq:area-statistic+formula}, implies that
\[ a_n = \sum\nolimits_{\pi \ssp \in \ssp \Pi(n)}
\rho^{\# \mathrm{blocks}(\pi)} \ssp
q^{\ssp \mathrm{area}(\pi)} \ssp
(1/q)^{\# ( \mathcal{C}(\pi) \ssp \cup \ssp
\mathcal{T}(\pi))} \ssp
(1/q)^{\mathrm{nest}(\pi)}. \]
This suggests that
\[ \mathrm{inv}(\pi) \, = \,
\mathrm{area}(\pi) \ssp - \ssp
\# \big( \mathcal{C}(\pi) \ssp \cup \ssp
\mathcal{T}(\pi) \big) \ssp - \ssp
\mathrm{nest}(\pi)
\]
for any set partition $\pi \in \Pi(n)$, but we do not prove this here.
\end{remark}

\subsection{Al-Salam--Chihara specialization}

Recall that the Al-Salam--Chihara specialization depends on two parameters
$\rho \in (0,1]$ and $q \in [1,\infty)$, and has
$x_k = \rho \, + [k-1]_q$ and $y_k = \rho \, [k]_q$ for $k \geq 1$.

\subsubsection{Orthogonal polynomials}
The orthogonal polynomials satisfy the three-term recurrence
\[
	P_{n+1}(t) = \big(t - \rho - [n]_q \big) P_n(t) - \rho [n]_q P_{n-1}(t).
\]
They appeared in
\cite{anshelevich2005linearization} and
\cite{KimStantonZeng2006}
under the notation
\(C_n(t, \rho ; q)\).
In the latter reference, $P_n(t)=C_n(t,\rho;q)$ were
identified with the
the Al-Salam--Chihara polynomials \( Q_n(x ; a ,b \mid q) \), after rescaling and incorporating a change of variables as follows:
\[
	P_n(t) = C_n(t, \rho ; q) = \left(\frac{\rho}{1-q}\right)^{n/2} Q_n \left( \frac{1}{2} \sqrt{\frac{1-q}{\rho}} \left(t - \rho - \frac{1}{1-q} \right); \, \frac{-1}{\sqrt{\rho(1-q)}}, 0 \,\Bigg|\, q \right).
\]
See \cite[Chapter~3.8]{Koekoek1996} for more details on the
Al-Salam--Chihara orthogonal polynomials $Q_n$.
Note that our our parameter $q$ is greater than one,
while in \cite{Koekoek1996} it is classically assumed
that $|q|<1$. Because of this, we cannot identify
the nonnegative Borel measure
$\upnu(dt)$
(which exists by \Cref{thm:stieltjes-moments-theorem}
and serves as the orthogonality measure for the $P_n(t)$'s)
with
the one coming from the Al-Salam--Chihara
polynomials in \cite{Koekoek1996}.

The existence of
different orthogonality measures for different ranges of parameters
is a known phenomenon, see, e.g.,
\cite{Askey1989} or
\cite{christiansen2004indeterminate}.
In particular, for $q>1$, the Al-Salam--Chihara
polynomials admit a different orthogonality measure
\cite{koornwinder2004onsalamchihara}. Let us
recall the necessary notation.
Denote
\begin{equation*}
	\widetilde{Q}_n(x;a,b\mid q)\coloneqq i^{-n}\ssp Q_n(ix;ia,ib\mid q),\qquad
	i=\sqrt{-1},\qquad n\ge0.
\end{equation*}
For $q>1$ we have, taking $b\to 0$ in \cite[(16)]{koornwinder2004onsalamchihara}:
\begin{equation}
	\label{eq:Al-Salam-Chihara-orthogonality-measure}
	\begin{split}
		&\sum_{k=0}^{\infty}
		\frac{1 + q^{-2k} a^{-2}}{1 + a^{-2}}
		\frac{(-a^{2}; q)_k}{(q; q)_k}\ssp a^{-4k} q^{\frac{3}{2}k(1-k)}
		\\&\hspace{120pt}\times
		\widetilde{Q}_n
		\bigl( \tfrac{1}{2} (a q^{k} - a^{-1} q^{-k}); a, 0 \mid q \bigr)
		\ssp\widetilde{Q}_m
		\bigl( \tfrac{1}{2} (a q^{k} - a^{-1} q^{-k}); a, 0 \mid q \bigr)
		\\&\hspace{30pt}=
		(-q^{-1}a^{-2};q^{-1})_\infty \ssp (-1)^n(q;q)_n
		\ssp
		\mathbf{1}_{m=n}.
	\end{split}
\end{equation}
One can readily verify that the weights
in \eqref{eq:Al-Salam-Chihara-orthogonality-measure} are nonnegative
when $a=-1/\sqrt{\rho(1-q)}$ and $q>1$. Moreover, matching the
variables in the polynomials, we see that the measure $\upnu(dt)$ is supported
on the following discrete set:
\begin{equation*}
	\bigl\{
		[k]_q+(1-q^{-k})\ssp \rho
	\bigr\}_{k\ge0}\subset \mathbb{R}_{\scriptscriptstyle\ge0}.
\end{equation*}

\subsubsection{Stieltjes moments}

The moment sequence $a_n$ of $\upnu(dt)$
can be derived from the continued fraction
$M(z)=J_{\, \vec{x}, \vec{y}} \,(z)$
\eqref{eq:moment-generating-function_theorem}--\eqref{eq:Jacobi-continued-fraction_theorem}. Moreover, in
\cite{KimStantonZeng2006}, the combinatorial interpretation of the $a_n$'s
was shown to be
\begin{equation*}
	a_n=B_{q,n}^{\mathrm{rc}}(\rho) \coloneqq \sum\nolimits_{\pi \, \in \, \Pi(n)} \rho^{\# \mathrm{blocks}(\pi)} q^{\mathrm{rc}(\pi)},
\end{equation*}
which includes a certain $q$-statistic counting the number
of restricted crossings in the set partition~$\pi$. We refer
to \cite{KimStantonZeng2006} for
the definition of $\mathrm{rc}(\pi)$.

\begin{remark}
	Again, we may appeal to
	\Cref{prop:fp-divergent-moment-formula}
	using the parameters $c_k = \rho$ and
	$1 + \epsilon_k = \rho^{-1} q^{k-1}$
	for all $k \geq 1$. Using \eqref{eq:nesting-formula}--\eqref{eq:area-statistic+formula}, we obtain
	\[ a_n = \sum\nolimits_{\pi \ssp \in \ssp \Pi(n)}
	\rho^{\# \mathrm{blocks}(\pi)} \ssp
	q^{\ssp \mathrm{nest}(\pi)}. \]
	This suggests that $\mathrm{nest}(\pi)$ coincides with the statistic
	$\mathrm{rc}(\pi)$ \cite{KimStantonZeng2006} for all $\pi \in \Pi(n)$, but we do not
	pursue a proof of this here.
\end{remark}

\subsection{$q$-Charlier specialization}

For $\rho, q \in (0,1]$, let
\[
x_k = \rho \ssp q^{2k-2} + [k-1]_q \big(1 + \rho (q-1) q^{k-2} \big)
\quad \text{and} \quad
y_k = \rho \ssp q^{2k-2} [k]_q \big(1 + \rho (q-1) q^{k-1} \big),
\qquad
k \geq 1.
\]

\subsubsection{Orthogonal polynomials}
The
orthogonal polynomials satisfy the three-term recurrence
\[
P_{n+1}(t) = \left( t - \rho \ssp q^{2n} - [n]_q \left(1 + \rho (q-1) q^{n-1} \right) \right) P_n(t)
- \rho \ssp q^{2n-2} [n]_q \left(1 + \rho (q-1) q^{n-1}
\right) P_{n-1}(t).
\]
They appear in
\cite{Zeng1993} under the notation \( V_n^{(\rho)}(t\, ; q) \).
Moreover, it follows from
\cite{Zeng1993} that these polynomials are related to the

After rescaling and a change of variables, these are exactly the $q$-Charlier polynomials from the Askey scheme \cite[Chapter~3.23]{Koekoek1996}. Namely, we have
(using the notation $C_n(x;a,q)$ instead of $C_n(q^{-x};a,q)$
as in \cite{Koekoek1996}):
\begin{equation}
	\label{eq:q-Charlier-identification}
	P_n(t) = V_n^{(\rho)}(t ; q) =
	(-\rho)^n q^{n(n-1)} C_n\left(
	(q - 1) \ssp t + 1; \rho(1-q^{-1}),q^{-1}\right).
\end{equation}
Like in the previous Al-Salam--Chihara case, here the $q$-Charlier
polynomials contain the parameter $q^{-1}>1$. Therefore,
the classical orthogonality measure \cite[Chapter~3.23]{Koekoek1996}
does not correspond to our nonnegative Borel measure $\upnu(dt)$.
Instead, we have the following orthogonality
for $q^{-1}$-Charlier polynomials:
\begin{equation*}
	\sum_{k=0}^{\infty}\frac{(-q)^{k}a^k}{(q;q)_k}\ssp
	C_n(q^k;a,q^{-1})\ssp C_m(q^k;a,q^{-1})=
	\frac{q^n \ssp (q^{-1};q^{-1})_n \ssp (-a^{-1}q^{-1};q^{-1})_n }
	{(-a\ssp q; q)_{\infty}}
	\ssp
	\mathbf{1}_{m=n}.
\end{equation*}
The $q$-Pochhammer symbol $(-a\ssp q;q)_\infty$ in the denominator
(as opposed to $(-a;q)_\infty$ in the numerator for
$0<q<1$, see \cite[(3.23.2)]{Koekoek1996})
comes from normalizing the orthogonality measure
to be a probability distribution.

In terms of our parameters, we have
$a=\rho(1-q^{-1})<0$, which ensures that the
orthogonality measure $\upnu(dt)$
for the polynomials $P_n(t)$
is nonnegative.
The support of $\upnu(dt)$
in $\mathbb{R}_{\scriptscriptstyle \geq 0}$
consists
of all $q$-integers $[k]_q$,
where
$k \in \mathbb{Z}_{\scriptscriptstyle \geq 0}$.

\subsubsection{Stieltjes moments}

The moments $a_n$ of $\upnu(dt)$ can be derived from the continued fraction
\eqref{eq:moment-generating-function_theorem}--\eqref{eq:Jacobi-continued-fraction_theorem}, and their combinatorial interpretation is
yet another $q$-variant of the Bell polynomials:
\[
	a_n=\widetilde{B}_{q,n}(\rho) \coloneqq \sum\nolimits_{\pi \in \Pi(n)} \rho^{\# \mathrm{blocks}(\pi)} \, q^{\widetilde{\mathrm{inv}}(\pi)}.
\]
Here
the statistic \( \widetilde{\mathrm{inv}}(\pi) \) is the number of so-called dual inversions of a set partition \( \pi \in \Pi(n) \). We refer to \cite{wachs1991pq} and \cite{Zeng1993} for details.

\subsection{Cigler-Zeng specialization}
\label{subsubsection:Cigler-Zeng-specialization+Al-Salam-Chihara}

For $q \geq q_0$ where $q_0 \approx 1.4656$ is
the unique root
of the equation $z^3 - z^2 = 1$, let
$x_k = q^{k-1}$ and
$y_k = q^k -1$ for all
$k \geq 1$.

\subsubsection{Orthogonal polynomials}
The orthogonal polynomials satisfy the three-term recurrence
\begin{equation}
\label{special-Cigler-Zeng-polynomials}
P_{n+1}(t) = \left( t - q^n \right)
P_n(t) - \left(q^n -1 \right) P_{n-1}(t).
\end{equation}
This is a special case of the more general recurrence
\begin{equation}
\label{general-Cigler-Zeng-polynomials}
P_{n+1}(t) = \left( t - xq^n \right)
P_n(t) + s[n]_q P_{n-1}(t),
\end{equation}
studied in \cite{cigler2011curious}. When $x=1$
and $s=1-q$, we recover \eqref{special-Cigler-Zeng-polynomials}.
It is shown in \cite{cigler2011curious} that solutions of
\eqref{general-Cigler-Zeng-polynomials}
can be expressed in terms of
Al-Salam-Chihara polynomials, namely as:
\[
	P_n(t) = \left(\frac{s}{q-1}\right)^{n/2} Q_n \left( \frac{t}{2} \sqrt{\frac{1-q}{s}}  \ssp ;
    x \sqrt{\frac{q-1}{s}} \ssp , 0 \,\Bigg|\, q \right).
\]
For $x=1$ and $s=1-q$,
we have $P_n(t) = i^n Q_n ( it/2 \ssp ; i , 0 \ssp | \ssp q) = \widetilde{Q}_n(t/2 \ssp ; 1 , 0
\ssp | \ssp q)$.
We require $q > q_0 >1$
which is outside the range of $q$-values
prescribed by the Askey scheme,
and so we must again use the
orthogonality
measure adapted for the
$\widetilde{Q}_n$-polynomials
given in
\eqref{eq:Al-Salam-Chihara-orthogonality-measure}.
When $x=1$ and
$s=1-q$, the support of $\upnu(dt)$
in $\mathbb{R}_{\scriptscriptstyle \geq 0}$
is $\{ q^k - q^{-k} : k \geq 0 \}$.

\subsubsection{Stieltjes moments}

The moments $a_n$ of $\upnu(dt)$ for general values of $x$ and $s$ are the chief concern of \cite{cigler2011curious}, where they are referred to as {\it new $q$-Hermite polynomials}. In general, $a_n = H_n (x, s \ssp | \ssp q)$ is expressed as a generating function for {\it incomplete matchings} of $n$, i.e., set partitions of $n$ whose blocks have size at most two. Specifically, they show
\[a_n \, = \,  H_n (x, s \ssp | \ssp q)
\, = \, \sum_{k \ssp = \ssp 0 }^n
\, \sum_{m \in M(n,k)} \,
q^{c(m)+ \mathrm{cr}(m)}
x^k
(-s)^{(n-k)/2},
\]
where the inner sum ranges over $M(n,k)$, the set of incomplete matchings $m$ of $n$ with $k$ unmatched vertices, and where $\mathrm{cr}(m)$ is the number of crossings of the matching $m$; see \cite[page~5]{cigler2011curious} for a definition of the $c(m)$ statistic.

\begin{remark}
Alternatively, when $x=1$ and $s=1-q$, we may apply \Cref{prop:fp-divergent-moment-formula}, using $c_k = 1$ and $1 + \epsilon_k = (q-1)q^{k-1}$, together with the observations
\Cref{lemma:nesting-formula,lemma:area-statistic+formula}
to obtain the following moment formula:
\begin{equation*}
a_n \, = \,
\sum\nolimits_{\pi \ssp \in \ssp \Pi(n)}
(q-1)^{\ssp \# \big(\mathcal{C}(\pi)
\ssp \cup \ssp \mathcal{T}(\pi) \big)}
\ssp q^{\ssp \mathrm{nest}(\pi)}.
\end{equation*}
\end{remark}

\subsection{Non-examples}

We showed that
Charlier, Al-Salam--Carlitz, Al-Salam--Chihara,
$q$-Char\-lier,
and the Cigler-Zeng Fibonacci positive specializations are related to
discrete orthogonality measures for orthogonal
polynomials of the respective types.
Notably, all the orthogonality measures we obtain
have \emph{discrete support}.
However, not all orthogonal polynomials correspond to Fibonacci positive specializations.

\subsubsection{Meixner polynomials}
The Meixner polynomials $P_n(t)$
are orthogonal with respect to the negative binomial distribution
$t\mapsto \frac{(\beta)_t c^t}{t!}$, $t\in \mathbb{Z}_{\ge0}$,
where $\beta>0$ and $0<c<1$ are parameters.
This distribution has discrete support.

The recurrence coefficients take the form
\begin{equation*}
	x_k=\frac{\beta c + (k - 1)(c + 1)}{1 - c},
\qquad
y_k=\frac{k(k + \beta - 1)c}{(1 - c)^2}.
\end{equation*}
Thus, the sequences $\vec t$ and $\vec c$ are
(see \Cref{prop:c_t_from_x_y}):
\begin{equation*}
	c_k=\frac{c\ssp (k-1+\beta)}{1-c},\qquad t_k=\frac{k}{c\ssp (k+\beta)}.
\end{equation*}
The coefficients $c_k$ are always positive.
Thanks to $c\in(0,1)$ in the denominator of $t_k$, the series
$A_\infty(m)$ \eqref{eq:A_infty_B_infty_series_definitions}
diverges. However,
then the inequalities $t_{k+1}-1-t_k\ge0$ take the form
\begin{equation*}
	c\le \frac{\beta}{(k+\beta)(k+1+\beta)}, \qquad k\ge1,
\end{equation*}
which implies that $c$ must be zero. This contradicts the assumption that $c>0$.
Thus, the Meixner polynomials
do not correspond to a Fibonacci positive specialization for any values of the parameters $\beta$ and $c$.

Meixner polynomials are
closely connected to a distinguished family of harmonic functions
(corresponding to the $z$-measures)
on the usual Young lattice $\mathbb{Y}$, which arise in harmonic analysis on
the infinite symmetric group \cite{borodin2006meixner}.

\subsubsection{Laguerre polynomials}

The Laguerre polynomials $P_n(t)$ are orthogonal with respect to the
Gamma distribution $t^\alpha e^{-t}\ssp dt$ on $\mathbb{R}_{\ge0}$, which is non-discrete.
Here $\alpha>-1$ is a parameter.
The recurrence coefficients are
\begin{equation*}
	x_k=2k+\alpha-1,\qquad y_k=k(k+\alpha).
\end{equation*}
One of the determinants in \eqref{eq:A_B_dets} equals $B_2(1) = -55 - 47 \alpha - 17 \alpha^{2} - \alpha^{3}$,
which is strictly negative for $\alpha>-1$.
Thus,
the Laguerre specialization is not Fibonacci positive for any parameter $\alpha>-1$.

\subsubsection{Discrete support of the orthogonality measure}

The examples and non-examples considered in
\Cref{sec:Fibonacci_and_Stieltjes_examples}
lead to the following question:
\begin{problem}
	\label{problem:discrete_support}
		Does there exist a Fibonacci positive specialization
		$(\vec{x}, \vec{y} \ssp)$ whose associated orthogonality
		measure $\upnu$ is continuous (rather than discrete)?
\end{problem}

A partial answer to this question for Fibonacci positive specializations of convergent type
can be deduced from \cite[pages 120-121]{Chihara1978}: If
$\lim_{k \rightarrow \infty} c_k < \infty$, then the
orthogonality measure associated to a Fibonacci positive
specialization of convergent type is discrete.
This observation provides a sufficient, not necessary,
condition for discreteness. We do not know an if-and-only-if
characterization of when the orthogonality measure is
discrete in the convergent case, nor do we have an analogous
result for specializations of divergent type.

\section{Shifted Charlier Specialization}
\label{sec:shifted_Charlier_Stieltjes}

Here we
consider the shifted Charlier specialization
$x_k = \rho + \sigma + k - 2$ and $y_k = (\sigma + k - 1)\rho$,
where \( k \geq 1 \) and \( \rho \in (0,1] \), \( \sigma \in [1, \infty) \),
see \Cref{def:shifted_Plancherel_and_Charlier}.
It is of divergent type when $\sigma=1$, and otherwise it is of
convergent type.
In this section, we discuss its Stieltjes moments and their generating function.

\subsection{Orthogonal polynomials for the shifted Charlier specialization}

The three-term recurrence for the orthogonal polynomials has the form
\begin{equation}
	\label{eq:shifted-Charlier-recurrence}
	P_{n+1}(t)
	=
	(t- \rho-\sigma-n+1)P_n(t)
	-
	\rho(\sigma+n-1)P_{n-1}(t).
\end{equation}
A similar recurrence is satisfied by the
so-called \emph{associated Charlier polynomials}
\cite{ismail1988linear},
\cite{ahbli2023new}:
\begin{equation*}
	a\ssp \mathscr{C}_{n+1}(x;a,\gamma)=
	(n+\gamma+a-x)\mathscr{C}_n(x;a,\gamma)
	-(n+\gamma)\mathscr{C}_{n-1}(x;a,\gamma).
\end{equation*}
Namely, we have the following identification:
\begin{equation}
	\label{eq:shifted-Charlier-identification}
	P_n(t)=(-\rho)^{n}\ssp\mathscr{C}_n(t; \rho,\sigma-1).
\end{equation}
The polynomials $P_n$
\eqref{eq:shifted-Charlier-recurrence}
can be expressed through
the hypergeometric function ${}_3F_2$
(see \eqref{eq:hypergeometric_function} for the notation).
This follows
from
\eqref{eq:shifted-Charlier-identification} and
\cite[(3.6)]{ahbli2023new}:
\begin{equation*}
	P_n(t)=
	\sum_{k=0}^{n} (-\rho)^{n-k} \ssp \frac{(-n)_k (\sigma-1 - t)_k}{k!}
	\ssp {}_3F_2 \left( \begin{array}{c} -k, \sigma-1, k-n \\ -n,
	\sigma-1-t\end{array} \ssp\middle|\ssp 1 \right).
\end{equation*}

\begin{remark}
	\label{rmk:Q_fake_shifted-Charlier}
	Denote by $Q_n(t)$ the orthogonal polynomials
	corresponding to the fake shifted Charlier
	specialization (\Cref{def:fake_shifted_Charlier}). One can check that they are related to the
	polynomials $P_n(t)$ via
	\begin{equation*}
	Q_n(t)=P_n(t;\sigma)+(\sigma-1) P_{n-1}(t;\sigma+1).
    \end{equation*}
		Since the properties of the polynomials $Q_n(t)$ are
		closely related to those of $P_n(t)$, we will not
		explore this other family of polynomials further.
\end{remark}

\subsection{Moment generating function}

Let us now discuss the moment generating function
$M(z)=J_{\, \vec{x}, \vec{y}} \,(z)$
\eqref{eq:moment-generating-function_theorem}--\eqref{eq:Jacobi-continued-fraction_theorem}
for the shifted Charlier specialization.
Define the fractional linear action of $2\times 2$ matrices on power series
$f(z)$~by
\begin{equation*}
\begin{pmatrix} a & b \\ c & d \end{pmatrix}
\boldsymbol{\cdot} f(z) \coloneqq \, {a \ssp f(z) + b \over {c\ssp f(z) + d}} .
\end{equation*}
\begin{lemma}
	\label{lemma:recurrence_for_power_series_shifted_Charlier}
	The moment generating function $M(z)=M(z;\rho,\sigma)$
	satisfies the following functional equation:
	\begin{equation}
		\label{eq:shifted-Charlier-continued-fraction-functional-equation}
		M(z;\rho,\sigma+1) =
		\begin{pmatrix} 1-(\sigma+\rho-1) \ssp z & -1 \\ \sigma\ssp \rho\ssp z^2 & 0 \end{pmatrix}
		\boldsymbol{\cdot}
		M(z;\rho,\sigma).
	\end{equation}
	In terms of the series coefficients $a_n=a_n(\rho,\sigma)$, this yields
	the quadratic recurrence
	\begin{equation}
	\label{eq:shifted-charlier-moment-recurrence}
		a_{n+1}(\rho,\sigma)  =  (\sigma+\rho-1)\ssp a_n(\rho,\sigma)  +
		\rho \ssp \sigma \sum_{k  =  0}^{n-1}   a_k(\rho,\sigma) \ssp
		a_{n-k-1}(\rho,\sigma +1 ),
	\end{equation}
	with the initial condition $a_0(\rho,\sigma)\equiv 1$.
\end{lemma}
In particular, when $\sigma=1$, we know that
$M(z;\rho,1)$ is the moment generating function for
the Poisson distribution \eqref{eq:Poisson-measure-definition}, and thus
$M(z;\rho,1)={}_1F_1(1;1-1/z;-\rho)$. We thus have
\begin{equation}
	\label{eq:shifted-Charlier-moment-generating-function-sequence-of-linear-fractional-transformations}
	M(z;\rho,k)=
	\begin{pmatrix}
	1 - (k-1) z & -1 \\
	(k-1) z^2 & 0
	\end{pmatrix} \, \cdots \,
	\begin{pmatrix}
	1 - z & -1 \\
	z^2 & 0
	\end{pmatrix}
	\, \boldsymbol{\cdot} \,
	{}_1F_1(1;1-1/z;-\rho).
\end{equation}
We are grateful to Michael Somos and Qiaochu Yuan
for helpful observations \cite{somos2022mathse} leading to
\Cref{lemma:recurrence_for_power_series_shifted_Charlier}.
\begin{proof}[Proof of \Cref{lemma:recurrence_for_power_series_shifted_Charlier}]
	The continued fraction for the shifted
	Charlier parameters has the form
	\begin{equation}
		\label{eq:shifted-Charlier-continued-fraction_in_proof}
		M(z; \rho,\sigma) =
		\frac{1}{1-(\sigma+\rho-1)\ssp z- \sigma\ssp \rho\ssp z^2 M(z;\rho,\sigma+1)},
	\end{equation}
	since the shifted sequences
	$(\vec{x}+1, \vec{y}+1)$
	correspond to the specialization under the shift $\sigma\mapsto \sigma+1$.
	Identity \eqref{eq:shifted-Charlier-continued-fraction_in_proof}
	is equivalent to the desired functional equation
	\eqref{eq:shifted-Charlier-continued-fraction-functional-equation}.

	The recurrence \eqref{eq:shifted-charlier-moment-recurrence}
	follows by writing the equation
	\eqref{eq:shifted-Charlier-continued-fraction_in_proof}
	as
	\begin{equation*}
		M(z;\rho,\sigma)\left( 1-(\sigma+\rho-1)\ssp z \right)=
		1+ \sigma\ssp \rho\ssp z^2 M(z;\rho,\sigma+1)\ssp M(z;\rho,\sigma),
	\end{equation*}
	and comparing the coefficients by $z^{n+1}$.
\end{proof}

\begin{remark}
	\label{rmk:meromorphic_Stieltjes}
	For integer values $\sigma=k\in \mathbb{Z}_{\ge1}$, the generating function
	$M(z; \rho, k)$ is derived by applying a sequence of
	fractional linear transformations
	\eqref{eq:shifted-Charlier-moment-generating-function-sequence-of-linear-fractional-transformations}
	to the meromorphic function
	$M(z; \rho, 1) = {}_1F_1(1; 1 - 1/z; -\rho)$. Thus,
	$M(z; \rho, k)$ is a meromorphic function of $z$.
	Consequently,
	the support of the measure $\upnu(dt)$ is
	discrete, which is likely the case also for all non-integer $\sigma>1$.
\end{remark}

We can solve the functional equation
\eqref{eq:shifted-Charlier-continued-fraction-functional-equation}
for $M(z;\rho,\sigma)$
in terms of the confluent hypergeometric function
$_1F_1$
(see \eqref{eq:hypergeometric_function} for the notation):

\begin{proposition}
	\label{prop:shifted_Charlier_moments}
	The moment generating function
	\( M(z) = M(z; \rho, \sigma) \)
	of the shifted Charlier specialization is given by
	\begin{equation}
		\label{eq:MGF_shifted_charlier_solution}
		M\big(z;\rho, \sigma\big)
		=
		\big(1 - z(\sigma-1) \big)^{-1} \,
        \frac{
		{_1F_1}\left(\sigma  ; \, \sigma - \frac{1}{z}  ; \, -\rho \right)
		}{
		{_1F_1}\left(\sigma - 1 ; \, \sigma -1 - \frac{1}{z}  ; \, -\rho \right)}
	\end{equation}
\end{proposition}
\begin{proof}
Equation \eqref{eq:shifted-Charlier-continued-fraction-functional-equation},
rewritten as the recurrence
\eqref{eq:shifted-charlier-moment-recurrence}
for the coefficients of a generating function in $z$
has a unique solution.
Moreover, identity
\cite[(13.3.13)]{NIST:DLMF}
implies that the right-hand side of
\eqref{eq:MGF_shifted_charlier_solution} satisfies the recurrence relation
\eqref{eq:shifted-Charlier-continued-fraction-functional-equation}.
Therefore, it remains to verify that the right-hand side
of \eqref{eq:MGF_shifted_charlier_solution}
is regular at $z=0$ (and hence is expanded as a power series in $z$).
See also \Cref{rmk:many_solutions_to_functional_equation} below
for examples of other solutions to \eqref{eq:shifted-Charlier-continued-fraction-functional-equation} which are not regular at $z=0$.

We have
\begin{equation*}
{}_1F_1\bigl(\sigma  ; \, \sigma - \tfrac{1}{z}  ; \, -\rho \bigr)
=
1+\sum_{r=1}^\infty
{\sigma (\sigma +1) \cdots( \sigma+r-1) \over
{(1 - \sigma z) \cdots (1 - (\sigma+r-1)z)}}
\, {\rho^r \over {r!}}  \, z^r,
\end{equation*}
and similarly,
\begin{equation}
	\label{eq:MGF_shifted_charlier_solution_2}
	\left( 1-z(\sigma-1) \right){_1F_1}
	\left(\sigma - 1 ; \, \sigma -1 - \tfrac{1}{z}  ; \, -\rho \right)
	=
	{_1F_1}\left(\sigma - 1 ;  \sigma - \tfrac{1}{z}  ;  -\rho \right)
- z (\sigma - 1)  {_1F_1}\left(\sigma  ;  \sigma - \tfrac{1}{z}  ;  -\rho \right)
\end{equation}
is a power series in $z$ with constant coefficient $1$.
Identity \eqref{eq:MGF_shifted_charlier_solution_2}
follows from \cite[(13.3.3)]{NIST:DLMF}.
Therefore, the right-hand side of \eqref{eq:MGF_shifted_charlier_solution}
is a power series in $z$, as desired.
\end{proof}

\begin{remark}
	\label{rmk:many_solutions_to_functional_equation}
	Curiosly, the functional equation
	\eqref{eq:shifted-Charlier-continued-fraction-functional-equation}
	has at least two solutions expressible as power series in \( z^{-1} \)
	with vanishing constant coefficient. If
	\(\mathcal{M}(z; \rho, \sigma) = \sum_{n \geq 1} m_n(\rho,\sigma) z^{-n}\)
	is a solution, then
	the recurrence relations for the coefficients take a different form:
	\begin{equation}
	\label{eq:variant-functional-equation}
		\begin{split}
			(1 - \rho - \sigma) \ssp m_1(\rho,\sigma) \, - \, 1
			&= \, \rho \ssp \sigma \ssp m_1(\rho, \sigma) \ssp m_1(\rho, \sigma +1),
			\\
			m_{n-1}(\rho,\sigma)
			\, - \, (\rho + \sigma -1)m_n(\rho, \sigma)
			&\displaystyle
			= \, \rho \ssp \sigma \sum_{k = 0}^{n-1}  \, m_{k+1}(\rho,\sigma) \,
			m_{n-k}(\rho,\sigma + 1),
			\qquad n \geq 2.
		\end{split}
	\end{equation}
	Two choices of valid initial conditions for \eqref{eq:variant-functional-equation} are
	\begin{equation*}
		m_1(\rho,\sigma)=(1-\sigma)^{-1}
	\end{equation*}
	and
	\[
	m_1(\rho,\sigma)
	 =
	\frac{\mathrm{U} \big(\sigma, \, \sigma, \, -\rho \big)}
	{(1 - \rho) \, \mathrm{U} \big(
	\sigma, \, \sigma -1, \, -\rho \big) \, - \,
	\rho \sigma
\mathrm{U} \big(\sigma +1, \, \sigma, \, -\rho \big)},
	\]
	where
	\[
	\mathrm{U}\big(\alpha, \,  \beta, \, \xi \big) \, = \, \frac{1}{\Gamma(\alpha)} \int_0^\infty
	e^{-\xi t} \, t^{\alpha -1} \, (1 +t)^{\beta - \alpha -1} \, dt
	\]
	is the Tricomi function.

	The Tricomi initial condition
	yields the solution
	\begin{equation}
		\label{eq:Tricomi-solution}
		\mathcal{M}(z; \rho, \sigma) =
		\frac{\mathrm{U} \big( \sigma, \, \sigma - 1/z, \, -\rho \big)}
		{(1 + z - z \rho) \ssp
		\mathrm{U} \big( \sigma, \, \sigma -1 - 1/z, \, -\rho \big) \, - \,
		z \ssp \rho \ssp \sigma\ssp
	\mathrm{U}\big( \sigma + 1, \, \sigma - 1/z, \, -\rho \big)}.
	\end{equation}
	The fact that \eqref{eq:Tricomi-solution}
	yields a solution to \eqref{eq:shifted-Charlier-continued-fraction-functional-equation}
	can be checked using contiguous relations
	similarly to the proof of \Cref{prop:shifted_Charlier_moments}.

	Solution \eqref{eq:Tricomi-solution} also arises by first decoupling equation
	\eqref{eq:shifted-Charlier-continued-fraction-functional-equation}
	through the Ansatz
	\begin{equation}
	\label{eq:decoupled-system}
	\begin{split}
		\mathcal{P}(z; \rho, \sigma+1)
		& = \, \big(1 - (\sigma + \rho -1) \ssp z \big)\mathcal{P}(z; \rho, \sigma) \, - \, \mathcal{Q}(z; \rho, \sigma), \\
		\mathcal{Q}(z; \rho, \sigma+1)
		& = \, \rho \ssp \sigma z^2
		\mathcal{P}(z; \rho, \sigma),
	\end{split}
	\end{equation}
	assuming that
	\(\mathcal{M}(z;\rho, \sigma) = \mathcal{P}(z; \rho, \sigma)/ \mathcal{Q}(z; \rho, \sigma)\). We then apply the Fourier transform to
	\eqref{eq:decoupled-system}, solve the resulting
	\(2 \times 2\) system of ordinary differential
	equations, and finally apply the inverse Fourier transform
	to return to the original function.

	It remains unclear whether other solutions to
	\eqref{eq:shifted-Charlier-continued-fraction-functional-equation}
	exist, or how they might be classified.
\end{remark}

\subsection{Stieltjes moments for the shifted Charlier specialization}

Here we describe the
moments $a_n(\rho,\sigma)$ for the shifted Charlier specialization in terms of bivariate statistics on set partitions.

\begin{definition}
\label{def:shifted_Charlier_statistics}
	For a set partition $\pi \in \Pi(n)$,
	let $\overline{g}_1(\pi)$ count
	the number of closers
	$i \in \mathbcal{C}(\pi)$ such that
	$\gamma_i(\pi) = 1$.
	Let $\#\ssp\mathrm{blocks}^\star(\pi)
	$ denote the number of non-singleton blocks in $\pi$, and $\#\ssp \mathbcal{S}(\pi)$ be the number
	of singletons in $\pi$.
\end{definition}

For the example
$\pi = 135 \big| 29 \big| 4 \big| 678$
(see \Cref{ex:arc-ensembles}, left),
we have
$\overline{g}_1(\pi)=1$,
$\#\ssp\mathrm{blocks}^\star(\pi)=3$,
and
$\#\ssp \mathbcal{S}(\pi)=1$.

\begin{proposition}
	\label{prop:shifted_Charlier_moments_combinatorial_interpretation}
	The $n$-th Stieltjes
	moment $a_n(\rho,\sigma)$ of the shifted Charlier specialization
	is given
	by the following
	variant of the Bell polynomial:
 \[
		a_n(\rho,\sigma) =
	\sum_{\pi \, \in \, \Pi(n)}
	\rho^{\# \mathrm{blocks}^\star(\pi)} \,
	\sigma^{\ssp \overline{g}_1(\pi) }\,
	\big( \rho + \sigma -1 \big)^{\#
	\mathbcal{S}(\pi)}.
	\]
\end{proposition}
\begin{proof}
	The proof is an adaptation of the
	method used to prove \Cref{prop:tp-moment-formula,prop:fp-divergent-moment-formula}.
	We start from the Motzkin path weighting in \Cref{sec:Stieltjes_moment_sequences},
	which assigns the weight $x_{k+1}=\rho+\sigma+k-1$ to a horizontal step
	at height $k$, and the weight $y_k=\rho(\sigma+k-1)$ to a $\nearrow$ step
	from height $k-1$ to height $k$.
	Then, we interpret this as a weighting scheme
	$\omega(\mathfrak{h})$ for Charlier histoires $\mathfrak{h} \in \mathfrak{H}_n$:
	\begin{enumerate}[$\bullet$]
		\item $\nearrow$ step is weighted by $\rho$;
		\item $\rightarrow$ step is weighted by
				$\rho + \sigma - 1$ if $\chi = 0$, or by $1$ if $\chi > 0$;
		\item $\searrow$ step is weighted by
			$\sigma$ if $\chi = 1$, or by $1$ if $\chi > 1$.
	\end{enumerate}
As before, let $\mathrm{pr}_n: \frak{H}_n \rightarrow \frak{M}_n$ be the projection map from
Charlier histoires to Motzkin paths which
simply forgets the histoire colors.
This projection yields the following weights for for the Motzkin paths:
\begin{enumerate}[$\bullet$]
\item each $\nearrow$ step at height $k$ is
weighted $\rho$;
\item each $\rightarrow$ step at height $k$ is
	weighted $\rho + \sigma +k  -1$;
\item each $\searrow$ step at height $k$
and is weighted $\sigma + k - 1$.
\end{enumerate}
Thus, our shifted Charlier moments have the form
$a_n(\rho, \sigma) =
\sum_{\frak{m} \in \frak{M}_n}
\mathrm{wt}(\frak{m})$.
Interpreting the weights through the statistics from
\Cref{def:shifted_Charlier_statistics}, we obtain the
desired result.
\end{proof}

\part{Random Fibonacci Words and Random Permutations}
\label{part:3}

In this part, we examine asymptotic behavior of clone
coherent measures on Fibonacci words arising from various
Fibonacci positive specializations introduced and discussed
in \Cref{part:1,part:2}.  In
\Cref{sec:initial_part_of_random_Fibonacci_word,sec:asymptotics_Charlier_specialization,sec:asymptotics_shifted_Plancherel_specialization},
we consider specializations under which the initial segment of
a random Fibonacci word has a large number of $1$'s or
$2$'s, and these numbers admit scaling limits. In
\Cref{sec:general_convergent_specializations}, we shift our
focus to specializations under which the random Fibonacci
words admit discrete-type asymptotics.  Finally, in
\Cref{sec:random_Fibonacci_words_to_random_permutations,sec:observables_from_Cauchy_identities},
we discuss how clone coherent measures on Fibonacci words
(and more general objects) can be used to define ensembles
of random permutations.  Using Cauchy identities for the
clone Schur functions, we extract asymptotic information
about the permutations.

\section{Initial Part of a Random Fibonacci Word}
\label{sec:initial_part_of_random_Fibonacci_word}

In this non-asymptotic section,
we obtain general identities
for the joint distributions (``correlations'') of sequences of $1$'s or $2$'s
in the initial segment of a random Fibonacci word
distributed according to an arbitrary clone coherent measure.

A Fibonacci word $w$ can be parsed in two
different ways. Looking at consecutive strings of 2's, define
$(h_1,h_2,\ldots )$ and $(\tilde h_1,\tilde h_2,\ldots )$ by
\begin{equation}
	\label{eq:hikes}
	w=2^{h_1}12^{h_2}1 \cdots 12^{h_m},
	\qquad
	h_j\in \mathbb{Z}_{\ge0};
	\qquad \qquad
	\tilde h_k\coloneqq
	\begin{cases}
		2h_k+1,& k\le m-1,\\
		2h_m,& k=m.
	\end{cases}
\end{equation}
The quantities $\tilde h_k$ appeared in \Cref{sub:Plancherel_measure_and_scaling} above.
Alternatively, we can look at consecutive strings
of 1's, and
define $(r_1,r_2,\ldots )$ and $(\tilde r_1,\tilde r_2,\ldots )$ by
\begin{equation}
	\label{eq:runs}
	w=1^{r_1}21^{r_2}2 \cdots 21^{r_p},\qquad r_j\in \mathbb{Z}_{\ge0};
	\qquad \qquad
	\tilde r_k\coloneqq
	\begin{cases}
		r_k+2,& k\le p-1,\\
		r_p,& k=p.
	\end{cases}
\end{equation}
In \eqref{eq:hikes} and \eqref{eq:runs}, the sequences $(h_1,h_2,\ldots )$ and $(r_1,r_2,\ldots )$
are called the \emph{hikes} and \emph{runs} of the word $w$, respectively.
We will use the shorthand notation
$r_{[i,j]}\coloneqq r_i+r_{i+1}+\ldots+r_j$, and similarly for
$\tilde r_{[i,j]}$, $h_{[i,j]}$, and $\tilde h_{[i,j]}$, and also for
open and half-open intervals.
In \eqref{eq:hikes} and \eqref{eq:runs}, the quantities $m$ and $p$ depend on $w$, and
we have
$\tilde h_{[1,m]}=\tilde r_{[1,p]}=|w|$.

Our goal is to obtain joint distributions for several initial runs or hikes
$r_j$ or $h_j$
under a clone coherent measure
\begin{equation}
	\label{eq:clone_coherent_measure}
	M_n(w)\coloneqq\dim (w) \cdot \varphi_{\vec x,\vec y} \, (w)
	=
	\dim(w) \cdot \frac{s_w(\vec x\mid \vec{y}\ssp)}{x_1\cdots x_{n} }
	,\qquad
	w\in \mathbb{YF}_n.
\end{equation}
As always, we assume that $x_i\ne 0$ for all $i$.
We start with runs:

\begin{proposition}
	\label{prop:runs_distribution}
	Fix $k\in \mathbb{Z}_{\ge1}$ and $r_1,\ldots,r_k\in \mathbb{Z}_{\ge0}$.
	Then for all $n \ge \tilde r_{[1,k]}$ we have
	\begin{equation}
		\label{eq:runs_distribution}
		M_n\left( w\colon r_1(w)=r_1, \ldots, r_k(w)=r_k \right)=
		\prod_{j=1}^{k}
		\frac{(n_j-r_j-1)\ssp B_{r_j}
		\bigl(n_j-r_j-2\bigr)}
		{x_{n_j}x_{n_j-1}\cdots x_{n_j-r_j-1}},
	\end{equation}
	where we denoted $n_j\coloneqq n-\tilde r_{[1,j)}$,
	and used the shorthand notation from \Cref{rmk:notation_Al_Blm}.
\end{proposition}
\begin{remark}
	\label{rmk:runs_distribution_not_sum_to_one}
	The sum over $r_1,\ldots,r_k $ of the quantities
	\eqref{eq:runs_distribution}
	is strictly less than $1$. Indeed, for example, if $k=1$,
	then the word $w$ must be of the form $1^{r_1}2u$,
	where the Fibonacci word $u$ is possibly empty.
	This
	excludes the possibility that $w=1^n$.
	See also
	\Cref{lemma:sum_to_one_rho} below for an explicit example.
\end{remark}

\begin{proof}[Proof of \Cref{prop:runs_distribution}]
	We have
	$M_n(w)=\dim (w)\ssp s_w(\vec x\mid \vec{y}\ssp)/(x_1\cdots x_n)$.
	Let $w=1^{r_1}2\cdots 1^{r_k}2 u$, where $u$ is a generic Fibonacci word with fixed weight $|u|=n-\tilde r_{[1,k]}\ge 0$.
	In particular, the event we consider in
	\eqref{eq:runs_distribution}
	requires the word $w$ to have at least $k$ letters~2, and $\tilde r_j=r_j+2$ for all $j=1,\ldots,k $.
	Using
	the recurrent definition \eqref{eq:clone_Schur_recurrence_def} of the clone Schur functions,
	we can write
	\begin{equation*}
		\frac{s_{w}(\vec x\mid \vec{y}\ssp)}{x_1\cdots x_n } \, = \,
		\frac{s_{u}(\vec x\mid \vec{y}\ssp)}{x_1\cdots x_{|u|} } \,
		\prod_{j=1}^{k} \,
		\frac{B_{r_j}
		\bigl(n-\tilde r_{[1,j]}\bigr)}
		{x_{n-\tilde r_{[1,j)}}\cdots \, x_{n-\tilde r_{[1,j]}+1} }.
	\end{equation*}
	Applying this relation to the Plancherel specialization and using
	\eqref{eq:Plancherel_clone_determinants}, we get
	\begin{equation*}
		\dim (w) \, = \,
		\dim (u) \cdot
        \prod_{j=1}^k
        \big(n- \tilde{r}_{[1,j]}+1 \big).
	\end{equation*}
	Summing $M_n(w)$ over all words $u$ eliminates the dependence on $u$ thanks to the probability normalization,
	and we obtain the desired product.
	Note that in the product in \eqref{eq:runs_distribution}, we changed the notation
	$n-\tilde r_{[1,j]}=n_j-r_j-2$.
\end{proof}

Let us turn to hikes.
Their joint distributions do not admit a simple product form
like \eqref{eq:runs_distribution}
due to runs of 1's arising for zero values of the hikes.
Let us denote
$d_j\coloneqq n-\tilde h_{[1,j)}$ (with $d_0=d_1=n$), and
recursively define
for $j=1,2,\ldots,m $:
\begin{equation}
	\label{eq:hikes_sequences_c_j}
	c_j\coloneqq
	\begin{cases}
		0,& \textnormal{if}\ j=1;\\
		c_{j-1}+1,& \textnormal{if}\ d_j=d_{j-1}-1;\\
		1,& \textnormal{otherwise}.
	\end{cases}
\end{equation}
The condition $d_j=d_{j-1}-1$ is equivalent to $h_{j-1}=0$.
For example, if $h=(2, 0, 0, 0, 0, 2, 0, 1)$,
then the word $w$ and
the sequences $d$ and $c$ have the following
form:
\begin{equation}
	\label{eq:hikes_example_sequences_c_d}
	w=221111122112,
	\qquad
	d=(17, 12, 11, 10, 9, 8, 3, 2),
	\qquad
	c=(0, 1, 2, 3, 4, 5, 1, 2).
\end{equation}

\begin{lemma}
	\label{lemma:hikes_product_for_Schur}
	Let a Fibonacci word $w=2^{h_1}1\cdots 2^{h_m}$
	be decomposed as in \eqref{eq:hikes}.
	Let $1\le k\le m$ be such that $h_k>0$. Then with the
	above notation $d_j,c_j$, we have
	\begin{equation}
		\label{eq:hikes_product_for_Schur}
			s_w(\vec x\mid \vec{y}\ssp)=s_u(\vec x\mid \vec{y}\ssp)\cdot
			\biggl( \, \prod_{i=1}^{k-1} \, \prod_{j=2}^{h_i} \, y_{d_i-2j+1}
            \! \biggr)
			\prod_{j=1}^{k}\frac{B_{c_j}\bigl(d_j-2\bigr)}
			{\mathbf{1}_{d_j\ne d_{j-1}-1}+
				B_{c_{j-1}}\bigl(d_{j}-1 \bigr)
				\ssp
				\mathbf{1}_{d_j=d_{j-1}-1}
			}
			,
	\end{equation}
	where $u=2^{h_{k}-1}1 2^{h_{k+1}}1\cdots 12^{h_m}$,
	and
	we used the shorthand notation from \Cref{rmk:notation_Al_Blm}.
\end{lemma}
For example, for the word in \eqref{eq:hikes_example_sequences_c_d}
and $k=6$, the last product in \eqref{eq:hikes_product_for_Schur} telescopes as
\begin{equation*}
	B_0(15)\ssp
	B_1(10)\ssp
	\frac{B_2(9)}{B_1(10)}\ssp
	\frac{B_3(8)}{B_2(9)}\ssp
	\frac{B_4(7)}{B_3(8)}\ssp
	\frac{B_5(6)}{B_4(7)}=B_0(15)B_5(6).
\end{equation*}
\begin{proof}[Proof of \Cref{lemma:hikes_product_for_Schur}]
	This is established similarly to the proof of \Cref{prop:runs_distribution}.
	The first
	product of the $y_j$'s in \eqref{eq:hikes_product_for_Schur} comes from the
	determinants $B_0$. The second product is telescoping
	to account for the recurrence involved in defining the clone Schur functions for words of the form
	$1^k2v$. Indeed, the entries of the sequence $c$
	are increasing by $1$ when there is a run of 1's in the word $w$
	(see the example in \eqref{eq:hikes_example_sequences_c_d}).
	This corresponds to the cases when $d_j=d_{j-1}-1$ in the denominator.
	Once the run of 1's ends, the next element of the sequence $c$ resets to $1$.
	Then $d_j\ne d_{j-1}-1$, the denominator is equal to $1$,
	and the index of the remaining determinant $B_{c_j}$ is precisely the length of the
	run of 1's in the word $w$.
	This completes the proof.
\end{proof}

\begin{proposition}
	\label{prop:hikes_distribution}
	Fix $k\in \mathbb{Z}_{\ge1}$ and $h_1,\ldots,h_k\in \mathbb{Z}_{\ge0}$.
	Then for all $n \ge \tilde h_{[1,k]}+2$ we have
	\begin{equation}
		\label{eq:hikes_distribution}
		\begin{split}
			&M_n\left( w\colon h_1(w)=h_1, \ldots, h_k(w)=h_k, h_{k+1}(w)>0 \right)
			\\&
			\hspace{5pt}=
			\biggl(\ssp\prod_{i=0}^{\tilde h_{[1,k]}+1}x_{n-i}^{-1}\biggr)
			\biggl(\prod_{i=1}^{k}\prod_{j=2}^{h_i}(d_i-2j+1)y_{d_i-2j+1}\biggr)
			\prod_{j=1}^{k+1}
			\frac{(d_j-1)B_{c_j}\bigl(d_j-2 \bigr)}
			{\mathbf{1}_{d_j\ne d_{j-1}-1}+
				d_j \ssp B_{c_{j-1}}\bigl(d_{j}-1 \bigr)
				\ssp
				\mathbf{1}_{d_j=d_{j-1}-1}
			},
		\end{split}
	\end{equation}
	where we use the notation $d_j,c_j$ introduced before
	\Cref{lemma:hikes_product_for_Schur}.
\end{proposition}
\begin{proof}
	Let $w=2^{h_1}1\cdots 1 2^{h_k}12v$, where $v=2^{h_{k+1}-1}1\cdots 12^{h_m}$.
	Here $v$ is a generic Fibonacci word with fixed weight $|v|=n-\tilde h_{[1,k]}-2\ge 0$.
	In particular, the event we consider in \eqref{eq:hikes_distribution}
	requires $w$ to have at least $k$ letters~1,
	and the number of hikes $m$ in
	\eqref{eq:hikes} satisfies $m\ge k+1$.

	Applying \Cref{lemma:hikes_product_for_Schur} twice --- once for
	$s_w(\vec x\mid \vec{y}\ssp)/(x_1\cdots x_n)$,
	and once for
	$s_w(\Pi)=\dim (w)$,
	we obtain the desired product times $\dim (v) \cdot s_v(\vec x\mid \vec{y}\ssp)/(x_1\cdots x_{|v|})$.
	Summing over the generic word $v$ eliminates the dependence on $v$ thanks to the probability normalization,
	and we obtain \eqref{eq:hikes_distribution}.
\end{proof}

Unlike for the runs in \Cref{prop:runs_distribution},
the result of \Cref{prop:hikes_distribution} does not uniquely
determine the joint distribution of the hikes $h_1,\ldots,h_k$.
Let us obtain an expression for the probability of the event $h_1=0$,
which will be useful for the scaling limit in
\Cref{sec:asymptotics_shifted_Plancherel_specialization} below.

\begin{lemma}
	\label{lemma:first_hike_distribution}
	For an arbitrary clone Schur measure $M_n$, we have
	\begin{equation*}
		M_n(w\colon h_1(w)=0)= \,
        1 \, - \, \frac{ (n-1) \ssp y_{n-1}}{x_{n-1}x_n}.
	\end{equation*}
\end{lemma}
\begin{proof}
	From the recurrent definition \eqref{eq:clone_Schur_recurrence_def} of the clone Schur functions,
	we get
	for any $v\in \mathbb{YF}_{n-2}$:
	\begin{equation*}
		M_n(w=2v)=\frac{(n-1)\ssp y_{n-1}}{x_{n-1}x_n}\ssp M_{n-2}(v).
	\end{equation*}
	Summing over all $v$ gives the probability that $h_1(w)>0$, and the result follows.
\end{proof}

\section{Asymptotics under the Charlier Specialization}
\label{sec:asymptotics_Charlier_specialization}

Consider the Charlier specialization
(\Cref{def:positive_specializations_class_II})
\begin{equation}
	\label{eq:defomed_Plancherel_specialization}
	x_k = k + \rho - 1 \quad \text{and} \quad y_k = k \rho ,\qquad  \rho \in (0, 1].
\end{equation}

\begin{definition}
	\label{def:eta_rho}
	For any $0<\rho<1$,
	let $\eta_{\rho}$ be a random variable on $[0,1]$ with the distribution
	\begin{equation}
			\label{eq:limit_distribution_deformed_Plancherel}
			\rho \ssp\delta_0(\alpha) + (1 - \rho)\ssp \rho (1 - \alpha)^{\rho - 1} \ssp d\alpha,\qquad
			\alpha\in[0,1].
	\end{equation}
	In words, $\eta_\rho$ is
	the convex combination of the point mass at $0$ and the Beta
	random variable $\mathrm{beta}(1, \rho)$, with weights
	$\rho$ and $1 - \rho$.
\end{definition}

Recall the run statistics $r_k(w)$ \eqref{eq:runs}, where $w$ is a Fibonacci word.

\begin{theorem}
	\label{thm:defomed_Plancherel_scaling}
	Let $w\in \mathbb{YF}_n$ be a random Fibonacci word distributed
	according to the deformed Plancherel measure $M_n$ \eqref{eq:clone_coherent_measure},
	\eqref{eq:defomed_Plancherel_specialization} with $0<\rho<1$.
	For any fixed $k\ge 1$, the joint distribution of the
	runs $(r_1(w), \ldots, r_k(w))$ has the scaling limit
	\begin{equation*}
		\frac{r_j(w)}{n-\sum_{i=1}^{j-1}r_i(w)}\xrightarrow[n\to\infty]{d}\eta_{\rho;j},\qquad j=1,\ldots,k,
	\end{equation*}
	where $\eta_{\rho;j}$ are independent copies of $\eta_\rho$.
\end{theorem}

Before proving \Cref{thm:defomed_Plancherel_scaling},
observe that we can
reformulate
this statement
in terms of the
residual allocation (stick-breaking) process, as in \Cref{def:GEM}:
\begin{equation*}
	\Bigl( \frac{r_1(w)}{n}, \frac{r_2(w)}{n}, \ldots \Bigr)
	\stackrel{d}{\longrightarrow}
	X=(X_1,X_2,\ldots ),
\end{equation*}
where $X_1=U_1$, $X_k=(1-U_1)\cdots (1-U_{k-1})U_k $ for $k\ge2$, and $U_k$
are independent copies of~$\eta_\rho$ (see \Cref{def:eta_rho}).
Unlike in the classical $\mathrm{GEM}$ distribution family, here
the variables $U_k$ can be equal to zero with positive probability $\rho$.
Thus, the random Fibonacci word
under the Charlier (deformed Plancherel) measure asymptotically
develops hikes of 2's of bounded length
(namely, these lengths are geometrically distributed with parameter~$\rho)$.
On the other hand, if we remove all
zero entries from the sequence
$X=(X_1,X_2,\ldots )$, then the resulting sequence is
distributed simply as $\mathrm{GEM}(\rho)$.

Note also that
for $\rho=1$, we have $U_k=1$ almost surely. This corresponds
to the fact that the deformed Plancherel measure reduces to
the usual Plancherel measure.
By \cite{gnedin2000plancherel} (see
\Cref{sub:Plancherel_measure_and_scaling}),
random Fibonacci words under the usual Plancherel measure have only a few $1$'s.
Thus, for $\rho=1$, the
scaling limit of the runs of $1$'s
is trivial, and instead one must consider the scaling limit of the hikes of $2$'s.
This is the subject of the next \Cref{sec:asymptotics_shifted_Plancherel_specialization}.

\medskip

In the rest of this section, we prove \Cref{thm:defomed_Plancherel_scaling}.
First, by \Cref{prop:runs_distribution}, we can express
the joint distribution of finitely many initial runs of 1's in a
random Fibonacci word in terms of a discrete
distribution
$\eta^{(m)}_\rho$ on $\left\{ 0,1,\ldots,m-1  \right\}$:
\begin{equation}
	\label{eq:runs_distribution_deformed_Plancherel_factor_eta_rho}
	\operatorname{\mathbb{P}}( \eta^{(m)}_\rho=r )=
	\begin{cases}\displaystyle
		\frac{(m-r-1)\ssp B_{r}
		\bigl(m-r-2\bigr)\ssp \Gamma(m+\rho-r-2)}
		{\Gamma(m+\rho)},& r=0,1,\ldots,m-2 ,\\\displaystyle
		\frac{\rho^m \Gamma(\rho)}{\Gamma(m+\rho)},& r=m-1.
	\end{cases}
\end{equation}
Here $B_r(m)$ are the determinants \eqref{eq:A_B_dets} with shifts (we use the notation of \Cref{rmk:notation_Al_Blm}).
By \Cref{lemma:sum_to_one_rho} which we establish below, we have
\begin{equation}
	\label{eq:runs_distribution_deformed_Plancherel_sum_not_one}
	\sum_{r=0}^{m-2}\frac{(m-r-1)\ssp B_{r}
	\bigl(m-r-2\bigr)\ssp \Gamma(m+\rho-r-2)}
	{\Gamma(m+\rho)}
	=1-\frac{\rho^m \Gamma(\rho)}{\Gamma(m+\rho)},
\end{equation}
so \eqref{eq:runs_distribution_deformed_Plancherel_factor_eta_rho}
indeed defines a probability distribution.

\Cref{prop:runs_distribution} states that the joint distribution
of a finite number of initial runs of 1's under the deformed Plancherel
measure has the product form
\begin{equation}
	\label{eq:runs_distribution_deformed_Plancherel_product_form}
	M_n\left( w\colon r_1(w)=r_1, \ldots, r_k(w)=r_k \right)=
	\prod_{j=1}^{k}
	\operatorname{\mathbb{P}}( \eta^{(n_j)}_\rho=r_j),
\end{equation}
where $n_j=n-\tilde r_{[1,j)}=
n-(2j-2)-r_1-\ldots-r_{j-1}$,
and $0\le r_j\le n_j-2$ for all $j=1,\ldots,k$.
By \eqref{eq:runs_distribution_deformed_Plancherel_sum_not_one},
we know that the sum of the probabilities
\eqref{eq:runs_distribution_deformed_Plancherel_product_form}
over all $r_j$ with $0\le r_j\le n_j-2$ is strictly less than $1$
(see also \Cref{rmk:runs_distribution_not_sum_to_one} and the proof of \Cref{lemma:sum_to_one_rho} below).
To get honest probability distributions,
we have artificially assigned the remaining
probability weights $\rho^{n_j} \Gamma(\rho)/\Gamma(n_j+\rho)$
to $r_j=n_j-1$.
Since
\begin{equation*}
	\rho^{n_j} \frac{\Gamma(\rho)}{\Gamma(n_j+\rho)}=
	\frac{\rho^{n_j}}{\rho(\rho+1)\cdots(\rho+n_j-1)}
\end{equation*}
rapidly
decays to $0$ as $n_j\to\infty$, these additional probability weights
can be ignored in the scaling limit.
More precisely,
by \Cref{lemma:B_def_Pl_scaling_limit} which we establish below, each random variable $\eta_\rho^{(n_j)}$, scaled by
$n_j^{-1}$, converges in distribution to $\eta_\rho$.
Thanks to the product form
of \eqref{eq:runs_distribution_deformed_Plancherel_product_form},
the scaled random variables
$r_j(w)/n_j$ become independent in the limit, and each of them converges in distribution to $\eta_\rho$.
This completes the proof of \Cref{thm:defomed_Plancherel_scaling}
modulo \Cref{lemma:sum_to_one_rho,lemma:B_def_Pl_scaling_limit}
which we now establish.

\begin{lemma}
	\label{lemma:sum_to_one_rho}
	Let $B_r(m)$ be the determinants \eqref{eq:A_B_dets} with shifts
	(\Cref{rmk:notation_Al_Blm}), and consider the Charlier
	specialization \eqref{eq:defomed_Plancherel_specialization} of the variables
	$x_i, y_i$.
	Then for any $m\ge2$,
	we have
	\begin{equation}
		\label{eq:runs_distribution_deformed_Plancherel_sum_to_one}
		\sum_{r=0}^{m-2}\frac{(m-r-1)\ssp B_{r}
		\bigl(m-r-2\bigr)\ssp \Gamma(m+\rho-r-2)}
		{\Gamma(m+\rho)}
		=1-\frac{\rho^m \Gamma(\rho)}{\Gamma(m+\rho)}.
	\end{equation}
\end{lemma}
\begin{proof}
	Consider the random word $w\in \mathbb{YF}_m$ under the deformed Plancherel measure $M_n$.
	From \Cref{prop:runs_distribution}, we know that the $r$-th summand in the left-hand side
	of \eqref{eq:runs_distribution_deformed_Plancherel_sum_to_one} is the probability
	that this word has the form $1^r2u$, for a (possibly empty) Fibonacci word $u$.
	Summing all these probabilities over $r=0,1,\ldots,m-2$, we obtain
	$1-M_n(w=1^m)$.
	We have
	\begin{equation*}
		M_n(w=1^m)=\frac{s_{1^m}(\Pi)\ssp s_{1^m}(\vec x\mid \vec{y}\ssp)}{\rho(\rho+1)\cdots(\rho+m-1) }=
		\frac{A_m(\Pi)\ssp A_m(\vec x\mid \vec{y}\ssp)}{\rho(\rho+1)\cdots(\rho+m-1)}=
		\frac{1\cdot \rho^m}{\rho(\rho+1)\cdots(\rho+m-1)}.
	\end{equation*}
	This completes the proof.
\end{proof}

\begin{lemma}
	\label{lemma:B_def_Pl_scaling_limit}
	Let $0<\rho<1$. Recall the distribution
	$\eta^{(m)}_\rho$
	\eqref{eq:runs_distribution_deformed_Plancherel_factor_eta_rho}.
	We have
	\begin{equation*}
		\frac{\eta^{(m)}_\rho}{m}\xrightarrow{d}\eta_\rho,\qquad m\to \infty,
	\end{equation*}
	where $\eta_\rho$ is described in \Cref{def:eta_rho}.
\end{lemma}
\begin{proof}
	Since
	$\operatorname{\mathbb{P}}(\eta^{(m)}_\rho=m-1)$
	rapidly decays to zero as $m\to\infty$, we can ignore
	this probability
	it in the limit.
	For an arbitrary specialization
	$(\vec{x},\vec{y}\ssp)$, the
	determinants $B_k(m)$
	satisfy the three-term recurrence:
	\begin{equation}
		\label{eq:B_k_general_three_term_recurrence}
		B_k(m)=x_{m+k+2}B_{k-1}(m)-y_{m+k+1}B_{k-2}(m), \quad k\ge2,
	\end{equation}
	with initial conditions $B_0(m)=y_{m+1}$ and $B_1(m)=x_{m+3}\ssp y_{m+1}-x_{m+1}\ssp y_{m+2}$.
	Substituting \eqref{eq:defomed_Plancherel_specialization}, we obtain
	\begin{equation}
		\label{eq:Bk_deformed_Plancherel_recurrence}
		\begin{split}
			B_k(m)&=(k+m+\rho+1)B_{k-1}(m)-\rho(k+m+1)B_{k-2}(m),
			\\
			B_0(m)&=\rho(m+1), \qquad B_1(m)=\rho(m+2-\rho).
		\end{split}
	\end{equation}
	This recurrence has a unique solution which
	has the form
	\begin{equation}
		\label{eq:Bk_deformed_Plancherel_via_Exp_integral}
		\begin{split}
			B_k(m)&=\rho^{k+1}(m+1)-\rho^{k+1}(1-\rho)(m+2)e^{-\rho}\ssp E_{m+3}(-\rho)
			\\&\hspace{80pt}+\rho(1-\rho)\frac{(m+k+2)!}{(m+1)!}\ssp e^{-\rho}\ssp E_{m+k+3}(-\rho),
		\end{split}
	\end{equation}
	where $E_r(z)$ is the exponential integral
	\begin{equation}
		\label{eq:exponential_integral_traditional}
		E_r(z)=\int_{1}^{\infty}t^{-r}e^{-zt}dt,\qquad r\ge0, \quad \Re(z)>0.
	\end{equation}
	Since our $z=-\rho<0$, formula
	\eqref{eq:exponential_integral_traditional} needs to be analytically continued
	\cite[\href{https://dlmf.nist.gov/8.19.E8}{(8.19.8)}]{NIST:DLMF}:
	\begin{equation}
		\label{eq:ExpIntegral_series}
		E_{r}(z)=\frac{(-z)^{r-1}}{(r-1)!}(\psi(r)-\ln z)-\sum_{
		\begin{subarray}{c}k=0\\
		k\neq r-1\end{subarray}}^{\infty}\frac{(-z)^{k}}{k!\ssp (1-r+k)},
		\qquad r=1,2,3,\ldots.
	\end{equation}
	Here $\psi(r)=\Gamma'(r)/\Gamma(r)$ is the digamma function.
	The logarithm $\ln z=\ln(-\rho)=\mathbf{i}\pi+\ln \rho$ has a branch cut,
	but all the summands in the remaining series are entire functions of $z$.
	Thus, formula \eqref{eq:ExpIntegral_series} produces the desired analytic continuation of $E_r(-\rho)$.

	Now,
	using the series representation
	\eqref{eq:ExpIntegral_series} for the exponential integral,
	we can show that
	\begin{equation}
		\label{eq:Exp_integral_asymptotics}
		\lim_{r\to+\infty}E_r(-\rho)=0,\qquad
		\lim_{r\to+\infty} r\ssp E_{r+1}(-\rho)=e^{\rho}.
	\end{equation}
	Indeed,
	$\psi(r)$ grows logarithmically with $r$,
	so the first summand is negligible as $r\to+\infty$ (even after multiplication by $r-1$).
	The series in \eqref{eq:ExpIntegral_series} converges uniformly in $z$, so
	we can take the limit of the individual terms and conclude that
	$E_r(-\rho)\to 0$ as $r\to+\infty$.
	The second limit in \eqref{eq:Exp_integral_asymptotics} follows from the
	recurrence
	$r\ssp E_{r+1}(z) + z\ssp E_r(z) = e^{-z}$
	\cite[(8.9.12)]{NIST:DLMF}.

	Assume that $r=\lfloor \alpha m \rfloor $, where $\alpha\in(0,1)$.
	By the standard Stirling asymptotics, the ratio
	$\Gamma(m+\rho-r-2)/\Gamma(m+\rho)$ in
	\eqref{eq:runs_distribution_deformed_Plancherel_factor_eta_rho} decays to zero as $e^{-\alpha\ssp m\ln m}$.
	Using \eqref{eq:Bk_deformed_Plancherel_via_Exp_integral}, \eqref{eq:Exp_integral_asymptotics},
	we see that
	\begin{equation*}
		(m-r-1)B_r(m-r-2)\sim
		(m-r-1)^2\rho^{r+1}-
		(m-r-1)
		\rho^{r+1}(1-\rho)
		+\rho(1-\rho)\frac{(m-1)!}{(m-r-2)!}.
	\end{equation*}
	The first two summands decay exponentially and are thus negligible since $\rho<1$.
	We have for the third summand:
	\begin{equation*}
		\frac{(m-1)!}{(m-r-2)!}\frac{\Gamma(m+\rho-r-2)}
		{\Gamma(m+\rho)}
		\ssp\rho(1-\rho)\sim
		m^{-1}\ssp
		\rho(1-\rho)\ssp(1-\alpha)^{\rho-1}
		.
	\end{equation*}
	The prefactor $m^{-1}$ corresponds to the scaling of our random variable $m^{-1}\eta^{(m)}_\rho$.
	Note that
	\begin{equation*}
		\int_0^1\rho(1-\rho)\ssp(1-\alpha)^{\rho-1} \ssp d\alpha=1-\rho,
	\end{equation*}
	and the remaining mass is concentrated at 0
	in the limit:
	\begin{equation*}
		\operatorname{\mathbb{P}}(\eta_\rho^{(m)}=0)=
		\frac{(m-1)^2 \rho }{(m+\rho -2) (m+\rho -1)}\to \rho, \qquad	m\to \infty.
	\end{equation*}
	This completes the proof of \Cref{lemma:B_def_Pl_scaling_limit},
	and finalizes the proof of \Cref{thm:defomed_Plancherel_scaling}.
\end{proof}

\section{Asymptotics under the Shifted Plancherel Specialization}
\label{sec:asymptotics_shifted_Plancherel_specialization}

Here we consider
both versions of the shifted Plancherel specialization
(\Cref{def:fake_shifted_Charlier,def:shifted_Plancherel_and_Charlier}):
\begin{equation}
	\label{eq:shifted_Plancherel_specialization_recall_two_at_the_same_time}
	y_k=k+\sigma-1,\qquad x_k=\begin{cases}
		k+\sigma-1,& k\ge 2,\\
		\textnormal{$1$ or $\sigma$},& k=1,
	\end{cases}
\end{equation}
where $\sigma\in [1,\infty)$
is a parameter.
Recall that the fake shifted Charlier specialization with $x_1=1$ is of divergent type,
while the case $x_1=\sigma$ is of convergent type when $\sigma>1$.

\subsection{Limiting distribution: Dependent stick-breaking}

First, let us introduce the family of random variables
which will serve as the scaling limit of the hikes of 2's
under the shifted Plancherel specialization.

\begin{definition}
	\label{def:xi_sigma}
	Let
	\begin{equation}
		\label{eq:beta_1_sigma_2_density}
		G(\alpha)\coloneqq 1-(1-\alpha)^{\frac{\sigma}{2}},\qquad
		g(\alpha)\coloneqq \frac{\sigma}{2}(1-\alpha)^{\frac{\sigma}{2}-1},
		\qquad
		\alpha\in [0,1],
	\end{equation}
	be the cumulative and density functions of the Beta distribution $\mathrm{beta}(1,\sigma/2)$.
	For any $\sigma\ge1$,
	let $\xi_{\sigma;1},\xi_{\sigma;2},\ldots $ be the sequence of random variables
	with the following joint cumulative distribution function (cdf):
	\begin{equation}
		\label{eq:joint_cdf}
		\operatorname{\mathbb{P}}
		\left( \xi_{\sigma;1}\le \alpha_1,\ldots,\xi_{\sigma;n}\le \alpha_n  \right)\coloneqq
		\sigma^{-n+1}G(\alpha_1)\cdots G(\alpha_n)
		+
		(\sigma-1)
		\sum_{j=1}^{n-1}\sigma^{-n+j}G(\alpha_1)\cdots G(\alpha_{n-j}).
	\end{equation}
	Denote the right-hand side of \eqref{eq:joint_cdf} by
	$F_n^{(\sigma)}(\alpha_1,\ldots,\alpha_n )$.
\end{definition}
\begin{lemma}
	\label{lemma:joint_cdf}
	The joint cdfs $F_n^{(\sigma)}$ for all $n\ge1$ are
	consistent, and uniquely define the distribution of
	$\xi_{\sigma;1},\xi_{\sigma;2},\ldots $.
	The marginal distribution of each $\xi_{\sigma;k}$ is
	\begin{equation*}
		( 1-\sigma^{-k+1} )\ssp\delta_0(\alpha)+\sigma^{-k+1}g(\alpha)\ssp d\alpha.
	\end{equation*}
	In particular, $\xi_{\sigma;1}$ is absolutely continuous and
	has the Beta distribution $\mathrm{beta}(1,\sigma/2)$,
	while $\xi_{\sigma;k}$ for each $k\ge0$ has an atom at $0$
	of mass $1-\sigma^{-k+1}$,
	and the remaining mass is distributed according to $\mathrm{beta}(1,\sigma/2)$.
\end{lemma}
When $\sigma=1$, the random variables $\xi_{\sigma;k}$
reduce to a collection of independent identically distributed
$\mathrm{beta}(1,1/2)$ random variables.
\begin{proof}[Proof of \Cref{lemma:joint_cdf}]
	Each $F^{(\sigma)}_n$ is a cdf, that is, it is continuous, increasing
	in each argument, satisfies the boundary conditions $F^{(\sigma)}_n(0,\ldots,0)=0$ and $F^{(\sigma)}_n(1,\ldots,1)=1$.
	The consistency
	\begin{equation*}
		F_n^{(\sigma)}(\alpha_1,\ldots,\alpha_{n-1},1 )=F_{n-1}^{(\sigma)}(\alpha_1,\ldots,\alpha_{n-1})
	\end{equation*}
	is straightforward.

	Let us check the nonnegativity of the rectangle probabilities
	under $F_n^{(\sigma)}$. If a rectangle is $n$-dimensional, then
	we can use the fact that
	\begin{equation}
		\label{eq:joint_cdf_derivative}
		\partial_{\alpha_1,\ldots,\alpha_n }F_n^{(\sigma)}(\alpha_1,\ldots,\alpha_n )=
		\sigma^{-n+1}\ssp g(\alpha_1)\cdots g(\alpha_n ),
	\end{equation}
	which produces nonnegative rectangle probabilities under $F_n^{(\sigma)}$
	by integration of \eqref{eq:joint_cdf_derivative}.
	If the rectangle
	$[a_1,b_1]\times \cdots \times [a_n,b_n]$
	is of lower dimension, then it must contain
	zero values $a_m=b_m=0$
	for each non-full axis,
	since under $F_n^{(\sigma)}$,
	there are no other lower-dimensional coordinate subspaces of positive mass.
	Observe that $F_n(\alpha_1,\ldots,\alpha_m,0,\alpha_{m+2},
	\ldots,\alpha_n  )$ does not depend on $\alpha_{m+2},\ldots,\alpha_n $.
	Thus, it suffices to check the nonnegativity for
	each $m$-dimensional rectangle
	of the form
	\begin{equation*}
		[a_1,b_1]\times \cdots\times[a_m,b_m]\times\left\{ 0 \right\}\times
		\cdots \times \{0\}.
	\end{equation*}
	We have
	\begin{equation*}
		\partial_{\alpha_1,\ldots,\alpha_m }F_n^{(\sigma)}(\alpha_1,\ldots,\alpha_m,0,\ldots,0 )=(\sigma-1)\ssp\sigma^{-m}\ssp g(\alpha_1)\cdots g(\alpha_m ),
	\end{equation*}
	which implies the nonnegativity.

	We have shown that $F_n^{(\sigma)}$, $n\ge1$, is a
	consistent family of cdfs, so by the Kolmogorov extension theorem,
	they uniquely determine the distribution
	of the family of random variables $\xi_{\sigma;1},\xi_{\sigma;2},\ldots $.
	The marginal distribution of each $\xi_{\sigma;k}$
	readily follows from its cdf
	$F_k^{(\sigma)}(1,\ldots,1,x_k )$, and so we are done.
\end{proof}
\begin{remark}
	\label{rmk:xi_sigma_construction}
	Alternatively, the random variables
	$\xi_{\sigma;k}$ can be constructed iteratively as follows.
	Toss a sequence of independent coins with probabilities of
	success $1,\sigma^{-1},\sigma^{-2},\ldots $.
	Let $N$ be the (random) number of successes until the first
	failure. We have
	\begin{equation}
		\label{eq:xi_sigma_construction_N_random_variable}
		\operatorname{\mathbb{P}}(N=n)=\sigma^{-\binom n2}\ssp(1-\sigma^{-n}),\qquad n\ge1.
	\end{equation}
	Then, sample $N$ independent
	$\mathrm{beta}(1,\sigma/2)$ random variables.
	Set $\xi_{\sigma;k}$,
	$k=1,\ldots,N $, to be these random variables,
	while $\xi_{\sigma;k}=0$ for $k>N$.
	It is worth noting
	that the random variables $\xi_{\sigma;k}$ are not
	independent, but $\xi_{\sigma;1},\ldots,\xi_{\sigma;n} $
	are conditionally independent given $N=n$.
\end{remark}

\subsection{Scaling limit under the shifted Plancherel specialization}

Recall the hike statistics $h_k(w)$ and $\tilde h_k(w)$ \eqref{eq:hikes}, where $w$ is a Fibonacci word.

\begin{theorem}
	\label{thm:shifted_Plancherel_scaling}
	Let $w\in \mathbb{YF}_n$ be a random Fibonacci word with distributed
	according to the clone coherent measure $M_n$
	\eqref{eq:clone_coherent_measure}
	with either of the shifted Plancherel specializations
	\eqref{eq:shifted_Plancherel_specialization_recall_two_at_the_same_time}.
	For any fixed $k\ge 1$, the joint distribution of the
	hikes $(\tilde h_1(w), \ldots, \tilde h_k(w))$ has the scaling limit
	\begin{equation*}
		\frac{\tilde h_j(w)}{n-\sum_{i=1}^{j-1}\tilde h_i(w)}\xrightarrow[n\to\infty]{d}\xi_{\sigma;j},\qquad j=1,\ldots,k,
	\end{equation*}
	where $\xi_{\sigma;j}$ are given by \Cref{def:xi_sigma}.
\end{theorem}

\begin{remark}
	\label{rmk:shifted_Plancherel_stick_breaking}
	In terms of the stick-breaking process, \Cref{thm:shifted_Plancherel_scaling}
	states that
	\begin{equation*}
		\Bigl( \frac{\tilde h_1(w)}{n}, \frac{\tilde h_2(w)}{n}, \ldots \Bigr)
		\stackrel{d}{\longrightarrow}
		X=(X_1,X_2,\ldots ),
	\end{equation*}
	where $X_1=U_1$, $X_k=(1-U_1)\cdots (1-U_{k-1})U_k $ for $k\ge2$, and $U_k=\xi_{\sigma;k}$ are
	\emph{dependent} random variables if $\sigma>1$.
	Due to the dependence structure of the $\xi_{\sigma;k}$'s
	(see \Cref{rmk:xi_sigma_construction}),
	a single zero in the sequence
	$\{X_j\}_{\ge1}$ makes all subsequent $X_j$'s zero.
	Thus, a growing random Fibonacci word under the shifted Plancherel measure
	has a growing number of $2$'s in (almost surely) finitely many
	initial hikes of lengths proportional to $n$.
	These initial hikes are then followed by a growing tail of $1$'s.
	We refer to
	\Cref{sub:shiifted_Plancherel_number_of_2s_discussion}
	for another approach to the asymptotics of the
	shifted Plancherel measure, and a detailed discussion of the
	limiting behavior of the total number of $2$'s.

	When $\sigma=1$, the $\xi_{\sigma;k}$'s
	do not have the point mass at $0$, and are
	independent and
	identically distributed as $\mathrm{beta}(1,1/2)$.
	The sequence $X$
	almost surely has no zeroes, and
	is distributed simply as $\mathrm{GEM}(1/2)$.
	In this special case, our
	\Cref{thm:shifted_Plancherel_scaling}
	reduces to the result of
	\cite{gnedin2000plancherel}
	recalled in \Cref{sub:Plancherel_measure_and_scaling}.
\end{remark}

\begin{proof}[Proof of \Cref{thm:shifted_Plancherel_scaling}]
	\textbf{Step 1.}
	We use \Cref{prop:hikes_distribution} which expresses
	the joint distribution of initial hikes $h_j(w)$ (where $j=1,\ldots,k$) as a product
	\eqref{eq:hikes_distribution}. This product involves the determinants
	$B_k(m)$ which for the shifted Plancherel specialization
	take the same simple form for all sizes, see
	\eqref{eq:B_for_shifted_Plancherel_explicit}, where $\gamma\equiv \sigma-1$.
	In the case of the fake shifted Charlier specialization,
	$B_k(m)$ for $m\ge1$ are the same as in \eqref{eq:B_for_shifted_Plancherel_explicit}.
	Note however that as $n\to\infty$, the quantity $B_k(0)$
	(which differs for the two specializations)
	does not enter the joint distribution of a fixed number of initial hikes.
	Thus, the limiting behavior of the initial hikes under the two
	specializations is the same, and we may consider the case $x_1=\sigma$,
	which is more convenient due to the uniformity of notation.

	For the shifted Plancherel specialization,
	the factors in the product \eqref{eq:hikes_distribution} are
	equal to
	\begin{equation}
		\label{eq:hikes_distribution_shifted_Plancherel_proof1}
		\begin{split}
			&\ssp\prod_{i=0}^{\tilde h_{[1,k]}+1}x_{n-i}^{-1}
			=
			\frac{\Gamma(n+\sigma-\tilde h_{[1,k]}-2)}{\Gamma(n+\sigma)}
			=
			\prod_{i=1}^{k}
			\frac{
			\Gamma(\sigma+d_{i+1}-2\cdot \mathbf{1}_{i=k})}
			{\Gamma(\sigma+d_{i})}
			,
			\\
			&\prod_{i=1}^{k}\prod_{j=2}^{h_i}(d_i-2j+1)\ssp y_{d_i-2j+1}
			=
			\prod_{i=1}^{k}
			\frac{2^{2h_i-2}\ssp \Gamma (\frac{d_i}{2} -
			\frac{1}{2}) \ssp\Gamma (\frac{d_i}{2} +
			\frac{\sigma}{2} - 1)}{\Gamma (\frac{d_i}{2} -
			h_i + \frac{1}{2})\ssp \Gamma (\frac{d_i}{2} - h_i +
			\frac{\sigma}{2})}
			,
		\end{split}
	\end{equation}
	and the last product involving the determinants $B_{c_j}$ is equal to
	\begin{equation}
		\label{eq:hikes_distribution_shifted_Plancherel_proof2}
		\prod_{i=1}^{k+1}
		\frac{(d_i-1)( d_i+\sigma-2 )}
		{\mathbf{1}_{h_{i-1}>0\text{ or }i=1}+
			d_i \ssp ( d_i+\sigma-1 )
			\ssp
			\mathbf{1}_{h_{i-1}=0}\ssp\mathbf{1}_{i>1}
		}.
	\end{equation}
	Here and in \eqref{eq:hikes_distribution_shifted_Plancherel_proof1},
	we used the notation $d_i=n-\tilde h_{[1,i)}$,
	and
	the fact
	that the condition $d_{i}=d_{i-1}-1$ in \eqref{eq:hikes_distribution}
	is equivalent to $h_{i-1}=0$ (for $i=1$, we have $d_1\ne d_0-1$,
	so
	$\mathbf{1}_{h_0>0}=1$).
	Note that for the shifted Plancherel specialization,
	the dependence on the quantities~$c_j$
	\eqref{eq:hikes_sequences_c_j} disappeared.

	\medskip\noindent
	\textbf{Step 2.}
	Let us now consider the asymptotic behavior of
	\eqref{eq:hikes_distribution_shifted_Plancherel_proof1}, \eqref{eq:hikes_distribution_shifted_Plancherel_proof2}
	as the $d_j$'s grow to infinity.
	We examine two cases depending on whether
	the hike is zero or is also growing.
	We have
	for the factors in \eqref{eq:hikes_distribution_shifted_Plancherel_proof2}
	for $i>1$:
	\begin{equation}
		\label{eq:hikes_distribution_shifted_Plancherel_proof_as1}
		\frac{(d_i-1)( d_i+\sigma-2 )}
		{\mathbf{1}_{h_{i-1}>0}+
			d_i \ssp ( d_i+\sigma-1 )
			\ssp
			\mathbf{1}_{h_{i-1}=0}
		}\sim
		\begin{cases}
			d_i^{2}, &h_{i-1}>0,\\
			1, & h_{i-1}=0.
		\end{cases}
	\end{equation}
	For the two products in \eqref{eq:hikes_distribution_shifted_Plancherel_proof1}, we have
	\begin{equation}
		\label{eq:hikes_distribution_shifted_Plancherel_proof_as2}
			\frac{\Gamma(\sigma+d_{i+1}-2\cdot \mathbf{1}_{i=k})}
			{\Gamma(\sigma+d_{i})}\ssp
			\frac{2^{2h_i-2}\ssp \Gamma (\frac{d_i}{2} -
			\frac{1}{2}) \ssp\Gamma (\frac{d_i}{2} +
			\frac{\sigma}{2} - 1)}{\Gamma (\frac{d_i}{2} -
			h_i + \frac{1}{2})\ssp \Gamma (\frac{d_i}{2} - h_i +
			\frac{\sigma}{2})}
			\sim
			d_i^{-3-2\cdot \mathbf{1}_{i=k}}(1-\alpha_i)^{\frac{\sigma}{2}-1-2\cdot \mathbf{1}_{i=k}}
			,
	\end{equation}
	where
	$h_i=\lfloor \alpha_i d_i/2 \rfloor $,
	$0\le \alpha_i<1$.
	Note that
	we inserted the factor $1/2$ since
	hikes count the 2's,
	so $\tilde h_i\sim \alpha_i d_i$.

	\medskip\noindent
	\textbf{Step 3.}
	Consider first the situation when all $\alpha_1,\ldots,\alpha_k$ are strictly positive.
	Then the product of the quantities \eqref{eq:hikes_distribution_shifted_Plancherel_proof_as1}, \eqref{eq:hikes_distribution_shifted_Plancherel_proof_as2} over all $i$ (which is asymptotically
	equivalent to \eqref{eq:hikes_distribution}) has the following behavior
	as $n\to\infty$:
	\begin{equation}
		\label{eq:hikes_distribution_shifted_Plancherel_proof_as3}
		M_n\left( w\colon
		h_1(w)=\Bigl\lfloor \frac{\alpha_1d_1}2 \Bigr\rfloor ,
		\ldots,
		h_k(w)=\Bigl\lfloor \frac{\alpha_k d_k}2 \Bigr\rfloor , h_{k+1}(w)>0 \right)
		\sim
		\sigma^{-k}\prod_{i=1}^{k}(d_i/2)^{-1}\cdot
		\frac{\sigma}2\ssp (1-\alpha_i)^{\frac{\sigma}{2}-1}.
	\end{equation}
	Here we used the fact that
	$d_{k+1}=d_k-2h_k-1$, so
	$d_{k+1}^2
	d_k^{-2}(1-\alpha_k)^{-2}\sim 1$.
	Since the density $\frac{\sigma}{2}(1-u)^{\frac{\sigma}{2}-1}$
	of $\mathrm{beta}(1, \sigma/2)$ integrates to $1$
	over $(0,1)$, we see that there is an asymptotic deficit of the probability
	mass equal to $1-\sigma^{-k}$.
	This deficit mass is supported by
	the event
	\begin{equation*}
		\bigcup_{i=1}^{k+1}\{w\colon h_i(w)=0\}.
	\end{equation*}

	By \Cref{lemma:first_hike_distribution}, we have
	\begin{equation*}
		M_n(w\colon h_1(w)=0)=1-\frac{n-1}{n+\sigma-1}\to 0, \qquad n\to \infty.
	\end{equation*}
	In particular, $M_n(w\colon h_1(w)>0)\to 1$, and the event
	$\left\{ h_1=0 \right\}$ is asymptotically negligible.

	Now consider the case when some of the $\alpha_i$'s are zero in the
	left-hand side of
	\eqref{eq:hikes_distribution_shifted_Plancherel_proof_as3}.
	Then, due to \eqref{eq:hikes_distribution_shifted_Plancherel_proof_as1},
	there is an extra factor of $d_j^{-2}$ for each $j\ge2$ with $\alpha_{j-1}=0$.
	This means that the probability that
	at least one of the $h_j(w)$'s is zero (for some $1\le j\le k$)
	while $h_{k+1}(w)>0$ is negligible in the limit.
	Therefore,
	we conclude that
	\begin{equation*}
		\lim_{n\to\infty}M_n(w\colon h_{k+1}(w)>0)
		=
		\lim_{n\to\infty}M_n(w\colon h_{1}(w)>0,\ldots,h_{k}(w)>0,h_{k+1}(w)>0)
		=\sigma^{-k}
	\end{equation*}
	for all $k\ge0$.
	This implies that for all $k\ge0$, all the deficit probability mass
	$1-\sigma^{-k}$
	from the left-hand side of \eqref{eq:hikes_distribution_shifted_Plancherel_proof_as3}
	is supported on
	the event $\left\{ h_{k+1}=0 \right\}$.

	\medskip\noindent
	\textbf{Step 4.}
	Define for each $k\ge1$ the joint cdf
	\begin{equation*}
		F_k(\alpha_1,\ldots, \alpha_k)\coloneqq
		\lim_{n\to\infty}
		M_n\left( w\colon h_1(w)\le \lfloor \alpha_1d_1/2 \rfloor , \ldots,
		h_k(w)\le \lfloor \alpha_k d_k/2 \rfloor \right)
	\end{equation*}
	of the scaled hikes $(\tilde h_1(w)/d_1, \ldots, \tilde h_k(w)/d_k)$.
	The observations in Step 3 imply that
	\begin{equation}
		\label{eq:scaling_limit_hikes_cdf}
		F_{k-1}(\alpha_1,\ldots,\alpha_{k-1})-F_k(\alpha_1,\ldots,\alpha_{k-1},0 )
		=
		\sigma^{-k+1} G(\alpha_1)\cdots G(\alpha_{k-1}),
	\end{equation}
	where $G(\cdot)$ is the cdf
	of the $\mathrm{beta}(1, \sigma/2)$ random variable
	given by \eqref{eq:beta_1_sigma_2_density}.
	Note that $G(0)=0$.
	We
	see that it remains to find the functions
	$F_k(\alpha_1,\ldots,\alpha_{k-1},0 )$ for all $k\ge1$.
	Iterating \eqref{eq:scaling_limit_hikes_cdf}, we see that these functions are
	consistent as long as there is at least one zero,
	that is,
	\begin{equation*}
		F_k(\alpha_1,\ldots,\alpha_{k-2},0,0 )=
		F_{k-1}(\alpha_1,\ldots,\alpha_{k-2},0),
	\end{equation*}
	and so on.

	Differentiate \eqref{eq:scaling_limit_hikes_cdf} in
	$\alpha_{1},\ldots,\alpha_{k-1} $.
	Then, because of the probability mass deficit at level $k-1$, we have
	\begin{equation*}
		\partial_{\alpha_1,\ldots,\alpha_{k-1} }\ssp
		F_{k-1}(\alpha_1,\ldots,\alpha_{k-1})=
		\sigma^{-k+2}
		g(\alpha_1)\cdots g(\alpha_{k-1}),
	\end{equation*}
	where $g(\alpha)$ is the density of the $\mathrm{beta}(1, \sigma/2)$
	random variable \eqref{eq:beta_1_sigma_2_density}.
	Therefore,
	\begin{equation}
		\label{eq:scaling_limit_hikes_cdf_final}
		\partial_{\alpha_1,\ldots,\alpha_{k-1} }
		\ssp
		F_{k}(\alpha_1,\ldots,\alpha_{k-1},0)=
		(\sigma-1)\ssp \sigma^{-k+1}\ssp g(\alpha_1)\cdots g(\alpha_{k-1}).
	\end{equation}

	Using \eqref{eq:scaling_limit_hikes_cdf_final}, we can now
	compute $F_k(\alpha_1,\ldots,\alpha_{k-1},0 )$ by induction on $k$ and
	iterative integration. We have $F_1(0)=0$,
	then
	\begin{equation*}
		\partial_{\alpha_1}
		F_2(\alpha_1,0)=
		(\sigma-1)\sigma^{-1}g(\alpha_1)
		\quad\Rightarrow\quad
		F_2(\alpha_1,0)=
		(\sigma-1)\sigma^{-1}G(\alpha_1)+F_2(0,0),
	\end{equation*}
	but by consistency, $F_2(0,0)=F_1(0)=0$.
	For general $k$, the first integration
	in $\alpha_{k-1}$ yields
	\begin{equation*}
		\partial_{\alpha_1,\ldots,\alpha_{k-2}}
		\ssp
		F_k(\alpha_1,\ldots,\alpha_{k-1},0)=
		(\sigma-1)\sigma^{-k+1}g(\alpha_1)\cdots g(\alpha_{k-2})
		G(\alpha_{k-1})+
		F_{k-1}(\alpha_1,\ldots,\alpha_{k-2},0).
	\end{equation*}
	This procedure of iterative integration yields the unique solution
	for $F_k(\alpha_1,\ldots,\alpha_{k-1},0 )$,
	and this leads to the formula for the joint cdf
	$F_k(\alpha_1,\ldots,\alpha_k)=
	F_k^{(\sigma)}(\alpha_1,\ldots,\alpha_k)$
	\eqref{eq:joint_cdf} of the limit of the scaled hikes.
	This completes the proof of \Cref{thm:shifted_Plancherel_scaling}.
\end{proof}

\section{General Convergent Specializations and Type-I Components}
\label{sec:general_convergent_specializations}

In this section, we
consider Fibonacci positive specializations under which random Fibonacci
words have a growing prefix of 1's, followed by a Fibonacci word
of almost surely finite size.
This leads to asymptotic behavior of discrete nature, different from the cases
discussed in \Cref{sec:asymptotics_Charlier_specialization,sec:asymptotics_shifted_Plancherel_specialization} above.

\subsection{Type-I components of coherent measures}
\label{sub:type_I_components}

The next definition follows \cite{KerovGoodman1997}:

\begin{definition}
\label{def:typeI-fibonacci-words}
A \emph{Type-I} Fibonacci word\footnote{Not to be confused with
Type-I Al-Salam--Carlitz polynomials or similar specializations,
as these objects are unrelated. The term ``Type-I'' in
Fibonacci words comes from connections to
Type-I factor representations of AF-algebras
associated to the branching graph \cite[Section~4]{KerovGoodman1997}.}
is an infinite Fibonacci word
formed by appending a prefix consisting of infinitely many
digits~$1$ to a Fibonacci word. A Type-I Fibonacci word
can be uniquely expressed as either $1^\infty$ or $1^\infty 2w$,
where $w$ is a finite suffix in $\mathbb{YF}$.
Denote the (countable) set of all Type-I words by $1^\infty\mathbb{YF}$,
and the subset of all Type-I words of the form $1^\infty 2w$,
$w\in \mathbb{YF}$,
by $1^\infty 2\mathbb{YF}\subset 1^\infty\mathbb{YF}$.
\end{definition}

A Type-I word $1^\infty w$ can be viewed as the equivalence
class of infinite saturated chains
$v_0 \nearrow v_1 \nearrow v_2 \nearrow \cdots$, starting
at $v_0 = \varnothing$, with $v_n = 1^{n - m}w$ for all
$n \geq m$, where $|w| = m$. We call infinite saturated chains of
this kind \emph{lonely paths}.

\begin{definition}
\label{def:measures-typeI-fibonacci-words}
If $\varphi\colon \mathbb{YF} \rightarrow \mathbb{R}_{\scriptscriptstyle \geq 0}$
is a nonnegative, normalized harmonic function and $1^\infty w$ is a Type-I word,
we define:
\[
	\muI(1^\infty w) := \lim_{n \rightarrow \infty} M_{m+n}(1^n w),
\]
where $w \in \mathbb{YF}_m$, and $M_{k}$ denotes the coherent measure
on $\mathbb{YF}_{k}$ associated to $\varphi$ by
\eqref{eq:coherent_measure_for_harmonic_function}.
We call the (in general, sub-probability)
measure $\muI(\cdot)$ on $1^\infty \mathbb{YF}$ the
\emph{Type-I component} of the harmonic function $\varphi$.

Note that
$0\le M_{n+m}(1^n w)\le 1 $. Moreover,
$ M_{n+m}(1^n w) $ is weakly decreasing in $n$,
and so the limit $\muI(1^\infty w) \in [0,1]$ exists
for all $w \in \mathbb{YF}$.
\end{definition}

Let $(\vec{x}, \vec{y} \, )$
be a Fibonacci positive specialization
\begin{equation*}
	x_k=c_k\ssp (1+t_{k-1}),\qquad y_k=c_k\ssp c_{k+1}\ssp t_k, \qquad k\ge1.
\end{equation*}
Here
$\vec{t}$ is a sequence
of either convergent or divergent type
(\Cref{def:finite_infinite_type_specializations}), and and
$\vec{c}$ is any sequence of positive real numbers.
Let $\varphi_{\vec{x}, \vec{y}}$ be the corresponding clone harmonic function,
and let $\muI$ be the associated Type-I component on $1^\infty\mathbb{YF}$.

\begin{lemma}
	\label{lemma:measure_of_all_ones}
	We have
	\begin{equation}
		\label{eq:measure_of_all_ones_infinite_product}
		\muI(1^\infty)=
		\prod_{i \ssp = \ssp 0}^{\infty} \, (1+t_i)^{-1}.
	\end{equation}
	Moreover, if $\vec t$ is of divergent type,
	then $\muI(1^\infty)$ vanishes.
\end{lemma}
\begin{proof}
	We have
	$\dim(1^n)=1$
	and
	$s_{1^n}(\vec x\mid \vec y\ssp )
	= c_1 \cdots c_n$
	for all $n$.
	Using \eqref{eq:clone_Schur_normalization}, we see that
	\begin{equation*}
		M_n(1^n)=\frac{s_{1^n}\bigl(\vec x\mid \vec y\ssp \bigr)}{x_1\cdots x_n }\ssp
		\dim (1^n)=
		\prod_{i \ssp  = \ssp 1}^{n-1}\frac{1}{1+t_i},
	\end{equation*}
	which
	converges to the desired infinite product
	\eqref{eq:measure_of_all_ones_infinite_product}.

	For a divergent type sequence $\vec{t}$, we have
	(using the notation \eqref{eq:A_infty_B_infty_series_definitions})
	\begin{equation*}	\infty=A_\infty(1)=1+t_1+t_1t_2+t_1t_2t_3+\ldots
		\ssp \le \ssp
		\prod_{i=1}^\infty(1+t_i),
	\end{equation*}
	where we used the fact that the $t_i$'s are nonnegative.
	As the reciprocal of the product in \eqref{eq:measure_of_all_ones_infinite_product} goes to infinity,
	we have $\muI(1^\infty)=0$.
\end{proof}

For a convergent type sequence $\vec{t}$, we either have
$\muI(1^\infty)=0$ or $0<\muI(1^\infty)<1$.
The next statement discusses the latter case.

\begin{proposition}
	\label{prop:measure_on_type_I_words}
	Let $\vec t$ be of convergent type,
	and let
	$\muI(1^\infty)>0$.
Then
\begin{equation}
	\label{eq:measure_of_all_ones_infinite_product_in_proposition}
	\sum_{w\in \mathbb{YF}_m}
	\muI(1^\infty2w)  =
	(m+1) \ssp B_\infty(m)
	\prod_{i \ssp = \ssp m}^{\infty}  (1+t_i)^{-1},
	\qquad m\ge0,
\end{equation}
where $B_\infty(m)$ is defined in
\eqref{eq:A_infty_B_infty_series_definitions}.
Moreover,
$\muI(1^\infty 2w) > 0$ for all $w\in \mathbb{YF}$.

In contrast, if $\muI(1^\infty)=0$, then $\muI(1^\infty\mathbb{YF})=0$.
\end{proposition}
\begin{proof}

	For any $w\in \mathbb{YF}_m$, we have by \eqref{eq:clone_Schur_recurrence_def}:
	\begin{equation*}
		s_{1^n 2w}(\vec x\mid \vec y\ssp )=B_{n-1}(m)\ssp s_w(\vec x\mid \vec y\ssp ).
	\end{equation*}
	Therefore, by \eqref{eq:dimension_recursion}, we can write
	\begin{equation}
		\label{eq:B_infty_proof}
		M_{m+n+2}(1^n2w)=
		\biggl(
			\, \prod_{i \ssp = \ssp 1}^{m}x_i
		\biggr)
		\big(m+1 \big)\ssp
		M_{m}(w)
		B_{n-1} (m)
		\prod_{i \ssp = \ssp 1}^{m+n+2}x_i^{-1}.
	\end{equation}
	The factor $B_{n-1}(m)$
	converges as $n\to \infty$ to $B_\infty(m)$
	\eqref{eq:A_infty_B_infty_series_definitions}.
	Assuming that $\muI(1^\infty)>0$,
	we see that the limit as $n\to \infty$ of \eqref{eq:B_infty_proof} is
	\begin{equation}
		\label{eq:B_infty_proof2}
		\muI(1^\infty 2w)=
		\lim_{n\to \infty}M_{m+n+2}(1^n2w)=
		\Bigg(
			\, \prod_{i \ssp = \ssp 1}^{m-1}(1+t_i)
		\Bigg)
		\big(m+1 \big)\ssp
		M_{m}(w) \ssp
		\frac{B_{\infty}(m)}{\prod_{i=1}^{\infty}(1+t_i)},
	\end{equation}
	which is positive.
	This means that $\muI(1^\infty 2w)>0$ for all $w\in \mathbb{YF}$.
	Summing \eqref{eq:B_infty_proof2} over all $w\in \mathbb{YF}_m$
	(which is a finite sum),
	we get
	the desired
	claim \eqref{eq:measure_of_all_ones_infinite_product_in_proposition}.

	In contrast, if $\muI(1^\infty)=0$,
	then the limit of \eqref{eq:B_infty_proof}
	is determined by the growing product of the $x_i^{-1}$'s,
	which goes to zero as $n\to \infty$. Therefore,
	the probability weight of each Type-I word vanishes in this case.
\end{proof}

\begin{remark}
	\label{rmk:type_I_components_examples_discussion_no_type_I_support}
	Both the shifted Plancherel specialization and the power specialization with $\upalpha=1$
	(see \Cref{sub:shifted_Charlier_defn,sub:power_spec_defn} for the definitions)
	are of convergent type, but are \emph{not} supported on Type-I words
	since $\muI(1^\infty)=0$.
	This is not surprising for the shifted Plancherel specialization,
	since by \Cref{thm:shifted_Plancherel_scaling},
	the corresponding random Fibonacci word
	starts with a growing prefix of~$2$'s.
	We briefly discuss the (conjectural)
	asymptotics of random Fibonacci words under the power specialization
	with $\upalpha=1$ in \Cref{sub:type_I_components_examples_power}
	below.
\end{remark}

\begin{proposition}
	\label{prop:full_type_I_support}
	Under the conditions of \Cref{prop:measure_on_type_I_words},
	if $\muI(1^\infty)>0$,
	then we have
	$\muI(1^\infty\mathbb{YF})=1$.
\end{proposition}
\begin{proof}
	Multiply identity
	\eqref{eq:finite-tversion-Important-Identity}
	from
	\Cref{rem:finite-tversion-Important-Identity} by
	$\prod_{k = 0}^n (1 + t_k)^{-1}$ (recall that $t_0=0$):
	\begin{equation}
		\label{eq:Cauchy_identity_in_proposition_proof_uniform_limit}
		\prod_{k \ssp = \, 0}^{n-1}
		\ssp (1 + t_k)^{-1}+
		\sum_{m \ssp = \ssp 0}^{n-2}
		\ssp (m+1) \ssp B_{n-m-2}(m)
		\prod_{k \ssp = \ssp m}^{n-1}
		\ssp (1 + t_k)^{-1}
		=1.
	\end{equation}
	We aim to take the limit as $n \rightarrow \infty$
	inside the sum in the left-hand side of
	\eqref{eq:Cauchy_identity_in_proposition_proof_uniform_limit}.
	This would yield
	\begin{equation}
		\label{eq:Cauchy_identity_in_proposition_proof3}
		\begin{split}
			1 &= \prod_{k = 0}^\infty (1 + t_k)^{-1}
			+ \sum_{m = 0}^{\infty} (m+1) B_{\infty}(m)
			\prod_{k = m}^{\infty} (1 + t_k)^{-1} \\
			&= \prod_{k = 0}^\infty (1 + t_k)^{-1}
			+ \sum_{m \geq 0} \sum_{|w| = m} (m+1) M_m(w) B_{\infty}(m)
			\prod_{k = m}^{\infty} (1 + t_k)^{-1} \\
			&= \muI(1^\infty) + \sum_{w \in \mathbb{YF}} \muI(1^\infty 2w),
		\end{split}
	\end{equation}
	which is the desired result.

	However,
	in order to justify
	the passage from
	\eqref{eq:Cauchy_identity_in_proposition_proof_uniform_limit}
	to \eqref{eq:Cauchy_identity_in_proposition_proof3}
	(the interchange of the limit and the summation),
	we need the convergence
	\begin{equation*}
		(m+1)B_{n-m-2}(m)\prod_{k=m}^{n-1}(1+t_k)^{-1}\to
		(m+1)B_{\infty}(m)\prod_{k=m}^{\infty}(1+t_k)^{-1},
		\qquad n\to\infty,
	\end{equation*}
	to be uniform in $m$ (and then apply the dominated convergence theorem for
	series).
	Note that in \eqref{eq:Cauchy_identity_in_proposition_proof_uniform_limit},
	we have $n\ge m+2$. That is, we can
	already turn the sum over $0\le m\le n-2$ in \eqref{eq:Cauchy_identity_in_proposition_proof_uniform_limit}
	into an infinite sum, by adding the zero terms for $m>n-2$.

	For the products, we have
	\begin{equation}
		\label{eq:uniform_convergence_products_in_proof}
		\left|
			\prod_{k=m}^{n-1}(1+t_k)^{-1}
			-
			\prod_{k=m}^{\infty}(1+t_k)^{-1}
		\right|
		=
		\left|
			\prod_{k=1}^{n-1}(1+t_k)^{-1}
			-
			\prod_{k=1}^{\infty}(1+t_k)^{-1}
		\right|
		\prod_{k=1}^{m-1}(1+t_k),
	\end{equation}
	where
	$\prod_{k=1}^{m-1}(1+t_k)$ is bounded in $m$,
	so \eqref{eq:uniform_convergence_products_in_proof}
	converges to zero uniformly in $m$.

	It remains to establish the uniform convergence in $m$ of
	(recall that $n\ge m+2$)
	\begin{equation*}
		(m+1)B_{n-m-2}(m)=
		\begin{cases}
			0,& n< m+2;\\
			(m+1)t_{m+1},& n=m+2;\\
			(m+1)\left(
				t_{m+1}-(1+t_m-t_{m+1})t_{m+2}A_{n-m-3}(m+3)
			\right),& n> m+2,
		\end{cases}
	\end{equation*}
	as $n\to\infty$.
	Observe that
	$m\ssp t_m\to 0$ as $m\to\infty$.

	Indeed, this follows from the convergence of the infinite product
	$\prod_{k=1}^{\infty}(1+t_k)$,
	which is equivalent to the convergence of the series
	$\sum_{k=1}^{\infty}t_k$ (since the $t_k$'s are nonnegative).
	Moreover, since the $t_m$'s eventually weakly decrease to zero
	(\Cref{proposition:cannot_nondecrease,proposition:limit_zero_of_convergent_type}),
	we can use Cauchy condensation test to conclude that
	$m\ssp t_m\to 0$.

	The convergence
	$m\ssp t_m\to 0$ implies that we can discard
	finitely many
	terms with $n-K-3< m\le n-2$
	from the sum in \eqref{eq:Cauchy_identity_in_proposition_proof_uniform_limit},
	as they converge to zero. Let us fix some $K>5$ once and for all.
	For the remaining terms, we can write
	\begin{equation*}
		\begin{split}
			&\bigl|
				(m+1)B_{n-m-2}(m)
				-
			(m+1)\left(
				t_{m+1}-(1+t_m-t_{m+1})t_{m+2}A_{\infty}(m+3)
			\right)\bigr|
			\\&\hspace{70pt}=
			(m+1)(1+t_m-t_{m+1})t_{m+2}
			\left(
				t_{m+3}\cdots t_{n}
				+
				t_{m+3}\cdots t_{n}t_{n+1}+\ldots
			\right).
		\end{split}
	\end{equation*}
	The factor $(m+3)t_{m+2}$ is bounded.
	Let the sequence $t_m$ eventually decrease starting from $m= m_0$.
	Pick $\varepsilon>0$, and find $N\ge m_0+K$ such that $t_{N-K}\le \varepsilon$.
	Then for all $n\ge N+3$ and $m\ge m_0$ (with the condition $m\le n-K-3$):
	\begin{equation*}
		t_{m+3}\cdots t_{n}
		+
		t_{m+3}\cdots t_{n}t_{n+1}+\ldots
		\le
		t_{N-K}\cdots t_{N+3}+
		t_{N-K}\cdots t_{N+3}t_{N+4}+\ldots
		\le \frac{\varepsilon}{1-\varepsilon},
	\end{equation*}
	which is small.
	This completes the proof.
\end{proof}

Let us restate the results of
\Cref{lemma:measure_of_all_ones}
and
\Cref{prop:measure_on_type_I_words}
in terms of the boundary
of the Young--Fibonacci lattice \cite{KerovGoodman1997},
\cite{BochkovEvtushevsky2020}, \cite{Evtushevsky2020PartII}.
Recall from \Cref{sub:YF_lattice}
that the
extremal (Martin)
boundary $\Upsilon_{\mathrm{ext}}(\mathbb{YF})$
is the set of all nonnegative, normalized, extremal harmonic functions
on $\mathbb{YF}$.
An arbitrary
nonnegative, normalized harmonic function $\varphi$
on $\mathbb{YF}$ can be represented as a Choquet integral
\eqref{eq:harmonic_function_as_integral_Choquet}
with respect to a probability measure $\mu$
on
$\Upsilon_{\mathrm{ext}}(\mathbb{YF})$.
The measure $\mu$ is uniquely determined by $\varphi$.

The set
$1^\infty \mathbb{YF}$ of all Type-I words (\Cref{def:typeI-fibonacci-words})
constitutes
a part of the boundary
$\Upsilon_{\mathrm{ext}}(\mathbb{YF})$.
Indeed, the extremal \emph{Type-I harmonic functions}
corresponding to Type-I words of the form $1^\infty 2w$
are given by \cite[Proposition~4.2]{KerovGoodman1997}:
\[
\Phi_{1^\infty2w}(v) \coloneqq
\begin{cases}
	\displaystyle \frac{\dim(v, 1^k 2w)}{\dim(2w)}, & \text{if $v \unlhd 1^k w$ for some $k \geq 0$}, \\[5pt]
	0, & \text{otherwise}.
\end{cases}
\]
Here $w\in\mathbb{YF}$ is fixed,
$\unlhd$ denotes the partial
order on $\mathbb{YF}$,
and $\dim(u,v)$ is the number of saturated chains in the Young--Fibonacci lattice
beginning at $u$ and ending at $v$.
Likewise, $\Phi_{1^\infty}(v)$
takes the values $1$ if $v=1^k$
for some $k \geq 0$, and
$0$ otherwise.

\begin{corollary}
\label{cor:typeI-expansion-convergent}
Let $(\vec{x}, \vec{y} \, )$
be a Fibonacci positive specialization
of
convergent type
such that $\muI(1^\infty) > 0$.
Then the measure $\mu$ on the boundary $\Upsilon_{\mathrm{ext}}(\mathbb{YF})$
coincides with its Type-I component~$\muI$.
Consequently, the
Choquet integral representation of the
clone harmonic function $\varphi_{\vec{x}, \vec{y}}$
involves only Type-I harmonic functions, and has the form:
\begin{equation*}
\varphi_{\vec x, \vec y}
\, = \,
\muI(1^\infty) \ssp \Phi_{1^\infty}
\, + \, \sum_{w \ssp \in \ssp \mathbb{YF}}
\, \muI(1^\infty2w) \ssp \Phi_{1^\infty2w}.
\end{equation*}
\end{corollary}
\begin{proof}
	This result follows from the fact that
	$\muI(1^\infty\mathbb{YF})=1$ (\Cref{prop:measure_on_type_I_words}),
	the Choquet integral
	representation
	\eqref{eq:harmonic_function_as_integral_Choquet},
	and the ergodicity of the
	Martin boundary
	established in the preprints
	\cite{BochkovEvtushevsky2020}, \cite{Evtushevsky2020PartII}.
\end{proof}

In the rest of this section, we consider two particular examples of
type-I component computations.

\subsection{Type-I component of the Al-Salam--Chihara specialization}
\label{sub:type_I_components_examples_al_salam_chihara}

Recall the Al-Salam--Chihara specialization
from
\Cref{def:positive_specializations_class_II};
it is of divergent type.
Recall that the parameters are
$x_k=\rho+[k-1]_q$, $y_k=\rho[k]_q$,
$k\ge1$,
where $0<\rho\le 1$ and $q\ge1$.
Note that as $q\to1$, we recover the Charlier specialization,
and the asymptotic behavior of the corresponding
clone coherent measures was considered in
\Cref{sec:asymptotics_Charlier_specialization}.
Let $M_n$ denote the coherent measures
corresponding to the Al-Salam--Chihara specialization.

By \Cref{lemma:first_hike_distribution}, we have
\begin{equation*}
	M_n(w\colon h_1(w)=0)=
	1-
	\frac{(n-1)(q-1) q^2 \rho \left(q^n - q\right)}{\left(q^n +
	\rho q^2 - (\rho + 1) q\right)\left(q^n + \rho\ssp q^3 - (\rho +1) q^2\right)}.
\end{equation*}
One readily checks that this expression converges to $1$ as $n\to\infty$
exponentially fast.
This means that with
probability exponentially close to $1$,
the growing random word $w\in \mathbb{YF}_n$
under the
Al-Salam--Chihara
clone coherent measure
does not start with a $2$.

For simplicity, here we only consider the case $\rho=1/q$.
\begin{proposition}
	\label{prop:Al-Salam-Chihara_Type-I}
	For the Al-Salam--Chihara specialization with $\rho=1/q$,
	we have
	\begin{equation*}
		\muI(1^\infty)=0
		\qquad \textnormal{and} \qquad
		\muI(1^\infty 2\mathbb{YF})=1.
	\end{equation*}
	This means that in the growing random word
	there will be finitely many (but at least one)
	occurrences of the digit $2$.
\end{proposition}
\begin{proof}
	Recall from the proof of \Cref{prop:class_II_positive_specializations,proposition:infinite_type_scaling}
	that we can take the $\vec t$-parameters to be
	$t_k=[k]_q/\rho=q\ssp [k]_q$, and then we must modify
	$x_k=1+t_{k-1}=[k]_q$ and $y_k=[k]_q$.
	By \Cref{lemma:measure_of_all_ones},
	we have $\muI(1^\infty)=\prod_{k=1}^{\infty}(1+q[k]_q)^{-1}$,
	which clearly diverges to zero.
	This proves the first claim.

	Let us find the limiting distribution of $r_1(w)$, the initial run of $1$'s.
	By \Cref{prop:runs_distribution}, we have
	\begin{equation*}
		M_{n+m+2} (1^n2\mathbb{YF}_m )
		=
		M_{n+m+2}(w\colon r_1(w)=n)
		=  (m+1)  B_n(m)  \prod_{k =  m+1}^{n+m+2} x_k^{-1}.
	\end{equation*}
	Using \eqref{eq:B_ell_m_finite_type_proof}, we have
	\begin{equation*}
		\begin{split}
			&M_{n+m+2} (1^n2\mathbb{YF}_m )
			=
			(m+1)
			\left( t_{m+1}-(1+t_m-t_{m+1})t_{m+2}\ssp A_{n-1}(m+3) \right)
			\prod_{k =  m+1}^{n+m+2} x_k^{-1}
			\\
			&\hspace{20pt}
			=
			(m+1)\ssp [m+1]_q
			\left(
				q+q(q-1)[m+2]_q
				\Bigl(
					1  +  \sum_{i=1}^{n-1}
					q^i  [m+3]_q \cdots [m+2+i]_q
				\Bigr)
			\right)
			\prod_{k =  m+1}^{n+m+2} [k]_q^{-1}.
		\end{split}
	\end{equation*}
	The product over $k$ from $m+1$ to $n+m+2$ diverges to zero
	as $n\to \infty$. Therefore, the only possible nonzero contribution
	must include the sum over $i$:
	\begin{equation}
		\label{eq:Al-Salam-Chihara_Type-I_proof_0}
		(m+1)\ssp q(q-1)\ssp [m+1]_q [m+2]_q
		\Biggl(\ssp\prod_{k =  m+1}^{n+m+2} [k]_q^{-1}\Biggr)
		\sum_{i=1}^{n-1}
		q^i  [m+3]_q \cdots [m+2+i]_q.
	\end{equation}
	We aim to show that
	\begin{equation}
		\label{eq:Al-Salam-Chihara_Type-I_proof}
		\lim_{n\to \infty}
		\Biggl(\ssp\prod_{k =  m+1}^{n+m+2} [k]_q^{-1}\Biggr)
		\sum_{i=1}^{n-1}
		q^i  [m+1]_q[m+2]_q \cdots [m+2+i]_q=(q-1)\ssp q^{-m-3}
	\end{equation}
	for all $m\ge0$.
	Indeed, after cancelling out, the sum in \eqref{eq:Al-Salam-Chihara_Type-I_proof}
	becomes
	\begin{equation*}
		\sum_{i\, =\,1}^{n-1}q^i\prod_{j\,=\,i+1}^{n}\frac{1}{[m+2+j]_q}
		=
		\sum_{i \, = \, 1}^{n-1}
		q^{n-i}\prod_{j\,=\,n-i+1}^{n}\frac{1}{[m+2+j]_q}
	\end{equation*}
	All terms in the latter sum except the first one decay to zero
	exponentially fast as $n\to\infty$.
	This is because for $i\ge2$, there are at least
	two factors of the form $[n+\mathrm{const}]_q$
	in the denominator, which cannot be compensated
	by $q^{n-i}$ in the numerator.
	Therefore, we can exchange the summation and the limit,
	and immediately obtain the desired outcome
	\eqref{eq:Al-Salam-Chihara_Type-I_proof}.

	Combined with
	\eqref{eq:Al-Salam-Chihara_Type-I_proof_0},
	observe that the limiting quantities sum to $1$ over $m$
	\begin{equation*}
		\sum_{m=0}^{\infty}(m+1)\ssp q\ssp (q-1)^2\ssp q^{-m-3}=1.
	\end{equation*}
	This implies that
	$\muI(1^\infty 2\mathbb{YF})=1$, as desired.
\end{proof}

\subsection{Type-I component of the power specialization with $\upalpha=1$}
\label{sub:type_I_components_examples_power}

Recall the power specializations $t_k = \varkappa /
k^{\upalpha}$ with $\upalpha \in \{1,2\}$, introduced in
\Cref{sub:power_spec_defn}; these are of
convergent type.
Set $t_k=\varkappa/k^\upalpha$, where $0<\varkappa\leq\varkappa_1^{(\upalpha)}$
(see \Cref{prop:convergent_type_two_examples}),
and $x_k=1+t_{k-1}$, $y_k=t_k$.

For $\upalpha=1$, we have by \Cref{lemma:measure_of_all_ones}:
\begin{equation*}
	\muI(1^\infty)=\prod_{k=1}^{\infty}\left(1+\frac{\varkappa}{k}\right)^{-1},
\end{equation*}
which diverges to zero for all $\varkappa$, $0<\varkappa\leq\varkappa_1^{(1)}\approx  0.844637$.
Thus, $\muI$
is identically zero, and there
is no Type-I support, as in the case of the shifted Plancherel specialization
(cf. \Cref{rmk:type_I_components_examples_discussion_no_type_I_support}).

On the other hand, for the power specialization with $\upalpha=1$,
we expect that
the run statistics $r_k$ \eqref{eq:runs}
admit a scaling limit, similarly to the Charlier specialization
considered in \Cref{sec:asymptotics_Charlier_specialization}.
	This is suggested by the characteristic quantity
\[ { y_{k} \over {x_k x_{k+1} }}
 =  {(k-1) \varkappa
 \over {\big( k + \varkappa -1 \big) \big( k + \varkappa \big)}},
\]
which has a very similar form to the corresponding quantity
in the Charlier case:
\[ { y_{k} \over {x_k x_{k+1} }}
 =  { k \rho
\over {\big( k + \rho -1 \big) \big( k + \rho \big)}}. \]
The difference is only in the shift of the index $k$, and
the renaming of $\rho$ to $\varkappa$.
We do not pursue the analysis of this scaling limit in the present paper.

\medskip

Turning to the case $\upalpha=2$,
we have by \Cref{lemma:measure_of_all_ones}:
\begin{equation*}
	\muI(1^\infty)=\prod_{k=1}^{\infty}\left(1+\frac{\varkappa}{k^2}\right)^{-1}
	=
	\frac{\pi  \sqrt{\varkappa }}{\sinh (\pi  \sqrt{\varkappa })}>0,
\end{equation*}
which means that \Cref{prop:measure_on_type_I_words}
applies for $\upalpha=2$, and $\muI(1^\infty\mathbb{YF})=1$.
We now state explicitly the probability normalization identity for this measure $\muI$.

We have by \eqref{eq:A_infty_B_infty_series_definitions}:
\begin{equation*}
	A_\infty(m)=
	1+\sum_{r=1}^{\infty}\frac{\varkappa^r}{m^2 (m+1)^2\cdots(m+r-1)^2 }
	=
	\sum_{r=0}^{\infty}\frac{r!}{(m)_r (m)_r r!}\ssp \varkappa^r
	=
	{}_1F_2(1;m,m;\varkappa).
\end{equation*}
Thus,
\begin{equation*}
	B_\infty(0)=
	\varkappa-\frac{\varkappa(1-\varkappa)}{4}\ssp
	{}_1F_2(1;3,3;\varkappa)
	=
	\frac{1}{\varkappa}+\frac{\varkappa-1}{\varkappa}
	\ssp I_0(2\sqrt{\varkappa})
\end{equation*}
(where $I_0$ is the modified Bessel function of the first kind),
and $B_\infty(m)$ for $m\ge 1$ is similarly defined by
\eqref{eq:A_infty_B_infty_series_definitions}.
We see that the normalization identity $\muI(1^\infty\mathbb{YF})=1$
(equivalent to
\eqref{eq:Cauchy_identity_in_proposition_proof3})
takes the form
\begin{equation}
	\label{eq:weird_hypergeometric_identity}
	\begin{split}
		&\sum_{m=1}^{\infty}
		\biggl[
		\frac{\varkappa}{(m+1)}
		\left( 1-
			\frac{m^2(m+1)^2+(2m+1)\varkappa}{m^2(m+2)^2}\ssp\ssp
			{}_1F_2(1;m+3,m+3;\varkappa)
		\right)
		\prod_{k=1}^{m-1}\left( 1+\frac{\varkappa}{k^2} \right)
		\\&\hspace{220pt}
		+
		\frac{(\varkappa-1)\varkappa^{m-2}}{(m-1)!^2}
		\biggr]
		=
		\frac{\sinh (\pi  \sqrt{\varkappa })}{\pi  \sqrt{\varkappa }}-
		\frac{\varkappa+1}{\varkappa},
	\end{split}
\end{equation}
where we used the standard series representation
for the Bessel function.
Let us emphasize that \eqref{eq:weird_hypergeometric_identity}
follows from \Cref{prop:measure_on_type_I_words}.
It is not clear how to prove this
identity directly,
without
referring to the parameters $\vec{t}$ of the
Fibonacci positive specialization.

\section{From Random Fibonacci Words to Random Permutations}
\label{sec:random_Fibonacci_words_to_random_permutations}

In this section, we develop a model
of random permutations and involutions
based on the Young--Fibonacci
Robinson--Schensted correspondence.
This model incorporates transition and
cotransition probabilities determined
by clone Schur measures.
These probability models exploit
specific features of the Young--Fibonacci lattice
that are absent in the
Young lattice.
Specifically, we introduce a system
of \emph{cotransition probabilities} defined by
an arbitrary positive harmonic
function $\varphi: \mathbb{YF} \to
\mathbb{R}_{> 0}$, and employ them to construct
measures on permutations and involutions.

We emphasize that, for a branching graph, cotransition
probabilities are usually defined canonically and do not
allow a parametrization by a harmonic function. When our
construction is applied to the Young lattice, the resulting
cotransition probabilities depending on a harmonic function
$\varphi$ are meaningful only in the classical case when
$\varphi$ is the Plancherel harmonic function. This reflects
an extra flexibility of the Young--Fibonacci lattice, which,
to the best of our knowledge, has not been observed
previously.

\subsection{The Young--Fibonacci Robinson--Schensted correspondence}
\label{sub:UD_YF_operators}

Both the
Young--Fibo\-nacci lattice
$\mathbb{YF}$ and the Young lattice $\mathbb{Y}$
of integer partitions
are examples of $1$-{\it differential posets}
\cite{stanley1988differential}, \cite{fomin1994duality}.
That is, they are:
\begin{enumerate}[\bfseries 1.]
    \item Ranked, locally finite posets
    $\big(\mathbb{P}, \unlhd\big)$
    with a unique minimal element
    $\varnothing \in \mathbb{P}$.
    \item Possess the \textit{up} and \textit{down} operators,
    denoted by
    $\boldsymbol{\mathcal{U}}$
    and
    $\boldsymbol{\mathcal{D}}$, respectively,
		which
    satisfy the Weyl commutation relation
		$
    \big[\boldsymbol{\mathcal{D}}, \boldsymbol{\mathcal{U}}\big]
		= \mathbf{Id}$.
    Here $\boldsymbol{\mathcal{U}}, \, \boldsymbol{\mathcal{D}}$
    act on the vector space $\mathbb{C}[\mathbb{P}]$
    of complex-valued functions on $\mathbb{P}$ as follows:
		\begin{equation*}
			\boldsymbol{\mathcal{U}} \delta_v
			\coloneqq
			\sum_{\substack{w \rhd v \\ |w| = |v| + 1}} \delta_w,
			\qquad
			\boldsymbol{\mathcal{D}} \delta_v
			\coloneqq
			\sum_{\substack{u \lhd v \\ |u| = |v| - 1}} \delta_u,
		\end{equation*}
		where $\delta_v \colon \mathbb{P} \to \mathbb{C}$
		is the indicator function supported at
		$v \in \mathbb{P}$.
\end{enumerate}

An immediate consequence of the Weyl commutation relation is
that $\boldsymbol{\mathcal{D}}^n \ssp \boldsymbol{\mathcal{U}}^n \ssp \delta_\varnothing = n! \ssp \delta_\varnothing$.
This is equivalent to the assertion that
\begin{equation}
\label{eq:sum_of_dimensions_squared}
\sum_{|w| = n} \dim_{\ssp \mathbb{P}}^2(w) = n!,
\end{equation}
where $|w|$ denotes the rank of $w \in \mathbb{P}$, and
$\dim_{\ssp \mathbb{P}}(w)$ represents the number of saturated chains
$w_0 \lhd \cdots \lhd w_n$ in $\mathbb{P}$, beginning at $w_0 = \varnothing$
and terminating at $w_n = w$. Formula \eqref{eq:sum_of_dimensions_squared}
suggests a potential bijection between the set of saturated chains
terminating at rank level $\mathbb{P}_n$,
and permutations $\boldsymbol\upsigma\in\mathfrak{S}_n$.
In the case of the Young lattice $\mathbb{Y}$, such a bijection exists
and is given by the celebrated Robinson--Schensted (RS) correspondence.

The theory of differential posets provides a framework that extends the
RS correspondence beyond the combinatorics of integer partitions. Fomin
\cite{fomin1994duality}, \cite{fomin1995schensted}
demonstrated this generalization
showing that an RS correspondence can be constructed for any differential poset
using his concept of growth processes.
Specifically, an explicit RS correspondence
for the Young--Fibonacci lattice
$\mathbb{YF}$ was developed in \cite{fomin1995schensted},
and later
reformulated into a theory of \emph{standard tableaux} by
Roby \cite{Roby91ThesisRSK}.
A subsequent variant was introduced by Nzeutchap in
\cite{nzeutchap2009young}, which circumvents the Fomin growth process.
In this subsection, we briefly review Nzeutchap's version of
the Young--Fibonacci RS correspondence,
and employ it to get random permutations and involutions.
We remark that other versions of the RS correspondence for $\mathbb{YF}$
are equally applicable for these purposes.
\medskip

Like a partition, a Fibonacci word $w = a_1 \cdots a_k$ of
rank $|w| = a_1+\ldots+a_k  = n$
can be visualized as an arrangement of boxes
called a \emph{Young--Fibonacci diagram}. This diagram
consists of $n$ boxes arranged from left to right into $k$
adjacent columns, where the $i$-th column consists of $a_i$
vertically stacked boxes.
The following example, where $w = 12112211$,
illustrates this concept in \Cref{fig:Young--Fibonacci-diagram}.
\begin{figure}[h]
\[
	\begin{young}
, &  & , & , &  &   \\
& &  &  &  &  & & \\
, 1 & , 2  & , 1 &  , 1 &  , 2 & , 2 & , 1 & , 1
\end{young}
\]
\caption{Young--Fibonacci diagram of $w = 12112211$.} \label{fig:Young--Fibonacci-diagram}
\end{figure}

A \emph{standard Young--Fibonacci tableau} (SYFT) of shape
$w \in \mathbb{YF}_n$ is a labeling of the boxes of the
Young--Fibonacci diagram associated with $w$ using indices from
$\{1, \dots, n\}$ such that:
(i) each index is used exactly once,
(ii) box entries are strictly increasing in columns,
and (iii) the top entry of any column has no entry
greater than itself to is right.
See \Cref{SYFT-example} for an example.

\begin{figure}[h]
\[
\begin{young}
, & 10 & , & , & 6  & 4   \\
11 &  8 & 9 & 7  & 5  & 2 & 3 & 1
\end{young}
\]
\caption{
	Example of a standard Young--Fibonacci tableaux of shape
	$w = 12112211$.}
\label{SYFT-example}
\end{figure}

\begin{remark}
\label{remark:linear-extentions}
A Fibonacci word $w = a_1 \cdots a_k$ can equivalently be depicted
by its {\it rooted tree} $\Bbb{T}_w$. This tree consists of a
horizontal spine with $k$ nodes, where the left-most node serves as
the root. Additionally, a vertical leaf-node is attached to the
$i$-th node on the spine whenever $a_i = 2$. Each edge of the tree
is oriented towards the root, thereby inducing a partial order
$\sqsubseteq$ on the nodes of the tree $\Bbb{T}_w$. Specifically,
$a \sqsubset b$ represents a covering relation if and only if the
nodes $a, b \in \Bbb{T}_w$ are joined by an edge directed from $b$
to $a$.
For example, the tree associated with $w = 12112211$ is illustrated
in \Cref{Young--Fibonacci-tree}.
\begin{figure}[h]
\begin{equation}
\xymatrix{
   & &*+[F]{} \ar[d] & & &*+[F]{} \ar[d] &*+[F]{} \ar[d]   \\
   &*+[F]{} &*+[F]{} \ar[l]  &*+[F]{} \ar[l] &*+[F]{} \ar[l] &*+[F]{} \ar[l]&*+[F]{} \ar[l] &*+[F]{} \ar[l] &*+[F]{} \ar[l]  }
\end{equation}
\caption{Rooted tree $\Bbb{T}_w$ for $w = 12112211$} \label{Young--Fibonacci-tree}
\end{figure}

From this perspective, a SYFT $T$ of shape $w$ corresponds to a
\emph{linear extension} of $\Bbb{T}_w$. This is achieved by
superimposing the cells (and entries) of $T$ onto the nodes of
$\Bbb{T}_w$, replacing each entry~$i$ with $n+1-i$, and then
interchanging the top and bottom entries in each column of height
two. The correspondence between SYFTs $T$ of shape $w$ and 
linear extensions of $\Bbb{T}_w$ is bijective.
See \Cref{Young--Fibonacci-tree-extension} for an illustration.
\begin{figure}[h]
\begin{equation}
\centerline{
\xymatrixrowsep{0.15in}
\xymatrixcolsep{0.2in}
\xymatrix{
   & &*+[F]{4} \ar[d] & & &*+[F]{7} \ar[d] &*+[F]{10} \ar[d]   \\
   &*+[F]{1} &*+[F]{2} \ar[l]  &*+[F]{3} \ar[l] &*+[F]{5} \ar[l] &*+[F]{6} \ar[l]&*+[F]{8} \ar[l] &*+[F]{9} \ar[l] &*+[F]{11} \ar[l]  }}
\end{equation}
\caption{A Llinear extension of $\Bbb{T}_w$
associated to the SYFT in
Figure \ref{SYFT-example}.} \label{Young--Fibonacci-tree-extension}
\end{figure}
\end{remark}

Equation (\ref{eq:dimension_formula}) for $\dim(w)$ is a
restatement of the general \emph{hook-length} formula for counting
linear extensions of a finite, rooted tree,
applied to $\mathbb{T}_w$.
Consequently, the number
of SYFTs of shape $w$ equals $\dim(w)$.
This result can also be understood by constructing a bijection
between SYFTs of shape $w \in \mathbb{YF}_n$ and saturated
chains $w_0 \nearrow \cdots \nearrow w_n$, where $w_0 =
\varnothing$ and $w_n = n$.
Nzeutchap defines such a bijection using an \emph{elimination
map} $\mathcal{E}_n$. This map sends a SYFT $T$ of shape
$w \in \mathbb{YF}_n$ to a SYFT $\mathcal{E}_n[T]$ of shape
$v \in \mathbb{YF}_{n-1}$ such that $w \searrow v$.
See
\Cref{fig:FRSK-example} for an illustration.
For details,
see \cite{nzeutchap2009young}.
\begin{figure}[htpb]
	\centering
	\[
	\begin{array}{ccccccccccccccc}
	\varnothing
	&\nearrow
	&1
	&\nearrow
	&2
	&\nearrow
	&12
	&\nearrow
	&22
	&\nearrow
	&212
	&\nearrow
	&222
	&\nearrow
	&2212 \\ \\
	\varnothing
	&\stackrel{\ssp \mathcal{E}_1}{\longleftarrow}
	&\text{\scriptsize \begin{young}
	\raisebox{.4mm}{1}
	\end{young}}
	&\stackrel{\ssp \mathcal{E}_2}{\longleftarrow}
	&\text{\scriptsize \begin{young}
	\raisebox{.4mm}{2} \\
	\raisebox{.4mm}{1}
	\end{young}}
	&\stackrel{\ssp \mathcal{E}_3}{\longleftarrow}
	&\text{\scriptsize \begin{young}
	, & \raisebox{.4mm}{2} \\ \raisebox{.4mm}{3} &
	\raisebox{.4mm}{1}
	\end{young}}
	&\stackrel{\ssp \mathcal{E}_4}{\longleftarrow}
	&\text{\scriptsize \begin{young}
	\raisebox{.4mm}{4} & \raisebox{.4mm}{2}
	\\ \raisebox{.4mm}{3} & \raisebox{.4mm}{1}
	\end{young}}
	&\stackrel{\ssp \mathcal{E}_5}{\longleftarrow}
	&\text{\scriptsize \begin{young}
	\raisebox{.4mm}{5} & , & \raisebox{.4mm}{2}
	\\ \raisebox{.4mm}{3} & \raisebox{.4mm}{4}
	& \raisebox{.4mm}{1}
	\end{young}}
	&\stackrel{\ssp \mathcal{E}_6}{\longleftarrow}
	&\text{\scriptsize \begin{young}
	\raisebox{.4mm}{6}
	& \raisebox{.4mm}{5} & \raisebox{.4mm}{2}
	\\ \raisebox{.4mm}{3} & \raisebox{.4mm}{4}
	& \raisebox{.4mm}{1}
	\end{young}}
	&\stackrel{\ssp \mathcal{E}_7}{\longleftarrow}
	&\text{\scriptsize \begin{young}
	\raisebox{.4mm}{7}
	& \raisebox{.4mm}{6} & , & \raisebox{.4mm}{2}
	\\ \raisebox{.4mm}{3} & \raisebox{.4mm}{4}
	& \raisebox{.4mm}{5} & \raisebox{.4mm}{1}
	\end{young}}
	\end{array}
	\]
	\caption{Example of elimination maps.}
	\label{fig:FRSK-example}
\end{figure}

The \emph{RS correspondence} for the Young--Fibonacci lattice $\mathbb{YF}$ is a bijection that maps a permutation $\boldsymbol\upsigma \in \mathfrak{S}_n$ to an ordered pair $\mathrm{P}(\boldsymbol\upsigma) \times \mathrm{Q}(\boldsymbol\upsigma)$ of SYFTs, both sharing the same shape $w \in \mathbb{YF}_n$.

Given a permutation $\boldsymbol\upsigma = (\boldsymbol\upsigma_1, \ldots, \boldsymbol\upsigma_n) \in \mathfrak{S}_n$, the \emph{insertion tableau} $\mathrm{P}(\boldsymbol\upsigma)$ and \emph{recording tableau} $\mathrm{Q}(\boldsymbol\upsigma)$ are constructed as follows:
\begin{enumerate}[\bf1.\/]
	\item Read $\boldsymbol\upsigma$ from right to left.
	\item
		For each index $\boldsymbol\upsigma_k$ (proceeding from right to
		left), match it with the maximal unmatched index to its
		left (including itself if no such index exists).
	\item To construct $\mathrm{P}(\boldsymbol\upsigma)$, place the matched indices
		into columns of height two (or leave single unmatched
		indices as columns of height one), in the order of reading
		$\boldsymbol\upsigma$ from right to left.
		Place the larger
		value of each pair at the top of its column. Assemble
		these columns from left to right in the tableau.
	\item To construct $\mathrm{Q}(\boldsymbol\upsigma)$, replace each entry $\boldsymbol\upsigma_k\in \mathrm{P}(\boldsymbol\upsigma)$ with its index $k$.
		For columns of height two, swap the entries between the
		top and bottom positions.
\end{enumerate}
\Cref{FRSK-example} illustrates this process for $\boldsymbol\upsigma = (2, 7, 1, 5, 6, 4, 3)$.

\begin{figure}[h]
	\raisebox{14.5pt}{\begin{minipage}[b]{0.45\linewidth}
\centering
\[
\xymatrixrowsep{0.10in}
\xymatrixcolsep{0.2in}
\xymatrix{
2 & 7  & 1  \ar `d/2pt[dl] `/4pt[ll] [ll] & 5  \ar@(dr,dl) & 6 & 4  \ar `d/2pt[dl] `/4pt[l] [l]  & 3  \ar `d/2pt[ddl]  `/4pt[lllll] [lllll]  \\
& & & & & & &  \\  & & & & & & & } \]
\end{minipage}}
\hspace{.5cm}
\begin{minipage}[b]{0.45\linewidth}
\centering
\[
\underbrace{\begin{young}  7 & 6 & , & 2     \\ 3 & 4 & 5 & 1  \end{young}}_{\text{$\mathrm{P}(\boldsymbol\upsigma)$-tableau}} \quad
\underbrace{\begin{young}  7 & 6 & , & 3     \\ 2 & 5 & 4 & 1  \end{young}}_{\text{$\mathrm{Q}(\boldsymbol\upsigma)$-tableau}}
\]
\end{minipage}
\caption{The Young--Fibonacci RS correspondence
for $\boldsymbol\upsigma =  (2,7,1,5,6,4,3)$.}
\label{FRSK-example}
\end{figure}

The Young--Fibonacci RS correspondence enjoys many of the
features of the classical RS correspondence
together with many novel features. For
example, we have \cite{nzeutchap2009young}:

\begin{enumerate}[\bf 1.\/]
    \item $\mathrm{P}(\boldsymbol\upsigma^{-1}) = \mathrm{Q}(\boldsymbol\upsigma)$ and $\mathrm{Q}(\boldsymbol\upsigma^{-1}) = \mathrm{P}(\boldsymbol\upsigma)$.

    \item $\mathrm{P}(\boldsymbol\upsigma) = \mathrm{Q}(\boldsymbol\upsigma)$ if and only if $\boldsymbol\upsigma$ is an involution. Furthermore, the cycle decomposition of an involution $\boldsymbol\upsigma$ can be determined from the columns of $\mathrm{P}(\boldsymbol\upsigma)$ as follows:
    \begin{enumerate}[\bf a.\/]
        \item The number of two-cycles in $\boldsymbol\upsigma$ is $\mathbcal{h}(w)$, the total number of digits $2$ in $w$.
        \item The number of fixed points of $\boldsymbol\upsigma$ is $\mathbcal{r}(w)$, the total number of digits $1$ in $w$.
    \end{enumerate}
    Here, $w \in \mathbb{YF}$ is the shape of $\mathrm{P}(\boldsymbol\upsigma) = \mathrm{Q}(\boldsymbol\upsigma)$.

	\item $\mathrm{Q}(\boldsymbol\upsigma') = \mathcal{E}_n \big[\mathrm{Q}(\boldsymbol\upsigma)\big]$, where $\boldsymbol\upsigma = (\boldsymbol\upsigma_1, \dots, \boldsymbol\upsigma_n)$, and $\boldsymbol\upsigma'$ is the \emph{standardization} of $(\boldsymbol\upsigma_1, \dots, \boldsymbol\upsigma_{n-1})$
		(that is, the permutation $\boldsymbol\upsigma'\in \mathfrak{S}_{n-1}$
		preserves the relative order of the entries of $\boldsymbol\upsigma$).

    \item The lexicographically minimal, reduced factorization $s_{j_1 \downarrow r_1} \cdots \ssp s_{j_k \downarrow r_k}$ of $\boldsymbol\upsigma$ can be derived from the (appropriately defined) \emph{inversions} in the tableaux $\mathrm{P}(\boldsymbol\upsigma)$ and $\mathrm{Q}(\boldsymbol\upsigma)$. Here we use the notations
	$s_j=(j,j+1)$ and $s_{j \downarrow r} = s_j \cdots s_{j-r+1}$.
	For further details, see \cite{hivert_scott2024diagram}.
\end{enumerate}

\subsection{Transition and cotransition measures for the Young--Fibonacci lattice}
\label{sub:transition-cotransition-measures}

We first recall the standard construction of cotransition and
transition probabilities
for the Young--Fibonacci lattice.
These notions are associated
with general branching graphs
(e.g., see \cite{borodin2016representations}).

Let
$M_n$ on $\mathbb{YF}_n$
be a coherent family of measures associated
with a positive normalized harmonic function $\varphi$ by \eqref{eq:coherent_measure_for_harmonic_function}.
The coherence property \eqref{eq:coherence_property_standard_down_transitions}
is equivalent to the fact that the $M_n$'s
are compatible with the (\emph{standard}) \emph{cotransition probabilities}
\begin{equation}
	\label{eq:standard_cotransition_probabilities}
	\mu_{\scriptscriptstyle\mathrm{CT}}^{\scriptscriptstyle\mathrm{std}}(w , v)
	\coloneqq \frac{\dim v}{\dim w}, \qquad
	v\in \mathbb{YF}_{n-1},\ w\in \mathbb{YF}_n.
\end{equation}
If $w\not \searrow v$, we set
$\mu_{\scriptscriptstyle\mathrm{CT}}^{\scriptscriptstyle\mathrm{std}}$
to zero.
Note that
$\mu_{\scriptscriptstyle\mathrm{CT}}^{\scriptscriptstyle\mathrm{std}}$ do not depend $\varphi$.

Using the cotransition probabilities
\eqref{eq:standard_cotransition_probabilities}, we can define the
joint distribution on $\mathbb{YF}_{n-1}\times \mathbb{YF}_n$
with marginals $M_{n-1}$ and $M_n$, whose
conditional distribution from level $n$ to $n-1$ is
given by $\mu_{\scriptscriptstyle\mathrm{CT}}^{\scriptscriptstyle\mathrm{std}}$.
The conditional distribution in the other direction is,
by definition, given by the \emph{transition probabilities}, which
now depend on~$\varphi$:
\begin{equation}
	\label{eq:transition_probabilities}
	\mu_{\scriptscriptstyle\mathrm{T}}^{\scriptscriptstyle\varphi}(v , w)
	\coloneqq \frac{\varphi(w)}{\varphi(v)}, \qquad
	v\in \mathbb{YF}_{n-1},\ w\in \mathbb{YF}_n
\end{equation}
(and this is zero if $w\not \searrow v$).

Using the transition probabilities \eqref{eq:transition_probabilities},
we can define probability distributions on arbitrary saturated chains
from $w_0=\varnothing$ to $\mathbb{YF}_n$:
\begin{equation}
	\label{eq:transition_probabilities_on_saturated_chains}
	\overline{\mu}_{\scriptscriptstyle\mathrm{T}}
	^{\ssp\scriptscriptstyle\varphi}\big(w_0 \nearrow \cdots \nearrow w_n\big)
	\coloneqq \prod_{k=1}^n \mu_{\scriptscriptstyle\mathrm{T}}
	^{\scriptscriptstyle\varphi}\big(w_{k-1} , w_k\big)
	= \varphi(w_n),\qquad w_n\in \mathbb{YF}_n.
\end{equation}
Note that the
distribution
\eqref{eq:transition_probabilities_on_saturated_chains}
is uniform for all chains that end at the same Fibonacci word $w_n$.
This is known as the \emph{centrality} property in the
works of Vershik and Kerov (e.g., see \cite{VK81AsymptoticTheory}).
The transition probabilities associated with a harmonic function
$\varphi$ define an infinite random walk on the Young--Fibonacci lattice
starting from $\varnothing$. The probability
that the random walk passes through a given Fibonacci word $w\in \mathbb{YF}_n$
is equal to $M_n(w)=\dim w\cdot \varphi(w)$.

\medskip

Let us now define a new family of cotransition probabilities
which are associated to a given positive normalized harmonic function $\varphi$.
We emphasize that this construction is specific to the
``reflective'' nature of the Young--Fibonacci lattice;
namely, that $v \searrow u$
if and only if $2u \searrow v$.

\begin{definition}[Cotransition probabilities for an arbitrary harmonic function]
	\label{def:cotransition_probabilities_for_harmonic_function}
	For $w\in \mathbb{YF}_n$ and $v\in \mathbb{YF}_{n-1}$,
	let us define the
	(\emph{generalized})
	\emph{cotransition probabilities}
	\begin{equation}
		\label{eq:cotransition_probabilities_for_harmonic_function}
		\mu_{\scriptscriptstyle\mathrm{CT}}
		^{\scriptscriptstyle\varphi}(w , v)
		\coloneqq
		\begin{cases}
			1, & \text{if $w = 1v$}, \\
			\displaystyle\frac{\varphi(v)}{\varphi(u)}, & \text{if $w = 2u$},\\
			0,& \text{if $w\not \searrow v$}.
		\end{cases}
	\end{equation}
	One can readily check that in the Plancherel case $\varphi
	=
	\varphi_{_\mathrm{PL}}$ \eqref{eq:Plancherel_harmonic_function},
	the cotransition probabilities
	\eqref{eq:cotransition_probabilities_for_harmonic_function}
	become the standard ones
	from \eqref{eq:standard_cotransition_probabilities}.
\end{definition}

\begin{proposition}
	\label{proposition:cotransition_probabilities_for_harmonic_function}
	Expression \eqref{eq:cotransition_probabilities_for_harmonic_function}
	indeed defines probabilities, that is,
	\begin{equation}
		\label{eq:cotransition_probabilities_for_harmonic_function_sums_to_one}
		\sum_{v\in \mathbb{YF}_{n-1}} \mu_{\scriptscriptstyle\mathrm{CT}}
		^{\scriptscriptstyle\varphi}(w , v) = 1,
		\qquad
		w\in \mathbb{YF}_n.
	\end{equation}
\end{proposition}
\begin{proof}
	If $w$ starts with $1$, then there is only one possibility for $v$
	corresponding to $w=1v$, and
	\eqref{eq:cotransition_probabilities_for_harmonic_function_sums_to_one}
	is evident.
	Otherwise, for $w=2u$, the edges $w\searrow v$
	are in one-to-one correspondence with the edges $u\nearrow v$.
	The harmonicity of $\varphi$ implies that
	\begin{equation*}
		\sum_{v\in \mathbb{YF}_{n-1}} \varphi(v) = \varphi(u),
	\end{equation*}
	which is equivalent to
	\eqref{eq:cotransition_probabilities_for_harmonic_function_sums_to_one}.
	This completes the proof.
\end{proof}

Similarly to
$\overline{\mu}_{\scriptscriptstyle\mathrm{T}}^{\scriptscriptstyle\varphi}$
\eqref{eq:transition_probabilities_on_saturated_chains},
we can define the cotransition probabilities on
all saturated chains that start at a fixed Fibonacci
word $w_n\in \mathbb{YF}_n$ and terminate at $w_0=\varnothing$:
\begin{equation}
	\label{eq:cotransition_probabilities_on_saturated_chains}
	\overline{\mu}^{\ssp \scriptscriptstyle \varphi}_{\scriptscriptstyle \mathrm{CT}} (
	w_n \searrow \cdots
	\searrow w_1 \searrow w_0 )
	\coloneqq\, \prod_{k=1}^n \,
	\mu^{\scriptscriptstyle \varphi}_{\scriptscriptstyle \mathrm{CT}}
	(w_k , w_{k-1} ).
\end{equation}
The measure \eqref{eq:cotransition_probabilities_on_saturated_chains}
is uniform on all chains if and only if
$\varphi=\varphi_{\scriptscriptstyle\mathrm{PL}}$, the
Plancherel harmonic function.
Let us write
$\overline{\mu}^{\ssp \scriptscriptstyle \varphi}_{\scriptscriptstyle \mathrm{CT}}(T)
= \overline{\mu}^{\ssp \scriptscriptstyle \varphi}_{\scriptscriptstyle
\mathrm{CT}}(w_n \searrow \cdots \searrow w_1 \searrow w_0)$
whenever $T$ is the SYFT associated
to the saturated chain
$w_0 \nearrow \cdots \nearrow w_n$
as in the example in \Cref{fig:Young--Fibonacci-diagram}.
We refer to \cite{nzeutchap2009young} for details,
and examples of the generalized cotransition probabilities
are given in
\Cref{fig:cotransition-weights}.

\begin{figure}[htpb]
	\centering
	\begin{equation*}
\xymatrixrowsep{0.43in}
\xymatrixcolsep{0.40in}
\xymatrix{ \varnothing \ar@{-}[d] & & & & & & &  \\
1 \ar@{-}[d] \ar@{-}[drrrrrrr] & & & & & & & \\
11  \ar@{-}[d] \ar@{-}[drrrr]^{\color{red}
\varphi(11) } & & & & & & & 2  \ar@{-}[d] \ar@{-}[dlll]_{\color{red}
\varphi(2)} \\
111  \ar@{-}[d] \ar@{-}[drr]^{\color{red}
\frac{\varphi(111)}{\varphi(11)}} & & & & 21 \ar@{-}[d] \ar@{-}[dll]_{\color{red}
\frac{\varphi(21)}{\varphi(11)}}
\ar@{-}[dr]^{\color{red}
\frac{\varphi(21)}{\varphi(2)} } & & &
12  \ar@{-}[d] \ar@{-}[dll]_{\color{red}
\frac{\varphi(12)}{\varphi(2)}} \\
1111  \ar@{-}[d]  \ar@{-}[dr]^{\color{red} \frac{\varphi(1111)}{\varphi(111)}}  & &211  \ar@{-}[d] \ar@{-}[dl]_{\color{red} \frac{\varphi(211)}{\varphi(111)} }
\ar@{-}[dr]^{\color{red} \! \frac{\varphi(211)}{\varphi(21)} } & &121
\ar@{-}[d] \ar@{-}[dl]_{\color{red} \!\frac{\varphi(121)}{\varphi(21)} }  &22 \ar@{-}[d]  \ar@{-}[dll]^(.35){\color{red} \ \ \frac{\varphi(22)}{\varphi(21)} }|\hole
\ar@{-}[dr]^{\color{red} \frac{\varphi(22)}{\varphi(12)} }
& &112  \ar@{-}[d] \ar@{-}[dl]_{\color{red} \frac{\varphi(112)}{\varphi(12)} }  \\
11111 & 2111 & 1211 & 221 & 1121 & 122 & 212 & 1112 }
\end{equation*}
\vspace{25pt}
\[
\begin{array}{rllllc}
\text{\begin{young}
3 & 2 & 1
\end{young}}
&: \, \big( \varnothing
&\nearrow  \, 1
&\nearrow  \, 11
&\nearrow  \, 111 \big) \, :
&1 \\ \\
\text{\begin{young}
, & 2 \\ 3  & 1
\end{young}}
&: \, \big( \varnothing
&\nearrow \, 1
&\nearrow \, 2
&\nearrow \, 12 \ \, \big) \, :
&1 \\ \\
\text{\begin{young}
3 & , \\ 2  & 1
\end{young}}
&: \, \big( \varnothing
&\nearrow \, 1
&\nearrow \, 11
&\nearrow \, 21 \ \, \big) \, :
&\displaystyle \varphi(11) \\ \\
\text{\begin{young}
3 & , \\ 1  & 2
\end{young}}
&: \, \big( \varnothing
&\nearrow \, 1
&\nearrow \, 2
&\nearrow \, 21 \ \, \big) \, :
&\displaystyle \varphi(2)
\end{array}
\]
\caption{Top: Nonzero cotransition weights (in red with $1$'s omitted).
Bottom: The four
saturated chains which
terminate in $\mathbb{YF}_3$, together
with their associated
SYFTs and cotransition weights.}
\label{fig:cotransition-weights}
\end{figure}

We will typically be interested
in the case when $\varphi = \varphi_{\vec x, \vec y}$
is a clone harmonic function coming from a
Fibonacci positive specialization
$(\vec{x}, \vec{y} \ssp )$.

\subsection{Building random permutations and involutions}
\label{sub:random-permutation-involutions}

To construct a random permutation in $\mathfrak{S}_n$, we observe
that the RS correspondence
from \Cref{sub:UD_YF_operators}
uniquely determines a permutation
$\boldsymbol\upsigma$ by three components, namely, a random shape $w \in \mathbb{YF}_n$,
and two random saturated chains in $\mathbb{YF}$, both terminating
at $w$. This construction proceeds as follows:
\begin{enumerate}[\bf{}1.\/]
	\item
		First, select a Fibonacci word $w \in \mathbb{YF}_n$ with probability
		$M_n(w)$, determined by a positive harmonic function $\uppi$.
	\item
		Next, generate two saturated chains
		terminating at $w$
		using the cotransition probabilities $\overline{\mu}^{\ssp
		\scriptscriptstyle \varphi}_{\scriptscriptstyle
		\mathrm{CT}}$ and $\overline{\mu}^{\ssp
		\scriptscriptstyle \psi}_{\scriptscriptstyle
		\mathrm{CT}}$, associated with two (possibly different)
		positive harmonic
		functions $\varphi$ and $\psi$.
		The chains are conditioned to end at the previously
		chosen Fibonacci word $w$.
	\item From these two chains (viewed as SYFTs),
		construct a permutation $\boldsymbol\upsigma$ using the RS correspondence.
\end{enumerate}
In this way, the triad $(\uppi, \varphi, \psi)$ of
positive harmonic functions
determines a random permutation $\boldsymbol\upsigma\in \mathfrak{S}_n$ for
every $n \geq 1$.

Similarly, to construct a random involution in $\mathfrak{S}_n$,
we pick $w\in \mathbb{YF}_n$ according to $M_n(w)$ (determined
by $\uppi$), and generate a single saturated chain terminating at $w$,
sampled according to the cotransition probabilities
$\overline{\mu}^{\ssp \scriptscriptstyle \varphi}
_{\scriptscriptstyle \mathrm{CT}}$.

Summarizing, we have the following
probability
measures on
permutations and involutions
in $\mathfrak{S}_n$
denoted by $\mu_n$ and $\nu_n$, respectively:
\begin{equation}
	\label{eq:random_permutation_involutions}
	\begin{split}
		\mu_n(\boldsymbol\upsigma)
		=
		\mu_n(\boldsymbol\upsigma\mid \uppi, \varphi, \psi)
		&\coloneqq
		\dim(w) \ssp \uppi(w) \,
		\overline{\mu}^{\ssp \scriptscriptstyle
		\varphi}_{\scriptscriptstyle \mathrm{CT}} \big(
		\mathrm{P}(\boldsymbol\upsigma)\big) \, \overline{\mu}^{\ssp
		\scriptscriptstyle \psi}_{\scriptscriptstyle \mathrm{CT}}
		\big( \mathrm{Q}(\boldsymbol\upsigma)\big),\\
		\nu_n(\boldsymbol\upsigma)
		=
		\nu_n(\boldsymbol\upsigma\mid \uppi, \varphi)
		&\coloneqq
		\dim(w) \ssp \uppi(w) \,
		\overline{\mu}^{\ssp \scriptscriptstyle \varphi}
		_{\scriptscriptstyle \mathrm{CT}} \big(
		\mathrm{P}(\boldsymbol\upsigma)\big).
	\end{split}
\end{equation}

For example, the distribution $\mu_3$ on $\mathfrak{S}_3$ has the form
(writing permutations in the one-line notation):
\begin{equation*}
	\begin{array}{rclrcl}
		\mu_3(123) &=& \uppi(111),& \mu_3(213) &=& \uppi(12),\\
		\mu_3(132) &=& 2 \ssp \uppi(21) \ssp \varphi(11) \ssp \psi(11),
		& \mu_3(321) &=& 2 \ssp \uppi(21) \ssp \varphi(2) \ssp \psi(2),\\
		\mu_3(312) &=& 2 \ssp \uppi(21) \ssp \varphi(11) \ssp \psi(2),
		& \mu_3(231) &=& 2 \ssp \uppi(21) \ssp \varphi(2) \ssp \psi(11).
	\end{array}
\end{equation*}

\begin{remark}[Plancherel cases]
	When $\uppi=\varphi=\psi=\varphi_{\scriptscriptstyle\mathrm{PL}}$,
	$\mu_n$ is simply the uniform measure on $\mathfrak{S}_n$.
	More generally, when $\varphi=\psi=\varphi_{\scriptscriptstyle\mathrm{PL}}$,
	each permutation $\boldsymbol\upsigma \in \mathfrak{S}_n$
	with the RS shape $w \in \mathbb{YF}_n$
	occurs with probability $\mu_n(\boldsymbol\upsigma)=\uppi(w)/\dim(w)$.

	For the model of
	random involutions,
	when $\varphi=\varphi_{\scriptscriptstyle\mathrm{PL}}$,
	each involution $\boldsymbol\upsigma \in \frak{S}_n$
	with the RS shape $w \in \mathbb{YF}_n$
	occurs with probability $\nu_n(\boldsymbol\upsigma)=\uppi(w)$.
\end{remark}

\section{Observables from Cauchy Identities and their Asymptotics}
\label{sec:observables_from_Cauchy_identities}

Here we illustrate the connection between random
permutations (and involutions) and clone Schur
functions discussed in \Cref{sec:random_Fibonacci_words_to_random_permutations} above.
Specifically, we compute the
expected numbers of fixed points and of two-cycles in a
random involution $\boldsymbol\upsigma \in \mathfrak{S}_n$
distributed according to the measure $\nu_n$ defined by
\eqref{eq:random_permutation_involutions}.
We subsequently apply this formula to the shifted Plancherel
specialization, providing another perspective on the
scaling limit of the corresponding random Fibonacci words,
which complements the results of \Cref{sec:asymptotics_shifted_Plancherel_specialization}.

\subsection{Quadridiagonal determinantal formula for a generating function}
\label{sub:quadridiagonal_determinantal_formula}

In the definition of the measure $\nu_n$ in
\eqref{eq:random_permutation_involutions},
let us set
$\varphi=\varphi_{\scriptscriptstyle\mathrm{PL}}$, and take
$\uppi=\varphi_{\vec{x},\vec{y}}$ for some
Fibonacci positive specialization $(\vec{x},\vec{y})$.

Recall from \Cref{sub:UD_YF_operators} that
the fixed points and two-cycles of an involution
$\boldsymbol\upsigma \in \mathfrak{S}_n$ correspond, respectively, to the
digits $1$ and $2$
in the shape
$w \in \mathbb{YF}_n$ associated with $\boldsymbol\upsigma$ under the
Young--Fibonacci RS correspondence.
Let $\mathbcal{r}(w)$ and $\mathbcal{h}(w)$,
respectively, denote the total number of $1$'s and $2$'s in $w$.
Rather than directly computing the expectations
of $\mathbcal{r}(w)$ and $\mathbcal{h}(w)$,
we introduce an auxiliary parameter $\tau$
and calculate the expectation of
$\tau^{\mathbcal{h}(w)}=\tau^{\# \ssp \mathrm{two \text{-} cycles}(\boldsymbol\upsigma)}$.
This approach leverages the first clone Cauchy identity
from \Cref{sub:clone-cauchy}.
Since $\mathbcal{r}(w)+2\mathbcal{h}(w)=n$, knowing the distribution of $\mathbcal{h}(w)$ would, in principle, determine that of $\mathbcal{r}(w)$ as well; however, we do not consider $\mathbcal{r}(w)$ further here.

\begin{proposition}
	\label{proposition:expectation-cycles-involution}
	The expected value of
	$\tau^{\# \ssp \mathrm{two \text{-} cycles}(\boldsymbol\upsigma) }$
	for a random
	involution $\boldsymbol\upsigma \in \frak{S}_n$
	distributed according to
	$\nu_n(\boldsymbol\upsigma\mid \varphi_{\vec{x},\vec{y}},
	\varphi_{\scriptscriptstyle\mathrm{PL}})$
	\eqref{eq:random_permutation_involutions}
	is given by the following quadridiagonal determinant:
	\begin{equation}
		\label{eq:expectation-cycles-involution_new}
		\operatorname{\mathbb{E}}_{\nu_n}\bigl[ \tau^{\# \ssp
		\mathrm{two \text{-} cycles}(\boldsymbol\upsigma) } \bigr]
		=
		(x_1 \cdots x_n)^{-1} \ssp
		\det \underbrace{\begin{pmatrix}
		x_1
		&(1 - \tau )  y_1
		& - \tau x_1 y_2
		&0
		&\cdots \\
		1
		&x_2
		&(1  - 2 \tau)y_2
		&- 2 \tau x_2 y_3
		& \\
		0
		&1
		&x_3
		&(1  - 3 \tau )y_3
		& \\
		0
		&0
		&1
		&x_4
		& \\
		\vdots & & & &\ddots
		\end{pmatrix}}_{n \times n \ \mathrm{quadridiagonal \, matrix}} .
	\end{equation}
\end{proposition}
\begin{proof}
	The left-hand side of \eqref{eq:expectation-cycles-involution_new}
	can be rewritten using clone Schur functions as
	\begin{equation*}
		\sum_{|w| = n}
		\dim(w) \ssp
		\varphi_{\vec{x}, \vec{y}} \ssp (w) \ssp
		\tau^{\mathbcal{h}(w)}
		=
		(x_1 \cdots x_n)^{-1} \ssp
		\sum_{|w| = n}
		\dim(w) \ssp
		s_w(\vec{x} \mid \vec{y} \ssp) \ssp
		\tau^{\mathbcal{h}(w)}.
	\end{equation*}
	Setting
	$p_k = x_k^{-1}$ and $q_k = k \tau x_{k}^{-1} x_{k+1}^{-1}$
	in the clone Cauchy identity
	\eqref{eq:first-clone-cauchy}
	and noticing that
	$h_w (\vec{p} \mid \vec{q} \ssp)
	= (x_1 \cdots x_n)^{-1}\dim(w) \ssp \tau^{\mathbcal{h} (w)}$
	under this specialization
	implies the desired quadridiagonal determinant.
\end{proof}

\begin{remark}
	If we set $p_k=\tau x_{k}^{-1}$ and $q_k=k x_k^{-1} x_{k+1}^{-1}$ in the proof of
	\Cref{proposition:expectation-cycles-involution},
	we would get the expected number of fixed points
	of a random involution
	distibuted according to
	$\nu_n(\boldsymbol\upsigma\mid
	\varphi_{\vec{x},\vec{y}},\varphi_{\scriptscriptstyle\mathrm{PL}})$.
\end{remark}

The expected number of two-cycles can be computed
in a standard way, by differentiating:
\begin{equation*}
	\operatorname{\mathbb{E}}_{\nu_n}
	\left[
		\# \ssp \mathrm{two \text{-} cycles}(\boldsymbol\upsigma)
	\right]=
	\frac{\partial}{\partial \tau}\bigg|_{\tau =1}
	\operatorname{\mathbb{E}}_{\nu_n}
	\bigl[
		\tau^{\# \ssp \mathrm{two \text{-} cycles}(\boldsymbol\upsigma)}
	\bigr].
\end{equation*}
This differentiation
of a quadridiagonal determinant
\eqref{eq:expectation-cycles-involution_new}
is not explicit
for a general Fibonacci
positive specialization
$(\vec{x}, \vec{y} \ssp )$. In the next
\Cref{sub:two_cycles},
we consider the particular case of
the shifted Plancherel specialization
$x_k=y_k=k+\sigma-1$ for $\sigma \in [1, \infty)$
(\Cref{def:shifted_Plancherel_and_Charlier} with $\rho=1$).

\subsection{Number of two-cycles under the shifted Plancherel specialization}
\label{sub:two_cycles}

Consider
\begin{equation}
	\label{eq:shifted-plancherel-expectation-calH-n-def}
	H_n(\sigma, \tau) \coloneqq \operatorname{\mathbb{E}}_{\nu_n}
	\bigl[ \tau^{\# \ssp \mathrm{two \text{-} cycles}(\boldsymbol\upsigma)} \bigr],\qquad
	\nu_n=\nu_n(\cdot\mid \varphi_{\vec{x},\vec{y}},\varphi_{\scriptscriptstyle\mathrm{PL}}),
	\qquad
	x_k=y_k=k+\sigma-1,
\end{equation}
where $\sigma \in [1, \infty)$ is the parameter of the shifted Plancherel specialization
(not to be confused with the random involution $\boldsymbol\upsigma$).
Denote also
\begin{equation*}
	G_n(\sigma)\coloneqq \operatorname{\mathbb{E}}_{\nu_n}
	\left[ \# \ssp \mathrm{two \text{-} cycles}(\boldsymbol\upsigma) \right].
\end{equation*}

We side-step the differentiation of the quadridiagonal determinant, and instead work directly
with the $H_n(\sigma,\tau)$'s, and their generating function
\begin{equation}
	\label{eq:shifted-plancherel-expectation-generating-function}
	H(\sigma, \tau ; z) \coloneqq \sum_{n \geq 0} H_n(\sigma, \tau) \ssp z^n.
\end{equation}

\begin{lemma}
	\label{lemma:calH_n_recurrence}
	In the case of the shifted Plancherel specialization,
	the quantities $H_n(\sigma, \tau)$ \eqref{eq:shifted-plancherel-expectation-calH-n-def}
	satisfy the inhomogeneous, two-step recurrence
	\begin{equation}
		\label{two-step-recurrence-shifted-plancherel}
		(n + \sigma - 1) \ssp H_n
		= H_{n-1}
		+ \tau(n - 1) H_{n-2}
		+ (n + \sigma - 1) \ssp
		\varphi_{\vec{x}, \vec{y}} \ssp (1^n)
		- \varphi_{\vec{x}, \vec{y}} \ssp (1^{n-1}).
	\end{equation}
\end{lemma}
\begin{proof}
	A crucial property of the shifted Plancherel specialization is that
	\begin{equation}
		\label{eq:shifted-plancherel-specialization-1w}
		\varphi_{\vec x, \vec y} \ssp (1w) =\frac{ \varphi_{\vec x, \vec y} \ssp (w) }{ x_n},\qquad
		\varphi_{\vec x, \vec y} \ssp (2v) =\frac{ \varphi_{\vec x, \vec y} \ssp (v) }{ x_n},
	\end{equation}
	for any Fibonacci word $w \in
	\mathbb{YF}_{n-1}$ which does not consist entirely of
	$1$-digits, and any Fibonacci word $v\in \mathbb{YF}_{n-2}$. Indeed, this is because for the shifted
	Plancherel specialization, we have for the second determinant
	in \eqref{eq:A_B_dets}:
	\begin{equation*}
		B_k(m)=m+\sigma=x_{m+1}, \qquad k\ge 0.
	\end{equation*}
	Moreover, $\dim (1w)=\dim(w)$. This implies \eqref{eq:shifted-plancherel-specialization-1w}.
	Now,
	\begin{equation*}
		 H_n(\sigma, \tau) =
		\sum_{|w|  =  n}
		\dim(w) \ssp
		\varphi_{\vec x, \vec y} \ssp (w)
		\ssp \tau^{\mathbcal{h}(w)}.
	\end{equation*}
	Split the sum into three parts: $w=1^n$, $w=1u$, and $w=2v$.
	Rewriting the second two sums in terms of $H_{n-1}$ and $H_{n-2}$, respectively,
	yields the desired recurrence \eqref{two-step-recurrence-shifted-plancherel}.
\end{proof}

Note that for the shifted Plancherel specialization, we have
\begin{equation}
	\label{eq:shifted-plancherel-specialization-1}
	(n+\sigma-1)\ssp \varphi_{\vec{x}, \vec{y}} \ssp (1^n) -
	\varphi_{\vec{x}, \vec{y}} \ssp (1^{n-1}) = \sigma-1, \qquad n \geq 1.
\end{equation}

\begin{lemma}
	\label{lemma:ode-evgf-shifted-plancherel}
	The generating function $H(\sigma, \tau ; z)$
	\eqref{eq:shifted-plancherel-expectation-generating-function}
	satisfies the first order ODE:
	\begin{equation}
		\label{eq:ode-H-sigma-tau-z}
		z(1 - \tau z^2) \ssp \partial_z \ssp H(\sigma, \tau ; z) \, + \,
		(\sigma - 1 - z - \tau z^2) \ssp H(\sigma,\tau ; z)
		\, = \, {\sigma -1 \over {1-z}}
	\end{equation}
\end{lemma}
\begin{proof}
	This immediately follows from the recurrence
	in \Cref{lemma:calH_n_recurrence}
	and the identity \eqref{eq:shifted-plancherel-specialization-1}.
\end{proof}

Consider first the case $\sigma=1$
(the usual Plancherel specialization $\uppi=\varphi_{\scriptscriptstyle\mathrm{PL}}$).
Then
the ODE in
\Cref{lemma:ode-evgf-shifted-plancherel}
admits an explicit solution:
\begin{equation*}
	H(1, \tau \ssp ; z)
	= \frac{1}{\sqrt{1 - \tau z^2}}
	\Bigg( \frac{1 + \sqrt{\tau}z}{1 - \sqrt{\tau}z} \Bigg)^{\frac{1}{2 \sqrt{\tau}}}.
\end{equation*}
Taking the $\tau$-derivative of the above expression
at $\tau=1$, we see that
\begin{equation}
	\label{eq:H_1_z}
	\sum_{n\ge0} G_n(1) \ssp z^n =
	{z \over {2 (1-z)^2}}  +
	{1 \over {4(1-z)}} \log \left(
	{1 -z \over {1+z}} \right),\qquad |z|<1.
\end{equation}
The dominating singularity of this function is $z=1$.
The first summand expands as $\sum_{n\ge0}\frac n2\ssp z^n$, and one can readily check that the
coefficients of the second summand are asymptotically bounded in $n$.
We conclude that
\begin{equation}
	\label{eq:limit-two-cycles-plancherel}
	\lim_{n\to\infty}\frac{G_n(1)}{n}=\frac{1}{2}.
\end{equation}

\begin{remark}
	The limit
	\eqref{eq:limit-two-cycles-plancherel}
	aligns with the result of \cite{gnedin2000plancherel}
	(cited in \Cref{sub:Plancherel_measure_and_scaling}),
	which states that under the Plancherel measure, the
	frequency of hikes of $2$'s in the random Fibonacci word
	(the number of two-cycles in the corresponding random permutation $\boldsymbol\upsigma\in \mathfrak{S}_n$)
	scales proportionally to $n$. Moreover, asymptotically, $1$'s (fixed points of $\boldsymbol\upsigma$)
	do not have a significance presence.
\end{remark}

Consider now the general case $\sigma \in [1, \infty)$.
It is not clear to the authors how to
express solutions to the ODE of \Cref{lemma:ode-evgf-shifted-plancherel},
even in terms of hypergeometric
functions. Nevertheless,
after differentiating
the ODE in $\tau$, setting $\tau=1$,
and using the
fact
$H(\sigma, 1; z) = (1-z)^{-1}$,
we obtain an ODE
for
\begin{equation*}
	G(\sigma;z)\coloneqq \partial_\tau |_{\tau =1} \ssp H(\sigma, \tau ; z)
	= \sum_{n \geq 0} G_n(\sigma) \ssp z^n,
	\qquad
	G_n(\sigma)=\operatorname{\mathbb{E}}_{\nu_n}
	\left[ \# \ssp \mathrm{two \text{-} cycles}(\boldsymbol\upsigma) \right].
\end{equation*}
The new ODE has the form:
\begin{equation}
	\label{eq:ode-H-sigma-z}
	z(1 - z^2) \ssp \partial_z G(\sigma;z)
	\, + \,
	(\sigma - 1 - z - z^2) \ssp
	G(\sigma;z)
	=  {z^2 \over {(1-z)^2}},
\end{equation}
whose solution can be expressed through the hypergeometric functions:
\begin{equation}
	\label{eq:H-sigma-z}
	\begin{split}
		G(\sigma;z)&=
		\frac{z^{1-\sigma} \ssp (1+z)^\sigma}
				 { 2^\sigma
                 (1 + \sigma) \ssp (1-z)^2} \,
		{}_2F_1 \left(
				-\frac{1 + \sigma}{2},
				-\sigma;
				\frac{1 - \sigma}{2};
				\frac{1-z}{1+z}
		\right)
		\\&\hspace{90pt}
		- \frac{\Gamma(1 + \sigma) \,
		\Gamma\big(\frac{1}{2} - \frac{\sigma}{2}\big)}
		{2^\sigma (1+\sigma) \ssp
		\Gamma\big(\frac{1}{2} + \frac{\sigma}{2}\big)} \,
		z^{1-\sigma} (1-z)^{-1} (1-z^2)^{(\sigma-1)/2}.
\end{split}
\end{equation}
The hypergeometric function makes sense unless $\sigma$ is an odd positive integer.
In the latter case, the singularities in the first and the second summand
cancel out, and the whole function $G(\sigma;z)$ is well-defined for all $\sigma\in[1,\infty)$.
We have $G(\sigma;0)=0$.
One can also verify that as $\sigma\to 1$, the solution
\eqref{eq:H-sigma-z} reduces to the right-hand side of \eqref{eq:H_1_z}.
Together with the known differentiation formula for the hypergeometric function, this implies that
\eqref{eq:H-sigma-z} is indeed a solution to \eqref{eq:ode-H-sigma-z}.

\begin{proposition}
	\label{proposition:limit-two-cycles-shifted-plancherel}
	The coefficients
	at $z=0$
	of the generating function $G(\sigma;z)$
	\eqref{eq:H-sigma-z}
	scale as follows:
	\begin{equation*}
		\lim_{n\to\infty}\frac{G_n(\sigma)}{n}=\frac{1}{\sigma+1},
		\qquad
		\sigma\in [1,\infty).
	\end{equation*}
\end{proposition}
\begin{proof}
	We need to analyze the singularities of $G(\sigma;z)$ in $z$.
	There are two singularities closest to the origin,
	$z=1$ and $z=-1$. At
	$z=1$, the
	first summand in \eqref{eq:H-sigma-z} clearly
	has a pole of order $2$
	and behaves as
	$(\sigma+1)^{-1}(z-1)^{-2}$.
	To complete the proof, it suffices to show that this is the dominant behavior.
	At $z=1$, the second summand in \eqref{eq:H-sigma-z}
	behaves as $\mathrm{const}\cdot (z-1)^{\frac{\sigma-3}{2}}$, which is
	less singular than $(z-1)^{-2}$.

	Consider now the singularity at $z=-1$.
	The second summand in \eqref{eq:H-sigma-z} is regular at $z=-1$.
	For the first summand, transform the hypergeometric function as
	\cite[(15.8.1)]{NIST:DLMF}
	\begin{equation*}
		{}_2F_1 \left(
				-\frac{1 + \sigma}{2},
				-\sigma;
				\frac{1 - \sigma}{2};
				\frac{1-z}{1+z}
		\right)
		=
		\left( \frac{2z}{1+z} \right)^{\frac{\sigma+1}{2}}
		{}_2F_1 \left(
				-\frac{1 + \sigma}{2},
				\frac{1+\sigma}{2};
				\frac{1 - \sigma}{2};
				\frac{z-1}{2z}
			\right).
	\end{equation*}
	Now, the hypergeometric function becomes regular at $z=-1$.
	The power $(1+z)^{-\frac{\sigma+1}{2}}$, combined with the
	prefactor, is also regular. This implies that the singularity
	at $z=-1$ does not contribute to the leading behavior of
	the coefficients, and so we are done.
\end{proof}

\subsection{Reconciling with initial hikes under the shifted Plancherel distribution}
\label{sub:shiifted_Plancherel_number_of_2s_discussion}

We see that for $\sigma>1$, the
asymptotic expected proportion of the $2$'s in a
random Fibonacci word with the shifted Plancherel
distribution is strictly less than $\frac{1}{2}$.
This phenomenon agrees with the scaling limit of initial
hikes of $2$'s in this Fibonacci word obtained in
\Cref{thm:shifted_Plancherel_scaling}.
Indeed, recall the random variables $\xi_{\sigma;k}$, $k\ge1$, from
\Cref{def:xi_sigma}. See also \Cref{rmk:xi_sigma_construction}
for an alternative description using conditional independence
on $N=n$, where $N$ is defined by
\eqref{eq:xi_sigma_construction_N_random_variable}.
Fix $k$. By \Cref{thm:shifted_Plancherel_scaling},
the scaling limit of the sum of the first $k$ initial hikes
in a random Fibonacci word with the shifted Plancherel
distribution is given by
\begin{equation}
	\label{eq:shifted_Plancherel_scaling_limit_remark_compute_via_xi}
	\lim_{n\to\infty}\frac{h_1(w)+\cdots+h_k(w) }{n}
	\to \frac{1}{2}\left( X_1+\cdots+X_k  \right)=
	\frac{1}{2}-\frac{1}{2}
	\prod_{j=1}^{k}(1-\xi_{\sigma;j}),
\end{equation}
where the $X_j$'s are obtained from the $\xi_{\sigma;j}$'s by the
stick-breaking construction
(see \Cref{rmk:shifted_Plancherel_stick_breaking}).
Let us compute the expectation of the right-hand side of
\eqref{eq:shifted_Plancherel_scaling_limit_remark_compute_via_xi}
with $k=\infty$
(one readily sees that
the limit of \eqref{eq:shifted_Plancherel_scaling_limit_remark_compute_via_xi}
as $k\to\infty$ is well-defined).
We have, using the fact that
$\operatorname{\mathbb{E}}(\mathrm{beta}(1,\sigma/2))=\frac{2}{2+\sigma}$:
\begin{equation*}
	\operatorname{\mathbb{E}}\biggl[\,
	\prod_{j=1}^{\infty}(1-\xi_{\sigma;j})
	\biggr]
	=
	\sum_{m=1}^{\infty}
	\operatorname{\mathbb{P}}(N=m)\ssp
	\Bigl( \frac{\sigma}{2+\sigma} \Bigr)^m
	=
	\sum_{m=1}^{\infty}
	\sigma^{-\binom m2}(1-\sigma^{-m})
	\Bigl( \frac{\sigma}{2+\sigma} \Bigr)^m.
\end{equation*}
One can check that
\begin{equation}
	\label{eq:expectation_to_compare}
	\frac{1}{2}\operatorname{\mathbb{E}}\biggl[\,\sum_{j=1}^{\infty}X_j\biggr]
	\le \frac{1}{\sigma+1},
\end{equation}
with equality at $\sigma=1$,
where the difference between the two sides of the inequality
is at most $\approx 0.015$, and vanishes as $\sigma\to\infty$.
See \Cref{fig:exp_compare} for an illustration
of the two sides of the inequality.
\begin{figure}[htpb]
	\centering
	\includegraphics[width=0.6\textwidth]{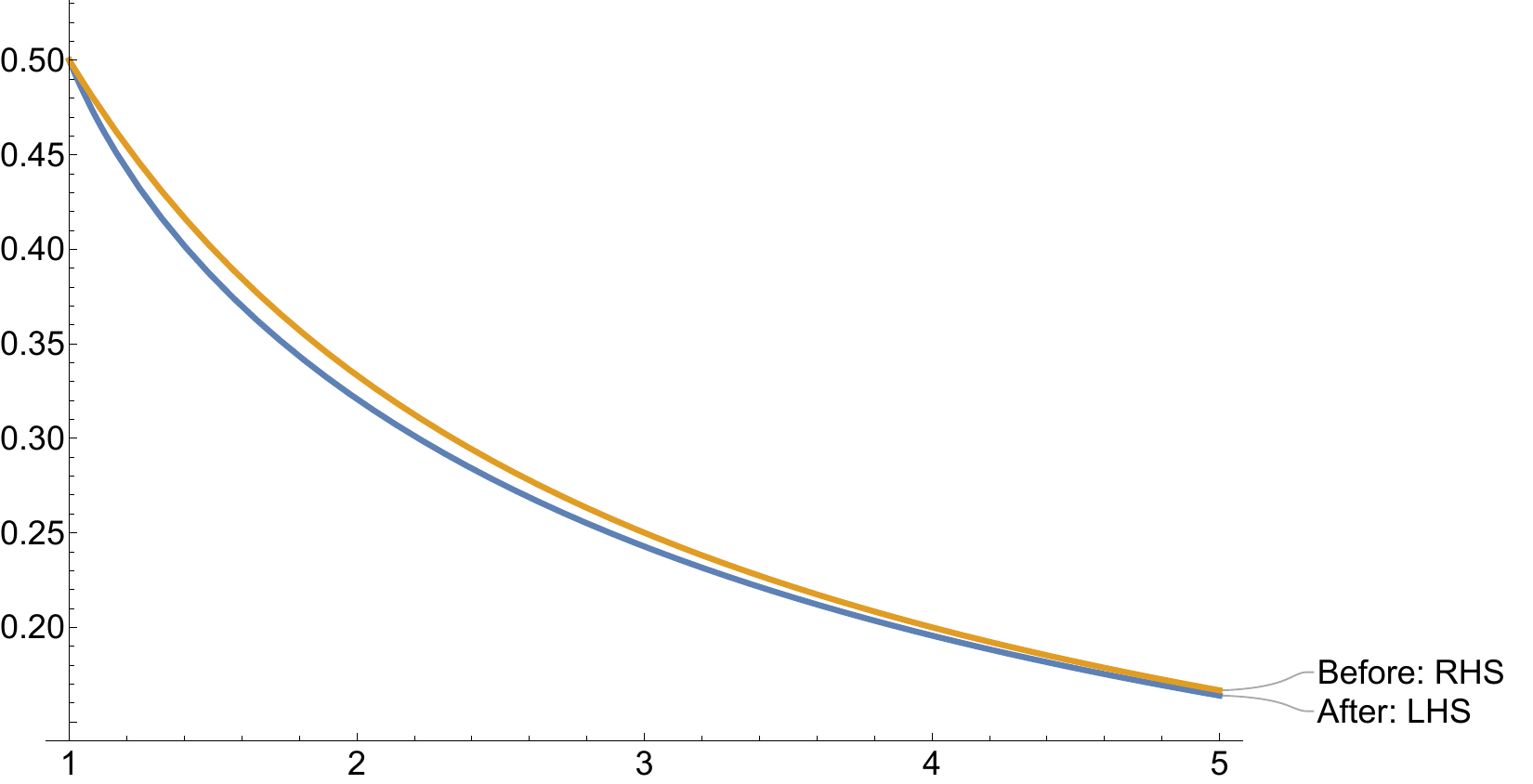}
	\caption{Comparing
		the two sides of
		\eqref{eq:expectation_to_compare},
		before
		and after
		the scaling limit
		of the initial hikes, for $1\le \sigma\le 5$.}
	\label{fig:exp_compare}
\end{figure}
The discrepancy between the two sides of
\eqref{eq:expectation_to_compare}
is due to the interchange of the limits
in~$n$ and~$k$ in \eqref{eq:shifted_Plancherel_scaling_limit_remark_compute_via_xi}.
Specifically, taking the limit $k \to \infty$ first
accounts for the total number of $2$'s in the random Fibonacci word,
whereas sending $n \to \infty$ first
considers only a finite number of initial hikes of $2$'s.
This reveals that additional digits $2$ remain hidden
in the growing random Fibonacci word
after long sequences of $1$'s.
These extra $2$'s contribute to the right-hand side of
\eqref{eq:expectation_to_compare}
but are absent from the left-hand side.

\subsection{The fake case}

One can perform the same computations as above for the
fake shifted Plancherel specialization
(\Cref{def:fake_shifted_Charlier} with $\rho=1$).
Equations
\eqref{eq:ode-H-sigma-tau-z} and \eqref{eq:H-sigma-z}
take the form, respectively,
\begin{equation*}
	z(1-\tau z^2)\partial_z \tilde H(\sigma,\tau;z)
	+
	(\sigma-1-z-\tau z^2)\ssp \tilde H(\sigma,\tau;z)
	=
	(\sigma-1)\frac{1+(\tau-1)z^2}{1-z},
\end{equation*}
and
\begin{equation*}
z(1-z^2)\partial_z \tilde G(\sigma;z) + (\sigma-1 - z - z^2 )\ssp\tilde G(\sigma;z)=
\frac{z^2 (\sigma-(\sigma-1)z )}{(1-z)^2}.
\end{equation*}
One can find the solution $\tilde G(\sigma;z)$ with $\tilde G(\sigma;0)=0$
in terms of hypergeometric functions.
Moreover, one sees that $\tilde G(1;z)=G(1;z)$ (given by \eqref{eq:H_1_z}),
as it should be.
Considering
the asymptotic behavior of the Taylor coefficients at $z=0$
reveals that they scale as $\sim\frac{n}{\sigma+1}$, exactly as for
the true shifted Plancherel specialization.
Therefore, the scaled lenghts of the initial
hikes of $2$'s
(\Cref{thm:shifted_Plancherel_scaling}),
\emph{and} the expected total number of $2$'s
in the random Fibonacci word
under both variants of the shifted Plancherel specialization
are asymptotically \emph{the same}.
It is not clear to the authors which asymptotic statistic
can distinguish between these two variants of the
shifted Plancherel specialization.

\newpage
\addtocontents{toc}{\protect\vspace{1em}}
\section{Concluding Remarks}
\label{sec:prospectives}

We conclude by outlining several directions for future work.
We first recall three questions already posed earlier in the
text ---
\Cref{problem:combinatorial_interpretation_coefficients_Fibonacci_positive_via_epsilons},
concerning the monomial expansion of clone Schur functions
in the $\epsilon$-variables;
\Cref{conj:alternative_q_Charlier}, on the Fibonacci positivity
of the alternative $q$-Charlier specialization;
\Cref{conj:splitting+dominance+upper-triangular} and
\Cref{prob:counting-N}, concerning the multiplicities $N(\pmb{\varkappa})$ arising in the
Stieltjes moment problem;
\Cref{problem:moment-sequences-tridiagonal-matrices-Fibonacci},
on understanding Borel measures associated to Fibonacci
positive specializations; \Cref{problem:Toda-flow},
which asks about the behavior of the Toda flow on Fibonacci
positive specializations; and
\Cref{problem:discrete_support} on whether the
support of the orthogonality measure associated to a
Fibonacci positive specialization is always discrete.
Beyond these, a number of other
intriguing directions remain open, and we highlight a few of
them below.

\subsection{Sch\"utzenberger promotion and combinatorial ergodicity}
Standard
Young--Fibo\-nacci tableaux (SYFTs) $T$ of shape $w \in
\mathbb{YF}$ are in bijection with linear extensions of a
binary, rooted tree $\mathbb{T}_w$ constructed from $w$
(see
\Cref{remark:linear-extentions}).
Like for any finite poset, there is a $\mathbb{Z}$-action on
the set of linear extensions of $\mathbb{T}_w$, which
implements Sch\"utzenberger \emph{promotion}
\cite{schutzenberger1972promotion, stanley2009promotion}.
Thus, one gets a $\mathbb{Z}$-action on the set of
SYFTs of shape $w \in \mathbb{YF}$ or, equivalently, on
saturated chains $ w_0 \nearrow \cdots \nearrow w_n $
starting at $w_0 = \varnothing$ and terminating at $w_n =
w$.  It would be very interesting to study the interplay
between this action and the probability distributions
$\mu^{\ssp \scriptscriptstyle \varphi}_{\scriptscriptstyle
\mathrm{CT}}$ and $\overline{\mu}^{\ssp \scriptscriptstyle
\varphi}_{\scriptscriptstyle \mathrm{CT}}$ associated to a
positive harmonic function $\varphi$.  This study
becomes particularly intriguing when viewed through the lens
of J.~Propp and T.~Roby's notion of \emph{combinatorial
ergodicity} \cite{propp2015homomesy}.

A related question concerns promotion and Type-I harmonic functions.
Recall that for $v \in \Bbb{YF}_k$ and $w \in \Bbb{YF}_n$
with $k \leq n$, the measure $M_k(v) = \dim(v) \ssp
\Phi_{1^\infty 2w}(v)$ represents the probability that $w_k
= v$, where $w_0 \nearrow \cdots \nearrow w_n$ is a
uniformly sampled random saturated chain starting at $w_0 =
\varnothing$ and terminating at $w_n = w$.
Now fix a saturated chain $\mathbf{u} = u_0 \nearrow \cdots
\nearrow u_n$ which terminates at $u_n = w$. For $v \in
\Bbb{YF}_k$ with $k \ll n$, consider the probability
$\zeta_{\mathbf{u} ; k}(v)$ that $w_k = v$, where $w_0
\nearrow \cdots \nearrow w_n$ is a uniformly sampled random
saturated chain from the promotion orbit
$\mathcal{O}_\mathbf{u}$ of~$\mathbf{u}$.
The measures $\zeta_{\mathbf{u} ; k}$ are not coherent.
However, in light of combinatorial ergodicity, one expects
that $\zeta_{\mathbf{u} ; k}$ approximates $M_k$
as $n \to \infty$.

\subsection{Truncations of the Young--Fibonacci lattice}

The theory of biserial clone Schur functions, along with the constructions introduced in \Cref{sub:clone-Kostka-numbers,sub:transition-cotransition-measures}, can be adapted to the $k$-th \emph{truncation} $\Bbb{YF}^{(k)}$ of the Young--Fibonacci lattice; see \cite{hivert_scott2024diagram}.
From a representation-theoretic perspective, $\mathbb{YF}^{(k)}$
is the Young-Fibonacci counterpart of the poset $\mathbb{Y}^{(k)}$,
which consists of partitions $\lambda \in \mathbb{Y}$
with at most $k$ parts.
Without going into detail, $\Bbb{YF}^{(k)}$ is an infinite, ranked poset that is part of an infinite filtration:
\[
\Bbb{YF}^{(1)} \subset \Bbb{YF}^{(2)} \subset \Bbb{YF}^{(3)} \subset \cdots \subset \Bbb{YF},
\]
where the Hasse diagram of $\Bbb{YF}^{(k)}$ sits inside $\Bbb{YF}^{(k+1)}$ as an induced subgraph.
The first two truncations, $\Bbb{YF}^{(1)}$ and $\Bbb{YF}^{(2)}$, are respectively the half-Pascal (Dyck) and Pascal lattices.
The next truncation, $\Bbb{YF}^{(3)}$, is illustrated
in \Cref{fig:YF_lattice_truncation}.

\begin{figure}[htpb]
	\centering
	\begin{equation*}
	\adjustbox{scale=0.8}{
		\xymatrixrowsep{0.3in}
		\xymatrixcolsep{0.15in}
		\xymatrix{ \varnothing \ar@{-}[d] & & & & & & &  \\
		1 \ar@{-}[d] \ar@{-}[drrrrrrr] & & & & & & & \\
		11  \ar@{-}[d] \ar@{-}[drrrr] & & & & & & & 2  \ar@{-}[d] \ar@{-}[dlll] \\
		111  \ar@{-}[d] \ar@{-}[drr] & & & & 21 \ar@{-}[d] \ar@{-}[dll] \ar@{-}[dr] & & & 12  \ar@{-}[d] \ar@{-}[dll] \\
		1111  \ar@{-}[d] \ar@{-}[dr]  & &211  \ar@{-}[dl] \ar@{-}[dr] & &121 \ar@{-}[d] \ar@{-}[dl]  &22 \ar@{-}[dll]|(0.52)\hole \ar@{-}[dr]
		& &112  \ar@{-}[d] \ar@{-}[dl]  \\
		11111 \ar@{-}[d]  \ar@{-}[dr] & 2111  \ar@{-}[dr] \ar@{-}[d] &  & 221 \ar@{-}[d] \ar@{-}[dl] \ar@{-}[dr]|\hole
        &1121 \ar@{-}[dr] \ar@{-}[dl] & & 212 \ar@{-}[d] \ar@{-}[dll]|(0.665)\hole & 1112 \ar@{-}[d]  \ar@{-}[dl] \\
        111111 &21111 & 2211 &2121 &222 &11121
        &2112 & 11112}
				}
	\end{equation*}
	\caption{The truncated poset $\Bbb{YF}^{(3)}$ up to level $n=6$
	(compare to the full Young--Fibonacci lattice
	in \Cref{fig:YF_lattice}).}
	\label{fig:YF_lattice_truncation}
\end{figure}

None of the truncations is a differential poset. Clone harmonic functions $\varphi_{\vec{x}, \vec{y}}$ on the Young--Fibonacci lattice restrict to $\Bbb{YF}^{(k)}$ and remain harmonic, provided that the specialization $(\vec{x},\vec{y})$ stabilizes suitably. Fibonacci positive specializations, as well as positive normalized harmonic functions for $\Bbb{YF}^{(k)}$, are defined in the usual way. The space of Fibonacci positive specializations for $\Bbb{YF}^{(k)}$ is finite-dimensional and is expected to admit a simple description.

In general, saturated chains in $\Bbb{YF}^{(k)}$ terminating
at a fixed endpoint are not known to be in bijection with
the set of linear extensions of any poset; hence,
Sch\"utzenberger promotion is not available. Nevertheless, a
natural $\mathbb{Z}$-action, realized by the adic shift,
exists and permits the investigation of combinatorial
ergodicity in the truncated setting.

Each truncation $\Bbb{YF}^{(k)}$ supports a restricted
version of the Young--Fibonacci RS correspondence, which
involves pattern-avoiding permutations. Accordingly, random
pattern-avoiding permutations can be studied using the
framework set up in
\Cref{sec:random_Fibonacci_words_to_random_permutations}.

\subsection{Martin boundary, Fibonacci positive specializations, and stick-breaking}

The relationship between the Martin boundary
$\Upsilon_{\mathrm{Martin}}(\mathbb{YF})$ of the Young--Fibonacci lattice
and the space of Fibonacci positive specializations remains only partially
understood. Given a Fibonacci positive specialization
$(\vec{x}, \vec{y})$ and its associated Type-I component $\muI$ (see \Cref{def:measures-typeI-fibonacci-words}),
we expect that
$\muI(1^{\infty}\mathbb{YF})$ takes the value $0$ or $1$.
Verifying this claim, or finding a counterexample, would already be informative.

A related problem is to determine whether those Fibonacci positive specializations
$(\vec{x}, \vec{y})$ satisfying $\muI(1^{\infty}\mathbb{YF}) = 0$
are precisely the ones whose clone measures, when projected to either initial
runs of $1$'s or initial hikes of $2$'s, lead to continuous models of random sequences in
$[0,1]^{\infty}$. At present we have two such continuous scaling limits,
both of stick-breaking nature:
\Cref{thm:defomed_Plancherel_scaling} with independent,
and
\Cref{thm:shifted_Plancherel_scaling} with conditional stick-breaking.
It would be interesting to know whether further continuous models of random
sequences in $[0,1]^{\infty}$ can be obtained from clone measures through the
$n \to \infty$ limit of joint initial run and/or hike distributions.

\subsection{Random permutation models}

In \cite{gnedin2000plancherel}, Gnedin and Kerov introduced
a surjection called the \emph{Fibonacci solitaire} (see also \cite{gnedin2002fibonacci})
which
maps permutations in $\mathfrak{S}_n$ to saturated chains
terminating in $\mathbb{YF}_n$. This solitaire is distinct
from the surjection $\upsigma \mapsto \mathrm{P}(\upsigma)$
obtained by simply forgetting the recording tableaux
$\mathrm{Q}(\upsigma)$ in Nzeutchap's RS correspondence.
The push-forward of the uniform measure on $\mathfrak{S}_n$
under the Fibonacci solitaire was shown in
\cite{gnedin2000plancherel} to be the Plancherel measure,
i.e., $\nu_n(\boldsymbol\upsigma\mid \uppi, \varphi)$ with
$\uppi = \varphi = \varphi_{\mathrm{PL}}$ in the notation of
\Cref{sub:random-permutation-involutions}. Beyond
this example, it is not clear which measures on
$\mathfrak{S}_n$ realize $\nu_n(\boldsymbol\upsigma\mid
\uppi, \varphi)$ as push-forward measures, even when
$\varphi = \varphi_{\mathrm{PL}}$, $\uppi =
\varphi_{\vec{x}, \vec{y}}$, and where $(\vec{x},
\vec{y}\ssp)$ is a Fibonacci positive
specialization.

The probability models for random permutations and
involutions introduced in
\Cref{sub:random-permutation-involutions} arise from the
Young--Fibonacci side of the RS correspondence. For an
arbitrary Fibonacci positive specialization (and even for
concrete examples such as the Charlier specialization) it
would be valuable to identify natural multivariate
statistics on permutations and involutions that make the
corresponding measures Gibbsian, i.e., make the probability
of a permutation proportional to the exponential of a
certain linear combination of these statistics.

Another key question is whether the distributions
$\mu_n(\boldsymbol\upsigma \mid \uppi, \varphi, \psi)$ and
$\nu_n(\boldsymbol\upsigma \mid \uppi, \varphi)$ can be
understood in connection with the Stieltjes moment problem.
For instance, the Plancherel specialization corresponds to
the uniform distribution on permutations
$\boldsymbol\upsigma \in \mathfrak{S}_n$, as this is
realized by $\mu_n(\boldsymbol\upsigma \mid \uppi, \varphi,
\psi)$ when $\uppi = \varphi = \psi =
\varphi_{\scriptscriptstyle \mathrm{PL}}$. It is known
\cite{fulman2024fixed}, \cite{arratia1992cycle} that the
distribution of fixed points of a uniformly sampled random
permutation $\boldsymbol\upsigma \in \mathfrak{S}_n$ tends
to the Poisson distribution $\upnu_{\scriptscriptstyle
\mathrm{Pois}}^{\scriptscriptstyle (1)}(dt)$ on $[0,
\infty)$ as $n \to \infty$. Notably,
$\upnu_{\scriptscriptstyle
\mathrm{Pois}}^{\scriptscriptstyle (1)}(dt)$ is the Borel
measure associated with the Plancherel Fibonacci positive
specialization by \Cref{thm:stieltjes-moments-theorem}.
It is natural to ask how this coincidence may be extended to
other Fibonacci positive specializations.

To investigate this link, one may first study the asymptotic
distribution of fixed points in permutations
$\boldsymbol\upsigma\in\mathfrak{S}_n$ drawn from
$\mu_n(\boldsymbol\upsigma\mid\uppi,\varphi,\psi)$ with
$\varphi=\psi=\varphi_{\scriptscriptstyle\mathrm{PL}}$ and
$\uppi=\varphi_{\vec{x},\vec{y}}$ an arbitrary clone
harmonic function.  Combinatorially, this amounts to
counting permutations with exactly $k\ge0$ fixed points
whose RS shape is a prescribed Fibonacci word
$w\in\mathbb{YF}_n$.  Such counts should be accessible,
since fixed points are easily detected in the
$\mathbb{YF}$-RS correspondence.  For the Charlier
specialization $(\vec{x},\vec{y}\ssp)$, one can then test
whether the limiting fixed-point distribution matches the
Poisson law
$\upnu_{\scriptscriptstyle\mathrm{Pois}}^{\scriptscriptstyle(\rho)}(dt)$,
the Borel measure in the Charlier case.

\subsection{Clone Cauchy identities and Okada's noncommutative theory}

The clone Cauchy identities from \Cref{sub:clone-cauchy} give rise to
two Gibbs measures on the Young--Fibonacci lattice~$\mathbb{YF}$:
\begin{equation}
	\label{eq:clone-prob-measures}
	\mathrm{prob}_{\scriptscriptstyle H}(w)
	:= \frac{h_w(\vec{p}\mid\vec{q})\,s_w(\vec{x}\mid\vec{y})}
												{H(\vec{x},\vec{y}\,;\vec{p},\vec{q})},
	\qquad
	\mathrm{prob}_{\scriptscriptstyle S}(w)
	:= \frac{s_w(\vec{p}\mid\vec{q})\,s_w(\vec{x}\mid\vec{y})}
												{S(\vec{x},\vec{y}\,;\vec{p},\vec{q})},
	\qquad w\in\mathbb{YF}.
\end{equation}
Here $H$ and $S$ are the normalizing constants obtained by summing,
the right-hand sides of the first and second clone
Cauchy identities
(\Cref{prop:first-clone-cauchy,prop:second-clone-cauchy}) over $n\ge0$, and
$(\vec{x},\vec{y})$ and $(\vec{p},\vec{q})$ are two
Fibonacci positive specializations.
The measures
\eqref{eq:clone-prob-measures} may be viewed as clone analogues of the
Schur measures on partitions introduced in~\cite{okounkov2001infinite}.
The next natural goal is to define and study \emph{clone Schur
processes} --- probability measures on sequences of Fibonacci words whose
joint distributions are expressed through suitable skew versions of
clone Schur functions.

The connection between the clone measures
\eqref{eq:clone-prob-measures} and Okada's noncommutative theory
\cite{okada1994algebras} is also worth exploring.
In the noncommutative
setting, clone Schur functions
$\mathrm{s}_w(\mathbf{x}\mid\mathbf{y})$ form a basis of the free
algebra $\mathbb{C}\langle\mathbf{x},\mathbf{y}\rangle$ generated by
two noncommuting variables $\mathbf{x},\mathbf{y}$.  To reproduce the
Cauchy identities one introduces another pair of noncommuting variables
$\mathbf{p},\mathbf{q}$ that commute with both $\mathbf{x}$ and
$\mathbf{y}$.  The noncommutative counterpart of the quadridiagonal
matrix in \eqref{eq:second-clone-cauchy} is
\[
S_n(\mathbf{x},\mathbf{y};\mathbf{p},\mathbf{q})
:=\underbrace{\begin{pmatrix}
\mathbf{A}&\mathbf{B}&\mathbf{C}&0&\cdots\\
1&\mathbf{A}&\mathbf{B}&\mathbf{C}&\\
0&1&\mathbf{A}&\mathbf{B}&\\
0&0&1&\mathbf{A}&\\
\vdots&&&&\ddots
\end{pmatrix}}_{n\times n},
\]
whose entries lie in the free algebra
$\mathfrak{A}\langle\mathbf{x},\mathbf{y}\rangle$ with coefficient ring
$\mathfrak{A}=\mathbb{C}\langle\mathbf{p},\mathbf{q}\rangle$:
\[
\mathbf{A}=\mathbf{p}\mathbf{x},\qquad
\mathbf{B}=\mathbf{q}(\mathbf{x}^2-\mathbf{y})+(\mathbf{p}^2-\mathbf{q})\mathbf{y},
\qquad
\mathbf{C}=\mathbf{q}\mathbf{p}\mathbf{y}\mathbf{x}.
\]
We conjecture that
every matrix minor (quasi-determinant) of
$S_n(\mathbf{x},\mathbf{y};\mathbf{p},\mathbf{q})$ is
\emph{coefficient-wise clone Schur positive}: the coefficient of
$\mathrm{s}_w(\mathbf{x}\mid\mathbf{y})$ in such a minor lies in the
positive cone spanned by the functions
$\mathrm{s}_v(\mathbf{p}\mid\mathbf{q})$ with $v\in\mathbb{YF}$.  This
would provide a new manifestation of total positivity, invisible under
the biserial specialization considered in
\Cref{sec:Fibonacci_positivity}.

\subsection{Quasisymmetric versions of clone Schur functions}

Quasisymmetric functions arise naturally from Nzeutchap's
Robinson--Schensted--Knuth correspondence for the Young--Fibonacci
lattice.  The correspondence is an injection from positive integer
sequences $\mathbb{N}^{\infty}$ to pairs $(\mathrm{P},\mathrm{Q})$ of
standard and semistandard Young--Fibonacci tableaux of the same shape
in~$\mathbb{YF}$.  An obvious quasisymmetric analogue
of $\mathrm{s}_w(\mathbf{x}\mid\mathbf{y})$ is  
the generating function $Q_w$ 
of all semistandard tableaux of shape $w$.
However, this function no longer obeys the $\mathbb{YF}$ branching rule
and its expansion in Gessel's fundamental quasisymmetric functions does
not match the expected clone version of
\cite[Theorem~7.19.7]{Stanley1999}.

A more promising approach exploits the graded Hopf-algebra duality
between noncommutative symmetric functions ($\mathbf{NSym}$) and
quasisymmetric functions ($\mathrm{QSym}$); see
\cite{gelfand1995noncommutative}, \cite{gessel1984multipartite}, \cite{malvenuto1995duality}.
We view $\mathbf{NSym}$ as the free algebra
$\mathbb{C}\langle\Psi_1,\Psi_2,\Psi_3,\ldots\rangle$ with multiplicative
basis $\Psi_{\alpha}:=\Psi_1^{\alpha_1}\cdots\Psi_k^{\alpha_k}$
indexed by compositions $\alpha=(\alpha_1,\ldots,\alpha_k)$.  These
$\Psi_{\alpha}$ are the \emph{Type-I noncommutative power sums}.  The
Hopf structure is the usual one for a free algebra; in particular,
each $\Psi_k$ is primitive:
$\Delta(\Psi_k)=1\otimes\Psi_k+\Psi_k\otimes1$.

The dual basis in $\mathrm{QSym}$ consists of the \emph{Type-I
quasisymmetric power sums} $\{\psi_{\alpha}\}$
\cite{ballantine2020quasisymmetric}.  Their product and coproduct are
given by shuffling and de-concatenation,
\[
\psi_{\alpha}\psi_{\beta}=
\sum_{\gamma\in\alpha\shuffle\beta}!\psi_{\gamma},\qquad
\Delta(\psi_{\gamma})=
\sum_{\alpha\cdot\beta=\gamma}\psi_{\alpha}\otimes\psi_{\beta},
\]
and each expands into monomial quasisymmetric functions via
\[
\psi_{\alpha}=\sum_{\beta\succeq\alpha}\frac{M_{\beta}}{\pi(\alpha,\beta)},
\]
where the sum ranges over compositions $\beta\models n$ coarsening
$\alpha$.

Okada's clone ring $\mathbb{C}\langle\mathbf{x},\mathbf{y}\rangle$
embeds into $\mathbf{NSym}$ by sending
$\mathbf{x},\mathbf{y}\mapsto\Psi_1,\Psi_2$; the induced Hopf structure
coincides with the standard one on the rank-two free algebra
$\mathbb{C}\langle\mathbf{x},\mathbf{y}\rangle$.  Interpreting a
Fibonacci word $w\in\mathbb{YF}$ as a composition with parts $1$ and
$2$ gives a multiplicative basis
$\{\Psi_w:w\in\mathbb{YF}\}$ for the
Okada's clone ring, whose expansion in clone Schur functions
is governed by the clone Kostka numbers:
\[
\Psi_w=\sum_{|v|=|w|}K_{v,w}\,
\mathrm{s}_v(\mathbf{x}\mid\mathbf{y}).
\]

The quasisymmetric subalgebra generated by $\psi_1,\psi_2$ (equivalently,
by the $\psi_w$ with $w\in\mathbb{YF}$) is dual to
$\mathbb{C}\langle\mathbf{x},\mathbf{y}\rangle$.  Functions dual to the
clone Schur basis are therefore
\[
Q_w^{(\mathrm{I})}:=\sum_{|v|=|w|}K_{v,w}\,\psi_{v}.
\]
Each $Q_w^{(\mathrm{I})}$ expands into monomial quasisymmetric
functions $M_{\beta}$, where the compositions $\beta$ may contain parts
larger than~$2$. One may view $Q_w^{(\mathrm{I})}$ as a quasisymmetric
analogue of $\mathrm{s}_w$.

A second, Type-II, power-sum theory yields another family
$Q_w^{(\mathrm{II})}$ of quasisymmetric analogues.  Understanding the
relationships among $Q_w$, $Q_w^{(\mathrm{I})}$, and
$Q_w^{(\mathrm{II})}$ could clarify the Young--Fibonacci RSK
correspondence and suggest quasisymmetric versions of the clone Cauchy
identities.  Hopf duality may also shed light on common branching
rules and on clone Littlewood--Richardson coefficients for these
functions.

\newpage


\medskip

\textsc{L. Petrov, University of Virginia, Charlottesville, VA, USA}

E-mail: \texttt{lenia.petrov@gmail.com}

\medskip

\textsc{J. Scott, University of Minnesota, Minneapolis, MN, USA}

E-mail: \texttt{scot1526@umn.edu}


\begin{bibdiv}
\begin{biblist}
\bib{ahbli2023new}{article}{
      author={Ahbli, K.},
       title={{New results on the associated Meixner, Charlier, Laguerre, and
  Krawtchouk polynomials}},
        date={2023},
     journal={arXiv preprint},
        note={arXiv:2306.05371 [math.CA]},
}

\bib{AESW51}{article}{
      author={Aissen, M.},
      author={Edrei, A.},
      author={Schoenberg, I.~J.},
      author={Whitney, A.},
       title={On the generating functions of totally positive sequences},
        date={1951},
        ISSN={0027-8424},
     journal={Proc. Nat. Acad. Sci. U. S. A.},
      volume={37},
       pages={303\ndash 307},
}

\bib{ASW52}{article}{
      author={Aissen, M.},
      author={Schoenberg, I.~J.},
      author={Whitney, A.},
       title={{On the generating functions of totally positive sequences I}},
        date={1952},
     journal={J. Analyse Math.},
      volume={2},
       pages={93\ndash 103},
}

\bib{anshelevich2005linearization}{article}{
      author={Anshelevich, M.},
       title={{Linearization coefficients for orthogonal polynomials using
  stochastic processes}},
        date={2005},
     journal={Ann. Probab.},
      volume={33},
      number={1},
       pages={114\ndash 136},
        note={arXiv:math/0301094 [math.CO]},
}

\bib{arratia1992cycle}{article}{
      author={Arratia, R.},
      author={Tavar{\'e}, S.},
       title={{The Cycle Structure of Random Permutations}},
        date={1992},
     journal={Ann. Probab.},
      volume={20},
       pages={1567\ndash 1591},
}

\bib{Askey1989}{incollection}{
      author={Askey, R.},
       title={{Continuous q-Hermite Polynomials when $q > 1$}},
        date={1989},
   booktitle={{$q$-Series and Partitions}},
      editor={Stanton, D.},
      series={The IMA Volumes in Mathematics and Its Applications},
      volume={18},
   publisher={Springer, New York, NY},
       pages={151\ndash 158},
}

\bib{ballantine2020quasisymmetric}{article}{
      author={Ballantine, C.},
      author={Daugherty, Z.},
      author={Hicks, A.},
      author={Mason, S.},
      author={Niese, E.},
       title={{Quasisymmetric power sums}},
        date={2020},
     journal={J. Combin. Theory Ser. A},
      volume={175},
      number={1},
       pages={105273},
        note={arXiv:1710.11613 [math.CO]},
}

\bib{bell1934exponential}{article}{
      author={Bell, E.~T.},
       title={{Exponential polynomials}},
        date={1934},
     journal={Ann. Math.},
      volume={35},
      number={2},
       pages={258\ndash 277},
}

\bib{blitvic2021permutations}{article}{
      author={Blitvi\'{c}, N.},
      author={Steingr\'{i}msson, E.},
       title={{Permutations, moments, measures}},
        date={2021},
     journal={Trans. AMS},
      volume={374},
      number={8},
       pages={5473\ndash 5508},
        note={arXiv:2001.00280 [math.CO]},
}

\bib{BochkovEvtushevsky2020}{article}{
      author={Bochkov, I.},
      author={Evtushevsky, V.},
       title={{Ergodicity of the Martin boundary of the Young--Fibonacci graph.
  I}},
        date={2020},
     journal={arXiv preprint},
        note={arXiv:2012.07447 [math.CO]},
}

\bib{BorFerr2008DF}{article}{
      author={Borodin, A.},
      author={Ferrari, P.},
       title={{Anisotropic growth of random surfaces in 2+1 dimensions}},
        date={2014},
     journal={Commun. Math. Phys.},
      volume={325},
       pages={603\ndash 684},
        note={arXiv:0804.3035 [math-ph]},
}

\bib{BorodinGorinSPB12}{incollection}{
      author={Borodin, A.},
      author={Gorin, V.},
       title={Lectures on integrable probability},
        date={2016},
   booktitle={{Probability and Statistical Physics in St.\ Petersburg}},
      series={Proceedings of Symposia in Pure Mathematics},
      volume={91},
   publisher={AMS},
       pages={155\ndash 214},
        note={arXiv:1212.3351 [math.PR]},
}

\bib{borodin2006meixner}{article}{
      author={Borodin, A.},
      author={Olshanski, G.},
       title={{Meixner polynomials and random partitions}},
        date={2006},
     journal={Moscow Mathematical Journal},
      volume={6},
      number={4},
       pages={629\ndash 655},
        note={arXiv:math/0609806 [math.PR]},
}

\bib{borodin2016representations}{book}{
      author={Borodin, A.},
      author={Olshanski, G.},
       title={{Representations of the Infinite Symmetric Group}},
      series={Cambridge Studies in Advanced Mathematics},
   publisher={Cambridge University Press},
        date={2016},
      volume={160},
}

\bib{BorodinPetrov2013Lect}{article}{
      author={Borodin, A.},
      author={Petrov, L.},
       title={{Integrable probability: From representation theory to Macdonald
  processes}},
        date={2014},
     journal={Probab. Surv.},
      volume={11},
       pages={1\ndash 58},
        note={arXiv:1310.8007 [math.PR]},
}

\bib{chen2021coefficientwise}{article}{
      author={Chen, X.},
      author={Deb, B.},
      author={Dyachenko, A.},
      author={Gilmore, T.},
      author={Sokal, A.~D.},
       title={{Coefficientwise total positivity of some matrices defined by
  linear recurrences}},
        date={2021},
     journal={S{\'e}m. Lothar. Combin.},
      volume={85B},
      number={30},
       pages={12},
        note={arXiv:2012.03629 [math.CO]},
}

\bib{Chihara1978}{book}{
      author={Chihara, T.~S.},
       title={{An Introduction to Orthogonal Polynomials}},
   publisher={Gordon and Breach},
        date={1978},
}

\bib{christiansen2004indeterminate}{thesis}{
      author={Christiansen, J.~S.},
       title={{Indeterminate Moment Problems within the Askey-scheme}},
        type={Ph.D. Thesis},
     address={Institute for Mathematical Sciences, University of Copenhagen},
        date={2004},
}

\bib{cigler2011curious}{article}{
      author={Cigler, J.},
      author={Zeng, J.},
       title={{A curious q-analogue of Hermite polynomials}},
        date={2011},
     journal={J. Comb. Theory Ser. A},
      volume={118},
      number={1},
       pages={9\ndash 26},
        note={arXiv:0905.0228 [math.CO]},
}

\bib{CorteelKimStanton2016}{incollection}{
      author={Corteel, S.},
      author={Kim, J.~S.},
      author={Stanton, D.},
       title={{Moments of orthogonal polynomials and combinatorics}},
        date={2016},
   booktitle={Recent trends in combinatorics},
       pages={545\ndash 578},
}

\bib{corwin2014brownian}{article}{
      author={Corwin, I.},
      author={Hammond, A.},
       title={{Brownian Gibbs property for Airy line ensembles}},
        date={2014},
     journal={Invent. math.},
      volume={195},
      number={2},
       pages={441\ndash 508},
        note={arXiv:1108.2291 [math.PR]},
}

\bib{deMedicisStantonWhite1995}{article}{
      author={de~Medicis, A.},
      author={Stanton, D.},
      author={White, D.},
       title={{The Combinatorics of q-Charlier Polynomials}},
        date={1995},
     journal={J. Combin. Theory Ser. A},
      volume={69},
       pages={87\ndash 114},
        note={arXiv:math/9307208 [math.CA]},
}

\bib{deb2023coefficientwise}{article}{
      author={Deb, B.},
      author={Sokal, A.~D.},
       title={{Coefficientwise Hankel-total positivity of the Schett
  polynomials}},
        date={2023},
     journal={arXiv preprint},
        note={arXiv:2311.11747 [math.CO]},
}

\bib{Edrei1952}{article}{
      author={Edrei, A.},
       title={On the generating functions of totally positive sequences. {II}},
        date={1952},
        ISSN={0021-7670},
     journal={J. Analyse Math.},
      volume={2},
       pages={104\ndash 109},
}

\bib{Edrei53}{article}{
      author={Edrei, A.},
       title={On the generating function of a doubly infinite, totally positive
  sequence},
        date={1953},
     journal={Trans. AMS},
      volume={74},
       pages={367\ndash 383},
}

\bib{Evtushevsky2020PartII}{article}{
      author={Evtushevsky, V.},
       title={{Ergodicity of the Martin boundary of the Young--Fibonacci graph.
  II}},
        date={2020},
     journal={arXiv preprint},
        note={arXiv:2012.08107 [math.CO]},
}

\bib{fallat2017total}{article}{
      author={Fallat, S.},
      author={Johnson, C.~R.},
      author={Sokal, A.~D.},
       title={{Total positivity of sums, Hadamard products and Hadamard powers:
  Results and counterexamples}},
        date={2017},
     journal={Linear Algebra Appl.},
      volume={520},
       pages={242\ndash 259},
        note={arXiv:1612.02210 [math.AC]},
}

\bib{flajolet1980combinatorial}{article}{
      author={Flajolet, P.},
       title={{Combinatorial aspects of continued fractions}},
        date={1980},
     journal={Discrete Math.},
      volume={32},
       pages={125\ndash 161},
}

\bib{fomin1994duality}{article}{
      author={Fomin, S.},
       title={{Duality of graded graphs}},
        date={1994},
     journal={J. Algebr. Comb.},
      volume={3},
      number={4},
       pages={357\ndash 404},
}

\bib{fomin1995schensted}{article}{
      author={Fomin, S.},
       title={{Schensted algorithms for dual graded graphs}},
        date={1995},
     journal={J. Algebr. Comb.},
      volume={4},
      number={1},
       pages={5\ndash 45},
}

\bib{FominZelevinsky1999}{article}{
      author={Fomin, S.},
      author={Zelevinsky, A.},
       title={{Double Bruhat cells and total positivity}},
        date={1999},
     journal={J. Amer. Math. Soc.},
      volume={12},
      number={2},
       pages={335\ndash 380},
        note={arXiv:math/9802056 [math.RT]},
}

\bib{FZTP2000}{article}{
      author={Fomin, S.},
      author={Zelevinsky, A.},
       title={{Total positivity: Tests and parametrizations}},
        date={2000},
     journal={The Mathematical Intelligencer},
      volume={22},
      number={1},
       pages={23\ndash 33},
        note={arXiv:math/9912128 [math.RA]},
}

\bib{Fomin1986}{article}{
      author={Fomin, S.~V.},
       title={{Generalized Robinson-Schensted-Knuth correspondence}},
        date={1988},
     journal={J. Math. Sci.},
      volume={41},
      number={2},
       pages={979\ndash 991},
        note={Translated from Zap. Nauchn. Semin. LOMI, vol. 155, pp. 156--175,
  1986},
}

\bib{fulman2024fixed}{article}{
      author={Fulman, J.},
       title={{Fixed points of non-uniform permutations and representation
  theory of the symmetric group}},
        date={2024},
     journal={arXiv preprint},
        note={arXiv:2406.12139 [math.RT]},
}

\bib{GekhtmanShapiro1997}{article}{
      author={Gekhtman, M.~I.},
      author={Shapiro, M.~Z.},
       title={{Completeness of real Toda flows and totally positive matrices}},
        date={1997},
     journal={Math. Z.},
      volume={226},
       pages={51\ndash 66},
}

\bib{gelfand1995noncommutative}{article}{
      author={Gelfand, I.},
      author={Krob, D.},
      author={Lascoux, A.},
      author={Leclerc, B.},
      author={Retakh, V.},
      author={Thibon, J.-Y.},
       title={{Noncommutative symmetric functions}},
        date={1995},
     journal={Adv. Math.},
      volume={112},
       pages={218\ndash 348},
        note={arXiv:hep-th/9407124},
}

\bib{gessel1984multipartite}{article}{
      author={Gessel, I.},
       title={{Multipartite P-partitions and inner products of skew Schur
  functions}},
        date={1984},
     journal={Contemp. Math.},
      volume={34},
       pages={289\ndash 301},
}

\bib{gnedin2000plancherel}{article}{
      author={Gnedin, A.},
      author={Kerov, S.},
       title={{The Plancherel measure of the Young--Fibonacci graph}},
        date={2000},
     journal={Math. Proc. Camb. Philos. Soc.},
      volume={129},
      number={3},
       pages={433\ndash 446},
}

\bib{gnedin2002fibonacci}{article}{
      author={Gnedin, A.},
      author={Kerov, S.},
       title={{Fibonacci solitaire}},
        date={2002},
     journal={Random Struct. Algorithms},
      volume={20},
      number={1},
       pages={71\ndash 88},
}

\bib{KerovGoodman1997}{article}{
      author={Goodman, F.},
      author={Kerov, S.},
       title={{The Martin Boundary of the Young-Fibonacci Lattice}},
        date={2000},
     journal={Jour. Alg. Comb.},
      volume={11},
       pages={17\ndash 48},
        note={arXiv:math/9712266 [math.CO]},
}

\bib{hivert_scott2024diagram}{article}{
      author={Hivert, F.},
      author={Scott, J.},
       title={{Diagram Model for the Okada Algebra and Monoid}},
        date={2024},
     journal={S{\'e}m. Lothar. Combin.},
        note={Proceedings of the 36th Conference on Formal Power Series and
  Algebraic Combinatorics (FPSAC 2024), Bochum, Article 25, 12 pp.
  arXiv:2404.16733 [math.RT]},
}

\bib{ismail1988linear}{article}{
      author={Ismail, M.},
      author={Letessier, J.},
      author={Valent, G.},
       title={{Linear Birth and Death Models and Associated Laguerre and
  Meixner Polynomials}},
        date={1988},
     journal={J. Approx. Theory},
      volume={55},
       pages={337\ndash 348},
}

\bib{Johnson1997}{book}{
      author={Johnson, N.~L.},
      author={Kotz, S.},
      author={Balakrishnan, N.},
       title={{Discrete multivariate distributions}},
   publisher={Wiley},
     address={New York},
        date={1997},
      volume={165},
}

\bib{Karlin1968}{book}{
      author={Karlin, S.},
       title={{Total positivity, Vol.1}},
   publisher={Stanford},
        date={1968},
}

\bib{kasraoui2006distribution}{article}{
      author={Kasraoui, A.},
      author={Zeng, J.},
       title={{Distribution of crossings, nestings and alignments of two edges
  in matchings and partitions}},
        date={2006},
     journal={Electron. J. Combin.},
      volume={13},
      number={1},
       pages={R33},
        note={arXiv:math/0601081 [math.CO]},
}

\bib{KimStantonZeng2006}{article}{
      author={Kim, D.},
      author={Stanton, D.},
      author={Zeng, J.},
       title={{The Combinatorics of the Al-Salam-Chihara q-Charlier
  polynomials}},
        date={2006},
     journal={S{\'e}m. Lothar. Combin.},
      volume={54},
       pages={article B54i},
}

\bib{Koekoek1996}{article}{
      author={Koekoek, R.},
      author={Swarttouw, R.F.},
       title={{The Askey-scheme of hypergeometric orthogonal polynomials and
  its q-analogue}},
        date={1996},
     journal={Technical report, Delft University of Technology and Free
  University of Amsterdam},
        note={arXiv:math/9602214 [math.CA], report no. OP-SF 20 Feb 1996.
  Updated version available at
  \url{https://fa.ewi.tudelft.nl/~koekoek/documents/as98.pdf}},
}

\bib{koornwinder2004onsalamchihara}{article}{
      author={Koornwinder, T.H.},
       title={{On $q^{-1}$-Al-Salam-Chihara polynomials}},
        date={2004},
        note={Informal note, version of January 29, 2004, available at:
  \url{https://staff.fnwi.uva.nl/t.h.koornwinder/art/informal/ASC.pdf}},
}

\bib{malvenuto1995duality}{article}{
      author={Malvenuto, C.},
      author={Reutenauer, C.},
       title={{Duality between quasi-symmetric functions and the Solomon
  descent algebra}},
        date={1995},
     journal={J. Algebra},
      volume={177},
       pages={967\ndash 982},
}

\bib{NakamuraZhedanov2004}{article}{
      author={Nakamura, Y.},
      author={Zhedanov, A.},
       title={{Toda Chain, Sheffer Class of Orthogonal Polynomials and
  Combinatorial Numbers}},
        date={2004},
     journal={Proc. Inst. Math. NAS Ukraine},
      volume={50},
       pages={450\ndash 457},
}

\bib{NIST:DLMF}{misc}{
      author={NIST},
       title={{NIST Digital Library of Mathematical Functions}},
        date={2024},
         url={https://dlmf.nist.gov/},
        note={Olver, F. W. J., {Olde Daalhuis}, A. B., Lozier, D. W.,
  Schneider, B. I., Boisvert, R. F., Clark, C. W., Miller, B. R., Saunders, B.
  V., Cohl, H. S., McClain, M. A., eds. Release 1.2.0 of 2024-03-15. Available
  at: \url{https://dlmf.nist.gov/}},
}

\bib{nzeutchap2009young}{article}{
      author={Nzeutchap, J.},
       title={{Young-Fibonacci insertion, tableauhedron, and Kostka numbers}},
        date={2009},
     journal={J. Comb. Theory, Ser. A},
      volume={116},
       pages={143\ndash 167},
}

\bib{okada1994algebras}{article}{
      author={Okada, S.},
       title={{Algebras associated to the Young-Fibonacci lattice}},
        date={1994},
     journal={Trans. Amer. Math. Soc.},
      volume={346},
       pages={549\ndash 568},
}

\bib{okounkov2001infinite}{article}{
      author={Okounkov, A.},
       title={{Infinite wedge and random partitions}},
        date={2001},
     journal={Selecta Math.},
      volume={7},
      number={1},
       pages={57\ndash 81},
        note={arXiv:math/9907127 [math.RT]},
}

\bib{okounkov2006gromov}{article}{
      author={Okounkov, A.},
      author={Pandharipande, R.},
       title={{Gromov-Witten theory, Hurwitz theory, and completed cycles}},
        date={2006},
     journal={Ann. Math.},
      volume={163},
      number={2},
       pages={517\ndash 560},
        note={arXiv:math/0204305 [math.AG]},
}

\bib{okounkov2003correlation}{article}{
      author={Okounkov, A.},
      author={Reshetikhin, N.},
       title={{Correlation function of Schur process with application to local
  geometry of a random 3-dimensional Young diagram}},
        date={2003},
     journal={Jour. AMS},
      volume={16},
      number={3},
       pages={581\ndash 603},
        note={arXiv:math/0107056 [math.CO]},
}

\bib{okounkov2006quantum}{incollection}{
      author={Okounkov, A.},
      author={Reshetikhin, N.},
      author={Vafa, C.},
       title={{Quantum Calabi-Yau and Classical Crystals}},
        date={2006},
   booktitle={Progress in mathematics},
      editor={Etingof, P.},
      editor={Retakh, V.},
      editor={Singer, I.~M.},
      volume={244},
   publisher={Birkh{\"a}user Boston},
       pages={597\ndash 618},
        note={arXiv:hep-th/0309208},
}

\bib{propp2015homomesy}{article}{
      author={Propp, J.},
      author={Roby, T.},
       title={{Homomesy in products of two chains}},
        date={2015},
     journal={Electron. J. Comb.},
      volume={22},
      number={3},
        note={arXiv:1310.5201 [math.CO]},
}

\bib{petreolle2023lattice}{book}{
      author={P{\'e}tr{\'e}olle, M.},
      author={Sokal, A.~D.},
      author={Zhu, B.-X.},
       title={{Lattice Paths and Branched Continued Fractions: An Infinite
  Sequence of Generalizations of the Stieltjes--Rogers and Thron--Rogers
  Polynomials, with Coefficientwise Hankel-Total Positivity}},
      series={Mem. Amer. Math. Soc.},
   publisher={AMS},
        date={2023},
      volume={291},
      number={1450},
        ISBN={978-1-4704-6268-0},
        note={arXiv:1807.03271 [math.CO]},
}

\bib{Roby91ThesisRSK}{thesis}{
      author={Roby, T.},
       title={{Applications and Extensions of Fomin's Generalization of
  Robinson-Schensted Correspondence to Differential Posets}},
        type={Ph.D. Thesis},
     address={MIT},
        date={1991},
}

\bib{Schoenberg1988}{article}{
      author={Schoenberg, I.J.},
       title={{Contributions to the Problem of Approximation of Equidistant
  Data by Analytic Functions. Part A.---On the Problem of Smoothing or
  Graduation. A First Class of Analytic Approximation Formulae}},
        date={1988},
     journal={I.J. Schoenberg Selected Papers},
       pages={3\ndash 57},
}

\bib{schutzenberger1972promotion}{article}{
      author={Sch{\"u}tzenberger, M.-P.},
       title={{Promotion des morphismes d'ensembles ordonn\'es}},
        date={1972},
     journal={Discrete Math.},
      volume={2},
       pages={73\ndash 94},
}

\bib{sokal2014coefficientwise}{misc}{
      author={Sokal, A.~D.},
       title={{Coefficientwise total positivity (via continued fractions) for
  some Hankel matrices of combinatorial polynomials}},
        date={2014},
        note={Talk at the S\'eminaire de Combinatoire Philippe Flajolet,
  Institut Henri Poincar\'e, Paris, 5 June 2014; transparencies available at
  \url{http://semflajolet.math.cnrs.fr/index.php/Main/2013-2014}},
}

\bib{sokal2020euler}{article}{
      author={Sokal, A.~D.},
       title={{The Euler and Springer numbers as moment sequences}},
        date={2020},
     journal={Expo. Math.},
      volume={38},
      number={1},
       pages={1\ndash 26},
        note={arXiv:1804.04498 [math.CO]},
}

\bib{somos2022mathse}{misc}{
      author={Somos, M.},
       title={{Math.StackExchange answer to "Strange polynomial analog of the
  Bell numbers"}},
         how={Mathematics Stack Exchange},
        date={2022},
         url={https://math.stackexchange.com/q/4541543},
        note={\url{https://math.stackexchange.com/q/4541543} (version:
  2022-10-01)},
}

\bib{stanley1988differential}{article}{
      author={Stanley, R.},
       title={Differential posets},
        date={1988},
     journal={Jour. AMS},
      volume={1},
      number={4},
       pages={919\ndash 961},
}

\bib{Stanley1999}{book}{
      author={Stanley, R.},
       title={Enumerative {C}ombinatorics. {V}ol. 2},
   publisher={Cambridge University Press},
     address={Cambridge},
        date={2001},
        note={With a Foreword by Gian-Carlo Rota and Appendix 1 by Sergey
  Fomin},
}

\bib{stanley2009promotion}{article}{
      author={Stanley, R.~P.},
       title={{Promotion and Evacuation}},
        date={2009},
     journal={Electron. J. Comb.},
      volume={16},
      number={2},
        note={arXiv:0806.4717 [math.CO]},
}

\bib{Thoma1964}{article}{
      author={Thoma, E.},
       title={Die unzerlegbaren, positive-definiten {K}lassenfunktionen der
  abz\"ahlbar unendlichen, symmetrischen {G}ruppe},
        date={1964},
     journal={Math. Zeitschr},
      volume={85},
       pages={40\ndash 61},
}

\bib{VK81AsymptoticTheory}{article}{
      author={Vershik, A.},
      author={Kerov, S.},
       title={Asymptotic theory of the characters of the symmetric group},
        date={1981},
     journal={Funktsional. Anal. i Prilozhen.},
      volume={15},
      number={4},
       pages={15\ndash 27, 96},
}

\bib{viennot1983theorie}{book}{
      author={Viennot, X.~G.},
       title={{Une th{\'e}orie combinatoire des polyn{\^o}mes orthogonaux
  g{\'e}n{\'e}raux}},
   publisher={UQAM, Montr{\'e}al},
     address={Montr{\'e}al, Canada},
        date={1983},
        note={Lecture Notes, 219 pages},
}

\bib{wachs1991pq}{article}{
      author={Wachs, M.},
      author={White, D.},
       title={{p, q-Stirling numbers and set partition statistics}},
        date={1991},
     journal={J. Combin. Theory Ser. A},
      volume={56},
      number={1},
       pages={27\ndash 46},
}

\bib{Zeng1993}{article}{
      author={Zeng, J.},
       title={{The q-Stirling numbers, continued fractions and the q-Charlier
  and q-Laguerre polynomials}},
        date={1995},
     journal={J. Comput. Appl. Math.},
      volume={57},
      number={3},
       pages={413\ndash 424},
}

\bib{josuat2011crossings}{article}{
      author={{Josuat-Verg\`es, M. and Rubey, M.}},
       title={{Crossings, Motzkin paths, and moments}},
        date={2011},
     journal={Discret. Math.},
      volume={311},
      number={18-19},
       pages={2064\ndash 2078},
        note={arXiv:1008.3093 [math.CO]},
}
\end{biblist}
\end{bibdiv}
\end{document}